\documentclass[reqno, 12pt]{amsart}

\usepackage{array}
\usepackage{amsmath}
\usepackage{amsfonts}
\usepackage{amssymb}
\usepackage{stmaryrd}
\usepackage{enumerate}
\usepackage{amsthm}
\usepackage{amsmath, amscd}
\usepackage{xy}
\xyoption{all}
\usepackage{supertabular}
\usepackage{marvosym}
\usepackage{wasysym}
\usepackage{hyperref}
\usepackage{tikz}
\usetikzlibrary{matrix,arrows,backgrounds}
	
\newtheorem{theorem}{Theorem}[section]

\newtheorem{lemma}[theorem]{Lemma}
\newtheorem{proposition}[theorem]{Proposition}
\newtheorem{corollary}[theorem]{Corollary}

\theoremstyle{definition}
\newtheorem{definition}[theorem]{Definition}
\newtheorem{definition-lemma}[theorem]{Definition/Lemma}
\newtheorem{remark}[theorem]{Remark}

\newcommand{\op}[1]{\operatorname{#1}}

\newcommand{\newterm}{\textsf}

\newcommand{\dbcoh}[1]{\operatorname{D}^{\operatorname{b}}(\operatorname{coh }#1)}
\newcommand{\dbqcoh}[1]{\operatorname{D}^{\operatorname{b}}(\operatorname{Qcoh }#1)}
\newcommand{\dqcoh}[1]{\operatorname{D}(\operatorname{Qcoh }#1)}

\newcommand{\dbcohG}[2]{\operatorname{D}^{\operatorname{b}}(\operatorname{coh }_{#1} #2)}
\newcommand{\dbqcohG}[2]{\operatorname{D}^{\operatorname{b}}(\operatorname{Qcoh }_{#1} #2)}

\newcommand{\dZbcohG}[3]{\operatorname{D}_{#3}^{\operatorname{b}}(\operatorname{coh }_{#1} #2)}
\newcommand{\dZbqcohG}[3]{\operatorname{D}_{#3}^{\operatorname{b}}(\operatorname{Qcoh }_{#1} #2)}

\newcommand{\dabs}{\op{D}^{\op{abs}}}

\def\Z{\mathbb{Z}}
\def\C{\mathbb{C}}

\def\Q{\mathop{\mathbb{Q}}}

\def\O{\mathcal{O}}

\def\P{\mathbb{P}}

\def\gm{\mathbb{G}_m}

\def\cL{\mathcal{L}}

\def\L{\mathop{\mathcal{L}}}

\def\ok{\otimes_k}

\def\hodgcat{\op{Ho}(\op{dg-cat}_k)}

\title[Kernels for equivariant factorizations and Hodge theory]{A category of kernels for equivariant factorizations and its implications for Hodge theory}

\author[Ballard]{Matthew Ballard}
\address{
  \begin{tabular}{l}
   Matthew Ballard  \\ 
   \hspace{.1in} University of South Carolina, Department of Mathematics, Columbia, SC, USA \\
   \hspace{.1in} Email: {\bf ballard@math.sc.edu} \\
  \end{tabular}
}

\author[Favero]{David Favero}
\address{
  \begin{tabular}{l}
   David Favero \\
   \hspace{.1in} University of Alberta, Department of Mathematics, Edmonton, AB, Canada \\
   \hspace{.1in} Email: {\bf favero@gmail.com} \\
  \end{tabular}
}

\author[Katzarkov]{Ludmil Katzarkov}
\address{
  \begin{tabular}{l}
   Ludmil Katzarkov \\
   \hspace{.1in} University of Miami, Department of Mathematics, Coral Gables, FL, USA \\ 
   \hspace{.1in} Universit\"at von Wien, Fakult\"at f\"ur Mathematik,  Wien, \"Osterreich \\
   \hspace{.1in} Email: {\bf lkatzark@math.uci.edu} \\
  \end{tabular}
}

\numberwithin{equation}{section}
\begin{document}
\renewcommand{\labelenumi}{\emph{\alph{enumi})}}

\begin{abstract}
 We provide a factorization model for the continuous internal Hom, in the homotopy category of $k$-linear dg-categories, between dg-categories of equivariant factorizations. This motivates a notion, similar to that of Kuznetsov, which we call the extended Hochschild cohomology algebra of the category of equivariant factorizations. In some cases of geometric interest, extended Hochschild cohomology contains Hochschild cohomology as a subalgebra and Hochschild homology as a homogeneous component. We use our factorization model for the internal Hom to calculate the extended Hochschild cohomology for equivariant factorizations on affine space.
 
 Combining the computation of extended Hochschild cohomology with the Hochschild-Kostant-Rosenberg isomorphism and a theorem of Orlov recovers and extends Griffiths' classical description of the primitive cohomology of a smooth, complex projective hypersurface in terms of homogeneous pieces of the Jacobian algebra. In the process, the primitive cohomology is identified with the fixed subspace of the cohomological endomorphism associated to an interesting endofunctor of the bounded derived category of coherent sheaves on the hypersurface. We also demonstrate how to understand the whole Jacobian algebra as morphisms between kernels of endofunctors of the derived category. 
 
 Finally, we present a bootstrap method for producing algebraic cycles in categories of equivariant factorizations. As proof of concept, we show how this reproves the Hodge conjecture for all self-products of a particular K3 surface closely related to the Fermat cubic fourfold.
\end{abstract}

\maketitle

\section{Introduction}

The subject of matrix factorizations has, in recent years, found itself at the crossroads between commutative algebra, homological algebra, theoretical physics, and algebraic geometry.  One of the deepest manifestations of this junction is D. Orlov's $\sigma$-model/Landau-Ginzburg correspondence \cite{Orl09} which intimately links projective varieties to equivariant factorization categories.  With Orlov's work as inspiration, this paper provides a thorough investigation of equivariant factorizations in broad generality.  The central technical result is a factorization model for B. T\"oen's internal Hom dg-category \cite{Toe} between these dg-categories. The novelty lies in the range of applications, including those to classical problems in algebraic geometry and Hodge theory.

In this article, we will examine some of the more immediate consequences of the main result, such as some special cases of the Hodge conjecture and a new proof of  Griffith's classical result \cite{Gri} relating the Dolbeault cohomology of a complex projective hypersurface to the Jacobian algebra of its defining polynomial.  In the sequel to this article \cite{BFKthesequel}, we will construct categorical coverings, calculate Rouquier dimension, investigate Orlov spectra, and connect our work to Homological Mirror Symmetry, all as applications of the central theorem presented here.  Now, before we delve into detailed statements, let us try to provide some context for the results.

Perhaps the simplest class of singular rings is that of hypersurface rings, i.e.\ rings which are the quotient of a regular ring by a single element (also called hypersurface singularities). In the foundational paper, \cite{EisMF}, D. Eisenbud introduced matrix factorizations and demonstrated their precise relationship with maximal Cohen-Macaulay (MCM) modules over a hypersurface singularity.  Building on Eisenbud's description, R.-O. Buchweitz introduced the proper categorical framework in \cite{Buc86}. Buchweitz showed that the homotopy category of matrix factorizations, the stable category of MCM modules over the associated hypersurface ring, and the stable derived category of the associated hypersurface ring are all equivalent descriptions of the same triangulated category. 

Outside of commutative algebra, interest in matrix factorizations grew due to intimate connections with physics; physical theories with potentials, called Landau-Ginzburg models, are ubiquitous. Building on the large body of work on Landau-Ginzburg models without boundary, (see, for example, C. Vafa's computation, \cite{Vafa}, of the closed string topological sector as the Jacobian algebra of the potential), M. Kontsevich proposed matrix factorizations as the appropriate category of D-branes for the topological B-model in the presence of a potential \cite[Section 7.1]{KL1}.

In physics, A. Kapustin and Y. Li confirmed Kontsevich's prediction and gave a mathematically conjectural description of the Chern character map and the pairing on Hochschild homology for the category of matrix factorizations, \cite{KL1} \cite{KL2}.

In mathematics, several foundational papers by Orlov soon followed: \cite{Orl04,Orl06,Orl09}. In particular, Orlov gave a global model for the stable bounded derived category of a Noetherian scheme possessing enough locally-free sheaves. He called this the category of singularities. Orlov also proved that the category of B-branes for an LG-model is equivalent to the coproduct of the categories of singularities of the fibers, and, to reiterate, the main inspiration for this work was the tight relationship he provided between the bounded derived categories of coherent sheaves on a projective hypersurface and the equivariant factorization category of affine space together with the defining polynomial.

In another early development, signaling the fertility of the marriage of physical inspiration to matrix factorizations, M. Khovanov and L. Rozansky categorified the HOMFLY polynomial using matrix factorizations, \cite{KR1,KR2}. In the process, Khovanov and Rozansky also introduced several important ideas to the study of matrix factorizations.  Central to their work is a construction which associates functors between categories of matrix factorizations to matrix factorizations of the difference potential. A strong, and precise, analogy exists between Khovanov and Rozansky's construction and the calculus of kernels of integral transforms between derived categories of coherent sheaves on algebraic varieties. Through this analogy, factorizations of the difference potential can be viewed as categorified correspondences for factorization categories.

Numerous further articles have elucidated the connection between factorization categories and Hodge theory. In \cite{KKP}, the third author, Kontsevich, and T. Pantev give explicit constructions describing the Hodge theory associated to the category of matrix factorizations. For the case of an isolated local hypersurface singularity, T. Dyckerhoff proved, in \cite{Dyc}, that the category of kernels introduced in \cite{KR1} is the correct one from the perspective of \cite{Toe}. More precisely, the dg-category of kernels from \cite{KR1} and \cite{Dyc} is quasi-equivalent to the internal homomorphism dg-category in the homotopy category of dg-categories. Using this result, Dyckerhoff rigorously established Kapustin and Li's description of the Hochschild homology of the dg-category of matrix factorizations. D. Murfet gave a mathematical derivation of the Kapustin-Li pairing \cite{Mur} which subsequently was expanded in \cite{DM}. In addition, E. Segal gave a description of the Kapustin-Li package in \cite{Seg1}.

Following this lead, several groups of authors extended Dyckerhoff's results. For a finite group, $G$, A. Polishchuk and A. Vaintrob gave a description of the Chern character, the bulk-boundary map, and proved an analog of Hirzebruch-Riemann-Roch in the case of the $G$-equivariant category of singularities of a local isolated hypersurface ring \cite{PV}. Orlov defined a category of matrix factorizations for a non-affine scheme with a global regular function and proved it is equivalent to the category of singularities of the associated hypersurface in the case when the ambient scheme is regular \cite{OrlMF}. K. Lin and D. Pomerleano also tackled non-affine matrix factorizations \cite{LP}. Contemporaneously, A. Preygel, using genuinely new ideas rooted in derived algebraic geometry, handled matrix factorizations on derived schemes, \cite{Pre}. 

Extending Dyckerhoff's results from the case of a local hypersurface to a global hypersurface, i.e.\ using a section of a line bundle instead of a global regular function, was also vigorously pursued. The first such results were obtained by A. C\u{a}ld\u{a}raru and J. Tu. in \cite{CT}. C\u{a}ld\u{a}raru and Tu defined a curved $A_{\infty}$-algebra associated to a hypersurface in projective space and computed the Borel-Moore homology of the curved algebra. Furthermore, in \cite{Tu}, Tu clarified the relationship between Borel-Moore homology and Hochschild homology. In \cite{PV2}, Polishchuk and Vaintrob gave a definition of a category of matrix factorizations on a stack satisfying appropriate conditions and proved that their category of matrix factorizations coincided with the category of singularities of the underlying hypersurface. In \cite{Pos1}, L. Positselski, using his work on co- and contra-derived categories of curved dg-modules over a curved dg-algebra, defined an enlargement of the category of matrix factorizations in 
the case of a section of line bundle. He also defined in \cite{Pos2}, a relative singularity category for an embedding of $Y$ in $X$ and proved that the relative singularity category of the hypersurface defined by a section of a line bundle coincides with his category of factorizations even if the ambient scheme is not regular. 

Continuing in this direction, this paper completely handles the case of a global hypersurface. Moreover, it also allows for an action of an affine algebraic group. Thus, in particular, it handles factorizations on any smooth algebraic stack with enough locally-free sheaves \cite{Totaro}. The first main result of our paper provides an internal description of the functor category between categories of equivariant factorizations i.e.\ as another category of equivariant factorizations. To state it appropriately, let us recall some work of T\"oen, with the simplifying assumption that $k$ is a field.

In \cite{Toe}, T\"oen studies the structure of the localization of the category of dg-categories over a field, $\op{dg-cat}_k$, by the class of quasi-equivalences. T\"oen calls this localization, the homotopy category of dg-categories, and denotes it as $\op{Ho}(\op{dg-cat}_k)$. For two dg-categories, $\mathsf C$ and $\mathsf D$, T\"oen then defines a dg-category, denoted $\mathbf{R}\! \op{Hom}(\mathsf C,\mathsf D)$, which is the internal Hom dg-category in $\op{Ho}(\op{dg-cat}_k)$. T\"oen defines $\mathbf{R}\! \op{Hom}_c(\mathsf C,\mathsf D)$ to be the full dg-subcategory of $\mathbf{R}\! \op{Hom}(\mathsf C,\mathsf D)$ whose objects induce coproduct-preserving functors between the homotopy categories. He calls such functors continuous. 

The category, $\mathbf{R}\! \op{Hom}_c(\mathsf C,\mathsf D)$, lies at the heart of T\"oen's derived Morita result of \cite{Toe}. Indeed, it seems to be a robust and general prescription for picking out the ``geometrically correct'' functor category for familiar dg/triangulated categories. Let us give attention to an important example: derived categories of sheaves on varieties, $X$ and $Y$. 

An object, $\mathcal K \in \dqcoh{X \times Y}$, gives a coproduct-preserving, exact functor,
\begin{displaymath}
 \mathbf{R}q_*(\mathcal K \overset{\mathbf{L}}{\otimes}_{\mathcal O_{X \times Y}} \mathbf{L}p^* \bullet) : \dqcoh{X} \to \dqcoh{Y},
\end{displaymath}
where $p: X \times Y \to X$ and $q: X \times Y \to Y$ are the projections. However, it is well-known that the category of exact, coproduct-preserving functors from $\dqcoh{X}$ to $\dqcoh{Y}$ is \textit{not} equivalent to $\dqcoh{X \times Y}$, see \cite{CS10} for an example. Passage from the category of chain complexes to triangulated categories is too brutal, we need to remember a bit more information. In \cite{Toe}, To\"en proves that, in $\op{Ho}(\op{dg-cat}_k)$, there is an isomorphism,
\begin{displaymath}
 \mathbf{R}\! \op{Hom}_c(\mathsf{Inj}(X),\mathsf{Inj}(Y)) \cong \mathsf{Inj}(X \times Y)
\end{displaymath}
where $\mathsf{Inj}(Z)$ is a particular dg-enhancement of $\dqcoh{Z}$. Similar work for varieties and other higher objects was carried out in \cite{BFN}.

Hence, the failure of a Morita-type result for derived categories is remedied by lifting to dg-categories and working in $\op{Ho}(\op{dg-cat}_k)$. This makes $\mathbf{R}\! \op{Hom}_c$ the correct functor category to study. However, in general, if two dg-categories, $\mathcal C$ and $\mathcal D$, come from some geometric framework, such as derived categories of sheaves, it is not clear a priori from T\"oen's definition of the internal Hom how $\mathbf{R}\! \op{Hom}_c(\mathcal C,\mathcal D)$ reflects the underlying geometry. One must identify $\mathbf{R}\! \op{Hom}_c(\mathcal C,\mathcal D)$ geometrically. This is the first goal of the paper.
 
Let us define our dg-categories of matrix factorizations. Let $k$ be an algebraically closed field of characteristic zero and let $G$ and $H$ be affine algebraic groups. Let $X$ and $Y$ be smooth varieties. Assume that $G$ acts on $X$ and $H$ acts on $Y$. Let $\mathcal L$ be an invertible $G$-equivariant sheaf on $X$ and let $w \in \op{H}^0(X,\mathcal L)^G$. Similarly, let $\mathcal L'$ be an invertible $H$-equivariant sheaf on $X$ and let $v \in \op{H}^0(Y,\mathcal L')^H$. Let $\mathsf{Inj}(X,G,w)$ and $\mathsf{Inj}(Y,H,v)$ be the dg-categories of equivariant factorizations with injective components. Let $\op{U}(\mathcal L)$ be the geometric vector bundle corresponding to $\mathcal L$ with the zero section removed and denote the regular function induced by $w$ on $\op{U}(\mathcal L)$ by $f_w$. Similarly, let $\op{U}(\mathcal L')$ be the geometric vector bundle corresponding to $\mathcal L'$ with the zero section removed and denote the regular function induced by $v$ on 
$\op{U}(\mathcal L')$ by $f_v$. Equip $\op{U}(\mathcal L) \times \op{U}(\mathcal L')$ with the 
natural $G \times H$-action and allow $\mathbb{G}_m$ to scale the fibers of $\op{U}(\mathcal L) \times \op{U}(\mathcal L')$ diagonally. Let
\begin{displaymath}
 (-f_w) \boxplus f_v := -f_w \ok 1 + 1 \ok f_v.
\end{displaymath}

The following is the main result of Section~\ref{sec: bimod and functor categories for graded MFs}.

\begin{theorem} \label{thm: intro thm}
 In the homotopy category of $k$-linear dg-categories, there is an equivalence,
\begin{displaymath}
 \mathbf{R}\! \op{Hom}_c(\mathsf{Inj}(X,G,w),\mathsf{Inj}(Y,H,v)) \cong \mathsf{Inj}(\op{U}(\mathcal L) \times \op{U}(\mathcal L'),G \times H \times \mathbb{G}_m, (-f_w) \boxplus f_v).
\end{displaymath}
\end{theorem}

This result follows work in the ungraded case by Dyckerhoff, \cite{Dyc}. Our methods in proving Theorem~\ref{thm: intro thm} are in line with \cite{LP} as we rely on generation statements for singularity categories and use Positselski's absolute derived category, \cite{Pos1,Pos2} as the model for our ``large'' triangulated category whose compact objects are (up to summands) coherent factorizations.

In contemporaneous and independent work, \cite{PVnew}, Polishchuk and Vaintrob prove Theorem~\ref{thm: intro thm} in the case $X$ and $Y$ are affine, $G$ and $H$ are finite extensions of $\mathbb{G}_m$, and both $w$ and $v$ have an isolated critical locus. Polishchuk and Vaintrob also give a computation of the Hochschild homology of the category of equivariant matrix factorizations in this case.  Despite the overlap in these foundational results, their inspiration and focus are ultimately distinct from the work here. They provide a purely algebraic version of FJRW-theory \cite{FRJ} by way of matrix factorizations.  The authors find this to be a beautiful illustration of the range and magnitude of this subject of study. 

One significant advantage of a geometric description of the internal Hom category is greater computational power. As defined by T\"oen, Hochschild cohomology of a dg-category is the cohomology of the dg-algebra of endomorphisms of the identity, viewed as an object of the internal Hom dg-category in $\op{Ho}(\op{dg-cat}_k)$. In the setting of $G$-equivariant factorizations, there is a natural extension, which we call extended Hochschild cohomology. For a dg-category, $\mathsf C$, we denote its homotopy category by $[\mathsf C]$. Let $\widehat{G}$ be the group of characters of $G$. The extended Hochschild cohomology is defined as
\begin{displaymath}
 \op{HH}_e^{(\chi,t)}(X,G,w) := \bigoplus_{\chi \in \widehat{G}, t \in \Z} \op{Hom}_{[\mathbf{R}\!\op{Hom}_c(\mathsf{Inj}(X,G,w),\mathsf{Inj}(X,G,w))]}(\op{Id},(\chi)[t]).
\end{displaymath}
Under certain assumptions on $X,G$, and $w$, the Hochschild homology of is a homogeneous component of $\op{HH}_e^{\bullet}(X,G,w)$.

We use Theorem~\ref{thm: intro thm} to compute the extended Hochschild cohomology of $(X,G,w)$ when $X$ is affine, $G$ is a finite extension of $\mathbb{G}_m$, and $w$ is semi-homogeneous regular function of non-torsion degree. The computation is along the lines of \cite{PV}. 

\begin{theorem} \label{theorem: intro thm HH}
 Let $G$ act linearly on $\mathbb{A}^n$ and let $w \in \Gamma(\mathbb{A}^n, \mathcal O_{\mathbb{A}^n}(\chi))^{G}$. Assume that the kernel of $\chi$, $K_{\chi}$, is finite and $\chi: G \to \mathbb{G}_m$ is surjective. Assume that the singular locus of the zero set, $Z_{(-w) \boxplus w}$, is contained in the product of the zero sets, $Z_w \times Z_w$.
 
 Then,
 \begin{gather*}
  \op{HH}^{(\rho,t)}_e(\mathbb{A}^n,G,w) \cong \\ \left( \bigoplus_{\substack{g \in K_{\chi}, l \geq 0 \\ t - \op{rk} W_g = 2u }} \op{H}^{2l}(\op{d} \! w_g)(\rho-\kappa_g+(u-l)\chi) \oplus \bigoplus_{\substack{g \in K_{\chi}, l \geq 0 \\ t - \op{rk} W_g = 2u+1 }} \op{H}^{2l+1}(\op{d} \! w_g)(\rho-\kappa_g+(u-l)\chi) \right)^G
 \end{gather*}
 where $\op{H}^{\bullet}(\op{d} \! w_g)$ denotes the Koszul cohomology of the Jacobian ideal of $w_g:= w|_{(\mathbb{A}^n)^g}$, $W_g$ is the conormal sheaf of $(\mathbb{A}^n)^g$, and $\kappa_g$ is the character of $G$ corresponding to $\Lambda^{\op{rk} W_g} W_g$.
 
 If, additionally, we assume the support of the Jacobian ideal $(\op{d} \! w)$ is $\{0\}$, then we have
 \begin{displaymath}
  \op{HH}^{(\rho,t)}_e(\mathbb{A}^n,G,w) \cong \left( \bigoplus_{\substack{g \in K_{\chi} \\  t - \op{rk} W_g = 2u }} \op{Jac}(w_g)(\rho-\kappa_g+u\chi) \oplus \bigoplus_{\substack{g \in K_{\chi} \\  t- \op{rk} W_g = 2u+1 }} \op{Jac}(w_g)(\rho-\kappa_g+u\chi) \right)^G.
 \end{displaymath}
 where $\op{Jac}(w)$ denotes the Jacobian algebra of $w$.
\end{theorem}

After building these foundations, we apply our results to Hodge theory. The primary observation is that Orlov's relationship between graded categories of singularities and derived categories of coherent sheaves \cite{Orl09} has some very interesting geometric consequences when combined with Theorem~\ref{thm: intro thm}. 

Let $\mathsf C$ be a saturated dg-category over $k$. The Hochschild homology of $\mathsf C$, $\op{HH}_*(\mathsf C)$, is an invariant that plays an important role in the noncommutative Hodge theory of $\mathsf C$, \cite{KKP}. When $X$ is a smooth proper algebraic variety over $k$, one can let $\mathsf C = \mathsf{Inj}_{\op{coh}}(X)$ be the dg-category of bounded below complexes of injective $\O_X$-modules with bounded and coherent cohomology. There is a Hochschild-Kostant-Rosenberg isomorphism, see \cite{HKR,Swan2,Kon}
\begin{displaymath}
 \phi_{\op{HKR}}: \op{HH}_i(\mathsf{Inj}_{\op{coh}}(X)) =: \op{HH}_i(X) \to \bigoplus_{q-p=i}\op{H}^p(X,\Omega^q_X).
\end{displaymath}
The HKR isomorphism allows one to study questions of Hodge theory by means of category theory. In Section~\ref{section: Griffiths}, we combine Orlov's theorem, the HKR isomorphism, and the computations of Theorem~\ref{theorem: intro thm HH} to reproduce a classic result of Griffiths \cite{Gri} describing the primitive cohomology of a projective hypersurface.

\begin{theorem} \label{theorem: Griffiths from Orlov}
 Let $Z$ be a smooth, complex projective hypersurface defined by a homogeneous polynomial $w \in \C[x_1,\ldots,x_n]$ of degree $d$. For each $0 \leq p \leq n/2-1$, Orlov's theorem and the HKR isomorphism induce an isomorphism,
 \begin{displaymath}
  \op{H}^{p,n-2-p}_{\op{prim}}(Z) \cong \op{Jac}(w)_{d(n-1-p)-n}.
 \end{displaymath}
\end{theorem}

In the process, we show that the primitive cohomology of $Z$ is exactly the fixed locus of the action of the endofunctor
\begin{align*}
 \{1\} := L_{\mathcal O_Z} \circ T_{\mathcal O_Z(1)} : \dbcoh{Z} & \to \dbcoh{Z} \\
 \mathcal E & \mapsto \op{Cone}\left( \oplus_{i \in \Z} \op{Hom}_{\dbcoh{Z}}(\mathcal O_Z,\mathcal E(i)[i]) \otimes_k \mathcal O_Z[-i] \overset{ev}{\to} \mathcal E(1)\right)
\end{align*}
on Hochschild homology, $\op{HH}_{\bullet}(Z)$. Furthermore, when $Z$ is Calabi-Yau, for the kernel, $\mathcal K \in \dbcoh{Z \times Z}$, of $\{1\}$, we have an injective homomorphism of graded rings, 
\begin{displaymath}
 \op{Jac}(w) \to \bigoplus_{i \geq 0} \op{Hom}_{\dbcoh{Z \times Z}}(\Delta_*\mathcal O_Z, \mathcal K^{\ast i})
\end{displaymath}
whose appropriate graded pieces are the isomorphisms of Theorem~\ref{theorem: Griffiths from Orlov}, at least after application of the HKR isomorphism. Thus, we have a categorical realization of Griffiths' fundamental result that sees the entire Jacobian algebra.

Following this categorical path further, we study algebraic cycles by understanding the image of the Chern character map in Hochschild homology. In Section~\ref{section: cycles}, we prove a result that allows one to bootstrap, via group homomorphisms, the Hodge conjecture for categories of equivariant matrix factorizations. We give one application of this procedure to the Hodge conjecture for varieties: we apply the results of Orlov in \cite{Orl09} and work of Kuznetsov \cite{Kuz09a, Kuz09b} to reprove the Hodge conjecture for $n$-fold products a certain $K3$ surface associated to a Fermat cubic fourfold. This case of the Hodge conjecture was originally handled in \cite{RM08}. We thank P. Stellari for pointing out the reference, \cite{RM08}.

\vspace{2.5mm}
\noindent \textbf{Acknowledgments:}
 The authors are greatly appreciative of the valuable insight gained from conversations and correspondence with Mohammed Abouzaid, Denis Auroux, Andrei C\u{a}ld\u{a}raru, Dragos Deliu, Colin Diemer, Tobias Dyckerhoff, Manfred Herbst, M. Umut Isik, Gabriel Kerr, Maxim Kontsevich, Alexander Kuznetsov, Jacob Lewis, Dmitri Orlov, Pranav Pandit, Tony Pantev, Anatoly Preygel, Victor Przyjalkowski, Ed Segal, Paul Seidel, and Paolo Stellari and would like to thank them for their time and patience.  Furthermore, the authors are deeply grateful to Alexander Polishchuk and Arkady Vaintrob for providing us with a preliminary version of their work \cite{PVnew} and for allowing us time to prepare the original version of this paper in order to synchronize posting of the articles due to the overlap. The first named author was funded by NSF DMS 0636606 RTG, NSF DMS 0838210 RTG, and NSF DMS 0854977 FRG. The second and third named authors were funded by NSF DMS 0854977 FRG, NSF DMS 0600800, NSF DMS 0652633 FRG, NSF DMS 0854977, NSF DMS 0901330, FWF P 24572 N25, by FWF P20778 and by an ERC Grant.
\vspace{2.5mm}

\section{Background on equivariant sheaves} \label{sec:graded rings}

For the entirety of this paper, $k$ will denote an algebraically-closed field of characteristic zero.

In this section, we recall some facts about quasi-coherent equivariant sheaves on separated, schemes/algebraic spaces of finite type following \cite{MFK}. A nice reference for basic facts, with a full set of details, is \cite[Chapter 3]{Blume}. The results here will be used in later sections. Let $X$ be a separated scheme of finite type over $k$ and $G$ be an affine algebraic group over $k$ acting on $X$. Denote by $m: G \times G \to G$, $i: G \to G$, and $e: \op{Spec} k \to G$, the group action, the inversion and the identity, respectively. Let $\sigma: G \times X \to X$ denote the $G$-action and $\pi: G \times X \to X$ the projection onto $X$. 

\begin{definition}
 A quasi-coherent \newterm{$G$-equivariant sheaf} on $X$ is a quasi-coherent sheaf, $\mathcal F$, on $X$ together with an isomorphism, $\theta: \sigma^* \mathcal F \to \pi^* \mathcal F$, satisfying, 
\begin{displaymath}
 \left((1_G \times \sigma) \circ (\tau \times 1_X \right))^*\theta \circ \left(1_G \times \pi\right)^*\theta = \left(m \times 1_X\right)^*\theta,
\end{displaymath}
on $G \times G \times X$ where $\tau: G \times G \times X \to G \times G \times X$ switches the two factors of $G$, and,
\begin{displaymath}
 s^{*} \theta = 1_{\mathcal F},
\end{displaymath}
 where $s: X \to G \times X$ is induced by $e$. If $\mathcal F$ is a coherent, respectively locally-free, sheaf on $X$, then we say the equivariant sheaf, $(\mathcal F, \theta)$, is coherent, respectively locally-free. 
 
 The isomorphism, $\theta$, is called the \newterm{equivariant structure}. We often refer to a quasi-coherent $G$-equivariant sheaf simply as $\mathcal E$. If the context is ambiguous, we denote the equivariant structure of $\mathcal E$ by $\theta^{\mathcal E}$.
\end{definition}

\begin{remark}
 For each closed point $g \in G$, we get an automorphism
 \begin{displaymath}
  \sigma_g := \sigma(g,\bullet) : X \to X.
 \end{displaymath}
 These satisfy $\sigma_{g_1} \circ \sigma_{g_2} = \sigma_{g_1g_2}$. If $\mathcal E$ is a quasi-coherent $G$-equivariant sheaf, then $\theta$ gives isomorphisms
 \begin{displaymath}
  \theta_g := \theta|_{\{g\} \times X} : \sigma_g^* \mathcal E \to \mathcal E.
 \end{displaymath}
 for each $g \in G$ with $\theta_{g_2g_1} = \theta_{g_1} \circ \sigma_{g_1}^*\theta_{g_2}$. Checking a subsheaf $\mathcal F$ of $\mathcal E$ inherits the equivariant structure, i.e. $\theta(\sigma^*\mathcal F) \subseteq \pi^*\mathcal F$, boils down to checking that it is preserved by each $\theta_g$. 
\end{remark}

\begin{definition}
 Let $\op{Qcoh}_G X$ be the Abelian category of quasi-coherent $G$-equivariant sheaves on $X$. Analogously, we let $\op{coh}_G X$ be the Abelian category of coherent $G$-equivariant sheaves.
\end{definition}

\begin{definition}
 Let $\mathcal E$ and $\mathcal F$ be quasi-coherent $G$-equivariant sheaves on $X$. The \newterm{tensor product} of $\mathcal E$ and $\mathcal F$ is the quasi-coherent sheaf $\mathcal E \otimes_{\mathcal O_X} \mathcal F$ together with the equivariant structure, $\theta^{\mathcal E} \otimes_{\mathcal O_{G \times X}} \theta^{\mathcal F}$. 
 
 The \newterm{sheaf of homomorphisms} from $\mathcal E$ to $\mathcal F$ is the quasi-coherent sheaf $\mathcal Hom_X(\mathcal E, \mathcal F)$ together with the equivariant structure $\theta^{\mathcal F} \circ (\bullet) \circ (\theta^{\mathcal E})^{-1}$. 
\end{definition}

\begin{definition}
 Let $X$ and $Y$ be separated, finite-type schemes equipped with actions, $\sigma_X$ and $\sigma_Y$, of $G$ and projections $\pi_X,\pi_Y$. A morphism of schemes, $f: X \to Y$, is \newterm{$G$-equivariant} if the diagram
 \begin{center}
 \begin{tikzpicture}[description/.style={fill=white,inner sep=2pt}]
  \matrix (m) [matrix of math nodes, row sep=3em, column sep=3em, text height=1.5ex, text depth=0.25ex]
  { G \times X & G \times Y \\ 
    X & Y \\ };
  \path[->,font=\scriptsize]
  (m-1-1) edge node[above] {$1 \times f$} (m-1-2)
  (m-1-1) edge node[left] {$\sigma_X$} (m-2-1)
  (m-1-2) edge node[right] {$\sigma_Y$} (m-2-2)
  (m-2-1) edge node[above] {$f$} (m-2-2)
  ;
 \end{tikzpicture}
 \end{center} 
 commutes. Given such an $f$, we get an adjoint pair of functors,
 \begin{align*}
  f^* : \op{Qcoh}_G Y & \to \op{Qcoh}_G X \\
  (\mathcal F, \theta) & \mapsto (f^*\mathcal F, (1 \times f)^*\theta), \\
  f_* : \op{Qcoh}_G X & \to \op{Qcoh}_G Y \\
  (\mathcal F, \theta) & \mapsto (f_*\mathcal F, (1 \times f)_*\theta).
 \end{align*}
\end{definition}

\begin{remark}
 The definition of $f_*$ and $f^*$ are sensible (as interpreted through natural isomorphisms) as $\sigma_X, \pi_X$ are flat and the squares
 \begin{center}
 \begin{tikzpicture}[description/.style={fill=white,inner sep=2pt}]
  \matrix (m) [matrix of math nodes, row sep=3em, column sep=3em, text height=1.5ex, text depth=0.25ex]
  { G \times X & G \times Y \\ 
    X & Y \\ };
  \path[->,font=\scriptsize]
  (m-1-1) edge node[above] {$1 \times f$} (m-1-2)
  (m-1-1) edge node[left] {$\sigma_X$} (m-2-1)
  (m-1-2) edge node[right] {$\sigma_Y$} (m-2-2)
  (m-2-1) edge node[above] {$f$} (m-2-2)
  ;
 \end{tikzpicture}
  \begin{tikzpicture}
    \matrix (m) [matrix of math nodes, row sep=3em, column sep=3em, text height=1.5ex, text depth=0.25ex]
  { G \times X & G \times Y \\ 
    X & Y \\ };
  \path[->,font=\scriptsize]
  (m-1-1) edge node[above] {$1 \times f$} (m-1-2)
  (m-1-1) edge node[left] {$\pi_X$} (m-2-1)
  (m-1-2) edge node[right] {$\pi_Y$} (m-2-2)
  (m-2-1) edge node[above] {$f$} (m-2-2)
  ;
 \end{tikzpicture}

 \end{center} 
 are Cartesian.
\end{remark}

\begin{definition}
 Given an affine algebraic group, $G$, we let 
 \begin{displaymath}
  \widehat{G} := \op{Hom}_{\op{alg \ grp}}(G,\mathbb{G}_m).
 \end{displaymath}
 The finitely-generated Abelian group, $\widehat{G}$, is called the \newterm{group of characters} of $G$. As $\widehat{G}$ is Abelian, we shall use additive notation for group structure on $\widehat{G}$.

 For a character, $\chi \in \widehat{G}$, we let $K_{\chi}$ denote the kernel of $\chi$. We also get an auto-equivalence
 \begin{align*}
  (\chi) : \op{Qcoh}_G X & \to \op{Qcoh}_G X \\
  \mathcal E & \mapsto \mathcal E \otimes_{\mathcal O_X} p^*\mathcal L_{\chi}
 \end{align*}
 where $p: X \to \op{Spec} k$ is the structure map and $\mathcal L_{\chi}$ is the object of $\op{Qcoh}_G (\op{Spec} k)$ corresponding to $\chi$. 
\end{definition}

\begin{lemma} \label{lemma: projection formula for equivariant pushforward}
 Let $G$ act on $X$ and $Y$. Assume we have an equivariant morphism, $f: X \to Y$. For $\mathcal E \in \op{Qcoh}_G Y$ locally-free and $\mathcal F \in \op{Qcoh}_G X$, there is a natural isomorphism
 \begin{displaymath}
  f_* \mathcal F \otimes_{\mathcal O_X} \mathcal E \cong f_* ( \mathcal F \otimes_{\mathcal O_X} f^* \mathcal E).
 \end{displaymath}
\end{lemma}

\begin{proof}
 This follows from the usual projection formula applied both to $\mathcal E$ and $\theta$.
\end{proof}

We will also need a more general pull-back functor. 

\begin{definition}
 Let $H$ and $G$ be affine algebraic groups and let $X$ and $Y$ be separated schemes of finite type equipped with actions, $\sigma_{H,X}: H \times X \to X$ and $\sigma_{G,Y}: G \times Y \to Y$. Let $\psi: H \to G$ be a homomorphism of algebraic groups. A \newterm{$\psi$-equivariant morphism}, or a \newterm{morphism equivariant with respect to} $\psi$, is a morphism of schemes, $f: X \to Y$, such that diagram
 \begin{center}
 \begin{tikzpicture}[description/.style={fill=white,inner sep=2pt}]
  \matrix (m) [matrix of math nodes, row sep=3em, column sep=3em, text height=1.5ex, text depth=0.25ex]
  { H \times X & G \times Y \\ 
    X & Y \\ };
  \path[->,font=\scriptsize]
  (m-1-1) edge node[above] {$\psi \times f$} (m-1-2)
  (m-1-1) edge node[left] {$\sigma_{H,X}$} (m-2-1)
  (m-1-2) edge node[right] {$\sigma_{G,Y}$} (m-2-2)
  (m-2-1) edge node[above] {$f$} (m-2-2)
  ;
 \end{tikzpicture}
 \end{center} 
 commutes. Given a $\psi$-equivariant morphism, $f$, we can define the pull-back functor,
 \begin{align*}
  f^*: \op{Qcoh}_G Y & \to \op{Qcoh}_H X \\
  (\mathcal F, \theta) & \mapsto (f^*\mathcal F, (\psi \times f)^* \theta).
 \end{align*}
 In the case that $X = Y$, we denote this functor by $\op{Res}_{\psi}$. If, in addition, $\psi: H \to G$ is a closed subgroup, the pull-back is called the \newterm{restriction functor} and denoted by $\op{Res}^G_H$. 
\end{definition}

\begin{remark}
While there is a bit of notational conflict here, we will always try to eliminate this confusion with exposition.
\end{remark}

\begin{definition}
 Let $G$ and $H$ be affine algebraic groups, $X$ and $Y$ separated schemes of finite type equipped with actions $G \times X \to X$ and $H \times Y \to Y$. Let $\pi_1: X \times Y \to X$ and $\pi_2 : X \times Y \to Y$ be the two projections. The projection, $\pi_1$, is equivariant with respect to the projection $G \times H \to G$ while $\pi_2$ is equivariant with respect to the projection $G \times H \to H$. Let $\mathcal E \in \op{Qcoh}_G X$ and $\mathcal F \in \op{Qcoh}_H Y$. The \newterm{exterior product} of $\mathcal E$ and $\mathcal F$ is the quasi-coherent $G \times H$-equivariant sheaf
 \begin{displaymath}
  \mathcal E \boxtimes \mathcal F := \pi^*_1 \mathcal E \otimes_{\mathcal O_{X \times Y}} \pi^*_2 \mathcal F.
 \end{displaymath}
\end{definition}

Let $H$ be a closed subgroup of $G$ and let $\sigma: H \times X \to X$ be an action of $G$ on $X$. The product, $G \times X$, carries an action of $H$ defined by
\begin{align*}
 \tau: H \times G \times X & \to G \times X \\
 (h,g,x) & \mapsto (m(g,i(h)),\sigma(h,x)). 
\end{align*}

\begin{lemma} 
 The fppf quotient of $G \times X$ by $H$ exists as a separated algebraic space of finite type over $k$. It is denoted by $G \overset{H}{\times} X$. 
\end{lemma}

\begin{proof}
 By Artin's Theorem, see \cite[Theorem 3.1.1]{Ana}, $G \overset{H}{\times} X$ exists as a separated algebraic space of finite type.  
\end{proof}

Let $\iota: X \to G \overset{H}{\times} X$ be the inclusion, $x \mapsto (e,x)$. This is equivariant with respect to the inclusion of $H$ in $G$. 

\begin{lemma} \label{lem: equivalence pullback}
 The pull-back functor, $\iota^*: \op{Qcoh}_G (G \overset{H}{\times} X) \to \op{Qcoh}_H X$, is an equivalence. Moreover, it induces an equivalence between the subcategories of coherent equivariant sheaves and an equivalence between the subcategories of locally-free equivariant sheaves. 
\end{lemma}

\begin{proof}
 This is an immediate consequence of faithfully-flat descent, see \cite[Lemma 1.3]{Tho2}.
\end{proof}

\begin{definition}
 Let $H$ be a closed subgroup of $G$ and assume we have an action, $\sigma: G \times X \to X$. The action, $\sigma$, descends to a $G$-equivariant morphism, $\alpha: G \overset{H}{\times} X \to X$. The \newterm{induction functor},
 \begin{displaymath}
  \op{Ind}^G_H: \op{Qcoh}_H X \to \op{Qcoh}_G X 
 \end{displaymath}
 is defined to be the composition, $\alpha_* \circ (\iota^*)^{-1}$. 
\end{definition}

\begin{lemma} \label{lemma: adjointness restriction induction}
 Let $H$ be a closed subgroup of $G$ and assume we have an action, $\sigma: G \times X \to X$. The functor, $\op{Ind}^G_H$, is right adjoint to the restriction, $\op{Res}^G_H$, and
 \[
 \op{Res}^G_H \cong \iota^* \circ \alpha^*.
 \]
\end{lemma}

\begin{proof}
 Note that the identity map on $X$ can be factored as
 \begin{displaymath}
 X \overset{\iota}{\to} G \overset{H}{\times} X \overset{\alpha}{\to} X.
 \end{displaymath}
 Thus, $\op{Res}^G_H = \iota^* \circ \alpha^*$ which is left adjoint to $\alpha_* \circ (\iota^*)^{-1}$.
\end{proof}

\begin{lemma} \label{lemma: facts about restriction induction}
 Let $H$ be a closed subgroup of $G$ and let $X$ be a separated scheme of finite type equipped with an action, $\sigma: G \times X \to X$. Let $p: G/H \times X \to X$ be the projection onto $X$.
 \begin{enumerate}
  \item The $H$-crossed product, $G \overset{H}{\times} X$, is a scheme, $G$-equivariantly isomorphic to $G/H \times X$, with the diagonal $G$-action.
  \item The functor, $\op{Res}^G_H$, is exact.
  \item For $\mathcal E \in \op{Qcoh}_H X$ and $\mathcal F \in \op{Qcoh}_G X$ locally-free, there is the following projection formula, i.e.\ a natural isomorphism,
  \begin{displaymath}
   \op{Ind}^G_H (\mathcal E \otimes_{\mathcal O_X} \op{Res}^G_H \mathcal F) \cong \op{Ind}^G_H \mathcal E \otimes_{\mathcal O_X} \mathcal F.
  \end{displaymath}
  \item There is a natural isomorphism
   \begin{displaymath}
    \op{Ind}^G_H \circ \op{Res}^G_H \cong p_*p^*
   \end{displaymath}
   of functors.
 \item If we, additionally, assume that $G/H$ is affine, then $\op{Ind}^G_H$ is exact. In particular, if $H$ is normal, then $\op{Ind}^G_H$ is exact. 
 \end{enumerate}
\end{lemma}

\begin{proof}
For a), as we are over $k$, the quotient of $G$ by $H$, as a fppf sheaf, exists as a quasi-projective scheme. By \cite[Theorem 16.1]{Waterhouse}, one can find a $G$-representation, $V$, with a subspace, $W$, whose stabilizer is exactly $H$. Let $n= \op{dim} W$. Passing to the Grassmannian, $G(n,V)$, $H$ is exhibited as the stabilizer of a closed point and by \cite[III,\textsection 3, Proposition 5.2]{DG} is representable by scheme with a locally-closed embedding into $G(n,V)$. 
 Now, the $H$-crossed product, $G \overset{H}{\times} X$, is $G$-equivariantly isomorphic to the product, $G/H \times X$, with the diagonal $G$-action, via the isomorphism
 \begin{align*}
  \Phi: G \overset{H}{\times} X & \to G/H \times X \\
  (g,x) & \mapsto (gH,gx).
 \end{align*}
 For $\alpha: G \overset{H}{\times} X \to X$, we have $\alpha = p \circ \Phi$. 

For b), recall that $\op{Res}^G_H \cong \iota^* \circ \alpha^*$. Then, 
 \begin{displaymath}
  \op{Res}^G_H \cong \iota^* \circ \Phi^* \circ p^*.
 \end{displaymath}
 Both $\iota^*$ and $\Phi^*$ are equivalences so both are exact while $p^*$ is exact as $G/H$ is flat over $k$. 
 
 For c), let $\mathcal E \in \op{Qcoh}_H X$ and $\mathcal F \in \op{Qcoh}_G X$ with $\mathcal F$ locally-free. Since $\iota^*$ is an equivalence, we can write $\mathcal E = \iota^* \mathcal E'$ for $\mathcal E' \in \op{Qcoh}_G G \overset{H}{\times} X$,
 \begin{align*}
  \op{Ind}^G_H ( \mathcal E \otimes_{\mathcal O_X} \op{Res}^G_H \mathcal F ) & \cong \alpha_* (\iota^*)^{-1} ( \mathcal E \otimes_{\mathcal O_X} \iota^* \alpha^* \mathcal F ) \\
  & \cong \alpha_* (\iota^*)^{-1} ( \iota^* \mathcal E' \otimes_{\mathcal O_X} \iota^* \alpha^* \mathcal F ) \\
  & \cong \alpha_* (\mathcal E' \otimes_{\mathcal O_{G \overset{H}{\times} X}} \alpha^* \mathcal F ) \\
  & \cong \alpha_* \mathcal E' \otimes_{\mathcal O_X} \mathcal F \\
  & \cong \op{Ind}^G_H \mathcal E \otimes_{\mathcal O_X} \mathcal F
 \end{align*}
 where we used the projection formula for $\alpha$ and the fact the $\iota^*$ is a monoidal functor.

 For d), we have isomorphisms
 \begin{align*}
  \op{Ind}^G_H \circ \op{Res}^G_H & \cong \alpha_* \circ (\iota^*)^{-1} \circ \iota^* \circ \alpha^* \\
  & \cong \alpha_* \circ \alpha^* \\
  & \cong p_* \circ \Phi_* \circ \Phi^* \circ p^* \\
  & \cong p_* \circ p^*.
 \end{align*}
 We used the fact that $\Phi_* \cong (\Phi^*)^{-1}$ as $\Phi$ is an isomorphism.

For e), the map $p$ is affine so $p_*$ is exact. Consequently, $\op{Ind}^G_H \cong \alpha_* \circ (\iota^*)^{-1} \cong p_* \circ \Phi_* \circ (\iota^*)^{-1}$ is a composition of exact functors. If $H$ is normal, then $G/H$ is an affine algebraic group, \cite[Theorem 16.3]{Waterhouse}.
 
\end{proof}

\begin{remark}
 Notice that when $H$ is not normal we may only consider $G/H$ as a scheme with an action of $G$ and not as an affine algebraic group. Furthermore, $G/H$ possesses a transitive $G$-action and, since the base field has characteristic zero, is generically smooth. Consequently, $G/H$ is a smooth variety. 
\end{remark}

\begin{lemma} \label{lemma: Res-Ind Abelian quotient}
 Let  $H$ be a closed normal subgroup of $G$. Assume that $G/H$ is Abelian. Then, there is a natural isomorphism
 \begin{displaymath}
  \op{Ind}^G_H \circ \op{Res}^G_H \mathcal E \cong \bigoplus_{\chi \in \widehat{G/H}} \mathcal E (\chi)
 \end{displaymath}
 where we view $\chi$ as a character of $G$ via the homomorphism $G \to G/H$. 
\end{lemma}

\begin{proof}
 By Lemma~\ref{lemma: facts about restriction induction}, we have an isomorphism
 \begin{displaymath}
  \op{Ind}^G_H \circ \op{Res}^G_H \cong p_*p^*
 \end{displaymath}
 where $p: G/H \times X \to X$ is the projection. Thus,
 \begin{displaymath}
  \op{Ind}^G_H \circ \op{Res}^G_H \mathcal E \cong p_*p^* \mathcal E \cong  \Gamma(G/H,\mathcal O_{G/H}) \otimes_k \mathcal E.
 \end{displaymath}
 Since $G/H$ is Abelian, $\Gamma(G/H,\mathcal O_{G/H}) \cong k[\widehat{G/H}]$ and 
 \begin{displaymath}
  \Gamma(G/H,\mathcal O_{G/H}) \otimes_k \mathcal E \cong \bigoplus_{\chi \in \widehat{G/H}} \mathcal E (\chi).
 \end{displaymath}
\end{proof}

\begin{lemma} \label{lemma: flat base change for equivariant sheaves}
 Let $\psi: G \to H$ be a flat homomorphism of affine algebraic groups. Let $G$ act on the algebraic varieties $Z$ and $X$ and $H$ act on the algebraic varieties $Y$ and $W$.  Assume we have a Cartesian square 
 \begin{center}
 \begin{tikzpicture}[description/.style={fill=white,inner sep=2pt}]
  \matrix (m) [matrix of math nodes, row sep=2em, column sep=2em, text height=1.5ex, text depth=0.25ex]
  { Z  &  Y  \\ 
    X  &  W \\ };
  \path[->,font=\scriptsize]
  (m-1-1) edge node[above] {$u'$} (m-1-2)
  (m-1-1) edge node[left] {$v'$} (m-2-1)
  (m-1-2) edge node[right] {$v$} (m-2-2)
  (m-2-1) edge node[above] {$u$} (m-2-2)
  ;
 \end{tikzpicture}
 \end{center}
where  $u'$ and $u$ are $\psi$-equivariant while $v'$ is $G$-equivariant and $v$ is $H$-equivariant.  Assume that $u$ is flat.  Then, we have a natural isomorphism of functors
 \begin{displaymath}
  u^* \circ v_* \cong v'_* \circ u'^* : \op{Qcoh}_H Y \to \op{Qcoh}_G X.
 \end{displaymath}
\end{lemma}

\begin{proof}
 For a $H$-equivariant quasi-coherent sheaf, $(\mathcal E, \theta)$, we have $u^* v_* \mathcal E \cong v'_* u'^* \mathcal E$ via flat base change. We also have a Cartesian diagram
 \begin{center}
 \begin{tikzpicture}[description/.style={fill=white,inner sep=2pt}]
  \matrix (m) [matrix of math nodes, row sep=2em, column sep=2em, text height=1.5ex, text depth=0.25ex]
  { G \times Z  &  H \times Y  \\ 
    G \times X  &  H \times W \\ };
  \path[->,font=\scriptsize]
  (m-1-1) edge node[above] {$\psi \times u'$} (m-1-2)
  (m-1-1) edge node[left] {$1_G \times v'$} (m-2-1)
  (m-1-2) edge node[right] {$1_H \times v$} (m-2-2)
  (m-2-1) edge node[above] {$\psi \times u$} (m-2-2)
  ;
 \end{tikzpicture}
 \end{center}
 and $\psi \times u$ is flat. So $(\psi \times u)^* (1 \times v)_* \cong (1 \times v')_* (\psi \times u)^*$ via flat base change, again. Using this fact on $\theta$, we get an equivariant isomorphism between $u^* v_* \mathcal E$ and $v'_* u'^* \mathcal E$. 
\end{proof}

\begin{definition}
 Let $X$ be a separated scheme of finite type over $k$. Let $\sigma: G \times X \to X$ act on $X$ and $N$ be a closed normal subgroup of $G$ such that $\sigma|_{N \times X} : N \times X \to X$ is the trivial action. Consider a quasi-coherent $G$-equivariant sheaf $(\mathcal E, \theta)$ and the restriction of $\theta$ to $N$
 \begin{displaymath}
  \theta|_{N \times X} : \sigma|_{N \times X}^* \mathcal E \cong \pi^* \mathcal E \to \pi^* \mathcal E.
 \end{displaymath}
 Via adjunction, we have a morphism,
 \begin{displaymath}
  \mathcal E \overset{u_{\mathcal E}}{\to} \Gamma(N, \mathcal O_N) \otimes_k \mathcal E \overset{1_{\Gamma(N, \mathcal O_N) \otimes_k \mathcal E} -\pi_* \theta|_{N \times X}}{\to} \Gamma(N, \mathcal O_N) \otimes_k \mathcal E.
 \end{displaymath}
 where $u: \op{Id} \to \pi_* \pi^*$ is the unit of adjunction. Let $\mathcal E^N$ be the kernel of this total morphism. Then, $\theta$ preserves $\mathcal E^N$ and the pair $(\mathcal E^N, \theta|_{\sigma^* \mathcal E^N})$ is a $G$-equivariant sheaf that naturally descends to a quasi-coherent $G/N$-equivariant sheaf on $X$. Denote the functor by 
 \begin{displaymath}
  (\bullet)^N : \op{Qcoh}_G X \to \op{Qcoh}_{G/N} X. 
 \end{displaymath}
 We shall often, interchangeably, view $\mathcal E^N$ as a $G$-equivariant sheaf or as a $G/N$-equivariant sheaf without additional notational adornment.
\end{definition}

\begin{remark} \label{rem: invariants}
 The local sections of the sheaf $\mathcal F^N$ over an open subset $U \subseteq X$ are 
 \begin{displaymath}
  \mathcal F^N(U) = \{ f \in \mathcal F(U) \mid \theta_n^{\mathcal F}(f) = f, \forall n \in N \}.
 \end{displaymath}
 In fact, this description can be taken as a definition of $\mathcal F^N$.
\end{remark}

\begin{lemma} \label{lemma: left adjoint to invariants}
 The functor $(\bullet)^N$ is right adjoint to $\op{Res}_{\pi}$ for the quotient homomorphism, $\pi: G \to G/N$.
\end{lemma}

\begin{proof}
 Let $\phi: \op{Res}_{\pi} \mathcal E \to \mathcal F$ be a $G$-equivariant morphism. Since $\theta_n^{\op{Res}_{\pi} \mathcal E} = \theta_{\pi(n)}^{\mathcal E} = 1_{\mathcal E}$, $N$ acts trivially on $\op{Res}_{\pi}$. As $\phi$ is $G$-equivariant, we have
 \begin{displaymath}
  \theta_n^{\mathcal F} \circ \phi = \phi \circ \theta_n^{\op{Res}_{\pi} \mathcal E} = \phi
 \end{displaymath}
 for all $n \in N$, and the image of $\op{Res}_{\pi} \mathcal E$ under $\phi$ must lie in $\mathcal F^N$. So any $G$-equivariant morphism from $\op{Res}_{\pi} \mathcal E$ factors through $\mathcal F^N$ uniquely. Of course, any $G$-equivariant morphism, $\op{Res}_{\pi} \mathcal E \to \mathcal F^N$, induces a $G$-equivariant morphism, $\op{Res}_{\pi} \mathcal E \to \mathcal F$, via composition with the inclusion, $\mathcal F^N \to \mathcal F$. Hence, we have an isomorphism
 \begin{displaymath}
  \op{Hom}_{\op{Qcoh}_G X}(\op{Res}_{\pi} \mathcal E, \mathcal F) \cong \op{Hom}_{\op{Qcoh}_G X}(\op{Res}_{\pi} \mathcal E, \mathcal F^N).
 \end{displaymath}
 As both $\op{Res}_{\pi} \mathcal E$ and $\mathcal F^N$ are $N$-invariant, any $G$-equivariant morphism, $\op{Res}_{\pi} \mathcal E \to \mathcal F^N$, uniquely descends to a  $G/N$-equivariant morphism. So,
 \begin{displaymath}
  \op{Hom}_{\op{Qcoh}_G X}(\op{Res}_{\pi} \mathcal E, \mathcal F^N) \cong \op{Hom}_{\op{Qcoh}_{G/N} X}(\mathcal E, \mathcal F^N).
 \end{displaymath}
\end{proof}

\begin{lemma} \label{lemma: projection formula for invariants}
 For any $\mathcal F_1 \in \op{Qcoh}_{G/N} X$ and $\mathcal F_2 \in \op{Qcoh}_G X$, there is a natural isomorphism of $G$-equivariant sheaves 
 \begin{displaymath}
  (\op{Res}_{\pi} \mathcal F_1 \otimes_{\mathcal O_X} \mathcal F_2)^N \cong \mathcal F_1 \otimes_{\mathcal O_X} \mathcal F_2^N.
 \end{displaymath}
\end{lemma}

\begin{proof}
 Since $\op{Res}_{\pi} \mathcal F_1$ is completely $N$-invariant, we have an isomorphism 
 \begin{displaymath}
  \theta_n^{\op{Res}_{\pi} \mathcal F_1 \otimes_{\mathcal O_X} \mathcal F_2} : = \theta_n^{\op{Res}_{\pi} \mathcal F_1} \otimes_{\mathcal O_X} \theta_n^{\mathcal F_2} \cong 1_{\mathcal F_1} \otimes \theta_n^{\mathcal F_2}.
 \end{displaymath}
 for all $n \in N$. Thus, $\theta_n^{\op{Res}_{\pi} \mathcal F_1 \otimes_{\mathcal O_X} \mathcal F_2}$ is the identity on a local section $f_1 \otimes f_2$ if and only if $\theta_n^{\mathcal F_2}$ is the identity on $f_2$.  The result follows from Remark~\ref{rem: invariants}.
\end{proof}

\begin{corollary} \label{corollary: right adjoint to full pullback}
 Let $N$ be a closed normal subgroup of an affine algebraic subgroup $G$. Assume that $G$ acts on $X$ and $G/N$ acts on $Y$. Let $f: X \to Y$ be a morphism equivariant with respect to the quotient homomorphism $\pi: G \to G/N$. We have the pullback $f^*: \op{Qcoh}_{G/N} Y \to \op{Qcoh}_G X$. Consider $Y$ with the induced $G$ action to have the pushforward $f_*: \op{Qcoh}_G X \to \op{Qcoh}_G Y$. The composition, $(f_*)^N$, is right adjoint to $f^*$.
\end{corollary}

\begin{proof}
 The functor $f^*$ is the composition of $f^*: \op{Qcoh}_{G} Y \to \op{Qcoh}_G X$ and $\op{Res}_{\pi}$. As we have adjunctions, $f^* \dashv f_*$ and $\op{Res}^N_G \dashv (\bullet)^N$, the latter by Lemma~\ref{lemma: left adjoint to invariants}, we get the desired statement.
\end{proof}

\begin{lemma} \label{lemma: commuting invariants with pushforward}
 Let $G$ act on $X$ and $Y$. Let $N$ be a closed normal subgroup which acts trivially on $X$ and $Y$ and let $f: X \to Y$ be a $G$-equivariant morphism. For any $\mathcal E \in \op{Qcoh}_G X$, there is a natural isomorphism
 \begin{displaymath}
  (f_* \mathcal E)^N \cong f_* \mathcal E^N.
 \end{displaymath}
\end{lemma}

\begin{proof}
 By definition, $(f_*\mathcal E)^N$ is the kernel of the composition
  \begin{displaymath}
 f_* \mathcal E \to \Gamma(N, \mathcal O_N) \otimes_k f_* \mathcal E \overset{1_{\Gamma(N, \mathcal O_N) \otimes_k f_*\mathcal E} - \pi_{Y*} (1_G \times f)_*\theta|_{N \times X}}{\to} \Gamma(N, \mathcal O_N) \otimes_k f_* \mathcal E
 \end{displaymath}
 where $\pi_Y: G \times Y \to Y$ is the projection.  The above is 
  \begin{displaymath}
 f_* \left( \mathcal E \to \Gamma(N, \mathcal O_N) \otimes_k \mathcal E \overset{1_{\Gamma(N, \mathcal O_N) \otimes_k \mathcal E} - \pi_{X*} \theta|_{N \times X}}{\to} \Gamma(N, \mathcal O_N) \otimes_k \mathcal E \right)
 \end{displaymath}
 where $\pi_X: G \times X \to X$ is the projection. This is the definition of $ f_*\mathcal E^N$.
\end{proof}

\begin{lemma} \label{lemma: base change for invariants}
 Let $N$ be a closed normal subgroup of $G$. Let $G$ act on $X$ and $Y$ with $N$ acting trivially on both $X$ and $Y$.  Let $f: X \to Y$ be a flat $G$-equivariant morphism. For each $\mathcal E \in \op{Qcoh}_G Y$, there is a natural isomorphism of $G$-equivariant sheaves
 \begin{displaymath}
  f^* \mathcal E^N \cong (f^* \mathcal E)^N.
 \end{displaymath}
\end{lemma}

\begin{proof}
 By definition, $(f^*\mathcal E)^N$ is the kernel of the composition
  \begin{displaymath}
 f^* \mathcal E \to \Gamma(N, \mathcal O_N) \otimes_k f^* \mathcal E \overset{1_{\Gamma(N, \mathcal O_N) \otimes_k f^*\mathcal E} - \pi_{X*} (1_G \times f)^*\theta|_{N \times X}}{\to} \Gamma(N, \mathcal O_N) \otimes_k f^* \mathcal E
 \end{displaymath}
 where $\pi_X: G \times X \to X$ is the projection.  Therefore, by flat base change this is equal to the kernel of the composition
 \begin{displaymath}
  f^* \mathcal E \to \Gamma(N, \mathcal O_N) \otimes_k f^* \mathcal E \overset{1_{\Gamma(N, \mathcal O_N) \otimes_k f^*\mathcal E} -f^* \pi_{Y*} \theta|_{N \times X}}{\to} \Gamma(N, \mathcal O_N) \otimes_k f^* \mathcal E
 \end{displaymath}
 where $\pi_Y: G \times Y \to Y$ is the projection.
 
 Since $f$ is flat, this is isomorphic to $f^*$ applied to the kernel of the composition
 \begin{displaymath}
  \mathcal E \to \Gamma(N, \mathcal O_N) \otimes_k  \mathcal E \overset{1_{\Gamma(N, \mathcal O_N) \otimes_k \mathcal E} - \pi_{Y*} \theta|_{N \times X}}{\to} \Gamma(N, \mathcal O_N) \otimes_k  \mathcal E.
 \end{displaymath}
 This kernel is the definition of $\mathcal E^N$.
\end{proof}

\begin{definition}
 Let $f: X \to Y$ be a morphism of separated schemes of finite type. We say that $X$ possesses an \newterm{$f$-ample family of line bundles} if there is a set of invertible sheaves, $\mathcal L_{\alpha}$, $\alpha \in A$, such that for any quasi-coherent sheaf, $\mathcal E$, the natural morphism 
 \begin{displaymath}
  \bigoplus_{\alpha \in A} \mathcal L_{\alpha} \otimes_{\mathcal O_X} f^*f_*(\mathcal L_{\alpha}^{\vee} \otimes_{\mathcal O_X} \mathcal E) \to \mathcal E
 \end{displaymath}
 is an epimorphism. If $f: X \to \op{Spec} k$ is the structure morphism, we shall simply refer to the set $\mathcal L_{\alpha}, \alpha \in A$ as an \newterm{ample family}.  When $X$ possess an ample family it is called \newterm{divisorial}. If $X$ and $Y$ possess an action of $G$, $f$ is $G$-equivariant, and each $\mathcal L_{\alpha}$ admits an equivariant structure, then we will say that the $f$-ample family is equivariant.
\end{definition}

\begin{remark}
 This is one of the multitude of equivalent definitions of an $f$-ample family \cite[Proposition 2.2.3]{SGA6II}.
\end{remark}

Let us recall the following result of Thomason.

\begin{theorem} \label{theorem: Thomason}
 Let $X$ be a normal scheme of finite type acted on by an affine algebraic group $G$. Assume that $X$ is divisorial. Then, $X$ possesses an equivariant ample family. In particular, for any coherent $G$-equivariant sheaf, $\mathcal E$, there exists a locally-free $G$-equivariant sheaf of finite rank, $\mathcal V$, and an epimorphism, $\mathcal V \to \mathcal E$. 
\end{theorem}

\begin{proof}
 The conclusion is true replacing $G$ by the connected component of the identity, $G_0$, by \cite[Lemma 2.10]{Tho2}. Applying \cite[Lemma 2.14]{Tho2} shows it is also true for $G$.
\end{proof}

\begin{remark}
In what follows, we often assume that a scheme is divisorial and implicitly use the theorem above to obtain an equivariant ample family.
\end{remark}

We finish the section by recalling a simple fact about the global dimensions of categories of equivariant sheaves. Let $G$ be an affine algebraic group and let $X$ be a separated scheme of finite type.  

\begin{definition}
 Recall that the \newterm{global dimension} of an Abelian category, $\mathcal A$, is the maximal $n$ such that $\op{Ext}^n_{\mathcal A}(A,B)$ is nonzero for some pair of objects, $A$ and $B$, of $\mathcal A$. Let $\op{gldim} \mathcal A$ denote the global dimension of $\mathcal A$.
 
 Let $\mathcal A$ be an Abelian category and let $A$ be an object. The \newterm{projective dimension} of $A$ is 
 \begin{displaymath}
  \op{pdim} A := \op{min} \{ s \mid \op{Ext}^s_{\mathcal A}(A,\bullet) = 0\}.
 \end{displaymath}
 It is defined to be infinite if no such $s$ exists. The object, $A$, is said to have \newterm{locally-finite projective dimension} if for each $A' \in \mathcal A$, there exists an $s_0$ such that
 \begin{displaymath}
  \op{Ext}_{\mathcal A}^s(A,A') = 0
 \end{displaymath}
 for all $s \geq s_0$. 
 
 Note that the global dimension of $\mathcal A$ is
 \begin{displaymath}
  \op{sup}_A \op{pdim} A.
 \end{displaymath}
\end{definition}

\begin{lemma} \label{lem:graded gldim = gldim}
 Let $\mathcal E$ be a quasi-coherent $G$-equivariant sheaf. If $\mathcal E$ has locally-finite projective dimension as an object of $\op{Qcoh} X$, then it has locally-finite projective dimension as an object of $\op{Qcoh}_G X$. Moreover, we have the following inequalities,
\begin{align*}
 \op{pdim} (\mathcal E,\theta) & \leq \op{pdim} \mathcal E + \op{gldim} \op{Qcoh}_G\op{Spec} k \\
 \op{gldim} \op{Qcoh}_G X & \leq \op{gldim} \op{Qcoh}X + \op{gldim} \op{Qcoh}_G\op{Spec} k.
\end{align*}
 In particular, if $X$ is smooth, then $\op{gldim} \op{Qcoh}_G X$ is finite.
\end{lemma}

\begin{proof}
 Since, by definition,
 \begin{displaymath}
  \op{Hom}_{\op{Qcoh}_G X}(E,F) = \op{Hom}_{\op{Qcoh} X}(E,F)^G,
 \end{displaymath}
 there is a spectral sequence
\begin{displaymath}
E_2^{p,q}: \op{Ext}_{\op{Qcoh}X}^p(E,F)^{\mathbf{R}^qG}  \Rightarrow  \op{Ext}^{p+q}_{\op{Qcoh}_G X}(E,F).
\end{displaymath}
 Let
 \begin{displaymath}
  p_0 := \op{sup} \{ p \mid \op{Ext}^p_{\op{Qcoh} X}(E,F) \not = 0 \}.
 \end{displaymath}
 As $\op{Ext}^q_{\op{Qcoh}_G \op{Spec} k }(k,M) = M^{\mathbf{R}^qG}$, we see that $\op{Ext}^r_{\op{Qcoh}_G X}(E,F)$ vanishes for
 \begin{displaymath}
  r > \op{gldim} \op{Qcoh} X + \op{gldim} \op{Qcoh}_G \op{Spec} k \geq p_0 + \op{gldim} \op{Qcoh}_G \op{Spec} k .
 \end{displaymath}
 This gives the stated inequality.
 
 Choose a closed embedding of $G \subset \op{GL}_n$. Then, $M^G \cong (\op{Ind}_G^{\op{GL}_n}M)^{\op{GL}_n}$ and the functor of $\op{GL}_n$-invariants is exact. Thus, $M^{\mathbf{R}^qG} = 0$ for $q > \op{dim} \op{GL}_n/G$ as $\op{Ind}_G^{\op{GL}_n}$ is the composition of $(\iota^*)^{-1}$ and the pushforward of $\op{GL}_n/G \to \op{Spec} k$. Since
 \begin{displaymath}
  \op{Ext}^s_{\op{Qcoh}_G \op{Spec} k}(V,W) \cong \op{Ext}^s_{\op{Qcoh}_G \op{Spec} k}(k,\op{Hom}_k(V,W)) \cong \op{Hom}_k(V,W)^{\mathbf{R}^qG}
 \end{displaymath}
 the global dimension of $\op{Qcoh}_G \op{Spec} k$ is finite.
 
 Thus, if $\mathcal E$ has locally-finite projective dimension as an object of $\op{Qcoh} X$, then it has locally-finite projective dimension as an object of $\op{Qcoh}_G X$. 
 
 If $X$ is smooth, it is well-known that
 \begin{displaymath}
  \op{gldim} \op{Qcoh} X = \op{dim} X.
 \end{displaymath}
\end{proof}

\begin{remark}
 In general, the global dimension of $\op{Qcoh}_G X$ can be strictly smaller than the global dimension of $\op{Qcoh} X$. Indeed, $\op{Qcoh}_G G$, with the left action of $G$ on itself, is equivalent to $\op{Qcoh} \op{Spec }k$ and, therefore, must have global dimension zero. We thank Kuznetsov for pointing this out. 
\end{remark}

\section{Equivariant factorizations} \label{sec:graded MFs}

Let $G$ be an affine algebraic group and let $X$ be a smooth variety equipped with an action $\sigma: G \times X \to X$. Let $w \in \Gamma(X,\mathcal L)^G$ be a $G$-invariant section of an invertible equivariant sheaf, $\mathcal L$. 
\begin{definition} \label{defn: big Fact}
 The \newterm{dg-category of factorizations} of $w$, is denoted by $\mathsf{Fact}(X,G,w)$. The objects of $\mathsf{Fact}(X,G,w)$ are pairs, 
 \begin{center}
 \begin{tikzpicture}[description/.style={fill=white,inner sep=2pt}]
  \matrix (m) [matrix of math nodes, row sep=3em, column sep=3em, text height=1.5ex, text depth=0.25ex]
  {  \mathcal E_{-1} & \mathcal E_0 & \mathcal E_{-1} \otimes_{\mathcal O_X} \mathcal L \\ };
  \path[->,font=\scriptsize]
  (m-1-1) edge node[above] {$\phi^{\mathcal E}_0$} (m-1-2)
  (m-1-2) edge node[above] {$\phi^{\mathcal E}_{-1}$} (m-1-3);
 \end{tikzpicture}
 \end{center}
 of morphisms in $\op{Qcoh}_G X$, satisfying 
 \begin{align*}
  \phi^{\mathcal E}_{-1} \circ \phi^{\mathcal E}_0 & = w \\
  (\phi^{\mathcal E}_0 \otimes \mathcal L) \circ \phi^{\mathcal E}_{-1} & = w.
 \end{align*}
 We denote such an object by $(\mathcal E_{-1},\mathcal E_0,\phi^{\mathcal E}_{-1},\phi^{\mathcal E}_0)$ or simply by $\mathcal E$ when there is no confusion. The morphism complex between two objects, $\mathcal E$ and $\mathcal F$, as a graded vector space, can be described as follows. For $n=2l$, we have 
\begin{displaymath}
 \op{Hom}^n_{\mathsf{Fact}(X,G,w)}(\mathcal E,\mathcal F) = \op{Hom}_{\op{Qcoh}_G X}(\mathcal E_{-1}, \mathcal F_{-1}\otimes_{\mathcal O_X} \mathcal L^l) \oplus \op{Hom}_{\op{Qcoh}_G X}(\mathcal E_0, \mathcal F_0\otimes_{\mathcal O_X} \mathcal L^l)
\end{displaymath}
and for $n=2l+1$, we have
\begin{displaymath}
 \op{Hom}^n_{\mathsf{Fact}(X,G,w)}(\mathcal E,\mathcal F) = \op{Hom}_{\op{Qcoh}_G X}(\mathcal E_{0},\mathcal F_{-1}\otimes_{\mathcal O_X} \mathcal L^{l+1}) \oplus \op{Hom}_{\op{Qcoh}_G X}(\mathcal E_{-1}, \mathcal F_0 \otimes_{\mathcal O_X} \mathcal L^l)
\end{displaymath}
 The differential applied to $(f_{-1},f_0) \in \op{Hom}^n_{\mathsf{Fact}(X,G,w)}(\mathcal E,\mathcal F)$ 
\begin{displaymath}
 = \begin{cases}
  \left((f_0 \circ \phi_0^{\mathcal E} - (\phi^{\mathcal F}_0\otimes_{\mathcal O_X} \mathcal L^l) \circ f_{-1}, \left( f_{-1} \otimes_{\mathcal O_X} \mathcal L \right) \circ \phi^{\mathcal E}_{-1} - ( \phi^{\mathcal F}_{-1}\otimes_{\mathcal O_X} \mathcal L^{l} ) \circ f_0\right) & \text{if }n=2l \\
  \left((f_0 \circ \phi_0^{\mathcal E} + (\phi^{\mathcal F}_{-1} \otimes_{\mathcal O_X} \mathcal L^{l}) \circ f_{-1}, \left(f_{-1} \otimes_{\mathcal O_X} \mathcal L\right) \circ \phi^{\mathcal E}_{-1} + (\phi^{\mathcal F}_0 \otimes_{\mathcal O_X} \mathcal L^{l+1}) \circ f_0\right) & \text{if }n=2l+1. \\
 \end{cases}
\end{displaymath}
\end{definition}

Given an additive subcategory of $\op{Qcoh}_G X$, we can form a corresponding dg-subcategory of $\mathsf{Fact}(X,G,w)$ by requiring the components, $\mathcal E_{-1}$ and $\mathcal E_0$, to be objects from that additive subcategory.

\begin{definition} \label{defn: multi facts}
 Denote by $\mathsf{fact}(X,G,w)$, $\mathsf{Vect}(X,G,w)$, $\mathsf{vect}(X,G,w)$, and $\mathsf{Inj}(X,G,w)$, respectively, the full dg-subcategory of $\mathsf{Fact}(X,G,w)$ whose components, respectively, are coherent, locally-free, locally-free of finite rank, and injective as quasi-coherent $G$-equivariant sheaves.
\end{definition}

\begin{remark}
 Categories of projective factorizations only prove useful when $X$ is affine and $G$ is reductive. Then, any locally-free $G$-equivariant sheaf of finite rank is projective.
\end{remark}

\begin{definition}
 The \newterm{shift}, denoted by $[1]$, sends a factorization, $\mathcal E$, to the factorization, 
 \begin{displaymath}
  \mathcal E[1] := (\mathcal E_0,\mathcal E_{-1}\otimes_{\mathcal O_X} \mathcal L,-\phi^{\mathcal E}_0,-\phi^{\mathcal E}_{-1} \otimes_{\mathcal O_X} \mathcal L).
 \end{displaymath}
\end{definition}

\begin{lemma}
 We have an equality
 \begin{displaymath}
  \op{Hom}^n_{\mathsf{Fact}(X,G,w)}(\mathcal E, \mathcal F) = \op{Hom}^0_{\mathsf{Fact}(X,G,w)}(\mathcal E, \mathcal F[n]).
 \end{displaymath}
\end{lemma}

\begin{proof}
 This is a straightforward check and is suppressed.
\end{proof}

One can pass to an associated Abelian category. It has the same objects as $\mathsf{Fact}(X,G,w)$, but morphisms between $\mathcal E$ and $\mathcal F$ are closed degree-zero morphisms in $\op{Hom}_{\mathsf{Fact}(X,G,w)}(\mathcal E, \mathcal F)$. Denote this category by $Z^0\mathsf{Fact}(X,G,w)$. The category, $Z^0\mathsf{Fact}(X,G,w)$, with component-wise kernels and cokernels is an Abelian category.

\begin{definition}
 Given a complex of objects from $Z^0\mathsf{Fact}(X,G,w)$,
\begin{displaymath}
 \cdots \to \mathcal E^b \overset{f^b}{\to} \mathcal E^{b+1} \overset{f^{b+1}}{\to} \cdots \overset{f^{t-1}}{\to} \mathcal E^{t} \to \cdots ,
\end{displaymath}
 the \newterm{totalization}, $\mathcal T$, is the factorization
\begin{align*}
 \mathcal T_{-1} & := \bigoplus_{i=2l} \mathcal E_{-1}^i \otimes_{\O_X} \cL^{-l} \oplus \bigoplus_{i=2l-1} \mathcal E_0^i \otimes_{\O_X} \cL^{-l} \\
 \mathcal T_0 & := \bigoplus_{i=2l} \mathcal E_{0}^i \otimes_{\O_X} \cL^{-l} \oplus \bigoplus_{i=2l+1} \mathcal E_{-1}^i \otimes_{\O_X} \cL^{-l} \\
 \phi^{\mathcal T}_0 & := \begin{pmatrix} \ddots & 0 & 0 & 0 & 0 \\ \ddots & -\phi_{-1}^{\mathcal E^{-1}} & 0 & 0 & 0 \\ 0 & f_0^{-1} & \phi_0^{\mathcal E^{0}} & 0 & 0 \\ 0 & 0 & f_{-1}^{0} & -\phi_{-1}^{\mathcal E^1} \otimes \mathcal L^{-1} & 0 \\ 0 & 0 & 0 & \ddots & \ddots \end{pmatrix} \\
 \phi^{\mathcal T}_{-1} & := \begin{pmatrix} \ddots & 0 & 0 & 0 & 0 \\ \ddots & -\phi_{0}^{\mathcal E^{-1}} \otimes \mathcal L & 0 & 0 & 0 \\ 0 & f_{-1}^{-1} \otimes \mathcal L & \phi_{-1}^{\mathcal E^{0}} & 0 & 0\\ 0 & 0 & f_0^{0} & -\phi_{0}^{\mathcal E^1} & 0 \\ 0 & 0 & 0 & \ddots & \ddots \end{pmatrix}
\end{align*}
 For any closed morphism of cohomological degree zero, $f:\mathcal E \to \mathcal F$, in $\mathsf{Fact}(X,G,w)$, we can form the cone factorization, $C(f)$, as the totalization of the complex
 \begin{displaymath}
  \mathcal E \overset{f}{\to} \mathcal F
 \end{displaymath}
 where $\mathcal F$ is in degree zero.
\end{definition}

\begin{proposition} \label{prop: facts are triangulated}
 The homotopy category, $[\mathsf{Fact}(X,G,w)]$, is a triangulated category.
\end{proposition}

\begin{proof}
 The translation is $[1]$ and the class of triangles is given by sequences of morphisms
 \begin{displaymath}
  \mathcal E \overset{f}{\to} \mathcal F \to C(f) \to \mathcal E[1].
 \end{displaymath}
 The proof now runs completely analogously to proving that the homotopy category of chain complexes of an Abelian category is triangulated. It is therefore suppressed.
\end{proof}

\begin{definition}{(Positselski)} 
 Let $\mathsf{Acyc}(X,G,w)$ denote the full subcategory of objects of $\mathsf{Fact}(X,G,w)$ consisting of totalizations of bounded exact complexes from $Z^0\mathsf{Fact}(X,G,w)$. Objects of $\mathsf{Acyc}(X,G,w)$ are called \newterm{acyclic}. Similarly, let $\mathsf{acyc}(X,G,w)$ denote the subcategory of totalizations of bounded exact complexes of coherent factorizations. 

 We will also need the analogs for factorizations with locally-free components. The full subcategory of objects of $\mathsf{Vect}(X,G,w)$ consisting of totalizations of bounded exact complexes from $Z^0\mathsf{Vect}(X,G,w)$ is denoted by $\mathsf{AcycVect}(X,G,w)$. Similarly, let $\mathsf{acycvect}(X,G,w)$ denote the subcategory of totalizations of bounded exact complexes of coherent locally-free factorizations.
\end{definition}

\begin{definition}{(Positselski)} \label{defn: abs der}
 The \newterm{absolute derived category} of $[\mathsf{Fact}(X,G,w)]$ is the Verdier quotient of $[\mathsf{Fact}(X,G,w)]$ by $[\mathsf{Acyc}(X,G,w)]$,
 \begin{displaymath}
  \op{D}^{\op{abs}}[\mathsf{Fact}(X,G,w)] := [\mathsf{Fact}(X,G,w)]/[\mathsf{Acyc}(X,G,w)].
 \end{displaymath}
 The \newterm{absolute derived category} of $[\mathsf{fact}(X,G,w)]$ is the Verdier quotient of $[\mathsf{fact}(X,G,w)]$ by $[\mathsf{acyc}(X,G,w)]$,
 \begin{displaymath}
  \op{D}^{\op{abs}}[\mathsf{fact}(X,G,w)] := [\mathsf{fact}(X,G,w)]/[\mathsf{acyc}(X,G,w)].
 \end{displaymath}
 
 The \newterm{absolute derived category} of $[\mathsf{Vect}(X,G,w)]$ is the Verdier quotient of $[\mathsf{Vect}(X,G,w)]$ by $[\mathsf{AcycVect}(X,G,w)]$
 \begin{displaymath}
  \op{D}^{\op{abs}}[\mathsf{Vect}(X,G,w)] := [\mathsf{Vect}(X,G,w)]/[\mathsf{AcycVect}(X,G,w)].
 \end{displaymath}
 The \newterm{absolute derived category} of $[\mathsf{vect}(X,G,w)]$ is the Verdier quotient of $[\mathsf{vect}(X,G,w)]$ by $[\mathsf{acycvect}(X,G,w)]$,
 \begin{displaymath}
  \op{D}^{\op{abs}}[\mathsf{vect}(X,G,w)] := [\mathsf{vect}(X,G,w)]/[\mathsf{acycvect}(X,G,w)].
 \end{displaymath}

 We say that two factorizations are \newterm{quasi-isomorphic} if they are isomorphic in the appropriate absolute derived category.
\end{definition}

We will also use versions of these categories with support conditions. Let $Z$ be a closed $G$-invariant subset of $X$ and set $U := X \setminus Z$. Let $j: U \to X$ be the inclusion. 

\begin{definition}
 The category, $\dabs_Z [\mathsf{Fact}(X,G,w)]$, is the kernel of the functor, 
 \begin{displaymath}
  j^*: \dabs_Z [\mathsf{Fact}(X,G,w)] \to \dabs_Z [\mathsf{Fact}(U,G,w|_U)].
 \end{displaymath}
 Define $\dabs_Z [\mathsf{fact}(X,G,w)]$, $\dabs_Z [\mathsf{Vect}(X,G,w)]$, $\dabs_Z [\mathsf{vect}(X,G,w)]$ analogously.
\end{definition}

Let us recall some useful facts, due essentially to Positselski, about $\op{D}^{\op{abs}}[\mathsf{Fact}(X,G,w)]$.

\begin{proposition} \label{prop: injective enhancement}
 Factorizations with injective components are right orthogonal to acyclic complexes in $[\mathsf{Fact}(X,G,w)]$. Moreover, the composition,
 \begin{displaymath}
  [\mathsf{Inj}(X,G,w)] \to [\mathsf{Fact}(X,G,w)] \to \op{D}^{\op{abs}}[\mathsf{Fact}(X,G,w)]
 \end{displaymath}
 is an equivalence.
\end{proposition}

\begin{proof}
 This is a version of \cite[Theorem 3.6]{Pos1} of Positselski. In this generality, it is a special case of \cite[Lemma 2.22 and Corollary 2.23]{BDFIK}.
\end{proof}

\begin{definition}
 We let $\mathsf{Inj}_{\op{coh}}(X,G,w)$ be the full dg-category of $\mathsf{Fact}(X,G,w)$ consisting of factorizations that have injective components and that are quasi-isomorphic to a factorization with coherent components.
\end{definition}

\begin{corollary} \label{cor: injective enhancement}
 The composition,
 \begin{displaymath}
 [\mathsf{Inj}_{\op{coh}}(X,G,w)] \to [\mathsf{Fact}(X,G,w)] \to \op{D}^{\op{abs}}[\mathsf{fact}(X,G,w)]
 \end{displaymath}
 is an equivalence.
\end{corollary}

\begin{proof}
 This is an immediate corollary of Proposition~\ref{prop: injective enhancement}.
\end{proof}

\begin{proposition} \label{prop: projective enhancement}
 The natural functor,
 \begin{displaymath}
  \op{D}^{\op{abs}}[\mathsf{Vect}(X,G,w)] \to \op{D}^{\op{abs}}[\mathsf{Fact}(X,G,w)],
 \end{displaymath}
 is an equivalence as is the natural functor,
 \begin{displaymath}
  \op{D}^{\op{abs}}[\mathsf{vect}(X,G,w)] \to \op{D}^{\op{abs}}[\mathsf{fact}(X,G,w)].
 \end{displaymath}
 
 Moreover, if $X$ is affine and $G$ is reductive, factorizations with locally-free components are left orthogonal to acyclic complexes in $[\mathsf{Fact}(X,G,w)]$ and the compositions
 \begin{align*}
  [\mathsf{Vect}(X,G,w)] & \to [\mathsf{Fact}(X,G,w)] \to \op{D}^{\op{abs}}[\mathsf{Fact}(X,G,w)] \\
  [\mathsf{vect}(X,G,w)] & \to [\mathsf{fact}(X,G,w)] \to \op{D}^{\op{abs}}[\mathsf{fact}(X,G,w)]
 \end{align*}
 are equivalences.
\end{proposition}

\begin{proof}
 We first check that any factorization is quasi-isomorphic to a locally-free factorization. Moreover, if the original factorization is coherent, then the locally-free factorization can be chosen to have finite rank. The argument is contained in the proof of \cite[Theorem 3.6]{Pos1}. Let $\mathcal E$ be a factorization. By Theorem \ref{theorem: Thomason}, we can find locally-free $G$-equivariant sheaves, $\mathcal V_{-1}$ and $\mathcal V_0$ and epimorphisms
 \begin{align*}
  \mathcal V_{-1} & \overset{f_{-1}}{\to} \mathcal E_{-1} \\
  \mathcal V_0 & \overset{f_0}{\to} \mathcal E_0.
 \end{align*}
 Form the factorization, $G^+(\mathcal V)$,
 \begin{displaymath}
  \mathcal V_0\otimes_{\mathcal O_X} \mathcal L^{-1} \oplus \mathcal V_{-1} \overset{\begin{pmatrix} 0 & 1 \\ w & 0 \end{pmatrix}}{\to} \mathcal V_{-1} \oplus \mathcal V_0 \overset{\begin{pmatrix} 0 & w \\ 1 & 0 \end{pmatrix}}{\to} \mathcal V_0 \oplus \mathcal V_{-1}\otimes_{\mathcal O_X} \mathcal L.
 \end{displaymath}
 The maps,
 \begin{align*}
  \mathcal V_0 \otimes_{\mathcal O_X} \mathcal L^{-1} \oplus \mathcal V_{-1} & \overset{( 0 \ f_{-1} )}{\to} \mathcal E_{-1} \\
  \mathcal V_{-1} \oplus \mathcal V_0 & \overset{( 0 \ f_{0} )}{\to} \mathcal E_{0},
 \end{align*}
 give an epimorphism in $Z^0\mathsf{Fact}(X,G,w)$. Thus, for any factorization, there exists a factorization with locally-free components mapping epimorphically onto it. We can construct an exact complex of objects of $Z^0\mathsf{Fact}(X,G,w)$
 \begin{displaymath}
  \cdots \to \mathcal V^s \to \cdots \to \mathcal V^1 \to \mathcal E \to 0
 \end{displaymath}
 where each $\mathcal V^j$ is a factorization with locally-free components. Let $\mathcal K^s$ be the kernel of $\mathcal V^s \to \mathcal V^{s-1}$ for $s > \op{dim} X$. Since $X$ is smooth, the components of $\mathcal K^s$ are locally-free. Thus, we have an exact sequence
 \begin{displaymath}
  0 \to \mathcal K^s \to \mathcal V^s \to \cdots \to \mathcal V^1 \to \mathcal E \to 0.
 \end{displaymath}
 In $\op{D}^{\op{abs}}[\mathsf{fact}(X,G,w)]$, we have an isomorphism 
 \begin{displaymath}
  \mathcal T \to \mathcal E
 \end{displaymath}
 where $\mathcal T$ is the totalization of $\mathcal K^s \to \mathcal V^s \to \cdots \to \mathcal V^1$. The factorization, $\mathcal T$, has locally-free components.
 
 Thus, the natural functors,
 \begin{align*}
  \op{D}^{\op{abs}}[\mathsf{Vect}(X,G,w)] & \to \op{D}^{\op{abs}}[\mathsf{Fact}(X,G,w)] \\
  \op{D}^{\op{abs}}[\mathsf{vect}(X,G,w)] & \to \op{D}^{\op{abs}}[\mathsf{fact}(X,G,w)],
 \end{align*}
 are essentially surjective. We next check fully-faithfulness. 
 
For fully-faithfulness, it suffices to show that given a short exact sequence
 \begin{equation} \label{equation: fact SES}
  0 \to \mathcal E^3 \to \mathcal E^2 \to \mathcal E^1 \to 0
 \end{equation}
 there exists a factorization, $\mathcal S \in \mathsf{AcycVect}(X,G,w)$, that is isomorphic to the totalization, $\mathcal T$, of \eqref{equation: fact SES} in $\op{D}^{\op{abs}}[\mathsf{Fact}(X,G,w)]$. Moreover, if $\mathcal E_i$ are all coherent, then $\mathcal S$ can be taken to have finite rank. 
 
 Using what we have already proven, we can find a locally-free factorization $\mathcal V^1_1$ and an epimorphism 
 \begin{displaymath}
  \mathcal V^1_1 \to \mathcal E^1.
 \end{displaymath}
 Next choose a locally-free factorization $\mathcal V^2_1$ and an epimorphism onto the fiber product
 \begin{displaymath}
  \mathcal V^2_1 \to \mathcal E^2 \times_{\mathcal E^1} \mathcal V^1_1.
 \end{displaymath}
 Let $\mathcal V^3_1$ be the kernel of the map $\mathcal V^2_1 \to \mathcal E^2 \times_{\mathcal E^1} \mathcal V^1_1 \to \mathcal V^1_1$. There is a commutative diagram 
 \begin{center}
 \begin{tikzpicture}[description/.style={fill=white,inner sep=2pt}]
  \matrix (m) [matrix of math nodes, row sep=1em, column sep=3em, text height=1.5ex, text depth=0.25ex]
  {  0 & \mathcal E^3 & \mathcal E^2 & \mathcal E^1 & 0 \\
     0 & \mathcal V_1^3 & \mathcal V_1^2 & \mathcal V_1^1 & 0 \\
  };
  \path[->,font=\scriptsize]
  (m-1-1) edge (m-1-2)
  (m-1-2) edge (m-1-3)
  (m-1-3) edge (m-1-4)
  (m-1-4) edge (m-1-5)
  (m-2-1) edge (m-2-2)
  (m-2-2) edge (m-2-3)
  (m-2-3) edge (m-2-4)
  (m-2-4) edge (m-2-5)
  (m-2-2) edge (m-1-2)
  (m-2-3) edge (m-1-3)
  (m-2-4) edge (m-1-4)
  ;
 \end{tikzpicture}
 \end{center}
 with the vertical morphisms being epimorphisms. Replacing \eqref{equation: fact SES} the kernels of the vertical morphisms, repeating the argument, and iterating, we get an exact sequence of short exact sequences
 \begin{center}
 \begin{tikzpicture}[description/.style={fill=white,inner sep=2pt}]
  \matrix (m) [matrix of math nodes, row sep=1em, column sep=3em, text height=1.5ex, text depth=0.25ex]
  {   & 0 & 0 & 0 &  \\
     0 & \mathcal E^3 & \mathcal E^2 & \mathcal E^1 & 0 \\
     0 & \mathcal V_1^3 & \mathcal V_1^2 & \mathcal V_1^1 & 0 \\
      & \vdots & \vdots & \vdots &  \\
     0 & \mathcal V_s^3 & \mathcal V_s^2 & \mathcal V_s^1 & 0 \\
      & 0 & 0 & 0 & \\
  };
  \path[->,font=\scriptsize]
  (m-2-1) edge (m-2-2)
  (m-2-2) edge (m-2-3)
  (m-2-3) edge (m-2-4)
  (m-2-4) edge (m-2-5)
  (m-3-1) edge (m-3-2)
  (m-3-2) edge (m-3-3)
  (m-3-3) edge (m-3-4)
  (m-3-4) edge (m-3-5)
  (m-5-1) edge (m-5-2)
  (m-5-2) edge (m-5-3)
  (m-5-3) edge (m-5-4)
  (m-5-4) edge (m-5-5)
  
  (m-2-2) edge (m-1-2)
  (m-2-3) edge (m-1-3)
  (m-2-4) edge (m-1-4)
  
  (m-3-2) edge (m-2-2)
  (m-3-3) edge (m-2-3)
  (m-3-4) edge (m-2-4)
  
  (m-4-2) edge (m-3-2)
  (m-4-3) edge (m-3-3)
  (m-4-4) edge (m-3-4)
  
  (m-5-2) edge (m-4-2)
  (m-5-3) edge (m-4-3)
  (m-5-4) edge (m-4-4)
  
  (m-6-2) edge (m-5-2)
  (m-6-3) edge (m-5-3)
  (m-6-4) edge (m-5-4)
  ;
 \end{tikzpicture}
 \end{center}
 where each $\mathcal V^i_j$ is locally-free, and of finite rank if each $\mathcal E^i$ is coherent. The long exact sequence of short exact sequences gives rise to a long exact sequence of the totalizations of these short exact sequences
 \begin{displaymath}
  0 \to \mathcal T_s \to \mathcal T_{s-1} \to \cdots \to \mathcal T_1 \to \mathcal T \to 0.
 \end{displaymath}
 Each $\mathcal T_j$ lies in $\mathsf{AcycVect}(X,G,w)$, or in $\mathsf{acycvect}(X,G,w)$ if each $\mathcal E^i$ is coherent. Thus, the totalization of $\mathcal T_s \to \mathcal T_{s-1} \to \cdots \to \mathcal T_1$ lies in $\mathsf{AcycVect}(X,G,w)$, or in $\mathsf{acycvect}(X,G,w)$ if each $\mathcal E^i$ is coherent, and is isomorphic to $\mathcal T$ in $\op{D}^{\op{abs}}[\mathsf{Fact}(X,G,w)]$.
 
 If we assume that $X$ is affine and $G$ is reductive, then any $G$-equivariant locally-free sheaf is projective.  The result in this case is a version of \cite[Theorem 3.6]{Pos1} of Positselski. For this generality, we argue as follows. By \cite[Lemma 2.22]{BDFIK}, factorizations with projective components are left orthogonal to acyclic factorizations. Thus, the compositions 
 \begin{align*}
  [\mathsf{Vect}(X,G,w)] & \to [\mathsf{Fact}(X,G,w)] \to \op{D}^{\op{abs}}[\mathsf{Fact}(X,G,w)] \\
  [\mathsf{vect}(X,G,w)] & \to [\mathsf{fact}(X,G,w)] \to \op{D}^{\op{abs}}[\mathsf{fact}(X,G,w)]
 \end{align*}
 are fully-faithful. As we have already seen they are essentially surjective, they must both be equivalences.
\end{proof}

For a definition of a compactly-generated triangulated category and compact generators, refer to Section~\ref{sec:generation of graded sing cat}.

\begin{proposition} \label{prop: we have a compactly-gen triangulated cat}
 The triangulated category, $\op{D}^{\op{abs}}[\mathsf{Fact}(X,G,w)]$, is compactly-generated. The objects of $\op{D}^{\op{abs}}[\mathsf{fact}(X,G,w)]$ are a set of compact generators. 
\end{proposition}

\begin{proof}
 The proof of this fact is a repetition of the argument of \cite[Theorem 3.11.2]{Pos1} using the fact that any quasi-coherent $G$-equivariant sheaf on $X$, hence any factorization, is a union of its coherent subsheaves \cite[Lemma 1.4]{Tho2}.  More precisely, one can use Lemma~\ref{lem: colimit gen up inf coprod} (which is a consequence of Thomason's result) and follow Positselski's argument verbatim.
\end{proof}

\begin{remark}
 It is a subtle problem to determine whether or not all compact objects of $\dabs[\mathsf{Fact}(X,G,w)]$ are isomorphic to objects of $\dabs[\mathsf{fact}(X,G,w)]$. By Proposition~\ref{prop: we have a compactly-gen triangulated cat} and \cite[Theorem 2.1]{Nee2}, every compact object is a summand of an object of $\dabs[\mathsf{fact}(X,G,w)]$ under a splitting in $\dabs[\mathsf{Fact}(X,G,w)]$. However, those summands may not be representable by coherent factorizations. See \cite{OrlFC} for an investigation of the relationship with completions of $X$. 
\end{remark}

To handle the possible idempotent incompleteness of our factorizations categories, we make the following definitions.

\begin{definition} \label{defn: idemp compl}
 Let $\overline{\mathsf{Inj}}_{\op{coh}}(X,G,w)$ be the full dg-subcategory of $\mathsf{Inj}(X,G,w)$ consisting of factorizations which are compact in $[\mathsf{Inj}(X,G,w)] \cong \dabs[\mathsf{Fact}(X,G,w)]$. 

 Let $\overline{\mathsf{vect}}(X,G,w)$ be the full dg-subcategory of $\mathsf{Vect}(X,G,w)$ consisting of factorizations which are compact in $\dabs[\mathsf{Vect}(X,G,w)] \cong \dabs[\mathsf{Fact}(X,G,w)]$. 

 Let $\overline{\dabs[\mathsf{fact}(X,G,w)]}$ denote the idempotent-completion of $\dabs[\mathsf{fact}(X,G,w)]$. Note that by Proposition~\ref{prop: injective enhancement}, we have
 \begin{displaymath}
  \left[\overline{\mathsf{Inj}}_{\op{coh}}(X,G,w)\right] \cong \overline{\dabs[\mathsf{fact}(X,G,w)]}.
 \end{displaymath}
 If $X$ is affine and $G$ is reductive, by Proposition~\ref{prop: projective enhancement}, we have
 \begin{displaymath}
  \left[\overline{\mathsf{vect}}(X,G,w)\right] \cong \overline{\dabs[\mathsf{fact}(X,G,w)]} 
 \end{displaymath}
\end{definition}

From a complex on the zero locus of $w$, one can form a factorization.

\begin{definition}
 Let $Y$ be the zero locus of $w$ in $X$. Denote by $\mathsf{Qcoh}_G Y$ the dg-category of chain complexes of quasi-coherent $G$-equivariant sheaves on $Y$.
 
 We have the dg-functor, see \cite[Section 3.7]{Pos2},
\begin{align*}
 \Upsilon: \mathsf{Qcoh}_G Y & \to \mathsf{Fact}(X,G,w) \\
  \mathcal C & \mapsto (\bigoplus_{l \in \Z} i_*\mathcal C^{2l-1}\otimes_{\mathcal O_X} \mathcal L^{-l}, \bigoplus_{l \in \Z} i_*\mathcal C^{2l} \otimes_{\mathcal O_X} \mathcal L^{-l}, \oplus_{l \in \Z} i_*d_{\mathcal C}^{2l-1}\otimes_{\mathcal O_X} \mathcal L^{-l}, \oplus_{l \in \Z} i_*d_{\mathcal C}^{2l} \otimes_{\mathcal O_X} \mathcal L^{-l}),
\end{align*}
 In the case that $\mathcal C$ is a coherent $G$-equivariant sheaf and the context allows, we will denote $\Upsilon \mathcal C$ simply by $\mathcal C$
\end{definition}

Note that $\Upsilon \mathcal C$ is the totalization of the chain complex 
\begin{displaymath}
 \cdots \to \Upsilon \mathcal C^b \to \cdots \to \Upsilon \mathcal C^t \to \cdots.
\end{displaymath}
It is clear that $\Upsilon$ takes bounded acyclic chain complexes in $\mathsf{Qcoh}_G Y$ to acyclic chain complexes on $[\mathsf{Fact}(X,G,w)]$. Thus, $\Upsilon$ descends to a functor 
\begin{displaymath}
 \Upsilon: \dbqcohG{G}{Y} \to \op{D}^{\op{abs}}[\mathsf{Fact}(X,G,w)].
\end{displaymath}
Moreover, $\Upsilon$ takes bounded complexes of coherent sheaves to coherent factorizations so it induces a functor
\begin{displaymath}
 \Upsilon: \dbcohG{G}{Y} \to \op{D}^{\op{abs}}[\mathsf{fact}(X,G,w)].
\end{displaymath}

Now, we give an explicit construction, due essentially to Eisenbud \cite[Section 7]{EisMF}, of a factorization associated to certain invariant closed subschemes in the zero locus of $w$. Consider an equivariant morphism
\begin{displaymath}
\mathcal E \overset{s}{\to} \mathcal O_X 
\end{displaymath}
with $\mathcal E$ locally-free of finite rank. We notationally identify $s$ with the corresponding global section of $\mathcal E^{\vee}$. Further, assume there exists an equivariant morphism $t: \mathcal O_X \to \mathcal E \otimes_{\mathcal O_X} \mathcal L$ making the diagram
\begin{center}
\begin{tikzpicture}[description/.style={fill=white,inner sep=2pt}]
\matrix (m) [matrix of math nodes, row sep=3em, column sep=3em, text height=1.5ex, text depth=0.25ex]
{  \mathcal E & \mathcal O_X \\
   \mathcal E\otimes_{\mathcal O_X} \mathcal L & \mathcal O_X\otimes_{\mathcal O_X} \mathcal L \\
};
\path[->,font=\scriptsize]
(m-1-1) edge node[above] {$s$} (m-1-2)
(m-1-2) edge node[above] {$t$} (m-2-1)
(m-2-1) edge node[above] {$s \otimes_{\mathcal O_X} \mathcal L$} (m-2-2)
(m-1-1) edge node[left] {$w$} (m-2-1)
(m-1-2) edge node[left] {$w$} (m-2-2)
;
\end{tikzpicture}
\end{center}
commute. 

\begin{definition} \label{definition: Koszul factorization}
 The \newterm{Koszul factorization} associated to the data $(\mathcal E, s, t)$ is defined as 
 \begin{align*}
  \mathcal K_{-1}(s,t) & := \bigoplus_{l \geq 0} (\Lambda^{2l+1} \mathcal E) \otimes_{\mathcal O_X} \mathcal L^l \\ 
  \mathcal K_0(s,t) & := \bigoplus_{l \geq 0} (\Lambda^{2l} \mathcal E)\otimes_{\mathcal O_X} \mathcal L^l \\
  \phi^{\mathcal K}_0, \phi^{\mathcal K}_{-1} & := \bullet \ \lrcorner \ s + \bullet \wedge t.
 \end{align*}
\end{definition}

\begin{proposition} \label{prop: Eisenbud stabilization}
 Assume that $(\mathcal E, s, t)$ as above exist. Let $\mathcal O_{Z_s}$ be the cokernel of $s$. If $\op{rank} \mathcal E = \op{codim} Z_s$, then $\mathcal K(s,t)$ is quasi-isomorphic to the factorization, $\Upsilon \mathcal O_{Z_s}$.
 
 Let $\mathcal O_{Z_{t^{\vee}}}$ be the cokernel of $\mathcal E^{\vee} \otimes_{\mathcal O_X} \mathcal L^{\vee} \overset{t^{\vee}}{\to} \mathcal O_{X}$. If $\op{rank} \mathcal E = \op{codim} Z_{t^{\vee}}$, then $\mathcal K(s,t)$ is quasi-isomorphic to the factorization $\Upsilon \left( \mathcal O_{Z_{t^{\vee}}} \otimes_{\mathcal O_X} \Lambda^{\op{rk} \mathcal E} \mathcal E [-\op{rk} \mathcal E]\right)$.
\end{proposition}

\begin{proof}
 Each is a straightforward application of \cite[Lemma 3.4]{BDFIK}, see also \cite[Section 3.2]{Becker}.
\end{proof}

\begin{lemma} \label{lemma: duality of Koszul factorizations}
 We have an isomorphism of factorizations,
 \begin{displaymath}
  \mathcal K(s,t)^{\vee} \cong \mathcal K(t^{\vee},s^{\vee}).
 \end{displaymath}
\end{lemma}

\begin{proof}
 This is immediate from the definitions.
\end{proof}

We describe some functors associated with natural operations on factorizations, mirroring those discussed in Section~\ref{sec:graded rings}.

\begin{definition}
 Let $X$ be a smooth variety equipped with an action of $G$. Assume we have $w, v \in \Gamma(X,\mathcal L)^G$. 
 We define a dg-functor,
 \begin{displaymath}
  \otimes_{\mathcal O_X} : \mathsf{Fact}(X,G,w) \otimes_k \mathsf{Fact}(X,G,v) \to \mathsf{Fact}(X,G,w+v),
 \end{displaymath}
 by setting
 \begin{align*}
 \left(\mathcal E \otimes_{\mathcal O_X} \mathcal F \right)_{-1} & := \mathcal E_{-1} \otimes_{\mathcal O_X} \mathcal F_0 \oplus \mathcal E_{0} \otimes_{\mathcal O_X} \mathcal F_{-1} \\
 \left(\mathcal E \otimes_{\mathcal O_X} \mathcal F \right)_0 & := \mathcal E_0 \otimes_{\mathcal O_X} \mathcal F_0 \oplus \mathcal E_{-1} \otimes_{\mathcal O_X} \mathcal F_{-1}\otimes_{\mathcal O_X} \mathcal L \\
 \phi^{\mathcal E \otimes_{\mathcal O_X} \mathcal F}_{0} & := \begin{pmatrix} \phi^{\mathcal E}_{0} \otimes_{\mathcal O_X} 1_{\mathcal F_{0}} & 1_{\mathcal E_0} \otimes_{\mathcal O_X} \phi^{\mathcal F}_{0} \\ - 1_{\mathcal E_{-1}} \otimes_{\mathcal O_X} \phi^{\mathcal F}_{-1}  & \phi^{\mathcal E}_{-1} \otimes_{\mathcal O_X} 1_{\mathcal F_{-1}} \end{pmatrix} \\
 \phi^{\mathcal E \otimes_{\mathcal O_X} \mathcal F}_{-1} & := \begin{pmatrix} \phi^{\mathcal E}_{-1} \otimes_{\mathcal O_X} 1_{\mathcal F_{0}} & - 1_{\mathcal E_{-1}} \otimes_{\mathcal O_X} \phi^{\mathcal F}_{0}\otimes_{\mathcal O_X} \mathcal L \\ 1_{\mathcal E_{0}} \otimes_{\mathcal O_X} \phi^{\mathcal F}_{-1}  & \phi^{\mathcal E}_{0}\otimes_{\mathcal O_X} \mathcal L \otimes_{\mathcal O_X} 1_{\mathcal F_{-1}} \end{pmatrix}
 \end{align*}
 Given $\alpha: \mathcal E \to \mathcal E'[r]$ and $\beta: \mathcal F  \to \mathcal F'[s]$, one has
 \begin{displaymath}
  \alpha \otimes_{\mathcal O_X} \beta : \mathcal E \otimes_{\mathcal O_X} \mathcal E' \to \mathcal F \otimes_{\mathcal O_X} \mathcal F'[r+s]
 \end{displaymath}
 defined by
 \begin{displaymath}
  \alpha \otimes_{\mathcal O_X} \beta = \begin{cases} 
                                         \left( \begin{pmatrix} \alpha_{-1} \otimes \beta_0 & 0 \\ 0 & \alpha_0 \otimes \beta_{-1} \end{pmatrix}, \begin{pmatrix} \alpha_0 \otimes \beta_0 & 0 \\ 0 & \alpha_{-1} \otimes \beta_{-1} \otimes \mathcal L \end{pmatrix} \right) & r,s \text{ even} \\
                                         \left( \begin{pmatrix} 0 & \alpha_{0} \otimes \beta_{-1} \\ -\alpha_{-1} \otimes \beta_0 & 0 \end{pmatrix}, \begin{pmatrix} 0 & -\alpha_{-1} \otimes \beta_{-1} \otimes \mathcal L \\ \alpha_{0} \otimes \beta_{0} & 0 \end{pmatrix} \right) & r \text{ even},s \text{ odd} \\
                                         \left( \begin{pmatrix} \alpha_{-1} \otimes \beta_0 & 0 \\ 0 & \alpha_0 \otimes \beta_{-1} \end{pmatrix}, \begin{pmatrix} \alpha_0 \otimes \beta_0 & 0 \\ 0 & \alpha_{-1} \otimes \beta_{-1} \otimes \mathcal L \end{pmatrix} \right) & r \text{ odd},s \text{ even} \\
                                         \left( \begin{pmatrix} 0 & - \alpha_{0} \otimes \beta_{-1} \\ \alpha_{-1} \otimes \beta_0 & 0 \end{pmatrix}, \begin{pmatrix} 0 & \alpha_{-1} \otimes \beta_{-1} \otimes \mathcal L \\ -\alpha_{0} \otimes \beta_{0} & 0 \end{pmatrix} \right) & r,s \text{ odd}
                                        \end{cases}
 \end{displaymath}

For a locally-free factorization, $\mathcal V$, the functor,
\begin{displaymath}
 \mathcal V \otimes_{\mathcal O_X} \bullet : [\mathsf{Fact}(X,G,v)] \to [\mathsf{Fact}(X,G,w+v)],
\end{displaymath}
preserves acyclic complexes and descends to a functor. 
\begin{displaymath}
 \mathcal V \otimes_{\mathcal O_X} \bullet : \op{D}^{\op{abs}}[\mathsf{Fact}(X,G,v)] \to \op{D}^{\op{abs}}[\mathsf{Fact}(X,G,w+v)].
\end{displaymath}
For $\mathcal E \in \mathsf{Fact}(X,G,w)$, we define
\begin{displaymath}
 \mathcal E \overset{\mathbf{L}}{\otimes}_{\mathcal O_X} \bullet := \mathcal V \otimes_{\mathcal O_X} \bullet.
\end{displaymath}
where $\mathcal V$ is a locally-free factorization quasi-isomorphic to $\mathcal E$.
\end{definition}

\begin{lemma} \label{lemma: derived tensor is well-defined}
 The functor, 
 \begin{displaymath}
  \mathcal E \overset{\mathbf{L}}{\otimes}_{\mathcal O_X} \bullet : \op{D}^{\op{abs}}[\mathsf{Fact}(X,G,v)] \to \op{D}^{\op{abs}}[\mathsf{Fact}(X,G,w+v)]
 \end{displaymath}
 is well-defined, i.e.\ it does not depend on the choice of representative of the quasi-isomorphism class. 
\end{lemma}

\begin{proof}
 By Proposition~\ref{prop: projective enhancement}, inclusion of $\mathsf{Vect}(X,G,v)$ into $\mathsf{Fact}(X,G,v)$ induces an equivalence
 \begin{displaymath}
  \dabs [\mathsf{Vect}(X,G,v)] \to \dabs [\mathsf{Fact}(X,G,v)].
 \end{displaymath}
 We may therefore view the derived functor on the absolute derived category of locally-free factorizations,
 \begin{displaymath}
  \mathcal E \overset{\mathbf{L}}{\otimes}_{\mathcal O_X} \bullet : \dabs [\mathsf{Vect}(X,G,v)] \to \dabs [\mathsf{Vect}(X,G,w+v)].
 \end{displaymath}
 Since tensoring with a locally-free sheaf is exact, tensoring with a locally-free factorization preserves acyclic factorizations and we have natural quasi-isomorphisms
 \begin{displaymath}
  \mathcal E \otimes_{\mathcal O_X} \mathcal W \cong \mathcal V \otimes_{\mathcal O_X} \mathcal W =: \mathcal E \overset{\mathbf{L}}{\otimes}_{\mathcal O_X} \mathcal W.
 \end{displaymath}
 when $\mathcal W$ is locally-free and $\mathcal V$ is locally-free and quasi-isomorphic to $\mathcal E$.
\end{proof}

\begin{definition}
 Let $X$ be a smooth variety equipped with an action of $G$. Assume we have $w \in \Gamma(X,\mathcal L)^G$. Let $p: X \to \op{Spec} k$ be the structure morphism. Let $(\mathcal C,d)$ be a bounded complex of vector spaces. Let $\mathcal E \in \mathsf{Fact}(X,G,w)$. Define a factorization $\mathcal E \otimes_k \mathcal C$ by
 \begin{displaymath}
  \mathcal E \otimes_k \mathcal C := \mathcal E \otimes_{\mathcal O_X} p^*(\Upsilon \mathcal C).
 \end{displaymath}
 Denote the corresponding functor by
 \begin{displaymath}
  \mathcal E \otimes_k \bullet : \mathsf{Qcoh}^{\op{b}}(\op{Spec} k) \to \mathsf{Fact}(X,G,w).
 \end{displaymath}
 This functor takes an exact chain complex to an acyclic factorization in $\mathsf{Fact}(X,G,w)$. Thus, it descends to a functor
 \begin{displaymath}
  \mathcal E \otimes_k \bullet: \dbqcoh{\op{Spec} k} \to \dabs [\mathsf{Fact}(X,G,w)]. 
 \end{displaymath}
\end{definition}

Next we give a version of sheaf Hom.

\begin{definition}
 Let $X$ be a smooth variety equipped with an action of $G$. Assume we have sections, $w, v \in \Gamma(X,\mathcal L)^G$. We define a dg-functor,
 \begin{displaymath}
  \mathcal Hom_X  : \mathsf{Fact}(X,G,w)^{\op{op}} \otimes_k \mathsf{Fact}(X,G,v) \to \mathsf{Fact}(X,G,v-w),
 \end{displaymath}
 by setting
 \begin{align*}
 \mathcal Hom_X (\mathcal E, \mathcal F)_{-1} & := \mathcal Hom_X (\mathcal E_{-1}, \mathcal F_0)\otimes_{\mathcal O_X} \mathcal L^{-1} \oplus \mathcal Hom_X (\mathcal E_0, \mathcal F_{-1}) \\
 \mathcal Hom_X (\mathcal E, \mathcal F)_{0} & := \mathcal Hom_X (\mathcal E_0, \mathcal F_0) \oplus \mathcal Hom_X (\mathcal E_{-1}, \mathcal F_{-1}) \\
 \phi^{\mathcal Hom_X (\mathcal E, \mathcal F)}_{0} & := \begin{pmatrix} (\bullet) \circ \phi_{-1}^{\mathcal E} & \phi^{\mathcal F}_0 \circ (\bullet) \\  (\phi_{-1}^{\mathcal F} \otimes_{\mathcal O_X} \mathcal L^{-1}) \circ (\bullet) & (\bullet) \circ \phi^{\mathcal E}_{0} \end{pmatrix} \\
 \phi^{\mathcal Hom_X (\mathcal E, \mathcal F)}_{-1} & := \begin{pmatrix} -(\bullet) \circ \phi_{0}^{\mathcal E} & \phi^{\mathcal F}_0 \circ (\bullet) \\ \phi_{-1}^{\mathcal F} \circ (\bullet)
   & -(\bullet) \circ (\phi^{\mathcal E}_{-1} \otimes_{\mathcal O_X} \mathcal L^{-1}) \end{pmatrix}
\end{align*}
 Given $\alpha: \mathcal E \to \mathcal E'[r]$ and $\beta: \mathcal F  \to \mathcal F'[s]$, one has
 \begin{displaymath}
  \mathcal Hom_X(\alpha, \beta) : \mathcal Hom_X(\mathcal E',\mathcal F) \to \mathcal Hom_X(\mathcal E,\mathcal F')[r+s]
 \end{displaymath}
 defined by
 \begin{displaymath}
    \begin{pmatrix} \beta_0 \circ (\bullet) \circ (\alpha_{-1} \otimes \mathcal L^{-l+1}) & 0 \\ 0 & \beta_{-1} \circ (\bullet) \circ (\alpha_0 \otimes \mathcal L^{-l}) \end{pmatrix}, 
    \begin{pmatrix} \beta_0 \circ (\bullet) \circ (\alpha_0 \otimes \mathcal L^{-l}) & 0 \\ 0 & \beta_{-1} \circ (\bullet) \circ (\alpha_{-1} \otimes \mathcal L^{-l}) \end{pmatrix} 
 \end{displaymath}
 if $r=2l,s=2j$,
 \begin{displaymath}
  \begin{pmatrix} 0 & \beta_{-1} \circ (\bullet) \circ (\alpha_{0} \otimes \mathcal L^{-l}) \\ -\beta_{0} \circ (\bullet) \circ (\alpha_{-1} \otimes \mathcal L^{-l+1}) & 0 \end{pmatrix}, 
  \begin{pmatrix} 0 & -\beta_{-1} \circ (\bullet) \circ (\alpha_{-1} \otimes \mathcal L^{-l}) \\ \beta_0 \circ (\bullet) \circ (\alpha_0 \otimes \mathcal L^{-l}) & 0 \end{pmatrix}
 \end{displaymath}
 if $r=2l,s=2j+1$,
 \begin{displaymath}
  \begin{pmatrix} -\beta_0 \circ (\bullet) \circ (\alpha_0 \otimes \mathcal L^{-l}) & 0 \\ 0 & \beta_{-1} \circ (\bullet) \circ (\alpha_{-1} \otimes \mathcal L^{-l}) \end{pmatrix}, 
  \begin{pmatrix} \beta_0 \circ (\bullet) \circ (\alpha_{-1} \otimes \mathcal L^{-l}) & 0 \\ 0 & - \beta_{-1} \circ (\bullet) \circ (\alpha_{0} \otimes \mathcal L^{-l-1}) \end{pmatrix}
 \end{displaymath}
 if $r=2l+1,s=2j$, and
 \begin{displaymath}
  \begin{pmatrix} 0 & \beta_{-1} \circ (\bullet) \circ (\alpha_{-1} \otimes \mathcal L^{-l}) \\ \beta_0 \circ (\bullet) \circ (\alpha_{0} \otimes \mathcal L^{-l}) & 0 \end{pmatrix}, 
  \begin{pmatrix} 0 & \beta_{-1} \circ (\bullet) \circ (\alpha_0 \otimes \mathcal L^{-l-1}) \\ \beta_{0} \circ (\bullet) \circ (\alpha_{-1} \otimes \mathcal L^{-l}) & 0 \end{pmatrix}
 \end{displaymath}
 if $r=2l+1,s=2j+1$.

For a locally-free factorization, $\mathcal V$, the functor,
\begin{displaymath}
 \mathcal Hom_X(\mathcal V, \bullet) : [\mathsf{Fact}(X,G,v)] \to [\mathsf{Fact}(X,G,v-w)],
\end{displaymath}
preserves acyclic complexes and descends to a functor. 
\begin{displaymath}
 \mathcal Hom_X(\mathcal V, \bullet) : \op{D}^{\op{abs}}[\mathsf{Fact}(X,G,v)] \to \op{D}^{\op{abs}}[\mathsf{Fact}(X,G,v-w)].
\end{displaymath}
For $\mathcal E \in \mathsf{Fact}(X,G,w)$, we define
\begin{displaymath}
 \mathbf{R}\mathcal Hom_X(\mathcal E, \bullet) := \mathcal Hom_X(\mathcal V, \bullet).
\end{displaymath}
where $\mathcal V$ is a locally-free factorization quasi-isomorphic to $\mathcal E$.
\end{definition}

\begin{lemma} \label{lemma: derived Hom is well-defined}
 The functor, 
 \begin{displaymath}
  \mathbf{R}\mathcal Hom_X(\mathcal E, \bullet) : \op{D}^{\op{abs}}[\mathsf{Fact}(X,G,v)] \to \op{D}^{\op{abs}}[\mathsf{Fact}(X,G,v-w)]
 \end{displaymath}
 is well-defined, i.e.\ it does not depend on the choice of representative of the quasi-isomorphism class. 
\end{lemma}

\begin{proof}
 The proof is completely analogous to that of Lemma~\ref{lemma: derived tensor is well-defined} and is therefore suppressed. 
\end{proof}

\begin{proposition} \label{proposition: Hom-tensor adjunction}
 Let $X$ be a smooth variety equipped with an action of an affine algebraic group $G$. Let $w,v \in \Gamma(X,\mathcal L)^G$ be invariant sections of an invertible equivariant sheaf, $\mathcal L$. For $\mathcal E \in \mathsf{Fact}(X,G,w), \mathcal F \in \mathsf{Fact}(X,G,v)$ and $\mathcal G \in \mathsf{Fact}(X,G,w+v)$, there are natural isomorphisms
 \begin{displaymath}
  \op{Hom}_{\mathsf{Fact}(X,G,w+v)}(\mathcal E \otimes_{\mathcal O_X} \mathcal F, \mathcal G) \cong \op{Hom}_{\mathsf{Fact}(X,G,w)}(\mathcal E, \mathcal Hom_X(\mathcal F,\mathcal G)).
 \end{displaymath}
\end{proposition}

\begin{proof}
 We first check this for $\op{Hom}^0$. We have
 \begin{gather*}
  \op{Hom}^0_{\mathsf{Fact}(X,G,w+v)}(\mathcal E \otimes_{\mathcal O_X} \mathcal F, \mathcal G) := \\
  \op{Hom}_{\op{Qcoh}_G X}( \left(\mathcal E \otimes_{\mathcal O_X} \mathcal F \right)_{-1}, \mathcal G_{-1} ) \oplus \op{Hom}_{\op{Qcoh}_G X}( \left(\mathcal E \otimes_{\mathcal O_X} \mathcal F \right)_{0}, \mathcal G_{0} ) \\
   := \op{Hom}_{\op{Qcoh}_G X}( \mathcal E_{-1} \otimes_{\mathcal O_X} \mathcal F_{0}, \mathcal G_{-1} ) \oplus \op{Hom}_{\op{Qcoh}_G X}( \mathcal E_{0} \otimes_{\mathcal O_X} \mathcal F_{-1}, \mathcal G_{-1} ) \oplus \\
   \op{Hom}_{\op{Qcoh}_G X}( \mathcal E_0 \otimes_{\mathcal O_X} \mathcal F_0, \mathcal G_{0} ) \oplus \op{Hom}_{\op{Qcoh}_G X}( \mathcal E_{-1} \otimes_{\mathcal O_X} \mathcal F_{-1} \otimes_{\mathcal O_X} \mathcal L, \mathcal G_{0} ).
 \end{gather*}
 Applying Hom-tensor adjunction for $G$-equivariant sheaves, we have an isomorphism
 \begin{gather*}
  \cong \op{Hom}_{\op{Qcoh}_G X}( \mathcal E_{-1}, \mathcal Hom_X(\mathcal F_{0}, \mathcal G_{-1}) ) \oplus \op{Hom}_{\op{Qcoh}_G X}( \mathcal E_{0}, \mathcal Hom_X ( \mathcal F_{-1}, \mathcal G_{-1}) ) \oplus \\
   \op{Hom}_{\op{Qcoh}_G X}( \mathcal E_0, \mathcal Hom_X(\mathcal F_0, \mathcal G_{0}) ) \oplus \op{Hom}_{\op{Qcoh}_G X}( \mathcal E_{-1}, \mathcal Hom_X(\mathcal F_{-1} \otimes_{\mathcal O_X} \mathcal L, \mathcal G_{0}) ) \\
   =: \op{Hom}_{\op{Qcoh}_G X}(\mathcal E_{-1}, \mathcal Hom_X(\mathcal F,\mathcal G)_{-1}) \oplus \op{Hom}_{\op{Qcoh}_G X}(\mathcal E_{0}, \mathcal Hom_X(\mathcal F,\mathcal G)_{0}) \\
   =: \op{Hom}^0_{\mathsf{Fact}(X,G,w+v)}(\mathcal E, \mathcal Hom_X(\mathcal F, \mathcal G)).
 \end{gather*}
 Since $\op{Hom}^0(\bullet, \bullet [n]) = \op{Hom}^n(\bullet,\bullet)$, this defines the natural transformation on the whole morphism space of $\mathsf{Fact}$. It is straightforward to check that these maps commute with the differentials.
\end{proof}

\begin{corollary} \label{corollary: derived Hom-tensor adjunction}
 We have an adjoint pair of derived functors
 \begin{align*}
  \bullet \overset{\mathbf{L}}{\otimes}_{\mathcal O_X} \mathcal F & : \dabs[\mathsf{Fact}(X,G,w)] \to \dabs[\mathsf{Fact}(X,G,w+v)] \\
  \mathbf{R}\mathcal Hom_X(\mathcal F, \bullet) & : \dabs[\mathsf{Fact}(X,G,w+v)] \to \dabs[\mathsf{Fact}(X,G,w)].
 \end{align*}
\end{corollary}

\begin{proof}
 This follows by replacing the first entry in a morphism space by a locally-free factorization and the second by an injective factorization and applying Proposition~\ref{proposition: Hom-tensor adjunction}. 
\end{proof}

\begin{definition}
 We will focus on a particular case of sheaf-Hom. Consider the factorization $\Upsilon \mathcal O_X$ of $0 \in \Gamma(X,\mathcal L)^G$. Denote it by $\mathcal O_X$. We get functors
 \begin{align*}
  (\bullet)^{\vee} & := \mathcal Hom_X(\bullet , \mathcal O_X) : \mathsf{Fact}(X,G,w)^{\op{op}} \to \mathsf{Fact}(X,G,-w) \\
  (\bullet)^{\mathbf{L} \vee} & := \mathbf{R} \mathcal Hom_X(\bullet , \mathcal O_X) : \dabs[\mathsf{Fact}(X,G,w)]^{\op{op}} \to \dabs[\mathsf{Fact}(X,G,-w)].
 \end{align*}
\end{definition}

\begin{lemma} \label{lemma: double dual is ok for coherent factorizations}
 The functor,
 \begin{displaymath}
 (\bullet)^{\mathbf{L} \vee}  : \dabs[\mathsf{fact}(X,G,w)]^{\op{op}} \to \dabs[\mathsf{fact}(X,G,-w)],
 \end{displaymath}
 is an equivalence.
\end{lemma}

\begin{proof}
 It is simple to check that for a locally-free factorization of finite rank, $\mathcal F$, we have a natural isomorphism
 \begin{displaymath}
  \mathcal F \cong \mathcal F^{\vee \vee}.
 \end{displaymath}
 Any object of $\dabs[\mathsf{fact}(X,G,w)^{\op{op}}]$ is quasi-isomorphic to a locally-free factorization of finite rank by Proposition~\ref{prop: projective enhancement}.
\end{proof}

\begin{lemma} \label{lemma: sheaf hom is tensor with dual for finite rank fact}
 Let $\mathcal V \in \mathsf{vect}(X,G,w)$. Then, there is an isomorphism
 \begin{displaymath}
  \mathcal V^{\vee} \otimes_{\mathcal O_X} \bullet \cong \mathcal Hom_X(\mathcal V, \bullet).
 \end{displaymath}
 Similarly, for $\mathcal E \in \mathsf{fact}(X,G,w)$, there is an isomorphism
 \begin{displaymath}
  \mathcal E^{\mathbf{L} \vee} \overset{\mathbf{L}}{\otimes}_{\mathcal O_X} \bullet \cong \mathbf{R}\mathcal Hom_X(\mathcal E, \bullet).
 \end{displaymath}
\end{lemma}

\begin{proof}
 The first isomorphism follows immediately from inspection of the definitions. The second is a quick consequence of the first. 
\end{proof}

Assume we have two smooth varieties, $X$ and $Y$, both carrying a $G$-action, and a morphism, $f: X \to Y$. Let $w \in \Gamma(Y,\mathcal L)^G$. We have pull-back and pushforward functors.

\begin{definition}
 \begin{align*}
  f^*: \mathsf{Fact}(Y,G,w) & \to \mathsf{Fact}(X,G,f^*w) \\
  (\mathcal E_{-1},\mathcal E_0,\phi^{\mathcal E}_{-1},\phi^{\mathcal E}_0) & \mapsto (f^* \mathcal E_{-1},f^* \mathcal E_0,f^* \phi^{\mathcal E}_{-1},f^* \phi^{\mathcal E}_0)
 \end{align*}
 and 
 \begin{align*}
  f_*: \mathsf{Fact}(X,G,f^*w) & \to \mathsf{Fact}(X,G,w) \\
  (\mathcal F_{-1},\mathcal F_0,\phi^{\mathcal F}_{-1},\phi^{\mathcal F}_0) & \mapsto (f_* \mathcal F_{-1},f_* \mathcal F_0,f_* \phi^{\mathcal F}_{-1},f_* \phi^{\mathcal F}_0).
 \end{align*}
 Note that by the projection formula $f_*(\mathcal F \otimes_{\mathcal O_X} f^*\mathcal L) \cong (f_*\mathcal F) \otimes_{\mathcal O_X} \mathcal L $ under which $f_*(f^*w)$ corresponds to $w$ so this is well-defined.
\end{definition}

\begin{definition}
 For a factorization, $\mathcal E$, of $0 \in \Gamma(X,\mathcal L)^G$. We let the \newterm{unfolding} of $\mathcal E$ be the complex $\Game \mathcal E \in \mathsf{Qcoh}_G(X)$ with
 \begin{displaymath}
  (\Game \mathcal E)_j = \begin{cases} \mathcal E_{-1} \otimes \mathcal L^l & j = 2l-1 \\ \mathcal E_0 \otimes \mathcal L^l & j = 2l. \end{cases}
 \end{displaymath}

 We shall also use a slightly different version of the pushforward. Let $X$ be equipped with an action of $G$ and consider the structure morphism, $p: X \to \op{Spec} k$. It is $G$-equivariant if we equip $\op{Spec} k$ with the trivial action. Then, we have a pushforward
 \begin{align*}
  p_*: \mathsf{Fact}(X,G,0) & \to \mathsf{Qcoh}_G (\op{Spec} k) \\
  \mathcal F & \mapsto p_* (\Game \mathcal F)
 \end{align*}
 where $p_*: \mathsf{Qcoh}_G(X) \to \mathsf{Qcoh}_G(\op{Spec} k)$ is the usual pushforward of equivariant sheaves.
\end{definition}

\begin{lemma} \label{lemma: pushdown sheaf Hom = Hom}
 Let $\mathcal E, \mathcal F \in \mathsf{Fact}(X,G,w)$. Then, we have an isomorphism of complexes
 \begin{displaymath}
  \left(p_* \mathcal{H}om_X(\mathcal E, \mathcal F)\right)^G \cong \op{Hom}_{\mathsf{Fact}(X,G,w)}(\mathcal E,\mathcal F).
 \end{displaymath}
\end{lemma}

\begin{proof}
 This is immediate from the definitions.
\end{proof}

\begin{lemma} \label{lemma: pull-push adjunction for factorizations}
 Push-forward, $f_*$, is right adjoint to pull-back, $f^*$. 
\end{lemma}

\begin{proof}
 Applying the standard adjunction between $f^*$ and $f_*$ for equivariant sheaves to the components of the factorization gives the statement.
\end{proof}

We also define their derived analogs.

\begin{definition}
 Define the left-derived functor of $f^*$ by
 \begin{align*}
  \mathbf{L}f^*: \op{D}^{\op{abs}}[\mathsf{Fact}(Y,G,w)] & \to \op{D}^{\op{abs}}[\mathsf{Fact}(X,G,f^*w)] \\
  \mathcal E & \mapsto f^*\mathcal V
 \end{align*}
 where $\mathcal V$ is a factorization with locally-free components quasi-isomorphic to $\mathcal E$. 
 
 Define the right-derived functor of $f_*$ by
 \begin{align*}
  \mathbf{R}f_* : \op{D}^{\op{abs}}[\mathsf{Fact}(X,G,f^*w)] & \to \op{D}^{\op{abs}}[\mathsf{Fact}(X,G,w)] \\
  \mathcal E & \mapsto f_* \mathcal I
 \end{align*}
 where $\mathcal I$ is a factorization with injective components quasi-isomorphic to $\mathcal E$.
\end{definition}

\begin{lemma} \label{lemma: push pull well-defined}
 Both $\mathbf{L}f^*$ and $\mathbf{R}f_*$ are well-defined, i.e.\ they do not depend on the choices of representatives of a quasi-isomorphism class.
\end{lemma}

\begin{proof}
 The derived push-forward is well-defined by Proposition~\ref{prop: injective enhancement} since $[\mathsf{Inj}(X,G,f^*w)] \cong \op{D}^{\op{abs}}[\mathsf{Fact}(X,G,f^*w)]$.
 
 The derived pull-back functor, $f^*$, is well-defined by Proposition~\ref{prop: projective enhancement} since $\op{D}^{\op{abs}}[\mathsf{Vect}(X,G,w)] \cong \op{D}^{\op{abs}}[\mathsf{Fact}(X,G,w)]$ and $f^*$ preserves acyclic complexes of locally-free sheaves.
\end{proof}

\begin{lemma} \label{lemma: derived projection formula for fact}
 For each, $\mathcal E \in \dabs[\mathsf{Fact}(Y,G,w)]$ and $\mathcal F \in \dabs[\mathsf{Fact}(X,G,f^*w)]$, there is a natural isomorphism 
 \begin{displaymath}
  \mathbf{R}f_* \mathcal F \overset{\mathbf{L}}{\otimes}_{\mathcal O_Y} \mathcal E \cong \mathbf{R}f_*( \mathcal F \overset{\mathbf{L}}{\otimes}_{\mathcal O_X} \mathbf{L}f^* \mathcal E).
 \end{displaymath}
\end{lemma}

\begin{proof}
 This follows from replacing $\mathcal E$ by a factorization with locally-free components, $\mathcal F$ by a factorization with injective components, and applying the projection formula, Lemma~\ref{lemma: projection formula for equivariant pushforward}, to the components of the factorizations.
\end{proof}

We also have an extension of pullback to allow for a group homomorphism.

\begin{definition}
 Assume we have two smooth varieties, $X$ and $Y$, and two affine algebraic groups, $G$ and $H$. Let $\psi: G \to H$ be a homomorphism and assume that $G$ acts on $X$ while $H$ acts on $Y$. Let $f: X \to Y$ be a $\psi$-equivariant morphism. Let $w \in \Gamma(Y,\mathcal L)^H$ so that $f^*w \in \Gamma(X, f^*\mathcal L)^G$. We have a functor,
 \begin{align*}
  f^*: \mathsf{Fact}(Y,H,w) & \to \mathsf{Fact}(X,G,f^*w) \\
  (\mathcal E_{-1},\mathcal E_0,\phi^{\mathcal E}_{-1},\phi^{\mathcal E}_0) & \mapsto (f^* \mathcal E_{-1},f^* \mathcal E_0,f^* \phi^{\mathcal E}_{-1},f^* \phi^{\mathcal E}_0).
 \end{align*}
 
 The left-derived functor of $f^*$ is
 \begin{align*}
  \mathbf{L}f^*: \op{D}^{\op{abs}}[\mathsf{Fact}(Y,H,w)] & \to \op{D}^{\op{abs}}[\mathsf{Fact}(X,G,f^*w)] \\
  \mathcal E & \mapsto f^*\mathcal V
 \end{align*}
 where $\mathcal V$ is a factorization with locally-free components quasi-isomorphic to $\mathcal E$.
\end{definition}

\begin{lemma} \label{lemma: general pull back well-defined}
 The functor, $\mathbf{L}f^*$, is well-defined, i.e.\ it does not depend on the choice of representatives of a quasi-isomorphism class.
\end{lemma}

\begin{proof}
 The proof is completely analogous to that of Lemma~\ref{lemma: push pull well-defined}.
\end{proof}

We also extend the restriction and induction functors. 

\begin{definition}
 Let $X$ be a smooth variety equipped with an action of an affine algebraic group, $G$. Let $w \in \Gamma(G,\mathcal L)^G$. Let $\psi: H \to G$ be a closed subgroup of $G$. 
 \begin{align*}
  \op{Res}^G_H: \mathsf{Fact}(X,G,w) & \to \mathsf{Fact}(X,H,w) \\
  (\mathcal E_{-1},\mathcal E_0,\phi^{\mathcal E}_{-1},\phi^{\mathcal E}_0) & \mapsto (\op{Res}^G_H \mathcal E_{-1},\op{Res}^G_H \mathcal E_0,\op{Res}^G_H \phi^{\mathcal E}_{-1},\op{Res}^G_H \phi^{\mathcal E}_0)
 \end{align*}
 and 
 \begin{align*}
  \op{Ind}^G_H: \mathsf{Fact}(X,H,w) & \to \mathsf{Fact}(X,G,w) \\
  (\mathcal F_{-1},\mathcal F_0,\phi^{\mathcal F}_{-1},\phi^{\mathcal F}_0) & \mapsto (\op{Ind}^G_H\mathcal F_{-1},\op{Ind}^G_H\mathcal F_0,\op{Ind}^G_H\phi^{\mathcal F}_{-1},\op{Ind}^G_H\phi^{\mathcal F}_0).
 \end{align*}
 The action on morphisms is clear.
 
 The restriction functor, $\op{Res}^G_H$, is exact so it immediately descends to
 \begin{displaymath}
  \op{Res}^G_H: \op{D}^{\op{abs}}[\mathsf{Fact}(X,G,w)] \to \op{D}^{\op{abs}}[\mathsf{Fact}(X,H,w)].
 \end{displaymath}

 The induction functor, $\op{Ind}^G_H$, is left-exact so we have its right-derived functor,
 \begin{align*}
  \mathbf{R}\!\op{Ind}^G_H : \op{D}^{\op{abs}}[\mathsf{Fact}(X,H,w)] & \to \op{D}^{\op{abs}}[\mathsf{Fact}(X,G,w)] \\
  \mathcal E & \mapsto f_* \mathcal I
 \end{align*}
 where $\mathcal I$ is a factorization with injective components quasi-isomorphic to $\mathcal E$.
\end{definition}

\begin{lemma} \label{lemma: Res-Ind fact adjunction}
 The functor, $\op{Res}^G_H$, is left adjoint to the functor, $\op{Ind}^G_H$.
\end{lemma}

\begin{proof}
 This is an immediate consequence of Lemma~\ref{lemma: adjointness restriction induction}.
\end{proof}

\begin{corollary} \label{corollary: Res-Ind derived fact adjunction}
 We have an adjoint pair of functors,
 \begin{align*}
  \op{Res}^G_H & : \op{D}^{\op{abs}}[\mathsf{Fact}(X,G,w)] \to \op{D}^{\op{abs}}[\mathsf{Fact}(X,H,w)] \\
  \mathbf{R}\!\op{Ind}^G_H & : \op{D}^{\op{abs}}[\mathsf{Fact}(X,H,w)] \to \op{D}^{\op{abs}}[\mathsf{Fact}(X,G,w)].
 \end{align*}
\end{corollary}

\begin{proof}
 This is an immediate consequence of Lemma~\ref{lemma: Res-Ind fact adjunction}.
\end{proof}

\begin{lemma} \label{lemma: derived projection formula for induction on fact}
 For each, $\mathcal E \in \dabs[\mathsf{Fact}(X,G,w)]$ and $\mathcal F \in \dabs[\mathsf{Fact}(X,H,w)]$, there is a natural isomorphism
 \begin{displaymath}
  \mathbf{R}\!\op{Ind}^G_H \mathcal F \overset{\mathbf{L}}{\otimes}_{\mathcal O_Y} \mathcal E \cong \mathbf{R}\!\op{Ind}^G_H( \mathcal F \overset{\mathbf{L}}{\otimes}_{\mathcal O_X} \op{Res}^G_H \mathcal E).
 \end{displaymath}
\end{lemma}

\begin{proof}
 This follows from replacing $\mathcal F$ by a factorization with injective components and applying the projection formula, Lemma~\ref{lemma: facts about restriction induction}, to the components of the factorizations.
\end{proof}

Finally, we extend the functor of invariants.

\begin{definition}
 Let $N$ be a closed normal subgroup of $G$. Let $X$ be a smooth variety equipped with an action of $G$ on which $N$ acts trivially. Let $\mathcal L$ be an invertible $G/N$-equivariant sheaf. Note that $\mathcal L$ inherits a $G$-equivariant structure. Consider a section $w \in \Gamma(X, \mathcal L)^G \cong \Gamma(X, \mathcal L)^{G/N}$. We define
 \begin{align*}
  (\bullet)^N : \mathsf{Fact}(X,G,w) & \to \mathsf{Fact}(X,G/N,w) \\
  (\mathcal E_{-1},\mathcal E_0,\phi^{\mathcal E}_{-1},\phi^{\mathcal E}_0) & \mapsto (\mathcal E_{-1}^N,\mathcal E_0^N,(\phi^{\mathcal E}_{-1})^N,(\phi^{\mathcal E}_0)^N).
 \end{align*}
 The derived functor of invariants is
 \begin{align*}
  (\bullet)^{\mathbf{R} N} : \dabs [\mathsf{Fact}(X,G,w)] & \to \dabs [\mathsf{Fact}(X,G/N,w)] \\
  \mathcal E & \mapsto \mathcal I^N
 \end{align*}
 where $\mathcal I$ is a factorization that has injective components and that is quasi-isomorphic to $\mathcal E$.
\end{definition}

\begin{definition}
 Let $\mathcal L$ be an invertible equivariant sheaf on $X$ and let $w \in \Gamma(X,\mathcal L)^G$. Let 
 \begin{displaymath}
  \op{V}(\mathcal L) := \underline{\op{Spec}}_X( \op{Sym} \mathcal L ) 
 \end{displaymath}
 denote the geometric vector bundle associated to $\mathcal L$. It carries an action of $G \times \mathbb{G}_m$ where $G$ acts via the equivariant structure on $\mathcal L$ and $\mathbb{G}_m$ dilates the fibers of the bundle. The section, $w$, defines a regular function, $f_w \in \Gamma(\op{V}(\mathcal L),\mathcal O_{\op{V}(\mathcal L)}(1))^{G \times \mathbb{G}_m}$ where $(1)$ denotes the projection character, $G \times \mathbb{G}_m \to \mathbb{G}_m$. Finally, let $\op{U}(\mathcal L)$ denote the complement of the zero section in $\op{V}(\mathcal L)$. Let $\pi: \op{U}(\mathcal L) \to X$ denote the projection. It is equivariant with respect to the projection, $G \times \mathbb{G}_m \to \mathbb{G}_m$.
\end{definition}

\begin{lemma} \label{lemma: ascent to trivialize L}
 The pull back functor,
 \begin{displaymath}
  \pi^*: \op{Qcoh}_G X \to \op{Qcoh}_{G \times \mathbb{G}_m} \op{U}(\mathcal L),
 \end{displaymath}
 is an equivalence. Moreover, $\pi^*$ induces equivalences between subcategories of coherent and locally-free equivariant sheaves.
\end{lemma}

\begin{proof}
 The variety, $\op{U}(\mathcal L)$, is a $\mathbb{G}_m$-torsor over $X$. Thus, the fppf quotient of $\op{U}(\mathcal L)$ by $\mathbb{G}_m$ is $X$. The statement of the lemma is a consequence of faithfully-flat descent.  In other words, the global quotient stack $[\op{U}(\mathcal L) / \gm]$ is represented by $X$, and therefore they have the same sheaf theory.
\end{proof}

\begin{lemma} \label{lemma: trivialize L - dg equiv}
 The pull back functor,
 \begin{displaymath}
  \pi^*: \mathsf{Fact}(X,G,w) \to \mathsf{Fact}(\op{U}(\mathcal L),G \times \mathbb{G}_m, f_w),
 \end{displaymath}
 is an equivalence of dg-categories. Moreover, $\pi^*$ restricts to equivalences,
 \begin{align*}
  \pi^* & : \mathsf{Inj}(X,G,w) \to \mathsf{Inj}(\op{U}(\mathcal L),G \times \mathbb{G}_m, f_w), \\
  \pi^* & : \mathsf{Vect}(X,G,w) \to \mathsf{Vect}(\op{U}(\mathcal L),G \times \mathbb{G}_m, f_w), \\
  \pi^* & : \mathsf{fact}(X,G,w) \to \mathsf{fact}(\op{U}(\mathcal L),G \times \mathbb{G}_m, f_w), \\
  \pi^* & : \mathsf{vect}(X,G,w) \to \mathsf{vect}(\op{U}(\mathcal L),G \times \mathbb{G}_m, f_w).
 \end{align*}
\end{lemma}

\begin{proof}
 This is an immediate consequence of Lemma~\ref{lemma: ascent to trivialize L}. 
\end{proof}

The following definitions seem to have no natural extension to the case of general equivariant line bundles. They will be essential later in the paper.

\begin{definition}
 Let $X$ and $Y$ be smooth varieties and let $w \in \Gamma(X,\mathcal O_X)$ and $v \in \Gamma(Y,\mathcal O_Y)$. We set
 \begin{displaymath}
  w \boxplus v := w \otimes 1 + 1 \otimes v \in \Gamma(X,\mathcal O_X) \otimes_k \Gamma(Y,\mathcal O_Y) \cong \Gamma(X \times Y, \mathcal O_{X \times Y}).
 \end{displaymath}
\end{definition}

We will have to deal with two potentials, $w, v \in \Gamma(X,\mathcal O_X)$, that are semi-invariant with respect to different characters of different groups. The largest group for which $w \boxplus v$ is semi-invariant is as follows.

\begin{definition}
 Let $G$ and $H$ be affine algebraic groups and let $\chi: G \to \mathbb{G}_m$ and $\chi': H \to \mathbb{G}_m$ be characters. Define a character of $G \times H$ by
 \begin{align*}
  \chi'-\chi : G \times H & \to \mathbb{G}_m \\
  (g,h) & \mapsto \chi(g)^{-1}\chi'(h).
 \end{align*}
 Let $G\times_{\mathbb G_m} H$ be the kernel of $\chi'-\chi$ or equivalently the fiber product of $G$ and $H$ over $\mathbb G_m$.
\end{definition}

\begin{definition} \label{defn: boxtimes}
 Let $X$ be a smooth variety equipped with an action of an affine algebraic group, $G$, and let $Y$ be a smooth variety equipped with an action of an affine algebraic group, $H$. Let $\chi: G \to \mathbb{G}_m$ and $\chi': H \to \mathbb{G}_m$ be characters. Let $w \in \Gamma(X,\mathcal O_X(\chi))^G$ and $v \in \Gamma(Y,\mathcal O_Y(\chi'))^H$.
 
 We have a dg-functor
\begin{displaymath}
 \boxtimes : \mathsf{Fact}(X,G,w) \otimes_k \mathsf{Fact}(Y,H,v) \to \mathsf{Fact}(X \times Y, G\times_{\mathbb{G}_m} H, w \boxplus v)
\end{displaymath}
 called the \newterm{exterior product}. It is defined as
 \begin{displaymath}
  \mathcal E \boxtimes \mathcal F := \op{Res}^{G \times H}_{G\times_{\mathbb{G}_m} H} (\pi_1^* \mathcal E ) \otimes_{\mathcal O_{X \times Y}} \op{Res}^{G \times H}_{G\times_{\mathbb{G}_m} H}  (\pi_2^* \mathcal F)
 \end{displaymath}
 Explicitly, we have
\begin{align*}
 \left(\mathcal E \boxtimes \mathcal F \right)_{-1} & := \op{Res}_{G\times_{\mathbb{G}_m} H}^{G \times H}(\mathcal E_{-1} \boxtimes \mathcal F_0 \oplus \mathcal E_{0} \boxtimes \mathcal F_{-1}) \\
 \left(\mathcal E \boxtimes \mathcal F \right)_0 & := \op{Res}_{G\times_{\mathbb{G}_m} H}^{G \times H}(\mathcal E_0 \boxtimes \mathcal F_0 \oplus \mathcal E_{-1}(\chi) \boxtimes \mathcal F_{-1}) 
\end{align*} 
\end{definition}

\begin{lemma} \label{lemma: box product is tensor product in the derived category}
 Assume that $\chi'-\chi$ is not torsion. Let $\mathcal E^1 \in \dabs[\mathsf{fact}(X,G,w)], \mathcal F^1 \in \dabs[\mathsf{fact}(Y,H,v)]$ and let $\mathcal E^2 \in \dabs[\mathsf{Fact}(X,G,w)], \mathcal F^2 \in \dabs[\mathsf{Fact}(Y,H,v)]$. Taking exterior products induces a natural isomorphism:
 \begin{gather*}
  \boxtimes: \bigoplus_{t \in \Z} \op{Hom}_{\dabs[\mathsf{Fact}(X, G, w)]}(\mathcal E^1, \mathcal E^2[-t]) \otimes_k \op{Hom}_{\dabs[\mathsf{Fact}(Y, H, v)]}(\mathcal F^1, \mathcal F^2[t]) \to \\ \op{Hom}_{\dabs[\mathsf{Fact}(X \times Y, G\times_{\mathbb{G}_m} H, w \boxplus v)]}(\mathcal E^1 \boxtimes \mathcal F^1, \mathcal E^2 \boxtimes \mathcal F^2).
 \end{gather*}
\end{lemma}

\begin{proof}
 We may assume that that all factorizations are locally-free in order to simplify notation for the derived functors in the proof.  We suppress the subscripts on Hom's and tensor products to help control notational girth. 
 
 We have the following chain of isomorphisms
 \begin{align}
  &  \op{Hom}(\mathcal E^1 \boxtimes \mathcal F^1, \mathcal E^2 \boxtimes \mathcal F^2)  \notag  \\
  := &  \op{Hom}(\op{Res}^{G \times H}_{G\times_{\mathbb{G}_m} H} (\pi^*_1\mathcal E^1 ) \otimes \op{Res}^{G \times H}_{G\times_{\mathbb{G}_m} H} (\pi^*_2 \mathcal F^1), \op{Res}^{G \times H}_{G\times_{\mathbb{G}_m} H} (\pi^*_1 \mathcal E^2 ) \otimes \op{Res}^{G \times H}_{G\times_{\mathbb{G}_m} H} (\pi^*_2 \mathcal F^2)  \notag \\
  \cong & \op{Hom}(\op{Res}^{G \times H}_{G\times_{\mathbb{G}_m} H} (\pi_1^* \mathcal E^1 ), \mathcal Hom (\op{Res}^{G \times H}_{G\times_{\mathbb{G}_m} H} (\pi_2^* \mathcal F^1), \op{Res}^{G \times H}_{G\times_{\mathbb{G}_m} H} (\pi_1^* \mathcal E^2 ) \otimes \op{Res}^{G \times H}_{G\times_{\mathbb{G}_m} H} (\pi_2^* \mathcal F^2) )) \notag  \\
  \cong & \op{Hom}(\op{Res}^{G \times H}_{G\times_{\mathbb{G}_m} H} (\pi_1^* \mathcal E^1 ), \op{Res}^{G \times H}_{G\times_{\mathbb{G}_m} H} (\pi_1^* \mathcal E^2 ) \otimes \mathcal Hom (\op{Res}^{G \times H}_{G\times_{\mathbb{G}_m} H} (\pi_2^* \mathcal F^1), \op{Res}^{G \times H}_{G\times_{\mathbb{G}_m} H} (\pi_2^* \mathcal F^2) ) ) \notag  \\
  \cong & \op{Hom}(\op{Res}^{G \times H}_{G\times_{\mathbb{G}_m} H} (\pi_1^* \mathcal E^1 ), \op{Res}^{G \times H}_{G\times_{\mathbb{G}_m} H} (\pi_1^* \mathcal E^2 ) \otimes \op{Res}^{G \times H}_{G\times_{\mathbb{G}_m} H} \pi_2^* \mathcal Hom ( \mathcal F^1, \mathcal F^2) )  \notag \\
  \cong & \op{Hom}(\pi_1^* \mathcal E^1 , \op{Ind}^{G \times H}_{G\times_{\mathbb{G}_m} H} (\op{Res}^{G \times H}_{G\times_{\mathbb{G}_m} H} (\pi_1^* \mathcal E^2 ) \otimes \op{Res}^{G \times H}_{G\times_{\mathbb{G}_m} H} \pi_2^* \mathcal Hom ( \mathcal F^1, \mathcal F^2) ) ) \notag  \\
  \cong &\op{Hom}(\pi_1^* \mathcal E^1 , \pi_1^* \mathcal E^2 \otimes  \op{Ind}^{G \times H}_{G\times_{\mathbb{G}_m} H} \op{Res}^{G \times H}_{G\times_{\mathbb{G}_m} H} \pi_2^* \mathcal Hom ( \mathcal F^1, \mathcal F^2) ) \notag \\
  \cong & \op{Hom}(\pi_1^* \mathcal E^1 , \pi_1^* \mathcal E^2 \otimes \bigoplus_{l \in \Z} \pi_2^* \mathcal Hom (\mathcal F^1, \mathcal F^2) (l(\chi'-\chi))) \label{eq: box product align}
 \end{align}
 The second line is by definition.  The third line is Corollary~\ref{corollary: derived Hom-tensor adjunction} i.e.\ tensor-Hom adjunction.  The fourth line can be seen by appealing to Lemma~\ref{lemma: sheaf hom is tensor with dual for finite rank fact} and associativity of tensor product using the fact that $\mathcal F^1$ is locally-free of finite rank to pull out a dual and put it back in.  Note that $\op{Res}^{G \times H}_{G \times_{\mathbb{G}_m} H}$ commutes with duals so the order of operations is not germane.  The fifth line follows from the fact that the functors $\op{Res}$ and $\pi^*_i$ are both monoidal, so they commute with $\otimes$ and $\mathcal Hom$.  The sixth line uses the adjunction of Corollary~\ref{corollary: Res-Ind derived fact adjunction}.  
 Since we have assumed that $\chi'-\chi$ is not torsion, we have an isomorphism
 \begin{displaymath}
  G \times H/ G \times_{\mathbb{G}_m} H \cong \mathbb{G}_m.
 \end{displaymath}
 As this quotient is affine, Ind is exact and $\mathbf{R}\!\op{Ind}^G_H \cong \op{Ind}^G_H$.  
 The seventh line is the projection formula for the induction functor, Lemma~\ref{lemma: derived projection formula for induction on fact}.
 The eighth line uses Lemma~\ref{lemma: Res-Ind Abelian quotient}.
 
 Let $q: Y \to \op{Spec} k$ and $p: X \to \op{Spec} k$ be the structure morphisms. Continuing with the isomorphisms from Equation~\eqref{eq: box product align} and using morphism spaces in $\dabs[\mathsf{Fact}(X,G,w)]$, we have
 \begin{align*}
  & \op{Hom}(\mathcal E^1 \boxtimes \mathcal F^1, \mathcal E^2 \boxtimes \mathcal F^2) \\
  \cong & \op{Hom}(\mathcal E^1 , (\mathbf{R}\pi_{1*}\pi_1^* \mathcal E^2 \otimes \bigoplus_{l \in \Z} \pi_2^* \mathcal Hom (\mathcal F^1, \mathcal F^2) (l(\chi'-\chi)))^{\mathbf{R}H}) \notag \\
  \cong & \op{Hom}(\mathcal E^1 , \mathcal E^2 \otimes (\mathbf{R}\pi_{1*}\bigoplus_{l \in \Z} \pi_2^* \mathcal Hom (\mathcal F^1, \mathcal F^2) (l(\chi'-\chi)))^{\mathbf{R}H}) \notag \\
  \cong & \op{Hom}(\mathcal E^1 , \mathcal E^2 \otimes (\mathbf{R}\pi_{1*}\pi_2^* \bigoplus_{l \in \Z} \mathcal Hom (\mathcal F^1, \mathcal F^2) (l(\chi'-\chi)))^{\mathbf{R}H}) \notag \\
  \cong & \op{Hom}(\mathcal E^1 , \mathcal E^2 \otimes (p^* \mathbf{R}q_{*} \bigoplus_{l \in \Z} \mathcal Hom (\mathcal F^1, \mathcal F^2) (l(\chi'-\chi)))^{\mathbf{R}H}) \notag \\
  \cong & \op{Hom}(\mathcal E^1 , \mathcal E^2 \otimes_k \bigoplus_{l \in \Z} \op{Hom}(\mathcal F^1,\mathcal F^2[2l])(-l\chi) \oplus \op{Hom}(\mathcal F^1,\mathcal F^2[2l+1])(-l\chi)[-1]) \\
  \cong & \op{Hom}(\mathcal E^1 , \bigoplus_{t \in \Z} \mathcal E^2[-t] \otimes_k \op{Hom} (\mathcal F^1, \mathcal F^2[t])) \\
  \cong & \bigoplus_{t \in \Z} \op{Hom}(\mathcal E^1 , \mathcal E^2[-t]) \otimes_k \op{Hom} (\mathcal F^1, \mathcal F^2[t]).
 \end{align*}

The first line uses that the right adjoint to $\pi_1^*$ is the composition $(\mathbf{R}\pi_{1*})^{\mathbf{R} H}$ by Corollary~\ref{corollary: right adjoint to full pullback}. The second line morally uses the projection formula.  However, we have not provided a projection formula in this general context.  We can work around this by deriving the two  projection formulas from Lemmas~\ref{lemma: projection formula for equivariant pushforward} and~\ref{lemma: facts about restriction induction} and rewriting the functor $\pi_1^* = (\pi_1^\prime)^* \circ \op{Res}_{r}$ where $r: G \times H \to G$ denotes the projection homomorphism and $\pi_1^\prime: X \times Y \to X$ denotes the $G \times H$ equivariant projection where $H$ acts trivially on $X$. 
The fourth line uses flat base change, Lemma~\ref{lemma: flat base change for equivariant sheaves}.
The fifth line uses Lemma~\ref{lemma: base change for invariants} to pull the invariants inside $p^*$. The sixth line comes from substitution of the isomorphism,
\begin{gather} 
 (\mathbf{R}q_{*} \bigoplus_{l \in \Z} \mathcal Hom (\mathcal F^1, \mathcal F^2) (l(\chi'-\chi)))^{\mathbf{R}H} \cong  \notag \\ \bigoplus_{l \in \Z} \op{Hom}(\mathcal F^1,\mathcal F^2[2l])(-l\chi) \oplus \op{Hom}(\mathcal F^1,\mathcal F^2[2l+1])(-l\chi)[-1]. \label{equation: morphisms spaces}
\end{gather}
Equation~\eqref{equation: morphisms spaces} is a consequence of Lemma~\ref{lemma: pushdown sheaf Hom = Hom} and the identity $(\chi') = [2]$. The sixth line uses that $\mathcal E^2 \otimes \bullet$ commutes with coproducts and a straightforward identification of the twists with shifts using $(\chi) = [2]$. The final line follows since $\mathcal E^1$ is a coherent factorization. By Proposition~\ref{prop: we have a compactly-gen triangulated cat} it is a compact object, and therefore,  $\op{Hom}(\mathcal E^1, \bullet)$ commutes with coproducts. 
The total isomorphism gives an inverse to $\boxtimes$.
\end{proof}

Finally, let us define a version of an integral transformation for factorizations. 

\begin{definition}
 Let $\mathcal P \in \mathsf{Fact}(X \times Y, G\times_{\mathbb{G}_m} H, (-w) \boxplus v)$. Equip $Y$ with the trivial $G$ action to give it a $G \times H$ action in full. View $\pi_2: X \times Y \to Y$ as $G \times H$-equivariant. Set
 \begin{align*}
  \Phi_{\mathcal P}^{X \to Y} : \mathsf{Fact}(X,G,w) & \to \mathsf{Fact}(Y,H,v) \\
  \mathcal E & \mapsto \left( \pi_{2*}(\pi_1^* \mathcal E \otimes_{\mathcal O_{X \times Y}} \op{Ind}^{G \times H}_{G\times_{\mathbb{G}_m} H} \mathcal P) \right)^G.
 \end{align*} 
 We will also denote the associated functor on the derived categories by
 \begin{align*}
  \Phi_{\mathcal P}^{X \to Y} : \dabs[\mathsf{Fact}(X,G,w)] & \to \dabs[\mathsf{Fact}(Y,H,v)] \\
  \mathcal E & \mapsto \left( \mathbf{R} \pi_{2*}(\mathbf{L} \pi_1^* \mathcal E \overset{\mathbf{L}}{\otimes}_{\mathcal O_{X \times Y}} \op{Ind}^{G \times H}_{G\times_{\mathbb{G}_m} H} \mathcal P) \right)^{\mathbf{R}G}.
 \end{align*}
 The object $\mathcal P$ is called the \newterm{kernel} of  $\Phi_{\mathcal P}^{X \to Y}$.
 
 View $\mathcal F$ as a factorization of $0 \in \Gamma(X,\mathcal O_X(\chi))^G$ as $\Upsilon \mathcal F$.  Define the factorization, 
 \begin{displaymath}
  \nabla(\mathcal F) := \op{Ind}^{G\times_{\mathbb{G}_m} G}_G \Delta_* \mathcal F := \Upsilon \op{Ind}^{G\times_{\mathbb{G}_m} G}_{G} \Delta_* \mathcal F.
 \end{displaymath}
 Set
 \begin{displaymath}
  \nabla := \nabla(\mathcal O_X).
 \end{displaymath}
\end{definition}

\begin{lemma} \label{lemma: we have the diagonal factorization}
 There is a natural transformation of dg-functors
 \begin{displaymath}
  \Phi_{\nabla(\mathcal F)} \to \bullet \otimes_{\mathcal O_X} \mathcal F 
 \end{displaymath}
 inducing an isomorphism of derived functors,
 \begin{displaymath}
  \Phi_{\nabla(\mathcal F)} \cong \bullet \overset{\mathbf{L}}{\otimes}_{\mathcal O_X} \mathcal F  : \dabs[\mathsf{Fact}(X,G,w)] \to \dabs[\mathsf{Fact}(X,G,w)].
 \end{displaymath}
 In particular, $\nabla$ is the kernel of the identity functor.
\end{lemma}

\begin{proof}
 For any $\mathcal E \in \mathsf{Fact}(X,G,w)$, we have a natural morphism
\begin{align*}
 \left( \pi_{2*} \left( \pi_1^* \mathcal E \otimes_{\mathcal O_{X \times X}} \op{Ind}^{G \times G}_{G} \Delta_* \mathcal F \right) \right)^G  & \mapsto \left( \pi_{2*} \op{Ind}^{G \times G}_{G} \Delta_* \left( \Delta^* \op{Res}^{G \times G}_G \pi_1^* \mathcal E \otimes_{\mathcal O_X} \mathcal F \right) \right)^G \\
& \cong   \left( \pi_{2*} \op{Ind}^{G \times G}_{G} \Delta_* \left( \mathcal E \otimes_{\mathcal O_X} \mathcal F \right) \right)^G \\
&   \cong   \mathcal E \otimes_{\mathcal O_X} \mathcal F
  \end{align*}
  The first line is from the projection formula for $\Delta^*,\Delta_*$, Lemma~\ref{lemma: projection formula for equivariant pushforward}, and $\op{Res}^{G \times G}_{G}, \op{Ind}^{G \times G}_{G}$, Lemma~\ref{lemma: facts about restriction induction}, applied component-wise to a factorization.  The second line comes from the isomorphism
 \begin{displaymath}
  \Delta^* \op{Res}^{G \times G}_{G} \pi_1^* \cong \Delta^* \pi_1^* \cong (\pi_1 \circ \Delta)^* \cong \op{Id}
 \end{displaymath}
 where for the first isomorphism we view $\pi_1$ as $G$-equivariant with respect to the diagonal action of $G$ on $X \times X$.  For the third line,  we use that
 \begin{displaymath}
  (\pi_{2*} \op{Ind}^{G \times G}_{G} \Delta_*)^G \cong \op{Id}
 \end{displaymath}
 as the functor, $(\pi_{2*} \op{Ind}^{G \times G}_{G} \Delta_*)^G$, is right adjoint to $\Delta^* \op{Res}^{G \times G}_{G} \pi_2^* \cong \op{Id}$.
 Combining the natural morphisms gives the natural transformation
 \begin{displaymath}
  \Phi_{\nabla(\mathcal F)} \to \bullet \otimes_{\mathcal O_X}  \mathcal F.
 \end{displaymath}
 
 The statement for the derived functors follows via the same argument, replacing the usual functors by their derived versions, and noting that derived projection formula is an isomorphism by Lemma~\ref{lemma: derived projection formula for fact}.
 \end{proof}

\begin{lemma} \label{lemma: derived trace is pullback via diagonal}
 Let $p: X \to \op{Spec} k$ be the structure map. There is a natural transformation of dg-functors
 \begin{displaymath}
  (p_* \Delta^* (\mathcal E^{\vee} \boxtimes \mathcal F))^G \to \op{Hom}_{\mathsf{Fact}(X,G,w)}(\mathcal E,\mathcal F)
 \end{displaymath}
 inducing a natural isomorphism
 \begin{displaymath}
  \left( \mathbf{R}p_* \mathbf{L} \Delta^* (\mathcal E^\vee \boxtimes \mathcal F) \right)^{\mathbf{R}G} \cong \mathbf{R}\!\op{Hom}_{\mathsf{Fact}(X,G,w)}(\mathcal E, \mathcal F)
 \end{displaymath}
 if we assume $\mathcal E \in \dabs [\mathsf{fact}(X,G,w)]$.
\end{lemma}

\begin{proof}
 We have
 \begin{align*}
  \left( p_* \Delta^* \mathcal E^\vee \boxtimes \mathcal F \right)^{G}  & =  \left( p_* \Delta^* \left( \op{Res}^{G \times G}_{G\times_{\mathbb{G}_m} G} \pi_1^* \mathcal E^{\vee} {\otimes}_{\mathcal O_{X \times X}} \op{Res}^{G \times G}_{G\times_{\mathbb{G}_m} G} \pi_2^* \mathcal F \right) \right)^{G} \\
  & \cong  \left( p_* (\mathcal E^{\vee} {\otimes}_{\mathcal O_{X}} \mathcal F )\right)^{G} \\
  & \to \left( p_* \mathcal Hom_X(\mathcal E, \mathcal F) \right)^{G} \\
  & \cong \op{Hom}_{\mathsf{Fact}(X,G,w)}(\mathcal E, \mathcal F).
 \end{align*}
 The first line is by definition.  The second line follows from by distributing $\Delta^*$ across the tensor product then observing that we have an isomorphism  $\op{Res}^{G \times G}_{G\times_{\mathbb{G}_m} G} \pi_1^* \cong (\pi_1^\prime)^*$ where $\pi_1^\prime: X \times X \to X$ is equivariant with respect to the first projection $G\times_{\mathbb{G}_m} G \to G$ and similarly,  $\op{Res}^{G \times G}_{G\times_{\mathbb{G}_m} G} \pi_2^* \cong (\pi_2^\prime)^*$.  Finally, $\pi_1 \circ \Delta \cong \pi_2 \circ \Delta \cong 1_X$.  The third line follows from the natural map
 \begin{displaymath}
  \mathcal E^{\vee} \otimes_{\mathcal O_X} \mathcal F \to \mathcal Hom_X(\mathcal E,\mathcal F).
 \end{displaymath}
 The fourth line is induced from the isomorphism of functors 
 \begin{displaymath}
  (p_* \mathcal Hom_X (\mathcal E, \mathcal F))^G \cong \op{Hom}_{\mathsf{Fact}(X,G,w)}(\mathcal E, \mathcal F)
 \end{displaymath}
 of Lemma~\ref{lemma: pushdown sheaf Hom = Hom}.
 
 The statement for the derived functors follows via analogous arguments replacing the usual functors by their derived version and using Lemma~\ref{lemma: sheaf hom is tensor with dual for finite rank fact} to know that the natural map
 \begin{displaymath}
  \mathcal E^{\mathbf{L} \vee} \otimes_{\mathcal O_X} \mathcal F \to \mathbf{R}\mathcal Hom_X(\mathcal E,\mathcal F)
 \end{displaymath}
 is an isomorphism if $\mathcal E$ is coherent.
\end{proof}

\begin{definition}
 The \newterm{trace functor} on $\dabs [\mathsf{Fact}(X \times X ,G\times_{\mathbb{G}_m} G,(-w) \boxplus w)]$ is the functor
 \begin{displaymath}
  \mathbf{L}\!\op{Tr} := \left( \mathbf{R}p_* \mathbf{L} \Delta^* \right)^{\mathbf{R}G} : \dabs [\mathsf{Fact}(X \times X ,G\times_{\mathbb{G}_m} G,(-w) \boxplus w)] \to \dbqcoh{\op{Spec} k}.
 \end{displaymath}
\end{definition}

\begin{lemma} \label{lemma: Tr is representable}
 Assume that $(G\times_{\mathbb{G}_m} G)/G \cong K_{\chi}$ is finite. There is an isomorphism of functors
 \begin{displaymath}
  \op{Tr} \cong \mathbf{R}\!\op{Hom}_{\mathsf{Fact}(X \times X ,G\times_{\mathbb{G}_m} G,(-w) \boxplus w)}(\nabla^{\mathbf{L} \vee}, \bullet )
 \end{displaymath}
 on $\dabs [\mathsf{fact}(X \times X ,G\times_{\mathbb{G}_m} G,(-w) \boxplus w)]$.
\end{lemma}

\begin{proof}
 As $(G\times_{\mathbb{G}_m} G)/G \cong K_{\chi}$ is finite, $\op{Ind}^{G\times_{\mathbb{G}_m} G}_G$ preserves coherent $G$-equivariant sheaves. For coherent factorizations, dualization is an anti-equivalence by Lemma~\ref{lemma: double dual is ok for coherent factorizations}. 

Now, we have 
\begin{align*}
 \mathbf{L}\!\op{Tr} & = \left( \mathbf{R}p_* \mathbf{L} \Delta^* \right)^{\mathbf{R}G} \\
 & \cong \mathbf{R}\!\op{Hom}_{\mathsf{Fact}(X,G,0)}( \mathcal O_X , \mathbf{L} \Delta^* (\bullet)) \\
& \cong  \mathbf{R}\!\op{Hom}_{\mathsf{Fact}(X,G,0)}( \mathbf{L} \Delta^* (\bullet)^{\mathbf{L} \vee} ,\mathcal O_X)  \\
& \cong   \mathbf{R}\!\op{Hom}_{\mathsf{Fact}(X \times X ,G\times_{\mathbb{G}_m} G, w \boxplus (-w))}( (\bullet)^{\mathbf{L} \vee} , \op{Ind}^{G\times_{\mathbb{G}_m} G}_G \Delta_* \mathcal O_X) \\
& \cong   \mathbf{R}\!\op{Hom}_{\mathsf{Fact}(X \times X ,G\times_{\mathbb{G}_m} G,(-w) \boxplus w)}(\nabla^{\mathbf{L} \vee}, \bullet ).
\end{align*}
The first line is a definition.  The second line is from the isomorphism of functors, 
\[
\mathbf{R}\!\op{Hom}_{\mathsf{Fact}(X,G,0)}( \mathcal O_X , \bullet) \cong (\mathbf{R}p_*)^{\mathbf{R}G}.
\]
The third line uses that $\Delta^*$ commutes with duals and $\bullet$ is assumed to be coherent.
The fourth line is adjunction between  $\op{Ind}^{G\times_{\mathbb{G}_m} G}_G \Delta_*$  and $\mathbf{L} \Delta^* $.
The fifth line uses dualization, coherence of $\bullet$, and the definition of $\nabla$.
\end{proof}

In the process of proving a generation statement for categories of factorizations, we will want to make use of some geometry. As such, we need an alternate, more geometric, characterization of these factorization categories. This characterization is due, in various generality, to Eisenbud \cite{EisMF}, Buchweitz \cite{Buc86}, Orlov \cite{Orl09,OrlMF}, Polishchuk-Vaintrob \cite{PV2}, and \cite{Pos2}.

Let us recall the definition of the singularity category.

\begin{definition} \label{defn: sing cat}
 Let $Y$ be a scheme of finite type over $k$ and let $G$ be an affine algebraic group acting on $Y$. Assume that $Y$ has enough locally-free $G$-equivariant sheaves. The $G$-equivariant \newterm{singularity category}, or $G$-equivariant \newterm{stable category}, of $Y$ is the Verdier quotient
 \begin{displaymath}
  \op{D}^{\op{sg}}_G(Y) := \dbcohG{G}{Y} / \op{perf}_G Y
 \end{displaymath}
 where $\op{perf}_G Y$ is the thick subcategory of locally-free $G$-equivariant sheaves of finite rank on $Y$. Let $Z$ be a closed $G$-invariant subset of $Y$, then we let $\op{D}^{\op{sg}}_{Z,G}(Y)$ be the kernel of the functor, $j^*: \op{D}^{\op{sg}}_G(Y) \to \op{D}^{\op{sg}}_G(U)$.
\end{definition}

Assume we have a smooth variety $X$ equipped with an action of an affine algebraic group $G$ and an invariant section $w \in \Gamma(X, \mathcal L)^G$ for an invertible equivariant sheaf, $\mathcal L$. Set $Y = Z_w$ to be the vanishing locus of $w$. Let $i: Y \to X$ denote the inclusion.

\begin{lemma} \label{lemma: enough locally-frees in a hypersurface}
 The scheme, $Y$, has enough locally-free $G$-equivariant sheaves. Moreover, every coherent $G$-equivariant sheaf on $Y$ admits an epimorphism from $i^*\mathcal V$ where $\mathcal V$ is locally-free of finite rank.
\end{lemma}

\begin{proof}
 As $X$ is smooth, it has enough locally-free $G$-equivariant sheaves by Theorem~\ref{theorem: Thomason}. Given any coherent $G$-equivariant sheaf on $Y$, $\mathcal E$, we can find a locally-free $G$-equivariant sheaf of finite rank, $\mathcal V$, and an epimorphism, $\psi: \mathcal V \to i_* \mathcal E$. The morphism, $i^*\psi : i^* \mathcal V \to  i^*i_* \mathcal E \cong \mathcal E$, remains an epimorphism as $i^*$ is right exact.
\end{proof}

Consider the functor,
\begin{align*}
 \op{cok}: [\mathsf{vect}(X,G,w)] & \to \op{D}^{\op{sg}}_G(Y) \\
			 \mathcal E & \mapsto \op{cok} \phi^{\mathcal E}_{0}.
\end{align*}

\begin{lemma} \label{lemma: exactness of cok}
 Assume that $w$ is not identically zero on any component of $X$. The functor, $\op{cok}$, is well-defined and exact.
\end{lemma}

\begin{proof}
 This is a special case of \cite[Lemma 3.12]{PV2}. 
\end{proof}

\begin{lemma} \label{lemma: cok descends to abs derived category}
 Assume that $w$ is not identically zero on any component of $X$. Let $Z$ be a closed $G$-invariant subset of $Y$. The functor, $\op{cok}$, descends to the absolute derived category, 
 \begin{displaymath}
  \op{cok} : \dabs_Z [\mathsf{vect}(X,G,w)] \to \op{D}^{\op{sg}}_{Z,G}(Y).
 \end{displaymath}
\end{lemma}

\begin{proof}
 Let us treat the situation $Z = Y$ first. In the case where $G$ is trivial, this is \cite[Proposition 3.2]{OrlMF}. The same argument works with the inclusion of $G$. We recall the argument for the convenience of the reader. Let 
 \begin{displaymath}
  0 \to \mathcal G \overset{q}{\to} \mathcal E \overset{p}{\to} \mathcal F \to 0
 \end{displaymath}
 be an exact sequence of factorizations and let $\mathcal T$ be the totalization. Recall that
 \begin{align*}
  \mathcal T_{-1} & := \mathcal G_{-1} \otimes_{\mathcal O_X} \mathcal L \oplus \mathcal E_0 \oplus \mathcal F_{-1}   \\
  \mathcal T_0 & := \mathcal G_0 \otimes_{\mathcal O_X} \mathcal L \oplus \mathcal E_{-1} \otimes_{\mathcal O_X} \mathcal L \oplus \mathcal F_0   \\
  \phi^{\mathcal T}_{0} & := \begin{pmatrix}  \phi_0^{\mathcal G} \otimes_{\mathcal O_X} \mathcal L  & 0 & 0 \\ q_{-1} \otimes_{\mathcal O_X} \mathcal L & -\phi_{-1}^{\mathcal E} & 0 \\ 0 & p_0 & \phi_0^{\mathcal F} \end{pmatrix} \\
  \phi^{\mathcal T}_{-1} & := \begin{pmatrix} \phi_{-1}^{\mathcal G} \otimes_{\mathcal O_X} \mathcal L & 0  & 0 \\  q_0 \otimes_{\mathcal O_X} \mathcal L & -\phi_{0}^{\mathcal E} \otimes_{\mathcal O_X} \mathcal L & 0 \\ 0 & p_{-1} \otimes_{\mathcal O_X} \mathcal L  & \phi_{-1}^{\mathcal F} \end{pmatrix}
 \end{align*}
 Consider the associated exact sequence over $\op{coh}_G X$
 \begin{gather*}
  0 \to \mathcal G_0 \overset{\begin{pmatrix} q_0 \\ \phi_{-1}^{\mathcal G} \end{pmatrix}}{\to} \mathcal E_{0} \oplus \mathcal G_{-1} \otimes_{\mathcal O_X} \mathcal L \overset{\begin{pmatrix} p_0 & 0 \\ -\phi_{-1}^{\mathcal E} & q_{-1} \otimes_{\mathcal O_X} \mathcal L \end{pmatrix}}{\to} \\ \mathcal F_0 \oplus \mathcal E_{-1} \otimes_{\mathcal O_X} \mathcal L \overset{\begin{pmatrix} \phi^{\mathcal F}_{-1} & p_1 \otimes_{\mathcal O_X} \mathcal L \end{pmatrix}}{\to} \mathcal F_{-1} \otimes_{\mathcal O_X} \mathcal L \to 0
 \end{gather*}
 and let $\mathcal U$ be the cokernel of the map $\mathcal G_{0} \to \mathcal E_{0} \oplus \mathcal G_{-1}\otimes_{\mathcal O_X} \mathcal L $. Let $(\alpha_0,\alpha_1):  \mathcal G_{-1}\otimes_{\mathcal O_X} \mathcal L \oplus \mathcal E_{0} \to \mathcal U$ be the epimorphism and $\begin{pmatrix} \beta_0 \\ \beta_1 \end{pmatrix} : \mathcal U \to \mathcal F_0 \oplus \mathcal E_{-1}\otimes_{\mathcal O_X} \mathcal L$ be the monomorphism.
 We have a commutative diagram 
 \begin{center}
 \begin{tikzpicture}[description/.style={fill=white,inner sep=2pt}]
  \matrix (m) [matrix of math nodes, row sep=3em, column sep=3em, text height=1.5ex, text depth=0.25ex]
  {  \mathcal E_{0} \oplus \mathcal G_{-1}\otimes_{\mathcal O_X} \mathcal L & \mathcal U \oplus \mathcal G_0\otimes_{\mathcal O_X} \mathcal L \\ 
    \mathcal F_{-1} \oplus \mathcal E_{0} \oplus \mathcal G_{-1}\otimes_{\mathcal O_X} \mathcal L  &  \mathcal F_0 \oplus \mathcal E_{-1}\otimes_{\mathcal O_X} \mathcal L \oplus \mathcal G_{0}\otimes_{\mathcal O_X} \mathcal L\\
     \mathcal F_{-1} & \mathcal F_{-1}\otimes_{\mathcal O_X} \mathcal L \\
  };
  \path[->,font=\scriptsize]
  (m-1-1) edge node[above] {$\begin{pmatrix} \alpha_0 & \alpha_1 \\ 0 & \phi_0^{\mathcal G}\otimes_{\mathcal O_X} \mathcal L \end{pmatrix}$} (m-1-2)
  (m-1-1) edge node[left] {$\begin{pmatrix} 0 & 0 \\ 1_{\mathcal E_0} & 0 \\ 0 &  1_{\mathcal G_{-1}\otimes_{\mathcal O_X} \mathcal L} \end{pmatrix}$} (m-2-1)
  (m-1-2) edge node[right] {$\begin{pmatrix} \beta_0 & 0 \\ \beta_1 & 0 \\ 0 &  1_{\mathcal G_0\otimes_{\mathcal O_X} \mathcal L}  \end{pmatrix}$} (m-2-2)
  (m-2-1) edge node[above] {$\phi^{\mathcal T}_{0}$} (m-2-2)
  (m-2-1) edge node[left] {$\begin{pmatrix} 1_{\mathcal F_{-1}} & 0 & 0 \end{pmatrix}$} (m-3-1)
  (m-2-2) edge node[right] {$\begin{pmatrix} \phi^{\mathcal F}_{-1} & p_1\otimes_{\mathcal O_X} \mathcal L & 0 \end{pmatrix}$} (m-3-2)
  (m-3-1) edge node[above] {$w$} (m-3-2)
  ;
 \end{tikzpicture}
 \end{center}
 with columns being short exact sequences. Thus, we have an exact sequence of cokernels, as coherent sheaves on $Y$,
 \begin{displaymath}
  0 \to \op{cok} \begin{pmatrix} \alpha_0 & \alpha_1 \\ 0 & \phi_0^{\mathcal G}\otimes_{\mathcal O_X} \mathcal L \end{pmatrix} \to \op{cok} \phi^{\mathcal T}_0 \to i^*\mathcal F_{-1}\otimes_{\mathcal O_X} \mathcal L \to 0.
 \end{displaymath}
 As $i^*(\mathcal F_{-1} \otimes_{\mathcal O_X} \mathcal L)$ is trivial in $\op{D}^{\op{sg}}_G(Y)$, we have an isomorphism
 \begin{displaymath}
  \op{cok} \begin{pmatrix} \alpha_0 & \alpha_1 \\ 0 & \phi_0^{\mathcal G}\otimes_{\mathcal O_X} \mathcal L \end{pmatrix} \cong  \op{cok} \phi^{\mathcal T}_0  
 \end{displaymath}
 in $\op{D}^{\op{sg}}_G(Y)$. We also have a commutative diagram
 \begin{center}
 \begin{tikzpicture}[description/.style={fill=white,inner sep=2pt}]
  \matrix (m) [matrix of math nodes, row sep=3em, column sep=3em, text height=1.5ex, text depth=0.25ex]
  {  \mathcal G_{0} & \mathcal G_0\otimes_{\mathcal O_X} \mathcal L \\ 
    \mathcal E_{0} \oplus \mathcal G_{-1}\otimes_{\mathcal O_X} \mathcal L &  \mathcal U \oplus \mathcal G_{0}\otimes_{\mathcal O_X} \mathcal L\\
     \mathcal U & \mathcal U \\
  };
  \path[->,font=\scriptsize]
  (m-1-1) edge node[above] {$w$} (m-1-2)
  (m-1-1) edge node[left] {$\begin{pmatrix} q_0 \\ \phi_{-1}^{\mathcal G} \end{pmatrix}$} (m-2-1)
  (m-1-2) edge node[right] {$\begin{pmatrix} 0 \\ 1_{\mathcal G_0\otimes_{\mathcal O_X} \mathcal L} \end{pmatrix}$} (m-2-2)
  (m-2-1) edge node[above] {$\begin{pmatrix} \alpha_0 & \alpha_1 \\ 0 & \phi_0^{\mathcal G}\otimes_{\mathcal O_X} \mathcal L \end{pmatrix}$} (m-2-2)
  (m-2-1) edge node[left] {$\begin{pmatrix} \alpha_0 & \alpha_1 \end{pmatrix}$} (m-3-1)
  (m-2-2) edge node[right] {$\begin{pmatrix} 1_{\mathcal U} & 0 \end{pmatrix}$} (m-3-2)
  (m-3-1) edge node[above] {$1_{\mathcal U}$} (m-3-2)
  ;
 \end{tikzpicture}
 \end{center}
 with columns being short exact sequences. Thus, we have an isomorphism of coherent sheaves
 \begin{displaymath}
  i^*\mathcal G_0 \cong \op{cok} \begin{pmatrix} \alpha_0 & \alpha_1 \\ 0 & \phi_0^{\mathcal G}\otimes_{\mathcal O_X} \mathcal L \end{pmatrix}.
 \end{displaymath}
 Thus, $\op{cok} \begin{pmatrix} \alpha_0 & \alpha_1 \\ 0 & \phi_0^{\mathcal G}\otimes_{\mathcal O_X} \mathcal L \end{pmatrix} \cong \op{cok}  \phi^{\mathcal T}_0 $ is trivial in $\op{D}^{\op{sg}}_G(Y)$.  This proves the statement when $Z=Y$.
 
 Now the general case follows from the case where $Z=Y$.  Indeed, it is clear that $\op{cok}$ commutes with restriction to open subsets. Thus, we have a commutative diagram of functors
 \begin{center}
 \begin{tikzpicture}[description/.style={fill=white,inner sep=2pt}]
  \matrix (m) [matrix of math nodes, row sep=3em, column sep=3em, text height=1.5ex, text depth=0.25ex]
  {  \dabs [\mathsf{vect}(X,G,w)] & \op{D}^{\op{sg}}_G(Y) \\ 
     \dabs [\mathsf{vect}(U,G,w|_U)] & \op{D}^{\op{sg}}_G(Y \cap U) \\
  };
  \path[->,font=\scriptsize]
  (m-1-1) edge node[above] {$\op{cok}$} (m-1-2)
  (m-1-1) edge node[left] {$j^*$} (m-2-1)
  (m-1-2) edge node[right] {$j^*$} (m-2-2)
  (m-2-1) edge node[above] {$\op{cok}$} (m-2-2)
  ;
 \end{tikzpicture}
 \end{center}
 where $j: U = X \setminus Z \to X$ is the inclusion. Thus, $\op{cok}$ induces a functor between the kernels of $j^*$. On the factorization side, this is $\dabs_Z [\mathsf{vect}(X,G,w)]$ while on the singularity side this is $\op{D}^{\op{sg}}_{Z,G}(Y)$.
\end{proof}

\begin{definition}
 Define the functor
 \begin{align*}
  \mathbf{L}\!\op{cok}: \dabs_Z [\mathsf{fact}(X,G,w)] & \to \op{D}^{\op{sg}}_{Z,G}(Y) \\
  \mathcal E & \mapsto \op{cok} \mathcal V
 \end{align*}
 where $\mathcal V$ is a factorization that has locally-free components and is quasi-isomorphic to $\mathcal E$.
\end{definition}

In the other direction, we use the functor $\Upsilon$. 

\begin{lemma} \label{lemma: upsilon descends to a functor on d-sing}
 Assume that $w$ is not identically zero on any component of $X$. The functor, $\Upsilon$, descends further to a functor
 \begin{displaymath}
  \Upsilon: \op{D}^{\op{sg}}_G(Y) \to \op{D}^{\op{abs}}[\mathsf{fact}(X,G,w)].
 \end{displaymath}
 Moreover, if $Z$ is a closed $G$-invariant subset of $Y$, then $\Upsilon$ induces a functor
 \begin{displaymath}
  \Upsilon: \op{D}^{\op{sg}}_{Z,G}(Y) \to \dabs_Z [\mathsf{fact}(X,G,w)].
 \end{displaymath}
\end{lemma}

\begin{proof}
 We treat the first statement. We need to check that $\Upsilon$ annihilates $\op{perf}_G Y$. By Lemma~\ref{lemma: enough locally-frees in a hypersurface}, it suffices to show that it annihilates $i^* \mathcal V$ for $\mathcal V$ a locally-free $G$-equivariant sheaf of finite rank on $X$. For a coherent $G$-equivariant sheaf on $X$, $\mathcal E$, define a factorization, $\mathcal H_{\mathcal E} := (\mathcal E, \mathcal E, w, 1_{\mathcal E})$. There is a short exact sequence
 \begin{displaymath}
  0 \to \mathcal H_{\mathcal V} \otimes \mathcal L^{-1} \to \mathcal H_{\mathcal V} \to \Upsilon (i^* \mathcal V) \to 0.
 \end{displaymath}
 Thus, $\Upsilon (i^* \mathcal V)$ is quasi-isomorphic to the cone $\mathcal H_{\mathcal V}\otimes \mathcal L^{-1} \to \mathcal H_{\mathcal V}$. It is straightforward to see that any $\mathcal H_{\mathcal V}$ is contractible. Thus, $\Upsilon (i^* \mathcal V)$ is zero in $\op{D}^{\op{abs}}[\mathsf{fact}(X,G,w)]$. 
 
 It is clear that $\Upsilon$ commutes with restriction to open subsets. Thus, we have a commutative diagram of functors
 \begin{center}
 \begin{tikzpicture}[description/.style={fill=white,inner sep=2pt}]
  \matrix (m) [matrix of math nodes, row sep=3em, column sep=3em, text height=1.5ex, text depth=0.25ex]
  {  \op{D}^{\op{sg}}_G(Y) & \dabs [\mathsf{vect}(X,G,w)] &  \\ 
     \op{D}^{\op{sg}}_G(U) & \dabs [\mathsf{vect}(V,G,w|_V)] \\
  };
  \path[->,font=\scriptsize]
  (m-1-1) edge node[above] {$\Upsilon$} (m-1-2)
  (m-1-1) edge node[left] {$j_U^*$} (m-2-1)
  (m-1-2) edge node[right] {$j_V^*$} (m-2-2)
  (m-2-1) edge node[above] {$\Upsilon$} (m-2-2)
  ;
 \end{tikzpicture}
 \end{center}
 where $j_U: U = Y \setminus Z \to X$ and $j_V: V = X \setminus Z \to X$ are the inclusions. Thus, $\Upsilon$ induces a functor between the kernels of $j_U^*$ and $j_V^*$. On the factorization side, this is $\dabs_Z [\mathsf{vect}(X,G,w)]$ while on the singularity side this is $\op{D}^{\op{sg}}_{Z,G}(Y)$.
\end{proof}

\begin{proposition} \label{proposition: upsilon surjective}
 Let $X$ be a smooth variety equipped with an action of an affine algebraic group $G$ and an invariant section $w \in \Gamma(X, \mathcal L)^G$ for an invertible equivariant sheaf, $\mathcal L$. Let $Y$ be the vanishing locus of $w$ and let $Z$ be a closed $G$-invariant subset of $Y$. Assume that $w$ is not identically zero on any component of $X$. The functor,
 \begin{displaymath}
  \Upsilon : \op{D}^{\op{sg}}_{Z,G}(Y) \to \op{D}^{\op{abs}}_Z[\mathsf{fact}(X,G,w)],
 \end{displaymath}
 is essentially surjective.
\end{proposition}

\begin{proof}
 Let us check that $\Upsilon \circ \mathbf{L}\!\op{cok} \cong \op{Id}$. Recall that, for a coherent $G$-equivariant sheaf on $X$, $\mathcal E$, we define a factorization, $\mathcal H_{\mathcal E} := (\mathcal E, \mathcal E, w, 1_{\mathcal E})$. There is a short exact sequence of factorizations
 \begin{displaymath}
  0 \to \mathcal H_{\mathcal V_{-1}} \to \mathcal V \to \Upsilon \op{cok} \mathcal V \to 0
 \end{displaymath}
 for a factorization with locally-free components. As $\mathcal H_{\mathcal V}$ is contractible, $\mathcal V$ is quasi-isomorphic to $\Upsilon \op{cok} \mathcal V$.
\end{proof}

\begin{remark}
 One can prove that $\Upsilon$ is an equivalence by using arguments in the proof \cite[Theorem 2.7]{Pos2} and accounting for a group action, see also \cite[Theorem 3.5]{OrlMF} and \cite[Theorem 3.14]{PV2}.  We skip this, as only essential surjectivity is necessary for the generation arguments of Section~\ref{sec:generation of graded sing cat}.
\end{remark}

We finish by recording an observation concerning how $\Upsilon$ interacts with exterior products.

\begin{lemma} \label{lemma: box products are preserved}
 Let $X$ and $Y$ be smooth varieties and let $G$ and $H$ be affine algebraic groups acting on, respectively, $X$ and $Y$. Let $w \in \Gamma(X,\mathcal O_X(\chi))^G$ and $v \in \Gamma(Y, \mathcal O_Y(\chi'))^H$ for characters $\chi: G \to \mathbb{G}_m$ and $\chi' : H \to \mathbb{G}_m$. Let $i_w: Z_w \to X$ be the zero locus of $w$, $i_v: Z_v \to Y$ be the zero locus of $v$, and $i_{w \boxplus v}: Z_{w \boxplus v} \to X \times Y$ be the zero locus of $w \boxplus v$. 
 
 For any $\mathcal E \in \op{coh}_G Z_w$ and $\mathcal F \in \op{coh}_H Z_v$, there are natural isomorphisms of $G\times_{\mathbb{G}_m} H$-equivariant factorizations of $w \boxplus v$,
 \begin{displaymath}
  (\Upsilon \mathcal E) \boxtimes (\Upsilon \mathcal F) \cong \Upsilon \op{Res}_{G\times_{\mathbb{G}_m} H}^{G \times H} (i_{w*}\mathcal E \boxtimes i_{v*}\mathcal F).
 \end{displaymath}
\end{lemma}
 
\begin{proof}
 It is straightforward to check that both of these factorizations are 
 \begin{displaymath}
  \Upsilon \op{Res}^{G \times H}_{G\times_{\mathbb{G}_m} H} (\pi_1^* i_{w*} \mathcal E \otimes_{\mathcal O_{X \times Y}} \pi_2^* i_{v*} \mathcal F).
 \end{displaymath}
\end{proof}

\section{Generation of equivariant derived categories} \label{sec:generation of graded sing cat}

To identify the internal Hom dg-categories for equivariant factorizations, we will need to prove a generation statement for our candidate categories. In this section, we lay the groundwork and establish results to which we will appeal in Section~\ref{sec: bimod and functor categories for graded MFs}. 

For a singular variety, $X$, equipped with a $G$-action, we want to find a nice set of generators for the bounded derived category of coherent $G$-equivariant sheaves, $\dbcohG{G}{X}$. One natural approach would be to study generation in a compactly-generated triangulated category whose category of compact objects is exactly $\dbcohG{G}{X}$. Such categories do exist. Since $\dbcohG{G}{X}$ admits an enhancement to a dg-category, we could use the derived category of dg-modules over the enhancement. Or, a more geometric construction due to Krause, \cite{Kra2}, uses the homotopy category of injective complexes of quasi-coherent sheaves, in the non-equivariant setting. This could be extended to handle our situation. 

However, this is not the approach we choose. Instead, we follow the method of Rouquier in \cite{Ro2} and focus on $\dbqcohG{G}{X}$, the bounded derived category of all quasi-coherent $G$-equivariant sheaves on $X$. The category, $\dbqcohG{G}{X}$, is not compactly-generated as it does not possess all coproducts. However, the definition of a compact object is still valid and useful for $\dbqcohG{G}{X}$. Indeed \cite[Proposition 6.15]{Ro2} implies that the category of compact objects of $\dbqcohG{G}{X}$ is exactly $\dbcohG{G}{X}$. A further advantage of studying $\dbqcohG{G}{X}$ comes from the fact that local cohomology of a coherent $G$-equivariant, or quasi-coherent sheaf, is always bounded and quasi-coherent, though usually never coherent. 

Let us recall some notions of generation. 

\begin{definition} \label{def: compactly generated}
 Given a triangulated category, $\mathcal T$, we say that a subcategory, $\mathcal S$, is \newterm{thick} if it is triangulated and closed under summands.

 Let $\mathcal S'$ be another subcategory. We say that a subcategory, $\mathcal S$, \newterm{generates} $\mathcal S'$, if the smallest full triangulated subcategory of $\mathcal T$ containing $\mathcal S$, and closed under finite coproducts and summands, contains $\mathcal S'$. If $\mathcal S' = \mathcal T$, we shall often say that $\mathcal S$ generates.

 We say that $\mathcal S$ \newterm{generates $\mathcal S'$ up to infinite coproducts} if the smallest full triangulated subcategory of $\mathcal T$ containing $\mathcal S$, and closed under arbitrary coproducts and summands, contains $\mathcal S'$.  If $\mathcal S' = \mathcal T$, we shall often say that $\mathcal S$ generates up to infinite coproducts.

 In addition, recall that an object $C$ of $\mathcal T$ is called \newterm{compact} if $\op{Hom}_{\mathcal T}(C,\bullet)$ commutes with all coproducts.

 A triangulated category, $\mathcal T$, is \newterm{compactly-generated} if it is co-complete, the compact objects form a set, and $\op{Hom}_{\mathcal T}(C,X) = 0$ for all compact objects, $C$, of $\mathcal T$ implies that $X \cong 0$.
\end{definition}

The following is a now-standard result on compactly-generated triangulated categories.

\begin{lemma} \label{lem: cpt gen is gen up to inf}
 Assume $\mathcal T$ is a co-complete triangulated category and the compact objects in $\mathcal T$ form a set. Then, $\mathcal T$ is compactly-generated if and only if the compact objects generate up to infinite coproducts.
\end{lemma}

\begin{proof}
 See \cite{Nee2} for a proof.
\end{proof}

The following result generalizes one direction of Lemma~\ref{lem: cpt gen is gen up to inf}.

\begin{lemma} \label{lem: big and little generation}
 Let $\mathcal T$ be a triangulated category. Let $\mathcal C, \mathcal C'$ be a subcategory of compact objects of $\mathcal T$. If $\mathcal C$ generates $\mathcal C'$ up to infinite coproducts, then $\mathcal C$ generates $\mathcal C'$. 
\end{lemma}

\begin{proof}
 See \cite[Proposition 2.2.4]{BV} or \cite[Corollary 3.13]{Ro2}.
\end{proof}

Let $X$ be a separated, reduced scheme of finite type over $k$ and let $G$ be an affine algebraic group acting on $X$, $\sigma: G \times X \to X$. We record some generation results about the category, $\dbcohG{G}{X}$. Their statements and proofs are in the style of Rouquier, \cite{Ro2}, see also the arguments in \cite{LP}. 

\begin{definition}
 Let $\mathcal E$ be a quasi-coherent $G$-equivariant sheaf on $X$. Let $Z$ be a $G$-invariant subscheme of $X$ determined by a sheaf of ideals, $\mathcal I_Z$. We say that $E$ is \newterm{scheme-theoretically supported} on $Z$ if $\mathcal I_Z \mathcal E = 0$. We say that $\mathcal E$ is \newterm{set-theoretically supported} on $Z$ if $j^*\mathcal E = 0$ for the inclusion $j: X \setminus Z \to X$. 
 
 Let $\dZbqcohG{G}{X}{Z}$ be the triangulated subcategory of $\dbqcohG{G}{X}$ consisting of complexes whose cohomology sheaves are set-theoretically supported on $Z$. 
\end{definition}

\begin{remark}
 Let $l: Z \to X$ be the inclusion of $Z$ into $X$. Then, a quasi-coherent $G$-equivariant sheaf is scheme-theoretically supported on $Z$ if and only if it is in the essential image of $l_*$.
\end{remark}

\begin{lemma} \label{lem: big and small generators of dbgrmod}
 Let $Z$ be a $G$-invariant closed subscheme of $X$. Let $\mathcal S, \mathcal S'$ be subcategories of $\dZbcohG{G}{X}{Z}$. If $\mathcal S$ generates $\mathcal S'$ up to infinite coproducts in $\dZbqcohG{G}{X}{Z}$, then $\mathcal S$ generates $\mathcal S'$.
\end{lemma}

\begin{proof}
 We apply Lemma~\ref{lem: big and little generation}. The compact objects of $\dZbqcohG{G}{X}{Z}$ are exactly the objects of $\dZbcohG{G}{X}{Z}$ by Proposition 6.15 of \cite{Ro2}. 
\end{proof}

We also record the following useful statement.

\begin{lemma} \label{lem: colimit gen up inf coprod}
 Any quasi-coherent $G$-equivariant sheaf, $\mathcal E$, is generated up to infinite coproducts by its coherent $G$-equivariant subsheaves.
\end{lemma}

\begin{proof}
 Any quasi-coherent $G$-equivariant sheaf, $\mathcal E$, is the colimit of its coherent $G$-equivariant subsheaves, see \cite[Lemma 1.4]{Tho2}. The colimit fits to into an exact sequence,
\begin{displaymath}
 0 \to \bigoplus_{\substack{\mathcal F \subset \mathcal E \\ \mathcal F \text{ coherent } }} \mathcal F \to \bigoplus_{\substack{\mathcal F \subset \mathcal E \\ \mathcal F \text{ coherent } }} \mathcal F \to \op{colim} \mathcal F \cong \mathcal E \to 0.
\end{displaymath}
 Here the morphism, 
\begin{displaymath}
 \bigoplus_{\substack{\mathcal F \subset \mathcal E \\ \mathcal F \text{ coherent } }} \mathcal F \to \bigoplus_{\substack{\mathcal F \subset \mathcal E \\ \mathcal F \text{ coherent } }} \mathcal F, 
\end{displaymath}
 is defined as follows. Given two coherent equivariant subsheaves, $\mathcal F$ and $\mathcal F^{\prime}$, the morphism $\mathcal F \to \mathcal F^{\prime}$ equals
\begin{displaymath}
 \begin{cases} 0 & \text{ if } \mathcal F \not \subseteq \mathcal F^{\prime} \\
               -i & \text{ if } i: \mathcal F \hookrightarrow \mathcal F^{\prime} \text{ is a proper inclusion } \\
               1 & \text{ if } \mathcal F = \mathcal F^{\prime}.
 \end{cases}
\end{displaymath}
 As such, $\mathcal E$ is isomorphic to a cone over an endomorphism of a coproduct of coherent equivariant sheaves. Thus, $\mathcal E$ is generated, up to infinite coproducts, by its coherent $G$-equivariant subsheaves.
\end{proof}

Let $Z$ be a $G$-invariant closed subset of $X$ and $l: Z \to X$ be the inclusion.

\begin{lemma} \label{lem: scheme-theoretic support to set-theoretic support}
 The category, $\dZbqcohG{G}{X}{Z}$, is generated up to infinite coproducts by the image of $l_* : \dbcohG{G}{Z} \to \dbcohG{G}{X}$.
\end{lemma}

\begin{proof}
 If we can generate the cohomology sheaves of a bounded complex, then we can generate said complex. So we may reduce to generating all quasi-coherent $G$-equivariant sheaves that are set-theoretically supported on $Z$.  By Lemma~\ref{lem: colimit gen up inf coprod}, it suffices to generate all coherent $G$-equivariant sheaves that are set-theoretically supported on $Z$. However, for a coherent sheaf set-theoretically supported on $Z$, there is an $n$ such that $\mathcal I_Z^n \mathcal E = 0$. Thus, we have a filtration
 \begin{displaymath}
  0 = \mathcal I_Z^n \mathcal E \subset \mathcal I_Z^{n-1} \mathcal E \subset \cdots \subset \mathcal I_Z \mathcal E \subset \mathcal E. 
 \end{displaymath}
 There are exact triangles 
 \begin{displaymath}
  \mathcal I_Z^s \mathcal E \to \mathcal I_Z^{s-1} \mathcal E \to \mathcal F_s \to \mathcal I_Z^s \mathcal E[1]
 \end{displaymath}
 with $\mathcal F_s$ scheme-theoretically supported on $Z$. Thus, we see can generate a coherent $G$-equivariant sheaf using coherent $G$-equivariant sheaves scheme-theoretically supported on $Z$ finishing the argument.
\end{proof}

Before continuing with the course of the argument, let us recall the definition of local cohomology for equivariant sheaves. For the arguments of this section, local cohomology complexes provide an efficient means of chopping complexes up with respect to their support.

Let $Z$ be a $G$-invariant closed subset of $X$ and $\mathcal E$ be a quasi-coherent $G$-equivariant sheaf on $X$. Set
\begin{align*}
 \mathcal H_Z \mathcal E(U) & := \{ e \in \Gamma(U, \mathcal E) \mid \exists~ n,~ \mathcal I_Z^ne = 0 \} \\
 \mathcal Q_Z \mathcal E & := j_*j^* \mathcal E
\end{align*}
where $j: X \setminus Z \to X$ is the inclusion of the complement of $Z$. There is a left exact sequence
\begin{displaymath}
 0 \to \mathcal H_Z \mathcal E \to \mathcal E \to \mathcal Q_Z \mathcal E.
\end{displaymath}
Moreover, if $\mathcal E$ is flasque, there is a short exact sequence
\begin{displaymath}
 0 \to \mathcal H_Z \mathcal E \to \mathcal E \to \mathcal Q_Z \mathcal E \to 0.
\end{displaymath}
The quasi-coherent sheaf, $\mathcal H_Z \mathcal E$, inherits the $G$-equivariant structure of $\mathcal E$. Let
\begin{align*}
 \mathbf{R} \mathcal H_Z  & : \dbqcohG{G}{X} \to \dbqcohG{G}{X} \\
 \mathbf{R} \mathcal Q_Z  & : \dbqcohG{G}{X} \to \dbqcohG{G}{X}
\end{align*}
be the associated right-derived functors. Note that there is a triangle of exact functors
\begin{equation} \label{eq: local cohomology triangle}
 \mathbf{R}\mathcal H_Z \to \op{Id} \to \mathbf{R}\mathcal Q_Z \to \mathbf{R}\mathcal H_Z [1].
\end{equation}

We now use the above discussion to  reduce generation arguments to  the $G$-invariant irreducible case. 

\begin{lemma} \label{lemma: generate irreducible components}
 Let $X = Z_1 \cup Z_2$ be the decomposition of $X$ into two $G$-invariant closed subsets, $Z_1$ and $Z_2$. Let $l_i: Z_i \to X$ denote the inclusion of $Z_i$ into $X$. The objects in the essential image of the pushforward, $l_{i*}: \dbcohG{G}{Z_i} \to \dbcohG{G}{X}$, for $i=1,2$ generate $\dbqcohG{G}{X}$ up to infinite coproducts.
\end{lemma}

\begin{proof}
 We appeal to the exact triangle in Equation~\eqref{eq: local cohomology triangle} to see that $\mathcal E$ is generated by $\mathbf{R} \mathcal Q_{Z_1} \mathcal E$ and $\mathbf{R}\mathcal H_{Z_1} \mathcal E$. Note that $\mathbf{R} \mathcal Q_{Z_1}\mathcal E$ is supported on the complement of $Z_1$. As $X = Z_1 \cup Z_2$, $\mathbf{R}\mathcal Q_{Z_1}\mathcal E$ is set-theoretically supported on $Z_2$ while $\mathbf{R}\mathcal H_{Z_1} \mathcal E$ is set-theoretically supported on $Z_1$. Applying Lemma~\ref{lem: scheme-theoretic support to set-theoretic support}, finishes the argument.
\end{proof}

We will need to pass to the singular locus so we record a simple lemma.

\begin{lemma} \label{lemma: group action preserves singular locus}
 Let $\sigma: G \times X \to X$ be an action of an affine algebraic group, $G$, on a reduced, separated scheme of finite type, $X$. Let $\op{Sing} X$ denote the closed subset of $X$ defined by
 \begin{displaymath}
  \op{Sing} X := \{ x \in X \mid \mathcal O_{X,x} \text{ is not regular}\}.
 \end{displaymath}
 Equip $\op{Sing} X$ with the reduced, induced structure sheaf. Then, the action of $G$ on $X$ restricts to $\op{Sing} X$.
\end{lemma}

\begin{proof}
 It suffices to verify that $\sigma_g := \sigma(g,\bullet) : X \to X$ preserves $\op{Sing} X$ for each $g \in G$. However, $\sigma_g$ is an automorphism of $X$ and hence must preserve $\op{Sing} X$. 
\end{proof}

We will need to use normality of a variety which is not guaranteed by the assumptions of the proceeding lemmas. We take a moment to comment on lifting the action of $G$ to the normalization in an equivariant manner. 

\begin{lemma} \label{lemma: normalization is equivariant}
 Let $\nu: \widetilde{X} \to X$ be the normalization of $X$. There is a unique action of $G$ on $\widetilde{X}$ making $\nu$ $G$-equivariant.
\end{lemma}

\begin{proof}
 Since $G$ is smooth, $G \times \widetilde{X}$ is normal. The map $\sigma \circ (1 \times \nu): G \times \widetilde{X} \to X$ is dominant and therefore factors uniquely through $\nu$. Let $\tilde{\sigma}: G \times \widetilde{X} \to \widetilde{X}$ be the unique lift. The uniqueness of the lift also allows one to verify that $\widetilde{\sigma}$ is an action of $G$ on $X$. With this lift, $\nu: \widetilde{X} \to X$ becomes $G$-equivariant.
\end{proof}

\begin{lemma} \label{lemma: generation for maps with ample family}
 Let $f: X \to Y$ be a $G$-equivariant morphism such that $X$ possesses an $f$-ample family of equivariant line bundles, $\mathcal L_{\alpha}, \alpha \in A$. The full subcategory of $\dbcohG{G}{X}$ consisting of objects of the form
 \begin{displaymath}
  \mathcal L_{\alpha} \otimes f^*\mathcal E
 \end{displaymath}
 for $\mathcal E \in \op{coh}_G Y$ and $\alpha \in A$ generates all coherent $G$-equivariant sheaves of locally-finite projective dimension in $\op{Qcoh} X$. Moreover, if $Y$ possesses enough locally-free $G$-equivariant sheaves of finite rank, then the full subcategory of $\dbcohG{G}{X}$ consisting of objects of the form
 \begin{displaymath}
  \mathcal L_{\alpha} \otimes f^*\mathcal V
 \end{displaymath}
 for $\mathcal V \in \op{coh}_G Y$ locally-free and $\alpha \in A$ generates all coherent $G$-equivariant sheaves of locally-finite projective dimension in $\op{Qcoh} X$.
\end{lemma}

\begin{proof}
 Let $\mathcal E$ be a coherent $G$-equivariant sheaf of locally-finite projective dimension in $\op{Qcoh} X$. There is a finite set $A' \subseteq A$ such that 
 \begin{displaymath}
  \bigoplus_{\alpha \in A'} \mathcal L_{\alpha} \otimes f^*f_*(\mathcal L_{\alpha}^{-1} \otimes \mathcal E) \to \mathcal E 
 \end{displaymath}
 is an epimorphism as $\L_{\alpha}$ is an $f$-ample family. For each $\alpha$, there exists a coherent $G$-equivariant subsheaf, $\mathcal F_{\alpha}$, of $f_*(\mathcal L_{\alpha}^{-1} \otimes \mathcal E)$ such that the restriction of the co-unit morphism remains an epimorphism
 \begin{displaymath}
  \bigoplus_{\alpha \in A'} \mathcal L_{\alpha} \otimes f^* \mathcal F_{\alpha} \to \mathcal E.
 \end{displaymath}
 If we assume that $Y$ possesses enough locally-free $G$-equivariant sheaves of finite rank, there is a locally-free $G$-equivariant sheaf, $\mathcal V_{\alpha}$, on $Y$ and an epimorphism, $\mathcal V_{\alpha} \to \mathcal F_{\alpha}$. Pulling back and composing, we have an epimorphism
 \begin{displaymath}
  \bigoplus_{\alpha \in A'} \mathcal L_{\alpha} \otimes f^* \mathcal V_{\alpha} \to \mathcal E.
 \end{displaymath}
 
 Taking kernels and iterating this process we may construct an exact sequence 
 \begin{displaymath}
  \cdots \to \mathcal G_s \to \cdots \to \mathcal G_1 \to \mathcal E \to 0
 \end{displaymath}
 where each $\mathcal G_i$ is a sum of objects of the $\mathcal L_{\alpha} \otimes f^*\mathcal F_\alpha$ for some finite set of $\alpha \in A'$. Moreover, if $Y$ possesses enough locally-free equivariant sheaves, we may take $\mathcal E$ to be locally-free.
 
 Let $\mathcal K_s$ be the kernel of $\mathcal G_s \to \mathcal G_{s-1}$. We have a short exact sequence 
 \begin{displaymath}
  0 \to \mathcal G_s \to \cdots \to \mathcal G_1 \to \mathcal E \to 0. 
 \end{displaymath}
 This represents an element of 
 \begin{displaymath}
  \op{Ext}^s_{\op{Qcoh}_G X}(\mathcal E, \mathcal K_s).
 \end{displaymath}
 As $\mathcal E$ is an object of locally-finite projective dimension in $\op{Qcoh} X$, from Lemma~\ref{lem:graded gldim = gldim}, there is an $s_0$ such that 
 \begin{displaymath}
  \op{Ext}^s_{\op{Qcoh}_G X}(\mathcal E, \mathcal K_s) = 0
 \end{displaymath}
 for $s \geq s_0$. Take $s$ larger than $s_0$. Then, there is a quasi-isomorphism,
 \begin{displaymath}
  \mathcal K_s[s] \oplus \mathcal E \simeq \mathcal G_s \to \cdots \to \mathcal G_1.
 \end{displaymath}
 Thus, $\mathcal E$ is generated by objects of the form
 \begin{displaymath}
  \mathcal L_{\alpha} \otimes f^*\mathcal E
 \end{displaymath}
 for $\mathcal E \in \op{coh}_G Y$ and $\alpha \in A$.
  If $Y$ has enough equivariant locally-free sheaves, then $\mathcal E$ is generated by objects of the form
 \begin{displaymath}
 \mathcal L_{\alpha} \otimes f^*\mathcal E
 \end{displaymath}
 for $\mathcal E \in \op{coh}_G Y$ locally-free and $\alpha \in A$.
\end{proof}

Next, we demonstrate how to produce a set of generators from a set of generators of the singular locus of $X$. 

\begin{lemma} \label{lem: gen dbgrmod}
Let $X$ be a divisorial variety. Let $\op{Sing} X$ be the singular locus of $X$ with its reduced, induced structure sheaf. Let $l: \op{Sing} X \to X$ denote the inclusion. Let $Y$ be a closed subset of $X$ that is $G$-invariant. Then, the subcategory, whose objects are 
 \begin{itemize}
  \item $\nu_*\mathcal V$ where $\mathcal V$ is a locally-free $G$-equivariant sheaves of finite rank on $\widetilde{X}$ plus
    \item the objects in the essential image of the pushforward,
  \begin{displaymath}
   l_*: \dbcohG{G}{Y \cap \op{Sing} X} \to \dbcohG{G}{X},
  \end{displaymath}
 \end{itemize}
 generate the subcategory $\dZbqcohG{G}{X}{Y}$ up to infinite coproducts.
 
 Moreover, if one assumes that $X$ has enough locally-free $G$-equivariant sheaves, then the subcategory, whose objects are 
 \begin{itemize}
  \item locally-free $G$-equivariant sheaves of finite rank on $X$, plus
  \item the objects in the essential image of the pushforward,
  \begin{displaymath}
   l_*: \dbcohG{G}{Y \cap \op{Sing} X} \to \dbcohG{G}{X},
  \end{displaymath}
 \end{itemize}
 generates $\dZbqcohG{G}{X}{Y}$ up to infinite coproducts.
\end{lemma}

\begin{proof} 
 To generate a bounded complex, it suffices to generate its cohomology sheaves. Therefore, we may reduce to generating quasi-coherent $G$-equivariant sheaves set-theoretically supported on $Y$ up to infinite coproducts. By Lemma~\ref{lem: colimit gen up inf coprod}, it then suffices to generate coherent $G$-equivariant subsheaves set-theoretically supported on $Y$ up to infinite coproducts. Let $\mathcal E$ be a coherent $G$-equivariant sheaf. Complete the unit of the adjunction, $\mathcal E \to \nu_* \nu^* \mathcal E$ to an exact triangle
 \begin{displaymath}
  \mathcal E \to \nu_* \nu^* \mathcal E \to \mathcal D \to \mathcal E[1].
 \end{displaymath}
 Since $\nu$ is an isomorphism on $U$, $\mathcal D$ is set-theoretically supported on $\op{Sing} X \cap Y$. Since $\mathcal D$ is coherent  it is generated by the essential image of $l_*$ by Lemmas~\ref{lem: big and small generators of dbgrmod} and~\ref{lem: scheme-theoretic support to set-theoretic support}. Thus, to generate $\mathcal E$ it suffices to generate $\nu_* \nu^* \mathcal E$. Note also that if $\mathcal E$ is a locally-free sheaf of finite rank, then we generate $\nu_*\nu^* \mathcal E$ as we are allowed to use $\mathcal E$. 
 
 Set $Z = \nu^{-1}(\op{Sing} X)$ and $U = \widetilde{X} \setminus Z$. If $\mathcal V$ is a locally-free $G$-equivariant sheaf of finite rank on $\widetilde{X}$, then we have an exact triangle,
 \begin{displaymath}
  \mathbf{R} \mathcal H_{Z \cap \nu^{-1}(Y)} \mathcal V \to \mathcal V \to \mathbf{R} \mathcal Q_{Z \cap \nu^{-1}(Y)} \mathcal V \to \mathbf{R} \mathcal H_{Z \cap \nu^{-1}(Y)} \mathcal V[1].
 \end{displaymath}
 Applying $\nu_*$, we have another exact triangle,
 \begin{displaymath}
  \nu_* \mathbf{R} \mathcal H_{Z \cap \nu^{-1}(Y)} \mathcal V \to  \nu_*\mathcal V \to  \nu_*\mathbf{R} \mathcal Q_{Z \cap \nu^{-1}(Y)} \mathcal V \to  \nu_*\mathbf{R} \mathcal H_{Z \cap \nu^{-1}(Y)} \mathcal V[1].
 \end{displaymath}
 The set-theoretic support of $ \nu_* \mathbf{R} \mathcal H_{Z \cap \nu^{-1}(Y)} \mathcal V$ is contained in $\op{Sing} X \cap Y$ as $\nu(Z) = \op{Sing} X$. By Lemma~\ref{lem: scheme-theoretic support to set-theoretic support}, $ \nu_* \mathbf{R} \mathcal H_{Z \cap \nu^{-1}(Y)} \mathcal V$ is generated up to infinite coproducts by the essential image of $l_*$. Thus, $ \nu_* \mathbf{R} \mathcal Q_{Z \cap \nu^{-1}(Y)} \mathcal V$ is generated up to infinite coproducts by the full subcategory consisting of $\nu_*\mathcal V$ where $\mathcal V$ is a locally-free $G$-equivariant  on $\widetilde{X}$ and the essential image of $l_*$.

 Let $\mathcal E$ be a coherent $G$-equivariant sheaf on $X$ supported on $Y$. We have a triangle,
 \begin{displaymath}
  \mathbf{R} \mathcal H_{Z \cap \nu^{-1}(Y)}\nu^*\mathcal E \to \nu^*\mathcal E \to \mathbf{R} \mathcal Q_{Z \cap \nu^{-1}(Y)} \nu^*\mathcal E \to \mathbf{R} \mathcal H_{Z \cap \nu^{-1}(Y)} \nu^*\mathcal E[1],
 \end{displaymath}
 Applying $\nu_*$, we get another triangle,
 \begin{displaymath}
  \nu_*\mathbf{R} \mathcal H_{Z \cap \nu^{-1}(Y)} \nu^*\mathcal E \to \nu_*\nu^*\mathcal E \to \nu_*\mathbf{R} \mathcal Q_{Z \cap \nu^{-1}(Y)} \nu^*\mathcal E \to \nu_*\mathbf{R} \mathcal H_{Z \cap \nu^{-1}(Y)} \nu^*\mathcal E[1],
 \end{displaymath}
 we see that to generate $\nu_*\nu^*\mathcal E$ it suffices to generate $\nu_*\mathbf{R} \mathcal H_{Z \cap \nu^{-1}(Y)} \nu^*\mathcal E$ and $\nu_*\mathbf{R} \mathcal Q_{Z \cap \nu^{-1}(Y)} \nu^*\mathcal E$. The complex, $\nu_*\mathbf{R} \mathcal H_{Z \cap \nu^{-1}(Y)} \nu^*\mathcal E$, is set-theoretically supported on $\op{Sing} X \cap Y$. By Lemma~\ref{lem: scheme-theoretic support to set-theoretic support}, $\nu_*\mathbf{R} \mathcal H_{Z \cap \nu^{-1}(Y)} \nu^*\mathcal E$ is generated up to infinite coproducts by the essential image of $l_*$. Thus, we reduce to generating $\nu_*\mathbf{R} \mathcal Q_{Z \cap \nu^{-1}(Y)} \nu^*\mathcal E$. 
 
 As $\nu$ is an affine morphism, the pullback of an ample family remains an ample family. Using Theorem \ref{theorem: Thomason}, we may construct an exact complex
 \begin{displaymath}
  \cdots \to \mathcal V_s \to \cdots \to \mathcal V_2 \to \mathcal V_1 \to \nu^*\mathcal E \to 0.
 \end{displaymath}
 where each $\mathcal V_i$ is a locally-free $G$-equivariant sheaf of finite rank. Apply $j^*$ where $j: U = \widetilde{X} \setminus (Z \cap \nu^{-1}(Y)) \to \widetilde{X}$ is the inclusion. As $j^*$ is exact, the complex
 \begin{displaymath}
  \cdots \to j^*\mathcal V_s \to \cdots \to j^*\mathcal V_2 \to j^*\mathcal V_1 \to j^*\nu^*\mathcal E \to 0
 \end{displaymath}
 remains exact. Let $\mathcal K_s$ be the kernel of the map $j^*\mathcal V_s \to j^*\mathcal V_{s-1}$. The exact sequence
 \begin{displaymath}
  0 \to \mathcal K_s \to j^*\mathcal V_s \to \cdots \to j^*\mathcal V_2 \to j^*\mathcal V_1 \to j^*\nu^*\mathcal E \to 0
 \end{displaymath}
 represents an element of $\op{Ext}^{s}_{\op{Qcoh}_G  U}(j^*\nu^* \mathcal E, \mathcal K_s)$. As $j^*\nu^*\mathcal E$ is supported on the smooth subset, $U \supset \tilde{X} \setminus Z$, this vanishes for $s \geq s_0$, for some $s_0$, by Lemma~\ref{lem:graded gldim = gldim}. Consequently, there is a quasi-isomorphism
 \begin{displaymath}
  j^*\nu^*\mathcal E \oplus \mathcal K_s[s] \simeq j^*\mathcal V_s \to \cdots \to j^*\mathcal V_2 \to j^*\mathcal V_1.
 \end{displaymath}
 Applying $\nu_*\mathbf{R}j_*$, we see that $\nu_*\mathbf{R}\mathcal Q_{Z \cap \nu^{-1}(Y)} \nu^*\mathcal E$ is generated by  $\nu_*\mathbf{R}\mathcal Q_{Z \cap \nu^{-1}(Y)} \mathcal V_i$ for $1 \leq i \leq s$. We have already observed that we can generate $\nu_* \mathbf{R} \mathcal Q_{Z \cap \nu^{-1}(Y)} \mathcal V$ up to infinite coproducts when $\mathcal V$ is locally-free of finite rank. We conclude that we can generate $\nu_*\mathbf{R}\mathcal Q_{Z \cap \nu^{-1}(Y)} \nu^*\mathcal E$ using the subcategory consisting of $\nu$-pushforwards of $G$-equivariant invertible sheaves on $\widetilde{X}$ and the essential image of $l_*$ up to infinite coproducts finishing the argument.
 
 If we assume that $X$ has enough locally-free $G$-equivariant sheaves, then we can repeat the previous argument replacing $\widetilde{X}$ by $X$.
\end{proof} 

\begin{corollary} \label{corollary: generate with line bundles and singular locus}
 Assume that $X$ possesses enough locally-free $G$-equivariant sheaves. Let $\op{Sing} X$ be the singular locus of $X$ with its reduced, induced structure sheaf. Let $l: \op{Sing} X \to X$ denote the inclusion. Let $Y$ be a closed subset. The subcategory, $\dZbcohG{G}{X}{Y}$, is generated by all locally-free $G$-equivariant sheaves of finite rank on $X$ and all objects in the essential image of $l_*: \dbcohG{G}{\op{Sing} X \cap Y} \to \dbcohG{G}{X}$.
\end{corollary}

\begin{proof}
The second part of Lemma~\ref{lem: gen dbgrmod} states that the subcategory consisting of all locally-free coherent $G$-equivariant sheaves and the essential image of $l_*$ generates $\dZbqcohG{G}{X}{Y}$ up to infinite coproducts. Thus, by Lemma~\ref{lem: big and small generators of dbgrmod}, the subcategory consisting of all locally-free $G$-equivariant sheaves of finite rank on $X$ and the essential image of $l_*$ generates $\dZbcohG{G}{X}{Y}$.
\end{proof}

\begin{remark}
 One may use induction by iteratively passing to singular loci to produce a slightly smaller generating subcategory for $\dbcohG{G}{X}$.
\end{remark}

\begin{definition}
 Assume that $X$ has enough $G$-equivariant locally-free sheaves. Let $U$ be an open $G$-invariant subset of $X$ and let $\op{Perf}_{U,G} X$ be the full subcategory of $\dbqcohG{G}{X}$ whose restriction to $\dbqcohG{G}{U}$ is quasi-isomorphic to a bounded complex of locally-free $G$-equivariant sheaves. Let $\op{perf}_{U,G} X$ be the subcategory of $\op{Perf}_{U,G} X$ consisting of complexes quasi-isomorphic to bounded complexes of coherent sheaves. 
\end{definition}

\begin{lemma} \label{lemma: generate complexes perfect on U}
 Assume that $X$ has enough $G$-equivariant locally-free sheaves. The category, $\op{Perf}_{U,G} X$, is generated up to infinite coproducts by locally-free $G$-equivariant sheaves of finite rank and the image of
 \begin{displaymath}
  l_*: \dbcohG{G}{Y} \to \dbcohG{G}{X}
 \end{displaymath}
 where $Y = X \setminus U$. 
\end{lemma}

\begin{proof}
 Let $\mathcal E \in \op{Perf}_{U,G} X$. Using the assumption that $G$ has enough locally-free $G$-equivariant sheaves and a standard argument (see for the example the proof of Lemma \ref{lemma: generation for maps with ample family}), we may construct a bounded complex of locally-free sheaves $\mathcal P$ and a morphism
 \begin{displaymath}
  \mathcal P \to \mathcal E
 \end{displaymath}
 whose cone is a quasi-coherent sheaf that is locally-free on $U$. Since, by Lemma~\ref{lem: colimit gen up inf coprod},  we may generate bounded complexes of locally-free sheaves $\mathcal P$ up to infinite coproducts with locally-free coherent sheaves, it suffices to generate this cone. We continue with the assumption that $\mathcal E$ is a quasi-coherent sheaf.

 There is an exact triangle
 \begin{displaymath}
  \mathbf{R}\mathcal H_Z \mathcal E \to \mathcal E \to \mathbf{R}\mathcal Q_Z \mathcal E \to \mathbf{R}\mathcal H_Z \mathcal E[1].
 \end{displaymath}
 It suffices to generate $\mathbf{R}\mathcal H_Z \mathcal E$ and $\mathbf{R}\mathcal Q_Z \mathcal E$. We can generate $\mathbf{R}\mathcal H_Z \mathcal E$ up to infinite coproducts by the image of $l_*$ by Lemma~\ref{lem: scheme-theoretic support to set-theoretic support}. Thus, we reduce to generating $\mathbf{R}\mathcal Q_Z \mathcal E$. 
 
 Using the assumption of having enough $G$-equivariant locally-free sheaves, we may construct an exact complex
 \begin{displaymath}
  \cdots \to \mathcal V_s \to \cdots \to \mathcal V_1 \to \mathcal E \to 0
 \end{displaymath}
 with $\mathcal V_i$ being locally-free $G$-equivariant sheaves. Apply $j^*$ to get an exact complex
 \begin{displaymath}
  \cdots \to j^*\mathcal V_s \to \cdots \to j^*\mathcal V_1 \to j^*\mathcal E \to 0.
 \end{displaymath}
 Let $\mathcal K_s$ be the kernel of the map $j^*\mathcal V_s \to j^*\mathcal \mathcal V_{s-1}$. The exact sequence
 \begin{displaymath}
  0 \to \mathcal K_s \to j^*\mathcal V_s \to \cdots \to j^*\mathcal V_1 \to j^*\mathcal E \to 0
 \end{displaymath}
 represents an element of $\op{Ext}^{s}_{\op{Qcoh}_G  U}(j^* \mathcal E, \mathcal K_s)$. As $j^*\mathcal E$ is perfect, this vanishes for $s \geq s_0$, for some $s_0$, by Lemma~\ref{lem:graded gldim = gldim}. Assuming $s \geq s_0$, there is a quasi-isomorphism
 \begin{displaymath}
  j^*\mathcal E \oplus \mathcal K_s[s] \simeq j^*\mathcal V_s \to \cdots \to j^*\mathcal V_2 \to j^*\mathcal V_1.
 \end{displaymath}
 Pushing this forward via $\mathbf{R}j_*$ shows that $\mathbf{R}\mathcal Q_Z \mathcal E$ is generated by $\mathbf{R}\mathcal Q_Z \mathcal V$ for $\mathcal V$ locally-free. Thus, we reduce to generating $\mathbf{R}\mathcal Q_Z \mathcal V$ for $\mathcal V$ locally-free. But, for such a $\mathcal V$, there is an exact triangle, 
 \begin{displaymath}
  \mathbf{R}\mathcal H_Z \mathcal V \to \mathcal V \to \mathbf{R}\mathcal Q_Z \mathcal V \to \mathbf{R}\mathcal H_Z \mathcal V[1].
 \end{displaymath}
 and we may generate $\mathbf{R}\mathcal H_Z \mathcal V$ and $\mathcal V$ up to infinite coproducts by locally-free $G$-equivariant sheaves of finite rank and the image of $l_*$ by Lemma~\ref{lem: scheme-theoretic support to set-theoretic support}.
\end{proof}

\begin{corollary} \label{corollary: generate complexes perfect on U}
 Assume that $X$ has enough $G$-equivariant locally-free sheaves. The category, $\op{perf}_{U,G} X$, is generated by locally-free $G$-equivariant sheaves of finite rank and the image of
 \begin{displaymath}
  l_*: \dbcohG{G}{Y} \to \dbcohG{G}{X}
 \end{displaymath}
 where $Y = X \setminus U$. 
\end{corollary}

\begin{proof}
 This follows from Lemma~\ref{lemma: generate complexes perfect on U} by applying Lemma~\ref{lem: big and small generators of dbgrmod}.
\end{proof}

The following lemma shows that generators restrict under changing of the group.

\begin{lemma} \label{lem: generation under change of grading}
 Let $X$ be a separated, reduced, divisorial scheme of finite type equipped with a $G$ action. Assume that $G/H$ is affine. Then, $\dbcohG{H}{X}$ is generated by the essential image of $\op{Res}_H^G: \dbcohG{G}{X} \to \dbcohG{H}{X}$. 
\end{lemma}

\begin{proof}
 Recall that $\op{Res}^G_H$ factors as $\iota^* \circ \alpha^*$ where $\alpha: G \overset{H}{\times} X \to X$ is induced by the action of $G$ on $X$ and $\iota: X \to G \overset{H}{\times} X$ is induced by the unit element of $G$. The functor, $\iota^*: \dbcohG{G}{G \overset{H}{\times} X} \to \dbcohG{H}{X}$, is an equivalence by Lemma~\ref{lem: equivalence pullback} so it suffices to show that the image of $\alpha^*: \dbcohG{G}{X} \to \dbcohG{G}{G \overset{H}{\times} X}$ generates. We factor $\alpha$ as
 \begin{center}
 \begin{tikzpicture}[description/.style={fill=white,inner sep=2pt}]
  \matrix (m) [matrix of math nodes, row sep=2em, column sep=2em, text height=1.5ex, text depth=0.25ex]
  {  G \overset{H}{\times} X & G/H \times X \\ 
   X &  \\ };
  \path[->,font=\scriptsize]
  (m-1-1) edge node[left] {$\alpha$} (m-2-1)
  (m-1-1) edge node[above] {$\Phi$} (m-1-2)
  (m-1-2) edge node[below] {$p$} (m-2-1)
  ;
 \end{tikzpicture}
 \end{center} 
 where
 \begin{align*}
  \Phi: G \overset{H}{\times} X & \to G/H \times X \\
  (g,x) & \mapsto (gH,\sigma(g,x))
 \end{align*}
 and $p$ is the projection. The morphism, $\Phi$, is an isomorphism so we reduce to checking that the image of $p^*: \dbcohG{G}{X} \to \dbcohG{G}{G/H \times X}$ generates.
 
 Let us handle the case that $\op{dim} X = 0$. Under our standing assumptions $X$ is reduced, therefore $X$ is smooth. Since $G/H$ is affine, $\mathcal O_{G/H}$ is ample and is naturally equivariant. Lemma~\ref{lemma: generation for maps with ample family} applies directly to show that the essential image of $p^*$ generates 
  
 Now assume we have proven the statement for $X$ with all components of $X$ having dimension $<n$ and assume that $\op{dim} X = n$. By Corollary~\ref{corollary: generate with line bundles and singular locus}, $\dbcohG{G}{G/H \times X}$ is generated by $\nu'_*\mathcal V$ where $\mathcal V$ are locally-free $G$-equivariant sheaves of finite rank, $\nu': \widetilde{G/H \times X} \to G/H \times X$ is the normalization, and the essential image of
 \begin{displaymath}
  l_* : \dbcohG{G}{\op{Sing} G/H \times X} \to \dbcohG{G}{G/H \times X}.
 \end{displaymath}
 Since $G/H$ is smooth $\op{Sing} (G/H \times X) = G/H \times \op{Sing} X$. Applying the induction hypothesis, the essential image of 
 \begin{displaymath}
  p^* : \dbcohG{G}{\op{Sing} X} \to \dbcohG{G}{\op{Sing} G/H \times X}
 \end{displaymath}
 generates. Thus, the essential image of
 \begin{displaymath}
  p^*: \dbcohG{G}{X} \to \dbcohG{G}{G/H \times X}
 \end{displaymath}
 generates the essential image of $l_*$. We are left to generate the coherent $G$-equivariant sheaves, $\nu'_* \mathcal V$, for $\mathcal V$ locally-free $G$-equivariant sheaves of finite rank on the normalization. 
 
 Since $G/H$ is smooth, $\widetilde{G/H \times X} \cong G/H \times \widetilde{X}$. We have a commutative diagram.
 \begin{center}
 \begin{tikzpicture}[description/.style={fill=white,inner sep=2pt}]
  \matrix (m) [matrix of math nodes, row sep=2em, column sep=2em, text height=1.5ex, text depth=0.25ex]
  {  G/H \times \widetilde{X} & G/H \times X \\ 
     \widetilde{X} & X \\ };
  \path[->,font=\scriptsize]
  (m-1-1) edge node[above] {$1 \times \nu$} (m-1-2)
  (m-1-1) edge node[left] {$\tilde{p}$} (m-2-1)
  (m-1-2) edge node[right] {$p$} (m-2-2)
  (m-2-1) edge node[above] {$\nu$} (m-2-2)
  ;
 \end{tikzpicture}
 \end{center} 
 
 Applying Lemma~\ref{lemma: generation for maps with ample family}, since $\mathcal O_{G/H \times \widetilde{X}}$ is $\widetilde{p}$-ample, any locally-free $G$-equivariant sheaf of finite rank, $\mathcal V$, is generated by the essential image of $\widetilde{p}^*$. 
 
 Therefore, $\nu'_* \mathcal V = (1 \times \nu)_* \mathcal V$ is generated by the essential image of $(1 \times \nu)_* \circ \widetilde{p}^*$. As $p$ is flat, 
 \begin{displaymath}
  (1 \times \nu)_* \circ \widetilde{p}^* \cong p^* \circ \nu_*.
 \end{displaymath}
 Thus, $\nu'_* \mathcal V$ is generated by the essential image of $p^*$ finishing the proof.
\end{proof}

The next proposition demonstrates that exterior products generate in the equivariant setting. 

\begin{proposition} \label{prop:box gen dsing}
 Let $G$ and $H$ be affine algebraic groups, and $X$ and $Y$ be separated, reduced, divisorial schemes of finite type equipped with actions $G \times X \to X$ and $H \times Y \to Y$.  The subcategory consisting of $\mathcal E \boxtimes \mathcal F$ for $\mathcal E \in \op{coh}_G X$ and $\mathcal F \in \op{coh}_H Y$ generates $\dbqcohG{G \times H}{X \times Y}$ up to infinite coproducts.
\end{proposition}

\begin{proof}
 By Lemma~\ref{lem: colimit gen up inf coprod}, it suffices to generate all coherent $G \times H$-equivariant sheaves up to infinite coproducts. 

 We proceed by induction on the dimension of $X \times Y$. Assume that $\op{dim} X \times Y = 0$. The morphism, 
 \begin{displaymath}
  h:= f \times g: X \times Y \to \op{Spec} k \times \op{Spec} k \cong \op{Spec} k,
 \end{displaymath}
 coming from the product of the structure maps, $f: X \to \op{Spec} k$ and $g: Y \to \op{Spec} k$, is affine and $G \times H$-equivariant therefore $\O_{X \times Y}$ is ample. By Lemma~\ref{lem:graded gldim = gldim}, any object of $\op{coh} X \times Y$ has locally-finite projective dimension since $X \times Y$ is smooth. Applying Lemma~\ref{lemma: generation for maps with ample family}, we see that the essential image of $h^*$ generates $\dbcohG{G \times H}{X \times Y}$. Moreover,
 \begin{displaymath}
  h^*(\mathcal E \boxtimes \mathcal F) \cong f^*\mathcal E \boxtimes g^* \mathcal F.
 \end{displaymath}
 So validity of the claim in the case $X = Y = \op{Spec} k$ implies validity of the claim for all $X \times Y$ of dimension zero. For a finite dimensional $G \times H$ representation, the evaluation morphism
 \begin{displaymath}
  \op{Hom}_{\op{Qcoh}_H \op{Spec} k}(\op{Res}^{G \times H}_H V,V) \otimes_k \op{Res}^{G \times H}_H V \to V
 \end{displaymath}
 is an epimorphism. Here, $\op{Hom}_{\op{Qcoh}_H}(\op{Res}^{G \times H}_H V,V)$ is a $G$-representation. By Lemma~\ref{lem:graded gldim = gldim} the category of $G$-representations has finite global dimension.  Thus, there are enough exterior products to resolve any $G \times H$-representation finishing the base case of the induction. 
 
 Assume we have proven the statement whenever $\op{dim} X \times Y < n$ and let us treat a product with $\op{dim} X \times Y = n$. From Lemma~\ref{lem: gen dbgrmod}, $\dbqcohG{G \times H}{X \times Y}$ is generated up to infinite coproducts by $\nu_* \mathcal V$ for locally-free $G \times H$-equivariant sheaves of finite rank on the normalization $\widetilde{X \times Y}$ and the essential image of 
 \begin{displaymath}
  l_*: \dbcohG{G \times H}{\op{Sing} X \times Y} \to \dbcohG{G \times H}{X \times Y}.
 \end{displaymath}
 The singular locus of $X \times Y$ is the union of two closed subsets: $(\op{Sing} X) \times Y$ and $X \times \op{Sing} Y$. From Lemma~\ref{lemma: generate irreducible components}, $\dbqcohG{G \times H}{\op{Sing} (X \times Y)}$ is generated up to infinite coproducts by the images of $\dbcohG{G \times H}{(\op{Sing} X) \times Y}$ and $\dbcohG{G\times H}{X \times \op{Sing} Y}$ under pushforward. Using the induction hypothesis, exterior products generate both $\dbcohG{G \times H}{(\op{Sing} X) \times Y}$ and $\dbcohG{G\times H}{X \times \op{Sing} Y}$. Thus, the essential image of $l_*$ is generated up to infinite coproducts by exterior products. Next, we turn to locally-free equivariant sheaves pushed forward from the normalization.

 The normalization of $X \times Y$ is the product of the normalizations, $\widetilde{X} \times \widetilde{Y}$, \cite[Corollary 6.14.3]{EGA4.2}. We have assumed that $X$ and $Y$ have ample families. Since normalization is affine, $\widetilde{X}$ and $\widetilde{Y}$ have ample families given by the pullbacks from $X$ and $Y$, respectively. The exterior product of ample families is again an ample family. Since $\widetilde{X} \times \widetilde{Y}$ is normal, taking sufficient powers of each line bundle, we get an ample family where all the line bundles admit equivariant structures, \cite[Lemma 2.10]{Tho2}. Thus, for any locally-free $G$-equivariant sheaf, $\mathcal V$, there is an exact sequence of equivariant sheaves
 \begin{displaymath}
  \cdots \to \mathcal F_s \to \cdots \to \mathcal F_1 \to \mathcal V \to 0
 \end{displaymath}
 where each $\mathcal F_i$ is an exterior product. The locally-free sheaf $\mathcal V$ has locally-finite projective dimension, and thus is a summand of the complex
 \begin{displaymath}
  \mathcal F_s \to \cdots \to \mathcal F_1
 \end{displaymath}
 for $s$ sufficiently large. We see that exterior products generate all locally-free equivariant sheaves on $\widetilde{X} \times \widetilde{Y}$. Pushing forward an exterior product under the normalization, map $\widetilde{X} \times \widetilde{Y} \to X \times Y$, yields another exterior product via the projection formula and flat base change. Thus, exterior products also generate $\nu_* \mathcal V$ for locally-free $G \times H$-equivariant sheaves of finite rank, $\mathcal V$, on the normalization. This finishes the proof. 
\end{proof}

\begin{corollary} \label{cor:box gen dsing}
 Let $G$ and $H$ be affine algebraic groups, $X$ and $Y$ separated, reduced schemes of finite type equipped with actions $G \times X \to X$ and $H \times Y \to Y$. The subcategory consisting of $\mathcal E \boxtimes \mathcal F$ for $\mathcal E \in \op{coh}_G X$ and $\mathcal F \in \op{coh}_H Y$ generates $\dbcohG{G \times H}{X \times Y}$.
\end{corollary}

\begin{proof}
 This follows from Proposition~\ref{prop:box gen dsing} by applying Lemma~\ref{lem: big and small generators of dbgrmod}.
\end{proof}

Next, we turn our attention to showing that exterior products of factorizations generate the appropriate category. We will demonstrate such generation for exterior products in the singularity category and then use that to pass to factorizations.

\begin{lemma} \label{lemma: generate d-sing by box products}
 Let $X$ and $Y$ be smooth varieties and let $G$ and $H$ be affine algebraic groups acting on, respectively, $X$ and $Y$. Let $w \in \Gamma(X,\mathcal O_X(\chi))^G$ and $v \in \Gamma(Y, \mathcal O_Y(\chi'))^H$ for characters $\chi: G \to \mathbb{G}_m$ and $\chi' : H \to \mathbb{G}_m$. Let $i_w: Z_w \to X$ be the zero locus of $w$, $i_v: Z_v \to Y$ be the zero locus of $v$, and $i_{w \boxplus v}: Z_{w \boxplus v} \to X \times Y$ be the zero locus of $w \boxplus v$. Let $l: \op{Sing} Z_w \times \op{Sing} Z_v \to Z_{w \boxplus v}$ be the inclusion.
 
 Objects of the form $l_* \op{Res}^{G \times H}_{G\times_{\mathbb{G}_m} H} (\mathcal E \boxtimes \mathcal F)$ for $\mathcal E \in \op{coh}_G \op{Sing} Z_w$ and $\mathcal F \in \op{coh}_H \op{Sing} Z_v$ generate $\op{D}^{\op{sg}}_{Z_w \times Z_v, G\times_{\mathbb{G}_m} H}(Z_{w \boxplus v})$.
\end{lemma}

\begin{proof}
 By Corollary~\ref{corollary: generate complexes perfect on U}, the inverse image of $\op{D}^{\op{sg}}_{Z_w \times Z_v, G\times_{\mathbb{G}_m} H}(Z_{w \boxplus v})$ in $\dbcohG{G\times_{\mathbb{G}_m} H}{Z_{w \boxplus v}}$ is generated by locally-free $G$-equivariant sheaves and objects of $\dbcohG{G\times_{\mathbb{G}_m} H}{Z_{w \boxplus v}}$ set-theoretically supported on $Z_w \times Z_v$. By Corollary~\ref{corollary: generate with line bundles and singular locus}, locally-free $G$-equivariant sheaves of finite rank on $Z_{w \boxplus v}$ and objects in the image of $l_*$ for the inclusion $l: \op{Sing} Z_{w \boxplus v} \cap (Z_w \times Z_v) \to Z_{w \boxplus v}$ generate $\dZbcohG{G\times_{\mathbb{G}_m} H}{Z_{w \boxplus v}}{Z_w \times Z_v}$.  So, in combination, we can generate $\op{D}^{\op{sg}}_{Z_w \times Z_v, G\times_{\mathbb{G}_m} H}(Z_{w \boxplus v})$ using the essential image of $l_*$.  It remains to check that exterior products generate $\dbcohG{G \times H}{\op{Sing} Z_{w \boxplus v} \cap (Z_w \times Z_v)}$.
 
 Note that $\op{Sing} Z_{w \boxplus v} \cap (Z_w \times Z_v) = \op{Sing} Z_w \times \op{Sing} Z_v$. By Lemma~\ref{lem: generation under change of grading}, the essential image of 
 \begin{displaymath}
  \op{Res}^{G \times H}_{G\times_{\mathbb{G}_m} H} : \dbcohG{G \times H}{  \op{Sing} Z_w \times  \op{Sing} Z_v} \to \dbcohG{G\times_{\mathbb{G}_m} H}{ \op{Sing} Z_w \times  \op{Sing} Z_v}
 \end{displaymath}
 generates. Notice also that $\op{Sing} Z_w \times \op{Sing} Z_v$ is divisorial simply by pulling back the ample family.  Hence, we may apply Corollary~\ref{cor:box gen dsing}, to see that $\dbcohG{G\times_{\mathbb{G}_m} H}{\op{Sing} Z_w \times \op{Sing} Z_v}$ is generated by $\mathcal E \boxtimes \mathcal F$ for $\mathcal E \in \op{coh}_G Z_w$ and $\mathcal F \in \op{coh}_H Z_v$.
\end{proof}

\begin{lemma} \label{lemma: generate factorizations by box products}
 Let $X$ and $Y$ be smooth varieties and let $G$ and $H$ be affine algebraic groups acting on, respectively, $X$ and $Y$. Let $w \in \Gamma(X,\mathcal O_X(\chi))^G$ and $v \in \Gamma(Y, \mathcal O_Y(\chi'))^H$ for characters $\chi: G \to \mathbb{G}_m$ and $\chi' : H \to \mathbb{G}_m$. Let $i_w: Z_w \to X$ be the zero locus of $w$ and let $i_v: Z_v \to Y$ be the zero locus of $v$. The derived category of coherent factorizations supported on $Z_w \times Z_v$, $\dabs_{Z_w \times Z_v} [\mathsf{fact}(X \times Y,G \times_{\mathbb{G}_m} H,w \boxplus v)]$, is generated by exterior products. 
\end{lemma}

\begin{proof}
 By Lemma~\ref{lemma: generate d-sing by box products}, objects of the form $l_* \op{Res}^{G \times H}_{G\times_{\mathbb{G}_m} H} \mathcal E \boxtimes \mathcal F$ for $\mathcal E \in \op{coh}_G \op{Sing} Z_w$ and $\mathcal F \in \op{coh}_H \op{Sing} Z_v$ generate $\op{D}^{\op{sg}}_{Z_w \times Z_v, G\times_{\mathbb{G}_m} H}(Z_{w \boxplus v})$. By Lemma~\ref{lemma: box products are preserved}, for any $\mathcal E \in \op{coh}_G Z_w$ and $\mathcal F \in \op{coh}_H Z_v$, there are natural isomorphisms of $G\times_{\mathbb{G}_m} H$-equivariant factorizations of $w \boxplus v$,
 \begin{displaymath}
  (\Upsilon \mathcal E) \boxtimes (\Upsilon \mathcal F) \cong \Upsilon \op{Res}_{G\times_{\mathbb{G}_m} H}^{G \times H} (i_{w*}\mathcal E \boxtimes i_{v*}\mathcal F).
 \end{displaymath} 
 Finally, by Proposition~\ref{proposition: upsilon surjective}, $\Upsilon$ is essentially surjective. Thus, $(\Upsilon \mathcal E) \boxtimes (\Upsilon \mathcal F)$ for $\mathcal E \in \op{coh}_G Z_w$ and $\mathcal F \in \op{coh}_H Z_v$ generate $\dabs_{Z_w \times Z_v} [\mathsf{fact}(X \times Y,G\times_{\mathbb{G}_m} H, w \boxplus v)]$.
\end{proof}

\section{Bimodule and functor categories for equivariant factorizations} \label{sec: bimod and functor categories for graded MFs}

\subsection{Morita products and functor categories for factorization categories} \label{section: morita products}

We now turn to studying tensor products and internal-homomorphism dg-categories of factorization categories in the homotopy category of dg-categories, $\hodgcat$. The main references for background are \cite{Kel,Toe}.

\begin{definition}
 A dg-functor, $f: \mathsf{C} \to \mathsf{D}$, is a \newterm{quasi-equivalence} if 
 \begin{displaymath}
  \op{H}^{\bullet}(f): \op{H}^{\bullet}(\op{Hom}_{\mathsf{C}}(c,c')) \to \op{H}^{\bullet}(\op{Hom}_{\mathsf{D}}(f(c),f(c')))
 \end{displaymath}
 is an isomorphism for all $c,c' \in \mathsf{C}$ and $[f]: [\mathsf{C}] \to [\mathsf{D}]$ is essentially surjective. 

 Let $\op{Ho(dg-cat)}_k$ denote the localization of $\op{dg-cat}_k$ at the class of quasi-equivalences. This category is called  \newterm{the homotopy category of dg-categories}. If $\mathsf{C}$ and $\mathsf{D}$ are quasi-equivalent, we shall write $\mathsf{C} \simeq \mathsf{D}$.
\end{definition}

\begin{definition}
 Let $\mathsf D$ be a dg-category. The category of left $\mathsf D$-modules, denoted $\mathsf D\op{-Mod}$, is the dg-category of dg-functors, $\mathsf D \to \mathsf{C}(k)$ where $\mathsf{C}(k)$ is the dg-category of chain complexes of vector spaces over $k$. The category of right $\mathsf D$-modules is the category of left $\mathsf D^{\op{op}}$-modules. 
 
 Each object $d \in \mathsf D$ provides a representable right module
 \begin{align*}
  h_d : \mathsf{D}^{\op{op}} & \to \mathsf{C}(k) \\
  d' & \mapsto \op{Hom}_{\mathsf{D}}(d',d).
 \end{align*}
 We denote the dg-Yoneda embedding by $h: \mathsf{D} \to \mathsf{D}^{\op{op}}\op{-Mod}$.
 
 The Verdier quotient of $[\mathsf{D}\op{-Mod}]$ by the subcategory of acyclic modules is called the \newterm{derived category of $\mathsf{D}$-modules} and is denoted by $\op{D}[\mathsf{D}\op{-Mod}]$. The smallest thick subcategory of $\op{D}[\mathsf{D}^{\op{op}}\op{-Mod}]$ containing the image of $[h]$ is called the \newterm{category of perfect $\mathsf{D}$-modules} and is denoted by $\op{perf}(\mathsf{D})$.
\end{definition}

\begin{remark}
 Throughout the paper, with the exception of the proof of Corollary~\ref{corollary: morita product of factorizations}, we will take $\mathsf C$ to be a \newterm{quasi-small} dg-category. A dg-category $\mathsf{D}$ is quasi-small if $[\mathsf{D}]$ is essentially small. In this case, we can choose a small full subcategory of $\mathsf{D}$ quasi-equivalent to $\mathsf{D}$ and work with that subcategory to define categories of modules and bimodules. This sidesteps certain set-theoretic issues in the quasi-small case. However, doing this in each example is tedious and not edifying. So we will suppress these arguments throughout the paper.
 
 When $\mathsf C$ is not quasi-small, but only $\mathbb{U}$-small, one only considers $\mathbb{U}$-small dg-modules. We suppress any of the set-theoretic issues as we do not ascend to a higher universe in the proof of Corollary~\ref{corollary: morita product of factorizations}.
\end{remark}

\begin{definition}
 Let $\mathsf{C}$ and $\mathsf{D}$ be two dg-categories. A \newterm{quasi-functor} $a: \mathsf{C} \to \mathsf{D}$ is a dg-functor
 \begin{displaymath}
  a: \mathsf{C} \to \mathsf{D}^{\op{op}}\op{-Mod}
 \end{displaymath}
 such that for each $c \in \mathsf{C}$,  $a(c)$ is quasi-isomorphic to $h_d$ for some $d \in \mathsf{D}$. Note that a quasi-functor corresponds to a bimodule $a \in \mathsf{C} \otimes_k \mathsf{D}^{\op{op}}\op{-Mod}$. Also note, that any quasi-functor induces a functor on homotopy categories which we denote by $[a]: [\mathsf{C}] \to [\mathsf{D}]$. In particular, it makes sense to extend the definition of quasi-equivalence to quasi-functors.
\end{definition}

\begin{lemma} \label{lemma: morphisms in hodgcat}
 The isomorphism classes of morphisms from $\mathsf{C}$ to $\mathsf{D}$ in $\op{Ho(dg-cat)}_k$ are in bijection with isomorphism classes of quasi-functors from $\mathsf{C}$ to $\mathsf{D}$ viewed as objects of $\op{D}[\mathsf{C} \otimes_k \mathsf{D}^{\op{op}}\op{-Mod}]$. In particular, two dg-categories are quasi-equivalent if and only if they are related by a quasi-functor that is a quasi-equivalence.
\end{lemma}

\begin{proof}
 This is an immediate consequence of the internal Hom constructed by T\"oen for $\op{Ho(dg-cat)}_k$, \cite[Theorem 6.1]{Toe}.
\end{proof}

The following provides a useful language to keep track of dg-categories.

\begin{definition}
 Let $\mathcal T$ be a triangulated category. An \newterm{enhancement} of $\mathcal T$ is a dg-category, $\mathsf C$, and an exact equivalence
 \begin{displaymath}
  \epsilon: [\mathsf C] \to \mathcal T.
 \end{displaymath}
 \end{definition}

We recall the following result concerning dg-quotients.

\begin{theorem} \label{theorem: drinfeld's dg-quotient}
 Let $\mathsf C$ be a small dg-category and let $\mathsf D$ be a full dg-subcategory. There exists a dg-category $\mathsf C /\mathsf D$, unique in $\op{Ho}(\op{dg-cat}_k)$, and dg-functor $\xi: \mathsf C \to \mathsf C/\mathsf D$ such that for any morphism $\eta: \mathsf C \to \mathsf A$ in $\op{Ho}(\op{dg-cat}_k)$ with $\eta|_{\mathsf{D}} = 0$ there exists a morphism $\lambda: \mathsf{C}/\mathsf{D} \to \mathsf{A}$ with $\eta \cong \lambda \circ \xi$.
\end{theorem}

\begin{proof}
 This is \cite[Theorem 3.4]{Drin}. The objects of $\mathsf C/\mathsf D$ in \cite[Section 3]{Drin} are exactly the objects of $\mathsf C$. Note that we use that $k$ is a field here.
\end{proof}

Let $X$ be a smooth variety, $G$ be an affine algebraic group acting on $X$, $\mathcal L$ be an invertible $G$-equivariant sheaf on $X$, and $w \in \Gamma(X,\mathcal L)^G$.

\begin{definition}
 Let $\dabs \mathsf{vect}(X,G,w)$ denote the dg-quotient as in Theorem~\ref{theorem: drinfeld's dg-quotient} of $\mathsf{vect}(X,G,w)$ by $\mathsf{acycvect}(X,G,w)$.
\end{definition}

\begin{corollary} \label{corollary: enhancements via dg-quotients}
 The dg-quotient $\dabs \mathsf{vect}(X,G,w)$ is an enhancement of $\dabs [\mathsf{fact}(X,G,w)]$.
\end{corollary}

\begin{proof}
 The result is an immediate consequence of Theorem~\ref{theorem: drinfeld's dg-quotient} and Proposition~\ref{prop: projective enhancement}.
\end{proof}

\begin{definition}
 We will need the following factorization of $0$. Let $\mathcal J$ be an injective resolution of $\mathcal O_X$ and consider the factorization, $\mathcal I^{\mathcal O} := \Upsilon \mathcal J$, of $0$.
\end{definition}

\begin{proposition} \label{proposition: equivalent enhancements}
 The dg-category $\mathsf{Inj}(X,G,w)$ is an enhancement of $\dabs [\mathsf{Fact}(X,G,w)]$. The dg-category $\mathsf{Inj}_{\op{coh}}(X,G,w)$ is an enhancement of $\dabs [\mathsf{fact}(X,G,w)]$. There is an isomorphism in $\hodgcat$ between $\mathsf{Inj}_{\op{coh}}(X,G,-w)$ and $\dabs \mathsf{vect}(X,G,w)^{\op{op}}$.
 
 If $X$ is affine and $G$ is reductive, then, additionally, $\mathsf{Vect}(X,G,w)$ is an enhancement of $\dabs [\mathsf{Fact}(X,G,w)]$ and $\mathsf{vect}(X,G,w)$ is an enhancement of $\dabs [\mathsf{fact}(X,G,w)]$. 
\end{proposition}

\begin{proof}
 The first two statements follow from Proposition~\ref{prop: injective enhancement}. While the final two follow from Proposition~\ref{prop: projective enhancement}. 
 For the third statement, consider the dg-functor,
 \begin{align*}
  \mathcal Hom_X(\bullet, \mathcal I^{\mathcal O}) & : \mathsf{vect}(X,G,w)^{\op{op}} \to \mathsf{Inj}_{\op{coh}}(X,G,-w),
 \end{align*}
which sends the subcategory $\mathsf{acycvect}(X,G,w)^{\op{op}}$ to acyclic factorizations with injective components. Thus, the induced functor
 \begin{displaymath}
  \mathcal Hom_X(\bullet, \mathcal I^{\mathcal O}) : \mathsf{acycvect}(X,G,w)^{\op{op}} \to \mathsf{Inj}_{\op{coh}}(X,G,-w)
 \end{displaymath}
 vanishes on homotopy categories. By \cite[Theorem 1.6.2]{Drin} and Lemma~\ref{lemma: double dual is ok for coherent factorizations}, $\mathsf{Inj}_{\op{coh}}(X,G,-w)$ is a dg-quotient of $\mathsf{vect}(X,G,w)^{\op{op}}$ by $\mathsf{acycvect}(X,G,w)^{\op{op}}$. By uniqueness, there is an isomorphism in $\hodgcat$ between $\mathsf{Inj}_{\op{coh}}(X,G,-w)$ and $\dabs \mathsf{vect}(X,G,w)^{\op{op}}$.
\end{proof}

\begin{corollary} \label{corollary: opposite category}
 There is an isomorphism in $\hodgcat$,
 \begin{displaymath}
  \mathsf{Inj}_{\op{coh}}(X,G,-w) \cong \mathsf{Inj}_{\op{coh}}(X,G,w)^{\op{op}}.
 \end{displaymath}
\end{corollary}

\begin{proof}
 The dg-functor 
 \begin{displaymath}
  \mathcal Hom_X(\bullet, \mathcal O_X) : \mathsf{vect}(X,G,w)^{\op{op}} \to \mathsf{vect}(X,G,-w)
 \end{displaymath}
 is an equivalence of dg-categories that preserves the subcategories of acyclic locally-free factorizations. Thus, it induces a quasi-equivalence
 \begin{displaymath}
  \dabs \mathsf{vect}(X,G,w)^{\op{op}} \cong \dabs \mathsf{vect}(X,G,-w).
 \end{displaymath}
 Applying Proposition~\ref{proposition: equivalent enhancements} finishes the argument.
\end{proof}

\begin{definition}
 Let $\mathsf{Inj}_Z(X,G,w)$ be the full subcategory of $\mathsf{Inj}(X,G,w)$ consisting of factorizations acyclic off of $Z$. Let $\mathsf{Inj}_{\op{coh},Z}(X,G,w)$ be the full subcategory of $\mathsf{Inj}_{\op{coh}}(X,G,w)$ consisting of factorizations acyclic off of $Z$. Let $\overline{\mathsf{Inj}}_{\op{coh},Z}(X,G,w)$ be the full subcategory of $\mathsf{Inj}(X,G,w)$ consisting of factorizations acyclic off of $Z$ and compact in $[\mathsf{Inj}_Z(X,G,w)]$.
 
 Let $\mathsf{Vect}_Z(X,G,w)$ be the full subcategory of $\mathsf{Vect}(X,G,w)$ consisting factorizations acyclic off of $Z$. Let $\mathsf{vect}_Z(X,G,w)$ be the full subcategory of $\mathsf{vect}(X,G,w)$ consisting factorizations acyclic off of $Z$. Let $\overline{\mathsf{Vect}}_{Z}(X,G,w)$ be the full subcategory of $\mathsf{Vect}(X,G,w)$ consisting of factorizations acyclic off of $Z$ and compact in $\dabs [\mathsf{Vect}_Z(X,G,w)]$.
\end{definition}

\begin{corollary}
 The dg-category $\mathsf{Inj}_Z(X,G,w)$ is an enhancement of $\dabs_Z [\mathsf{Fact}(X,G,w)]$. The dg-category $\mathsf{Inj}_{\op{coh},Z}(X,G,w)$ is an enhancement of $\dabs_Z [\mathsf{fact}(X,G,w)]$. 
 
 If $X$ is affine and $G$ is reductive, then, additionally, $\mathsf{Vect}_Z(X,G,w)$ is an enhancement of $\dabs_Z [\mathsf{Fact}(X,G,w)]$ and $\mathsf{vect}_Z(X,G,w)$ is an enhancement of $\dabs_Z [\mathsf{fact}(X,G,w)]$. Moreover, $\mathsf{Inj}_{\op{coh},Z}(X,G,w)$ is quasi-equivalent to $\mathsf{vect}_Z(X,G,w)$. 
\end{corollary}

\begin{proof}
 This is an immediate consequence of Proposition~\ref{proposition: equivalent enhancements} given the definitions above.
\end{proof}

\begin{theorem} \label{theorem: boxproduct of factorization dg-cats is a factorization dg-cat}
 Let $X$ and $Y$ be smooth varieties and let $G$ and $H$ be affine algebraic groups acting on, respectively, $X$ and $Y$. Let $w \in \Gamma(X,\mathcal O_X(\chi))^G$ and $v \in \Gamma(Y, \mathcal O_Y(\chi'))^H$ for characters $\chi: G \to \mathbb{G}_m$ and $\chi' : H \to \mathbb{G}_m$. Let $i_w: Z_w \to X$ be the zero locus of $w$ and let $i_v: Z_v \to Y$ be the zero locus of $v$.

 Assume that $\chi'-\chi$ is not torsion. The dg-functor,
 \begin{align*}
  \lambda_{w \boxplus v}: \mathsf{Inj}_{Z_w \times Z_v}(X \times Y, G\times_{\mathbb{G}_m} H, w \boxplus v) & \to (\mathsf{Inj}_{\op{coh}}(X,G,w) \otimes_k \mathsf{Inj}_{\op{coh}}(Y,H,v))^{\op{op}}\op{-Mod} \\
  \mathcal I & \mapsto \op{Hom}_{\mathsf{Fact}(X \times Y, G\times_{\mathbb{G}_m} H, w \boxplus v)}( \bullet \boxtimes \bullet , \mathcal I).
 \end{align*}
 induces an equivalence
 \begin{displaymath}
  \epsilon_{w \boxplus v} : \dabs_{Z_w \times Z_v} [\mathsf{Fact}(X \times Y, G\times_{\mathbb{G}_m} H, w \boxplus v)] \to \op{D}( (\mathsf{Inj}_{\op{coh}}(X,G,w) \otimes_k \mathsf{Inj}_{\op{coh}}(Y,H,v))^{\op{op}}\op{-Mod})
 \end{displaymath}
 satisfying
 \begin{displaymath}
  \epsilon_{w \boxplus v}(\mathcal E \boxtimes \mathcal F) \cong h_{\mathcal E \otimes_k \mathcal F}.
 \end{displaymath}
 
 If, in addition, $X$ and $Y$ are affine and $G$ and $H$ are reductive, then the dg-functor
 \begin{align*}
  \lambda_{w \boxplus v}: \mathsf{Vect}_{Z_w \times Z_v}(X \times Y, G\times_{\mathbb{G}_m} H, w \boxplus v) & \to (\mathsf{vect}(X,G,w) \otimes_k \mathsf{vect}(Y,H,v))^{\op{op}}\op{-Mod}) \\
  \mathcal V & \mapsto \op{Hom}_{\mathsf{Fact}(X \times Y, G\times_{\mathbb{G}_m} H, w \boxplus v)}( \bullet \boxtimes \bullet ,\mathcal V).
 \end{align*}
 induces an equivalence
 \begin{displaymath}
  \epsilon_{w \boxplus v} : \dabs_{Z_w \times Z_v} [\mathsf{Fact}(X \times Y, G\times_{\mathbb{G}_m} H, w \boxplus v)] \to \op{D}( (\mathsf{vect}(X,G,w) \otimes_k \mathsf{vect}(Y,H,v))^{\op{op}}\op{-Mod})
 \end{displaymath}
 satisfying
 \begin{displaymath}
  \epsilon_{w \boxplus v}(\mathcal E \boxtimes \mathcal F) \cong h_{\mathcal E \otimes_k \mathcal F}.
 \end{displaymath}
\end{theorem}

\begin{proof}
 We just need to check that the induced functor,
 \begin{gather*}
  \epsilon_{w \boxplus v} : \dabs_{Z_w \times Z_v} [\mathsf{Fact}(X \times Y, G\times_{\mathbb{G}_m} H, w \boxplus v)] \cong [\mathsf{Inj}_{Z_w \times Z_v}(X \times Y, G\times_{\mathbb{G}_m} H, w \boxplus v)] \\ \overset{[\lambda_{w \boxplus v}]}{\to} [\mathsf{Inj}_{\op{coh}}(X,G,w) \otimes_k \mathsf{Inj}_{\op{coh}}(Y,H,v))^{\op{op}}\op{-Mod}] \\ \to \op{D}((\mathsf{Inj}_{\op{coh}}(X,G,w) \otimes_k \mathsf{Inj}_{\op{coh}}(Y,H,v))^{\op{op}}\op{-Mod})
 \end{gather*}
 is an equivalence. Note that $\epsilon_{w \boxplus v}$ commutes with coproducts since the exterior products, $\mathcal E \boxtimes \mathcal F$, are compact in $\dabs [\mathsf{Fact}(X \times Y, G\times_{\mathbb{G}_m} H, w \boxplus v)]$ when $\mathcal E \in \mathsf{Inj}_{\op{coh}}(X,G,w)$ and $\mathcal F \in \mathsf{Inj}_{\op{coh}}(Y,H,v)$. The triangulated category, $[\mathsf{Inj}_{Z_w \times Z_v}(X \times Y, G\times_{\mathbb{G}_m} H, w \boxplus v)]$, is compactly generated by Proposition~\ref{prop: we have a compactly-gen triangulated cat} and the objects, $h_{\mathcal E \otimes \mathcal F}$, for a $\mathcal E \in \mathsf{Inj}_{\op{coh}}(X,G,w)$ and $\mathcal F \in \mathsf{Inj}_{\op{coh}}(Y,H,v)$, form a compact set of generators for the category, $\op{D}((\mathsf{Inj}_{\op{coh}}(X,G,w) \otimes_k \mathsf{Inj}_{\op{coh}}(Y,H,v))^{\op{op}}\op{-Mod})$.
 
 Thus to check that $\epsilon_{w \boxplus v}$ is an equivalence it suffices to check that it takes a compact generating set to a compact generating set and is fully-faithful on those sets. Let us first show that there is a quasi-isomorphism of bimodules
 \begin{displaymath}
  h_{\mathcal E \otimes \mathcal F} := \op{Hom}_{\mathsf{Inj}_{\op{coh}(X,G,w)}}(\bullet, \mathcal E) \otimes_k \op{Hom}_{\mathsf{Inj}_{\op{coh}(Y,H,v)}}(\bullet, \mathcal F) \simeq \op{Hom}_{\mathsf{Fact}(X \times Y, G\times_{\mathbb{G}_m} H, w \boxplus v)}( \bullet \boxtimes \bullet, \mathcal I_{\mathcal E \boxtimes \mathcal F})
 \end{displaymath}
 where we have a morphism of factorizations $\mathcal E \boxtimes \mathcal F \to \mathcal I_{\mathcal E \boxtimes \mathcal F}$ whose cone is acyclic and where the components of $\mathcal I_{\mathcal E \boxtimes \mathcal F}$ have injective components. We have the natural morphism
 \begin{gather*}
  \op{Hom}_{\mathsf{Inj}_{\op{coh}(X,G,w)}}(\bullet, \mathcal E) \otimes_k \op{Hom}_{\mathsf{Inj}_{\op{coh}(Y,H,v)}}(\bullet, \mathcal F) \overset{ \boxtimes }{\to} \op{Hom}_{\mathsf{Fact}(X \times Y, G\times_{\mathbb{G}_m} H, w \boxplus v)}( \bullet \boxtimes \bullet, \mathcal E \boxtimes \mathcal F) \\ \to \op{Hom}_{\mathsf{Fact}(X \times Y, G\times_{\mathbb{G}_m} H, w \boxplus v)}( \bullet \boxtimes \bullet, \mathcal I_{\mathcal E \boxtimes \mathcal F}).
 \end{gather*}
 where the later morphism is given by composing with $\mathcal E \boxtimes \mathcal F \to \mathcal I_{\mathcal E \boxtimes \mathcal F}$. By Lemma~\ref{lemma: box product is tensor product in the derived category}, this is a quasi-isomorphism. Again, appealing to Lemma~\ref{lemma: box product is tensor product in the derived category} shows that $\epsilon_{w \boxplus v}$ is fully-faithful on exterior products. It remains to check that exterior products are generators for $\dabs_{Z_w \times Z_v} [\mathsf{Fact}(X \times Y, G\times_{\mathbb{G}_m} H, w \boxplus v)]$, but this is Lemma~\ref{lemma: generate factorizations by box products}.
 
 The statements with $X$ and $Y$ affine and $G$ and $H$ reductive follow via an analogous argument.  Indeed, in this case, taking $G$ invariants is exact and locally-free objects are projective so we can work with locally-free objects in the exact same manner.
\end{proof}

\begin{definition}
 Let $\mathsf{C}$ be a dg-category. The category $\mathsf{C}\op{-Mod}$ possesses the structure of a model category with $f: F \to G$ being a fibration, respectively a weak equivalence, if $f(c): F(c) \to G(c)$ is an epimorphism in each degree, respectively a quasi-isomorphism, for each $c \in \mathsf{C}$. This determines the cofibrations: they are those morphisms satisfying the left lifting property with respect to all acyclic fibrations, i.e. those maps that are fibrations and weak equivalences.
 
 Any object of $\mathsf{C}\op{-Mod}$ is fibrant. We let $\widehat{\mathsf{C}}$ be the subcategory of cofibrant objects in $\mathsf{C}^{\op{op}}\op{-Mod}$. The dg-category $\widehat{\mathsf{C}}$ is an enhancement of $\op{D}[\mathsf{C}^{\op{op}}\op{-Mod}]$. We let $\widehat{\mathsf{C}}_{\op{pe}}$ be the full sub-dg-category of $\widehat{\mathsf{C}}$ consisting of all objects that are compact in $\op{D}[\mathsf{C}^{\op{op}}\op{-Mod}]$. As any representable dg-module is cofibrant, we have a dg-functor \begin{displaymath}
  h: \mathsf{C} \to \widehat{\mathsf{C}}_{\op{pe}}.
 \end{displaymath}
 
 Following the lead of T\"oen, we introduce the following product. Assume that $\mathsf{C}$ is small and let $\mathsf{D}$ be another small dg-category over $k$. The \newterm{Morita product} of $\mathsf{C}$ and $\mathsf{D}$ is
 \begin{displaymath}
  \mathsf{C} \circledast \mathsf{D} := \widehat{( \mathsf{C} \otimes_k \mathsf{D} )}_{\op{pe}}
 \end{displaymath}
 viewed as an object of $\hodgcat$. Because we view it as an object of $\hodgcat$, it is unique up to quasi-equivalence.
\end{definition}

\begin{remark}
 The cofibrant objects of $\mathsf{C}^{\op{op}}\op{-Mod}$ are exactly the summands of semi-free dg-modules \cite{FHT}. One can check that summands of semi-free dg-modules have the appropriate lifting property. Furthermore, for any dg-module, $M$, there exists a semi-free dg-module, $F$, and an acyclic fibration, $F \to M$. If we assume that $M$ is cofibrant, this must split.
\end{remark}

\begin{corollary} \label{corollary: morita product of factorizations}
 Let $X$ be a smooth variety, $G$ be an affine algebraic group acting on $X$, $\mathcal L$ be an invertible $G$-equivariant sheaf on $X$, and $w \in \Gamma(X,\mathcal L)^G$. Let $Y$ be a smooth variety, $H$ be an affine algebraic group acting $X$, $\mathcal L'$ be an invertible $H$-equivariant sheaf on $Y$, and $v \in \Gamma(Y,\mathcal L')^H$. There are isomorphisms in $\hodgcat$
 \begin{displaymath}
  \mathsf{Inj}(\op{U}(\mathcal L) \times \op{U}(\mathcal L'), G \times H \times \mathbb{G}_m , f_w \boxplus f_v) \cong \widehat{ \mathsf{Inj}_{\op{coh}}(X,G,w) \otimes_k \mathsf{Inj}_{\op{coh}}(Y,H,v) }
 \end{displaymath}
 and
 \begin{displaymath}
  \mathsf{Inj}_{\op{coh}}(X,G,w) \circledast \mathsf{Inj}_{\op{coh}}(Y,H,v) \cong \overline{\mathsf{Inj}}_{\op{coh}}(\op{U}(\mathcal L) \times \op{U}(\mathcal L'), G \times H \times \mathbb{G}_m, f_w \boxplus f_v).
 \end{displaymath}
 
 Assume in addition that $X$ and $Y$ are affine and $G$ and $H$ are reductive. Then, there are isomorphisms in $\hodgcat$
 \begin{displaymath}
  \mathsf{Vect}(\op{U}(\mathcal L) \times \op{U}(\mathcal L'), G \times H \times \mathbb{G}_m, f_w \boxplus f_v) \cong \widehat{ \mathsf{vect}(X,G,w) \otimes_k \mathsf{vect}(Y,H,v)}
 \end{displaymath}
 and
 \begin{displaymath}
  \mathsf{vect}(X,G,w) \circledast \mathsf{vect}(Y,H,v) \cong \overline{\mathsf{vect}}(\op{U}(\mathcal L) \times \op{U}(\mathcal L'), G \times H \times \mathbb{G}_m, f_w \boxplus f_v).
 \end{displaymath}
 
 In the special case that $\mathcal L = \mathcal O_X(\chi)$ and $\mathcal L' = \mathcal O_Y(\chi')$ for characters $\chi: G \to \mathbb{G}_m$ and $\chi': H \to \mathbb{G}_m$, if we assume that $\chi$ or $\chi'$ is not torsion, then there are isomorphisms in $\hodgcat$
 \begin{displaymath}
  \mathsf{Inj}_{Z_w \times Z_v}(X \times Y, G \times_{\mathbb{G}_m} H, w \boxplus v) \cong \widehat{\mathsf{Inj}_{\op{coh}}(X,G,w) \otimes_k \mathsf{Inj}_{\op{coh}}(Y,H,v) }
 \end{displaymath}
 and
 \begin{displaymath}
  \mathsf{Inj}_{\op{coh}}(X,G,w) \circledast \mathsf{Inj}_{\op{coh}}(Y,H,v) \cong \overline{\mathsf{Inj}}_{\op{coh} Z_w \times Z_v}(X \times Y , G \times_{ \mathbb{G}_m} H, w \boxplus v).
 \end{displaymath}
 
 Assume in addition that $X$ and $Y$ are affine and $G$ and $H$ are reductive. Then, there are isomorphisms in $\hodgcat$
 \begin{displaymath}
  \mathsf{Vect}_{Z_w \times Z_v}(X \times Y, G \times_{ \mathbb{G}_m} H, w \boxplus v) \cong \widehat{ \mathsf{vect}(X,G,w) \otimes_k \mathsf{vect}(Y,H,v)}
 \end{displaymath}
 and
 \begin{displaymath}
  \mathsf{vect}(X,G,w) \circledast \mathsf{vect}(Y,H,v) \cong \overline{\mathsf{vect}}_{Z_w \times Z_v}(X \times Y, G \times_{ \mathbb{G}_m} H, w \boxplus v).
 \end{displaymath}
\end{corollary}

\begin{proof}
 By Lemma~\ref{lemma: trivialize L - dg equiv}, we have equivalences of dg-categories
 \begin{align*}
  \mathsf{Fact}(X,G,w) & \cong \mathsf{Fact}(\op{U}(\mathcal L),G \times \mathbb{G}_m, f_w) \\
  \mathsf{Fact}(Y,H,w) & \cong \mathsf{Fact}(\op{U}(\mathcal L'),H \times \mathbb{G}_m, f_v).
 \end{align*}
 Replacing $(X,G,\mathcal L,w)$ and $(Y,H,\mathcal L',v)$ by $(\op{U}(\mathcal L),G \times \mathbb{G}_m, \mathcal O_{\op{U}(\mathcal L)}(1), f_w)$ and $(\op{U}(\mathcal L'),H \times \mathbb{G}_m, \mathcal O_{\op{U}(\mathcal L')}(1), f_v)$, we may assume that $\mathcal L$ and $\mathcal L'$ are (non-equivariantly) trivial as sheaves on $X$ and $Y$, respectively, and continue the argument.   Finally, as $f_w$ and $f_v$ are both linear along the fibers, Euler's formula using the fiber coordinates shows that $f_w$ vanishes only along the singular locus of $f_w$ and similarly for $f_v$. Thus, $f_w$ and $f_v$ both vanish along the singular locus of $f_w \boxplus f_v$. Consequently, factorizations supported away from $Z_{f_w} \times Z_{f_v}$ are automatically acyclic.  Thus, we are reduced to proving the special case of the statement.

 We have a dg-functor,
 \begin{displaymath}
  \lambda_{w \boxplus v}: \mathsf{Inj}_{Z_w \times Z_v}(X \times Y, G\times_{\mathbb{G}_m} H, w \boxplus v) \to (\mathsf{Inj}_{\op{coh}}(X,G,w) \otimes_k \mathsf{Inj}_{\op{coh}}(Y,H,v))^{\op{op}}\op{-Mod}
 \end{displaymath}
 and an inclusion
 \begin{displaymath}
  \widehat{ \mathsf{Inj}_{\op{coh}}(X,G,-w) \otimes_k \mathsf{Inj}_{\op{coh}}(Y,H,v) } \to (\mathsf{Inj}_{\op{coh}}(X,G,w) \otimes_k \mathsf{Inj}_{\op{coh}}(Y,H,v))^{\op{op}}\op{-Mod}.
 \end{displaymath}
 We then have a dg-functor
 \begin{align*}
  a: \mathsf{Inj}_{Z_w \times Z_v}(X \times Y, G\times_{\mathbb{G}_m} H, w \boxplus v) & \to (\widehat{ \mathsf{Inj}_{\op{coh}}(X,G,-w) \otimes_k \mathsf{Inj}_{\op{coh}}(Y,H,v) })^{\op{op}}\op{-Mod} \\
  M & \mapsto \op{Hom}_{(\mathsf{Inj}_{\op{coh}}(X,G,w) \otimes_k \mathsf{Inj}_{\op{coh}}(Y,H,v))^{\op{op}}\op{-Mod}}(\bullet, M). 
 \end{align*}
 For any $N \in \mathsf{Inj}_{Z_w \times Z_v}(X \times Y, G\times_{\mathbb{G}_m} H, w \boxplus v)$, there exists an $M \in \widehat{ \mathsf{Inj}_{\op{coh}}(X,G,-w) \otimes_k \mathsf{Inj}_{\op{coh}}(Y,H,v) }$ and a quasi-isomorphism $f: M \to N$. The induced natural transformation
 \begin{displaymath}
  \op{Hom}(\bullet, f) : \op{Hom}(\bullet, M) \to \op{Hom}(\bullet, N)
 \end{displaymath}
 is a quasi-isomorphism if we restrict the argument to lie in $\widehat{ \mathsf{Inj}_{\op{coh}}(X,G,-w) \otimes_k \mathsf{Inj}_{\op{coh}}(Y,H,v) }$. Thus, $a(N)$ is quasi-isomorphic to $h_M$. Given $M$ quasi-isomorphic to $N$ and $M'$ quasi-isomorphic to $N'$, we have natural isomorphisms
 \begin{align*}
  \op{Hom}_{\op{D}[(\widehat{ \mathsf{Inj}_{\op{coh}}(X,G,-w) \otimes_k \mathsf{Inj}_{\op{coh}}(Y,H,v) })^{\op{op}}\op{-Mod}]}& (a(N),a(N')) \\ 
  & \cong \op{Hom}_{\op{D}[(\widehat{ \mathsf{Inj}_{\op{coh}}(X,G,-w) \otimes_k \mathsf{Inj}_{\op{coh}}(Y,H,v) })^{\op{op}}\op{-Mod}]}(h_M,h_{M'}) \\
  & \cong \op{Hom}_{[\widehat{ \mathsf{Inj}_{\op{coh}}(X,G,-w) \otimes_k \mathsf{Inj}_{\op{coh}}(Y,H,v) }]}(M,M') \\
  & \cong \op{Hom}_{\op{D}[(\mathsf{Inj}_{\op{coh}}(X,G,w) \otimes_k \mathsf{Inj}_{\op{coh}}(Y,H,v))^{\op{op}}\op{-Mod}]}(M,M') \\
  & \cong \op{Hom}_{\op{D}[(\mathsf{Inj}_{\op{coh}}(X,G,w) \otimes_k \mathsf{Inj}_{\op{coh}}(Y,H,v))^{\op{op}}\op{-Mod}]}(N,N') \\
  & \cong \op{Hom}_{[\mathsf{Inj}_{Z_w \times Z_v}(X \times Y, G\times_{\mathbb{G}_m} H, w \boxplus v]}(N,N')
 \end{align*}
 where the first isomorphism is due to the fact that $a(N)$ is quasi-isomorphic to $h_M$ and $a(N')$ is quasi-isomorphic to $h_{M'}$, the second uses the Yoneda embedding, the third uses that $\widehat{ \mathsf{Inj}_{\op{coh}}(X,G,-w) \otimes_k \mathsf{Inj}_{\op{coh}}(Y,H,v) }$ is an enhancement of $\op{D}[(\mathsf{Inj}_{\op{coh}}(X,G,w) \otimes_k \mathsf{Inj}_{\op{coh}}(Y,H,v))^{\op{op}}\op{-Mod}]$, the fourth uses the assumed quasi-isomorphisms, and the final isomorphism uses that $\mathsf{Inj}_{Z_w \times Z_v}(X \times Y, G\times_{\mathbb{G}_m} H, w \boxplus v)$ is an enhancement of $\op{D}[(\mathsf{Inj}_{\op{coh}}(X,G,w) \otimes_k \mathsf{Inj}_{\op{coh}}(Y,H,v))^{\op{op}}\op{-Mod}]$, i.e.~ Theorem~\ref{theorem: boxproduct of factorization dg-cats is a factorization dg-cat}. 
 
 Thus, $a$ is a quasi-functor inducing a quasi-equivalence
 \begin{displaymath}
  \mathsf{Inj}_{Z_w \times Z_v}(X \times Y, G\times_{\mathbb{G}_m} H, w \boxplus v) \simeq \widehat{ \mathsf{Inj}_{\op{coh}}(X,G,-w) \otimes_k \mathsf{Inj}_{\op{coh}}(Y,H,v) }.
 \end{displaymath} 
 
 The isomorphism in $\hodgcat$ 
 \begin{displaymath}
  \mathsf{Inj}_{Z_w \times Z_v}(X \times Y, G\times_{\mathbb{G}_m} H, (-w) \boxplus v) \cong \widehat{ \mathsf{Inj}_{\op{coh}}(X,G,-w) \otimes_k \mathsf{Inj}_{\op{coh}}(Y,H,v) }
 \end{displaymath}
 induces an isomorphism between the compact objects,
 \begin{displaymath}
  \mathsf{Inj}_{\op{coh}}(X,G,w) \circledast \mathsf{Inj}_{\op{coh}}(Y,H,v) \cong \overline{\mathsf{Inj}}_{\op{coh},Z_w \times Z_v}(X \times Y, G\times_{\mathbb{G}_m} H, w \boxplus v).
 \end{displaymath}

 In the case that $X$ and $Y$ are affine and $G$ and $H$ are reductive, an analogous argument suffices.  Indeed, as noted before, taking $G$ invariants is exact and locally-free objects are projective so we can work with locally-free objects in the exact same manner.
\end{proof}

\begin{remark}
 In the case that $\mathcal L = \mathcal O_X(\chi)$ and $\mathcal L' = \mathcal O_Y(\chi')$, the quotient stack $[\op{U}(\mathcal O_X(\chi)) \times \op{U}(\mathcal O_Y(\chi'))/(G \times H \times \mathbb{G}_m)]$ is isomorphic to $[X \times Y \times \mathbb{G}_m / (G \times H)]$ via the morphism
 \begin{align*}
  \phi: \op{U}(\mathcal O_X(\chi)) \times \op{U}(\mathcal O_Y(\chi')) \cong \mathbb{G}_m \times X \times \mathbb{G}_m \times Y & \to X \times Y \times \mathbb{G}_m \\ 
  (\alpha, x, \beta, y) & \mapsto (x,y, \alpha^{-1} \beta).
 \end{align*}
 The quotient stack $[X \times Y \times \mathbb{G}_m / (G \times H)]$ is isomorphic to $[X \times Y / G \times_{\mathbb{G}_m} H]$ as the map 
 \begin{align*}
  (G \times H) \overset{G \times_{\mathbb{G}_m} H}{\times} (X \times Y) & \to X \times Y \times \mathbb{G}_m \\
  (g,h,x,y) & \mapsto (x,y, \chi(g)^{-1}\chi'(h))
 \end{align*}
 is an isomorphism assuming that $\chi'-\chi: G \times H \to \mathbb{G}_m$ is not torsion. This gives a direct comparison for the two LG models describing the Morita product in the case $\mathcal L = \mathcal O_X(\chi)$ and $\mathcal L' = \mathcal O_Y(\chi')$.
\end{remark}

One of the many great results of \cite{Toe} is the following. It provides a description of the continuous internal Hom dg-category in $\op{Ho}(\op{dg-cat}_k)$.

\begin{theorem} \label{thm: Toen's Morita}
 Let $\mathsf{C}$ and $\mathsf{D}$ be small dg-categories over $k$. Then, there is an isomorphism in $\hodgcat$ 
\begin{displaymath}
 \mathbf{R}\!\op{Hom}_c(\widehat{\mathsf{C}},\widehat{\mathsf{D}}) \cong \widehat{\mathsf{C}^{\op{op}} \otimes_k \mathsf{D}}.
\end{displaymath}
 Given a module, $F \in \widehat{\mathsf{C}^{\op{op}} \otimes_k \mathsf{D}}$, the corresponding dg-functor, $\Psi_F: \mathsf{C} \to \widehat{\mathsf{D}}$, sends $c \in \mathsf{C}$ to $F(c,\bullet) \in \widehat{\mathsf{D}}$. This uniquely determines a dg-functor, $\Psi_F: \widehat{\mathsf{C}} \to \widehat{\mathsf{D}}$, for which $[\Psi_F]$ commutes with coproducts.
\end{theorem}
\begin{proof}
 As stated, this result is \cite[Corollary 7.6]{Toe}. 
\end{proof}

\begin{remark}
 T\"oen's result is more general. The field, $k$, can be replaced by a commutative ring. The derivation of the tensor product, $\otimes_k$, is then required.
\end{remark}

Applying Theorem~\ref{thm: Toen's Morita}, we can give the following description of the continuous internal Hom dg-category for equivariant factorizations.

\begin{theorem} \label{theorem: continuous homs are factorizations}
 Let $X$ and $Y$ be smooth varieties and let $G$ and $H$ be affine algebraic groups. Assume that $G$ acts on $X$ and $H$ acts on $Y$. Let $\chi: G \to \mathbb{G}_m$ and $\chi': H \to \mathbb{G}_m$ be characters and let $w \in \Gamma(X, \mathcal O_X(\chi))^G$ and $v \in \Gamma(Y,\mathcal O_Y(\chi'))^H$.
 
 There is an isomorphism in $\hodgcat$ 
 \begin{displaymath}
  \mathbf{R}\!\op{Hom}_c(\widehat{\mathsf{Inj}_{\op{coh}}(X,G,w)},\widehat{\mathsf{Inj}_{\op{coh}}(Y,H,v)}) \cong \mathsf{Inj}_{Z_w \times Z_v}(X \times Y, G\times_{\mathbb{G}_m} H, (-w) \boxplus v)
 \end{displaymath}
 such that the induced map on homotopy categories corresponding to $\mathcal I \in \mathsf{Inj}_{Z_w \times Z_v}(X \times Y, G\times_{\mathbb{G}_m} H, (-w) \boxplus v)$ is $\Phi_{\mathcal I}$.
 
 If $X$ is affine and $G$ is reductive, then there is an isomorphism in $\hodgcat$ 
 \begin{displaymath}
  \mathbf{R}\!\op{Hom}_c(\widehat{\mathsf{vect}(X,G,w)},\widehat{\mathsf{vect}(Y,H,v)}) \cong \mathsf{Vect}_{Z_w \times Z_v}(X \times Y, G\times_{\mathbb{G}_m} H, (-w) \boxplus v)
 \end{displaymath}
 such that the induced map on homotopy categories corresponding to $\mathcal P \in \mathsf{Vect}_{Z_w \times Z_v}(X \times Y, G\times_{\mathbb{G}_m} H, (-w) \boxplus v)$ is $\Phi_{\mathcal P}$.
\end{theorem}

\begin{proof}
We have isomorphisms in $\hodgcat$,
\begin{align*}
  \mathbf{R}\!\op{Hom}_c(\widehat{\mathsf{Inj}_{\op{coh}}(X,G,w)},\widehat{\mathsf{Inj}_{\op{coh}}(Y,H,v)}) & \cong \widehat{ \mathsf{Inj}_{\op{coh}}(X,G,w)^{\op{op}} \otimes_k \mathsf{Inj}_{\op{coh}}(Y,H,v) }\\
  & \cong \widehat{ \mathsf{Inj}_{\op{coh}}(X,G,-w) \otimes_k \mathsf{Inj}_{\op{coh}}(Y,H,v) } \\
  & \cong    \mathsf{Inj}_{Z_w \times Z_v}(X \times Y, G\times_{\mathbb{G}_m} H, (-w) \boxplus v).
  \end{align*}
The first line follows from Theorem~\ref{thm: Toen's Morita}.  The second line comes from Proposition~\ref{proposition: equivalent enhancements} which states that $\mathsf{Inj}_{\op{coh}}(X,G,w)^{\op{op}}$ is quasi-equivalent to $\mathsf{Inj}_{\op{coh}}(X,G,-w)$.  
 The third line is an application of Corollary~\ref{corollary: morita product of factorizations}.
 
 Next, we need to check that the induced functor on homotopy categories for a given $\mathcal I \in \mathsf{Inj}_{Z_w \times Z_v}(X \times Y, G\times_{\mathbb{G}_m} H, (-w) \boxplus v)$ is $\Phi_{\mathcal I}$ up to isomorphism. Recall that the isomorphism of $\mathsf{Inj}_{\op{coh}}(X,G,w)^{\op{op}}$ and $\mathsf{Inj}_{\op{coh}}(X,G,-w)$ follows from the diagram of dg-functors
 \begin{center}
 \begin{tikzpicture}[description/.style={fill=white,inner sep=2pt}]
  \matrix (m) [matrix of math nodes, row sep=2em, column sep=3em, text height=1.5ex, text depth=0.25ex]
  { \mathsf{Inj}_{\op{coh}}(X,G,w)^{\op{op}} & \mathsf{vect}(X,G,-w) \\ 
    \mathsf{Inj}_{\op{coh}}(X,G,-w) & \mathsf{vect}(X,G,w)^{\op{op}}  \\
  };
  \path[->,font=\scriptsize]
  (m-1-2) edge node[above] {$\mathcal Hom_X(\bullet, \mathcal I^{\mathcal O})$} (m-1-1)
  (m-2-2) edge node[above] {$\mathcal Hom_X(\bullet, \mathcal I^{\mathcal O})$} (m-2-1)
  (m-1-2) edge node[right] {$\mathcal Hom_X(\bullet, \mathcal O_X)$} (m-2-2)
  (m-2-2) edge (m-1-2)
  ;
 \end{tikzpicture}
 \end{center}
 
 The induced dg-functor on the image of
 \begin{displaymath}
  \mathcal Hom_X(\bullet, \mathcal I^{\mathcal O}): \mathsf{vect}(X,G,-w) \to \mathsf{Inj}_{\op{coh}}(X,G,w)^{\op{op}}
 \end{displaymath}
 is
 \begin{align*}
 &  \mathcal Hom_X (\mathcal E, \mathcal I^{\mathcal O}) \otimes \mathcal J  & \mapsto & \op{Hom}_{\mathsf{Fact}}(\mathcal Hom_X(\mathcal E^{\vee}, \mathcal I^{\mathcal O}) \boxtimes \mathcal J, \mathcal I) \\
\cong & & \mapsto &  \op{Hom}_{\mathsf{Fact}}(\op{Res}^{G \times H}_{G\times_{\mathbb{G}_m} H} \pi_2^* \mathcal J,\mathcal Hom_{X \times Y}(\op{Res}^{G \times H}_{G\times_{\mathbb{G}_m} H} \pi_1^*\mathcal Hom_X(\mathcal E^{\vee}, \mathcal I^{\mathcal O}), \mathcal I)) \\
\cong &    & \mapsto & \op{Hom}_{\mathsf{Fact}}(\op{Res}^{G \times H}_{G\times_{\mathbb{G}_m} H} \pi_2^* \mathcal J, \op{Res}^{G \times H}_{G\times_{\mathbb{G}_m} H} \pi^*_1 \mathcal E^{\vee} \otimes_{\mathcal O_{X \times Y}} \mathcal I) \\ 
 \cong &  & \mapsto & \op{Hom}_{\mathsf{Fact}}(\mathcal J, \pi_{2*}\op{Ind}^{G \times H}_{G\times_{\mathbb{G}_m} H} \left( \op{Res}^{G \times H}_{G\times_{\mathbb{G}_m} H} \pi^*_1 \mathcal E^{\vee} \otimes_{\mathcal O_{X \times Y}} \mathcal I \right)) \\
 \cong &   &\mapsto &\op{Hom}_{\mathsf{Fact}}(\mathcal J, \pi_{2*}\left( \pi^*_1 \mathcal E^{\vee} \otimes_{\mathcal O_{X \times Y}} \op{Ind}^{G \times H}_{G\times_{\mathbb{G}_m} H} \mathcal I \right)). 
 \end{align*}
The first line uses tensor-Hom adjunction, Proposition~\ref{proposition: Hom-tensor adjunction}.
 The second line uses the natural isomorphism, $\op{Res}^{G \times H}_{G\times_{\mathbb{G}_m} H}\pi^*_1 \mathcal E^{\vee} \otimes_{\mathcal O_{X \times Y}} \mathcal I  \to \mathcal Hom_{X \times Y}(\op{Res}^{G \times H}_{G\times_{\mathbb{G}_m} H} \pi_1^*\mathcal Hom_X(\mathcal E^{\vee}, \mathcal I^{\mathcal O}), \mathcal I))$.
 The third line uses the adjunctions, $\pi_{2}^* \dashv \pi_{2*}$ and $\op{Res}^{G \times H}_{G\times_{\mathbb{G}_m} H} \dashv \op{Ind}^{G \times H}_{G\times_{\mathbb{G}_m} H}$.  The fourth line applies the projection formula, Lemma~\ref{lemma: facts about restriction induction}.
 
 As $\mathcal Hom_X (\mathcal E, \mathcal I^{\mathcal O})$ is quasi-isomorphic to $\mathcal E^{\vee}$, from the aligned display, we see that the induced map on the homotopy categories is $\Phi_{\mathcal I}$. The case of $X$ affine and $G$ reductive is handled in an analogous, even simpler, manner.
\end{proof}

\begin{proof}[Proof of Theorem~\ref{thm: intro thm}]
 By Lemma~\ref{lemma: trivialize L - dg equiv}, we have equivalences of dg-categories
 \begin{align*}
  \mathsf{Fact}(X,G,w) & \cong \mathsf{Fact}(\op{U}(\mathcal L),G \times \mathbb{G}_m, f_w) \\
  \mathsf{Fact}(Y,H,w) & \cong \mathsf{Fact}(\op{U}(\mathcal L'),H \times \mathbb{G}_m, f_v).
 \end{align*}
 Theorem~\ref{theorem: continuous homs are factorizations} applied to $(\op{U}(\mathcal L),G \times \mathbb{G}_m, f_w)$ and $(\op{U}(\mathcal L'),H \times \mathbb{G}_m, f_v)$ gives the statement.
\end{proof}

\subsection{Hochschild invariants} \label{section: HH}

In this section, we compute the Hochschild invariants in a simple case: $G$ acting linearly on $\mathbb{A}^n$. We start out a bit more generally. Let $G$ act on $X$ and let $w \in \Gamma(X, \mathcal O_{X}(\chi))^{G}$. For the whole of this section, we assume
\begin{displaymath}
 \op{Sing} Z_{(-w) \boxplus w} \subseteq Z_w \times Z_w
\end{displaymath}
so we may remove the support restrictions in the results of Section~\ref{section: morita products}.

\begin{definition}
 Let $\mathsf{C}$ be a small dg-category. The \newterm{Hochschild cohomology} of $\mathsf{C}$ is the graded vector space
 \begin{displaymath}
  \bigoplus_{t \in \Z} \op{Hom}_{\op{D}(\mathsf{C}^{\op{op}} \otimes \mathsf{C}\op{-Mod})}(\mathsf{C},\mathsf{C}[t]).
 \end{displaymath}
 where $\mathsf{C}$ is the bimodule given by
 \begin{displaymath}
  \mathsf{C}(c,c') = \op{Hom}_{\mathsf{C}}(c,c').
 \end{displaymath}
 When $\mathsf{C} = \mathsf{Inj}_{\op{coh}}(X,G,w)$, we denote the Hochschild cohomology by $\op{HH}^{\bullet}(X,G,w)$.
 
 We have a trace functor
 \begin{align*}
  \op{Tr}: \mathsf{C}^{\op{op}} \otimes \mathsf{C} & \to \mathsf{C}(k) \\
  (c,c') & \mapsto \op{Hom}_{\mathsf{C}}(c,c').
 \end{align*}
 This admits an extension to $\mathsf{C} \otimes \mathsf{C}^{\op{op}}\op{-Mod}$ by 
 \begin{displaymath}
  F \mapsto F {\otimes}_{\mathsf{C} \otimes \mathsf{C}^{\op{op}}} \mathsf{C}.
 \end{displaymath}
 The \newterm{Hochschild homology} of $\mathsf{C}$ is defined to be the homology of
 \begin{displaymath}
  \mathsf{C} \overset{\mathbf{L}}{\otimes}_{\mathsf{C}^{\op{op}} \otimes \mathsf{C}} \mathsf{C}.
 \end{displaymath}
 When $\mathsf{C} = \mathsf{Inj}_{\op{coh}}(X,G,w)$, we denote the Hochschild cohomology by $\op{HH}_{\bullet}(X,G,w)$.
\end{definition}

\begin{lemma} \label{lemma: Hochschild cohomology of a general factorization}
 Let $X$ and $Y$ be smooth varieties and let $G$ and $H$ be affine algebraic groups acting on, respectively, $X$ and $Y$. Let $w \in \Gamma(X,\mathcal O_X(\chi))^G$ and $v \in \Gamma(Y, \mathcal O_Y(\chi'))^H$ for characters $\chi: G \to \mathbb{G}_m$ and $\chi' : H \to \mathbb{G}_m$. Assume that $\op{Sing} Z_{(-w) \boxplus w} \subseteq Z_w \times Z_w$. 

 We have isomorphisms
 \begin{displaymath}
  \op{HH}^t(X,G,w) \cong \op{Hom}_{\dabs [\mathsf{Fact}(X \times Y, G \times_{\mathbb{G}_m} G, w \boxplus (-w)]}(\nabla, \nabla[t]).
 \end{displaymath}
 
 We also have isomorphisms
 \begin{displaymath}
  \op{HH}_t(X,G,w) \cong \op{H}^t (\mathbf{L}\!\op{Tr} \nabla)
 \end{displaymath}
 where $\mathbf{L}\!\op{Tr}$ is trace functor on $\dabs [\mathsf{Fact}(X \times Y, G \times_{\mathbb{G}_m} G,(-w) \boxplus w]$.
\end{lemma}

\begin{proof}
 By Theorem~\ref{theorem: boxproduct of factorization dg-cats is a factorization dg-cat}, we have an equivalence
 \begin{align*}
  \dabs [\mathsf{Fact}(X \times Y, G \times_{\mathbb{G}_m} G, w \boxplus(-w)] & \to \op{D}(\mathsf{Inj}_{\op{coh}}(X,G,w)^{\op{op}} \otimes \mathsf{Inj}_{\op{coh}}(X,G,w)\op{-Mod}) \\
  \mathcal P & \mapsto \mathbf{R}\!\op{Hom}(\bullet \boxtimes \bullet^{\mathbf{L}\vee},\mathcal P). 
 \end{align*}
 The assumption on the singular support of $Z_{(-w) \boxplus w}$ allows us to remove the support condition.
 
 We have natural quasi-isomorphisms
 \begin{align*}
  \mathbf{R}\!\op{Hom}(\mathcal E\boxtimes \mathcal F^{\mathbf{L}\vee} ,\nabla) & = \mathbf{R}\!\op{Hom}(\mathcal E \boxtimes \mathcal F^{\mathbf{L}\vee} ,\op{Ind}^{G \times_{\mathbb{G}_m} G}_G \Delta_* \mathcal O_X) \\
  & \simeq \mathbf{R}\!\op{Hom}(\mathbf{L}\Delta^* \op{Res}_G^{G \times_{\mathbb{G}_m} G}\mathcal E \boxtimes \mathcal F^{\mathbf{L}\vee} ,\mathcal O_X) \\
  & \simeq \mathbf{R}\!\op{Hom}(\mathcal E \otimes \mathcal F^{\mathbf{L}\vee},\mathcal O_X) \\
  & \simeq \mathbf{R}\!\op{Hom}(\mathcal E,\mathcal F).
 \end{align*}
 The first line is the definition of $\nabla$. The second line is an application of the adjunctions $\op{Res}^{G \times_{\mathbb{G}_m} G}_G \dashv \op{Ind}^{G \times_{\mathbb{G}_m} G}_G$, Corollary~\ref{corollary: Res-Ind derived fact adjunction}, and $\mathbf{L} \Delta^* \dashv \Delta_*$, derived from Lemma~\ref{lemma: pull-push adjunction for factorizations}. The third line comes from the identity $\mathbf{L}\Delta^* \circ \pi_i^* \cong \op{Id}$ for $i=1,2$. The final line is tensor-Hom adjunction, Corollary~\ref{corollary: derived Hom-tensor adjunction}, and the assumption that $\mathcal F$ is quasi-isomorphic to a coherent factorization so 
 \begin{displaymath}
  \mathcal F^{\mathbf{L} \vee \mathbf{L} \vee} \cong \mathcal F.
 \end{displaymath}
 
 We turn to the statement concerning Hochschild homology. Under the equivalence
 \begin{displaymath}
  \dabs [\mathsf{Fact}(X \times Y, G \times_{\mathbb{G}_m} G, (-w) \boxplus w] \to \op{D}(\mathsf{Inj}_{\op{coh}}(X,G,w) \otimes \mathsf{Inj}_{\op{coh}}(X,G,w)^{\op{op}}\op{-Mod})
 \end{displaymath}
 the categorical trace corresponds to the trace functor
 \begin{displaymath}
  (\mathbf{R}p_* \mathbf{L}\Delta^*)^{\mathbf{R}G}
 \end{displaymath}
 by Lemma~\ref{lemma: derived trace is pullback via diagonal}.
\end{proof}

\begin{remark}
 As transposing the two copies of $X$ induces an equivalence 
 \begin{displaymath}
  \dabs [\mathsf{Fact}(X \times X, G \times_{\mathbb{G}_m} G, w \boxplus(-w)] \cong \dabs [\mathsf{Fact}(X \times X, G \times_{\mathbb{G}_m} G, (-w) \boxplus w]
 \end{displaymath}
 which preserves the diagonal, we can compute Hochschild invariants in either derived category of factorizations.
\end{remark}

The Hochschild cohomology is a subalgebra of a larger algebra.

\begin{definition}
 The \newterm{extended Hochschild cohomology} of $(X,G,w)$ is the $\widehat{G} \times \Z$-graded $k$-algebra
 \begin{displaymath}
  \bigoplus_{\rho \in \widehat{G}, t \in \Z} \op{Hom}_{\dabs [\mathsf{fact}(X \times X, G \times_{\mathbb{G}_m} G, (-w) \boxplus w)]}( \nabla, \nabla(\rho)[t]).
 \end{displaymath}
 We denote the extended Hochschild cohomology by $\op{HH}_e^{\bullet}(X,G,w)$.
\end{definition}

\begin{remark}
 The ring $\op{HH}_e^{\bullet}(X,G,w)$ is a factorization analog of generalized Hochschild cohomology of a variety $X$ with support in $T \in \dbcoh{X \times X}$ and coefficients in $E \in \dbcoh{X \times X}$, $\op{HH}^{\bullet}_T(X,E)$ defined by Kuznetsov \cite{Kuz09b}. Here, we take $E$ to be the diagonal and $T$ to be the kernels of twist functors.
\end{remark}

\begin{lemma} \label{lemma: extended HH contains H cohom}
 There is a natural isomorphism,
 \begin{displaymath}
  \op{HH}^{t}(X,G,w) \to \op{HH}_e^{(0,t)}(X,G,w).
 \end{displaymath}
\end{lemma}

\begin{proof}
 This is clear.
\end{proof}

To compute $\op{HH}_e^{\bullet}(X,G,w)$, we first must identify the complex 
\begin{displaymath}
 \mathbf{L} \Delta^* \op{Ind}^{G \times_{\mathbb{G}_m} G}_{G} \Delta_* \mathcal O_{X}
\end{displaymath}
of coherent $G$-equivariant sheaves on $X$. Let $K_{\chi}$ be the kernel of $\chi$.

\begin{lemma} \label{lemma: identity fiber product}
 There is a $G \times_{\mathbb{G}_m} G$-equivariant isomorphism,
 \begin{align*}
  \Sigma: G \times_{\mathbb{G}_m} G \overset{G}{\times} X \times X & \to K_{\chi} \times X \times X \\
  (g_1,g_2,x_1,x_2) & \mapsto (g_1g_2^{-1},\sigma(g_1,x_1),\sigma(g_2,x_2)),
 \end{align*}
 where $G \times_{\mathbb{G}_m} G$ acts on $K_{\chi}$ via
 \begin{displaymath}
  (g_1,g_2) \cdot g := g_1 g g_2^{-1}.
 \end{displaymath}
\end{lemma}

\begin{proof}
 The inverse morphism is 
 \begin{align*}
  K_{\chi} \times X \times X & \to G \times_{\mathbb{G}_m} G \overset{G}{\times} X \times X \\
  (g,x_1,x_2) & \mapsto (g,e,\sigma(g^{-1},x_1),x_2).
 \end{align*}
\end{proof}

Consider the $G \times_{\mathbb{G}_m} G$-equivariant subvariety defined by 
\begin{displaymath}
 O(\Delta) := \{ (g,x_1,x_2) \mid  \sigma(g,x_2) = x_1 \} \subset K_{\chi} \times X \times X.
\end{displaymath}

\begin{lemma} \label{lemma: Push-Ind diagonal}
 Under the composition of the equivalence of Lemma~\ref{lem: equivalence pullback} and the equivalence $\Sigma_*$, the $G$-equivariant sheaf $\Delta_* \mathcal O_X$ corresponds to the structure sheaf of $O(\Delta)$ in $K_{\chi} \times X \times X$ i.e.\ 
 \[
  \iota^* \Sigma^* \mathcal O_{O(\Delta)} \cong \Delta_* \mathcal O_X.
 \]
\end{lemma}

\begin{proof}
 Recall that the equivalence of Lemma~\ref{lem: equivalence pullback} is induced by $\iota^*$ where $\iota: X \times X \to G \times_{\mathbb{G}_m} G \overset{G}{\times} X \times X$ is the inclusion along the identity.
  Note that $\Sigma \circ \iota$ remains the inclusion along the identity, but now of $X \times X$ into $K_{\chi} \times X \times X$. Since both $\Sigma^*$ and $\iota^*$ are equivalences before deriving, they are exact. Thus, the statement of the lemma is equivalent to checking that the equation defining $O(\Delta)$ restricts to the diagonal when we restrict to $\{e \} \times X \times X$. This is clear.
\end{proof}

From now on, we assume that $K_{\chi}$ is finite.  Consider the coherent sheaf 
\begin{displaymath}
 \bigoplus_{ g \in K_{\chi} } \mathcal O_{\Gamma^t(\sigma_g)}
\end{displaymath}
where 
\begin{displaymath}
 \Gamma^t(\sigma_g) := \{ (x_1,x_2) \in X \times X \mid \sigma(g,x_2) = x_1 \}
\end{displaymath}
is the transpose of the graph of $\sigma_g$. 

\begin{lemma} \label{lemma: finite kernel Ind Delta}
The coherent sheaf $\bigoplus_{ g \in K_{\chi} } \mathcal O_{\Gamma^t(\sigma_g)}$ possesses a natural $G \times_{\mathbb{G}_m} G$-equivariant structure such that there is an isomorphism of coherent $G \times_{\mathbb{G}_m} G$-equivariant sheaves
 \begin{displaymath}
  \op{Ind}^{G \times_{\mathbb{G}_m} G}_G \Delta_* \mathcal O_X \cong p_*(\mathcal O_{O(\Delta)} ) \cong \bigoplus_{ g \in K_{\chi} } \mathcal O_{\Gamma^t(\sigma_g)}.
 \end{displaymath}
 where $p: K_{\chi} \times X \times X \to X \times X$ is the projection.
\end{lemma}

\begin{proof}
The second isomorphism is clear from the (now) standing assumption that $K_{\chi}$ is finite and induces the natural equivariant structure on $\bigoplus_{ g \in K_{\chi} } \mathcal O_{\Gamma^t(\sigma_g)}$.
 
  For the first isomorphism, we recall that, in general, $\op{Ind}^G_H$ is the composition $\alpha_* \circ (\iota^*)^{-1}$ where $\iota: X \to G \overset{H}{\times} X$ is the inclusion along the identity and $\alpha: G \overset{H}{\times} X \to X$ is the morphism induced by the action of $G$ on $X$. In our case, we have the commutative diagram
 \begin{center}
 \begin{tikzpicture}[description/.style={fill=white,inner sep=2pt}]
  \matrix (m) [matrix of math nodes, row sep=3em, column sep=3em, text height=1.5ex, text depth=0.25ex]
  {  G \times_{\mathbb{G}_m} G \overset{G}{\times} X \times X & & K_\chi \times X \times X \\
   & X \times X & \\ };
  \path[->,font=\scriptsize]
  (m-1-1) edge node[above]{$\Sigma$} (m-1-3)
  (m-1-1) edge node[below]{$\alpha$} (m-2-2)
  (m-1-3) edge node[below]{$p$} (m-2-2)
  ;
 \end{tikzpicture}
 \end{center}
 
 Now, by  Lemma~\ref{lemma: Push-Ind diagonal}, we have 
  \begin{displaymath}
  (\iota^*)^{-1}\Delta_* \mathcal  O_X \cong \Sigma^* \mathcal O_{O(\Delta)}.
 \end{displaymath}
 Applying $\alpha_*$ to both sides we get
 \begin{displaymath}
  \op{Ind}^{G \times_{\mathbb{G}_m} G}_G \Delta_* \mathcal O_X \cong p_*(\mathcal O_{O(\Delta)} )
 \end{displaymath}
where the simplification on the right hand side comes either by flat base change for the isomorphism $\Sigma$ or by using the isomorphism  $\Sigma^{-1}_* = \Sigma^*$.
\end{proof}

From this point forward, we restrict our attention to $X = \mathbb{A}^n$ equipped with a linear action of $G$ such that $K_\chi$ is finite.   It is easy to see that this implies that $G$ is reductive.  Write $\mathbb{A}^n = \op{Spec} \op{Sym}(V)$. Then, we have a right exact sequence
\begin{displaymath}
 V \otimes_k \mathcal O_{\mathbb{A}^n \times \mathbb{A}^n}  \overset{s}{\to} \mathcal O_{\mathbb{A}^n \times \mathbb{A}^n} \to \Delta_* \mathcal O_{\mathbb{A}^n} \to 0
\end{displaymath}
where the first morphism is
\begin{displaymath}
 v \otimes f \mapsto f (v \otimes 1 - 1 \otimes v).
\end{displaymath}
The potential $(-w) \boxplus w$ vanishes on $\Delta_* \mathcal O_{\mathbb{A}^n}$. Since $X$ is affine and $G$ is reductive, locally-free coherent equivariant sheaves are projective objects. Thus, there exists a morphism
\begin{displaymath}
 t: \mathcal O_{\mathbb{A}^n \times \mathbb{A}^n} \to V \otimes_k \mathcal O_{\mathbb{A}^n \times \mathbb{A}^n}
\end{displaymath}
making the diagram
\begin{center}
\begin{tikzpicture}[description/.style={fill=white,inner sep=2pt}]
 \matrix (m) [matrix of math nodes, row sep=3em, column sep=3em, text height=1.5ex, text depth=0.25ex]
 { V \otimes_k \mathcal O_{\mathbb{A}^n \times \mathbb{A}^n} & \mathcal O_{\mathbb{A}^n \times \mathbb{A}^n} \\
   V \otimes_k \mathcal O_{\mathbb{A}^n \times \mathbb{A}^n} & \mathcal O_{\mathbb{A}^n \times \mathbb{A}^n} \\ };
 \path[->,font=\scriptsize]
 (m-1-1) edge node[above]{$s$} (m-1-2)
 (m-1-1) edge node[left]{$(-w) \boxplus w$} (m-2-1)
 (m-1-2) edge node[right]{$(-w) \boxplus w$} (m-2-2)
 (m-2-1) edge node[above]{$s$} (m-2-2)
 (m-1-2) edge node[above]{$t$} (m-2-1)
 ;
\end{tikzpicture}
\end{center}
commute. 

Similarly, given $g \in G$, we can twist this diagram by $\sigma_g$ as follows. We have a right exact sequence
\begin{displaymath}
 V \otimes_k \mathcal O_{\mathbb{A}^n \times \mathbb{A}^n}  \overset{s_g}{\to} \mathcal O_{\mathbb{A}^n \times \mathbb{A}^n} \to \mathcal O_{\Gamma^t(\sigma_g)} \to 0
\end{displaymath}
where the first morphism is
\begin{displaymath}
 v \otimes f \mapsto f (g^{-1} \cdot v \otimes 1 - 1 \otimes v).
\end{displaymath}
Here $g^{-1} \cdot v$ is the element of $\op{Sym} V$ given by the automorphism of rings dual to $\sigma_g : \mathbb{A}^n \to \mathbb{A}^n$. For $g \in K_{\chi}$, $(-w) \boxplus w$ vanishes on $\mathcal O_{\Gamma^t(\sigma_g)}$ so there exists a 
\begin{displaymath}
 t_g: \mathcal O_{\mathbb{A}^n \times \mathbb{A}^n} \to V \otimes_k \mathcal O_{\mathbb{A}^n \times \mathbb{A}^n}
\end{displaymath}
making the diagram
\begin{center}
\begin{tikzpicture}[description/.style={fill=white,inner sep=2pt}]
 \matrix (m) [matrix of math nodes, row sep=3em, column sep=3em, text height=1.5ex, text depth=0.25ex]
 { V \otimes_k \mathcal O_{\mathbb{A}^n \times \mathbb{A}^n} & \mathcal O_{\mathbb{A}^n \times \mathbb{A}^n} \\
   V \otimes_k \mathcal O_{\mathbb{A}^n \times \mathbb{A}^n} & \mathcal O_{\mathbb{A}^n \times \mathbb{A}^n} \\ };
 \path[->,font=\scriptsize]
 (m-1-1) edge node[above]{$s_g$} (m-1-2)
 (m-1-1) edge node[left]{$(-w) \boxplus w$} (m-2-1)
 (m-1-2) edge node[right]{$(-w) \boxplus w$} (m-2-2)
 (m-2-1) edge node[above]{$s_g$} (m-2-2)
 (m-1-2) edge node[above]{$t_g$} (m-2-1)
 ;
\end{tikzpicture}
\end{center}
commute.

\begin{lemma} \label{lemma: resolution of the diagonal}
 There are quasi-isomorphisms of $G \times_{\mathbb{G}_m} G$-equivariant factorizations,
 \begin{displaymath}
  \bigoplus_{g \in K_{\chi}} \mathcal K(s_g,t_g) \cong \bigoplus_{g \in K_{\chi}} \mathcal O_{\Gamma^t(\sigma_g)} \cong \op{Ind}^{G \times_{\mathbb{G}_m} G}_{G} \Delta_* \mathcal O_{\mathbb{A}^n}.
 \end{displaymath}
\end{lemma}

\begin{proof}
 The second isomorphism is already stated in Lemma~\ref{lemma: finite kernel Ind Delta}. The first quasi-isomorphism follows from an immediate application of Proposition~\ref{prop: Eisenbud stabilization}.
\end{proof}

Since each $\mathcal K(s_g,t_g)$ is a factorization with locally-free components, to compute
\begin{displaymath}
 \mathbf{L} \Delta^* \op{Ind}^{G \times_{\mathbb{G}_m}G}_{G} \Delta_* \mathcal O_{\mathbb{A}^n}
\end{displaymath}
we may compute 
\begin{displaymath}
 \Delta^* \left( \bigoplus_{g \in K_{\chi}} \mathcal K(s_g,t_g) \right).
\end{displaymath}

We record the following lemma as a reminder of the structure of $\Delta^* \mathcal K(s_g,t_g)$.

\begin{lemma} \label{lemma: restricting the twisted factorizations to the diagonal}
 The factorization $\Delta^* \mathcal K(s_g,t_g)$ has components
 \begin{align*}
  \Delta^* \mathcal K(s_g,t_g)_{-1} & = \bigoplus_{l \geq 0} \Lambda^{2l+1} V \otimes_k \mathcal O_{\mathbb{A}^n}(l\chi) \\
  \Delta^* \mathcal K(s_g,t_g)_{0} & = \bigoplus_{l \geq 0} \Lambda^{2l} V \otimes_k \mathcal O_{\mathbb{A}^n}(l\chi)
 \end{align*}
 and morphisms given by
 \begin{displaymath}
  \bullet \ \lrcorner \ \Delta^*s_g + \bullet \wedge \Delta^*t_g
 \end{displaymath}
 where 
 \begin{align*}
  \Delta^* s_g : V \otimes_k \mathcal O_{\mathbb{A}^n} & \to \mathcal O_{\mathbb{A}^n} \\
  v \otimes f & \mapsto f(g^{-1} \cdot v - v). 
 \end{align*}
\end{lemma}

\begin{proof}
 This is clear from the definition of the Koszul factorization, $\mathcal K(s_g,t_g)$.
\end{proof}

\begin{definition}
 Let $g \in G$. Set 
 \begin{displaymath}
  V_g := \{ v \in V \mid g^{-1} \cdot v = v\}.
 \end{displaymath}
 The ideal sheaf of $(\mathbb{A}^n)^g$ corresponds to $\{ g^{-1} \cdot f - f \mid f \in \op{Sym} V\}$. This determines a subspace $W_g \subseteq V$. Note that there is an equivariant splitting $V = V_g \oplus W_g$. 
 
 Let $\kappa_g : G \to \mathbb{G}_m$ be the character corresponding to $\Lambda^{\op{dim} W_g} W_g$. More precisely, $\mathcal O_{\mathbb{A}^n}(\kappa_g)$ is the invertible sheaf corresponding to the free graded module of rank $1$, $\Lambda^{\op{dim} W_g} W_g \otimes_k \op{Sym} V$. 
\end{definition}

\begin{lemma} \label{lemma: identifying the twisted sectors}
 There is a quasi-isomorphism between $\Delta^* \mathcal K(s_g,t_g)$ and the Koszul factorization $i_{g*}\mathcal K(0,\op{d} \! w_g)$ where
 \begin{displaymath}
  i_g: (\mathbb{A}^n)^g \to \mathbb{A}^n
 \end{displaymath}
 is the inclusion, $0$ is the morphism
 \begin{displaymath}
  V_g \otimes_k \mathcal O_{(\mathbb{A}^n)^g} \overset{0}{\to} \mathcal O_{(\mathbb{A}^n)^g},
 \end{displaymath}
 and $w_g$ is the restriction of $w$ to $(\mathbb{A}^n)^g$.
\end{lemma}

\begin{proof}
 Consider the pullback of $s_g$ and $t_g$ to $(\mathbb{A}^n)^g \times (\mathbb{A}^n)^g$ via 
 \begin{displaymath}
  i_g \times i_g : (\mathbb{A}^n)^g \times (\mathbb{A}^n)^g \to \mathbb{A}^n \times \mathbb{A}^n.
 \end{displaymath}
 We have 
 \begin{displaymath}
  (i_g \times i_g)^*s_g(v) = v \otimes 1 - 1 \otimes v
 \end{displaymath}
 and a commutative diagram
 \begin{center}
 \begin{tikzpicture}[description/.style={fill=white,inner sep=2pt}]
  \matrix (m) [matrix of math nodes, row sep=3em, column sep=3em, text height=1.5ex, text depth=0.25ex]
  { V_g \otimes_k \mathcal O_{(\mathbb{A}^n)^g \times (\mathbb{A}^n)^g} & \mathcal O_{(\mathbb{A}^n)^g \times (\mathbb{A}^n)^g} \\
    V_g \otimes_k \mathcal O_{(\mathbb{A}^n)^g \times (\mathbb{A}^n)^g} & \mathcal O_{(\mathbb{A}^n)^g \times (\mathbb{A}^n)^g} \\ };
  \path[->,font=\scriptsize]
  (m-1-1) edge node[above]{$(i_g \times i_g)^*s_g$} (m-1-2)
  (m-1-1) edge node[left]{$(-w_g) \boxplus w_g$} (m-2-1)
  (m-1-2) edge node[right]{$(-w_g) \boxplus w_g$} (m-2-2)
  (m-2-1) edge node[above]{$(i_g \times i_g)^*s_g$} (m-2-2)
  (m-1-2) edge node[above]{$(i_g \times i_g)^*t_g$} (m-2-1)
  ;
 \end{tikzpicture}
 \end{center}
 Let $\Delta_g: (\mathbb{A}^n)^g \to (\mathbb{A}^n)^g \times (\mathbb{A}^n)^g$ be the diagonal embedding. Then, $\Delta_g^* (i_g \times i_g)^*t_g = \op{d} \! w_g$. As the diagram
 \begin{center}
 \begin{tikzpicture}[description/.style={fill=white,inner sep=2pt}]
  \matrix (m) [matrix of math nodes, row sep=3em, column sep=3em, text height=1.5ex, text depth=0.25ex]
  { (\mathbb{A}^n)^g & (\mathbb{A}^n)^g \times (\mathbb{A}^n)^g \\
    \mathbb{A}^n & \mathbb{A}^n \times \mathbb{A}^n \\ };
  \path[->,font=\scriptsize]
  (m-1-1) edge node[above]{$\Delta_g$} (m-1-2)
  (m-1-1) edge node[left]{$i_g$} (m-2-1)
  (m-1-2) edge node[right]{$i_g \times i_g$} (m-2-2)
  (m-2-1) edge node[above]{$\Delta $} (m-2-2)
  ;
 \end{tikzpicture}
 \end{center}
 commutes, we have $i_g^*\Delta^* t_g = \Delta_g^* (i_g \times i_g)^* t_g = \op{d} \! w_g$ while $i_g^*\Delta^*s_g = \Delta_g^* (i_g \times i_g)^* s_g =  0$. Thus,
 \begin{displaymath}
  i_g^* \Delta^* \mathcal K(s_g,t_g) \cong \mathcal K(0,\op{d} \! w_g).
 \end{displaymath}
 
 Now,  associated to the adjunction $i_g^* \dashv i_{g*}$, we have a morphism
 \begin{displaymath}
  \pi: \Delta^* \mathcal K(s_g,t_g) \to i_{g*}i_g^*\Delta^* \mathcal K(s_g,t_g) \cong i_{g*}  \mathcal K(0,\op{d} \! w_g)
 \end{displaymath}
 which we claim is a quasi-isomorphism.

To verify this claim, we check that the kernel of $\pi$, $\op{ker}(\pi)$, is acyclic.  The components of $\op{ker}(\pi)$ are
 \begin{align*}
  \op{ker}(\pi)_{-1} & = \bigoplus_{\substack{l \geq 0, a > 0 \\ a+b = 2l+1}} \Lambda^a W_g \otimes_k \Lambda^b V_g \otimes_k \mathcal O_{\mathbb{A}^n}(l\chi) \\ 
  \op{ker}(\pi)_0 & = \mathcal I_{(\mathbb{A}^n)^g} \oplus \bigoplus_{\substack{l \geq 0, a > 0 \\ a+b = 2l}} \Lambda^a W_g \otimes_k \Lambda^b V_g \otimes_k \mathcal O_{\mathbb{A}^n}(l\chi).
 \end{align*}
 Let 
 \begin{displaymath}
  \mathcal J^j := \op{ker}(\bullet \ \lrcorner \ \Delta^*s_g): \Lambda^j W_g \otimes_k \mathcal O_{\mathbb{A}^n} \to \Lambda^{j-1} W_g \otimes_k \mathcal O_{\mathbb{A}^n}
 \end{displaymath}
 and $\mathcal J^0 := \mathcal I_{(\mathbb{A}^n)^g}$. As $\Delta^* s_g$ vanishes on $V_g$ and $\Delta^* t_g$ has image in $V_g$, we have a filtration $F^j \op{ker}(\pi)$. In the case $j=2u$, it is
 \begin{align*}
  F^j \op{ker}(\pi)_{-1} & = \bigoplus_{\substack{b \geq j, a > 0 \\ a+b = 2l+1}} \Lambda^a W_g \otimes_k \Lambda^b V_g \otimes_k \mathcal O_{\mathbb{A}^n}(l\chi) \\ 
  F^j \op{ker}(\pi)_0 & = \Lambda^j V_g \otimes_k \mathcal J^j(u\chi) \oplus \bigoplus_{\substack{b \geq j, a > 0 \\ a+b = 2l}} \Lambda^a W_g \otimes_k \Lambda^b V_g \otimes_k \mathcal O_{\mathbb{A}^n}(l\chi).
 \end{align*}
 In the case $j=2u+1$, it is 
 \begin{align*}
  F^j \op{ker}(\pi)_{-1} & = \Lambda^j V_g \otimes_k \mathcal J^j(u\chi) \oplus \bigoplus_{\substack{b \geq j, a > 0 \\ a+b = 2l+1}} \Lambda^a W_g \otimes_k \Lambda^b V_g \otimes_k \mathcal O_{\mathbb{A}^n}(l\chi) \\ 
  F^j \op{ker}(\pi)_0 & = \bigoplus_{\substack{b \geq j, a > 0 \\ a+b = 2l}} \Lambda^a W_g \otimes_k \Lambda^b V_g \otimes_k \mathcal O_{\mathbb{A}^n}(l\chi).
 \end{align*}
 The associated graded factorization, $F^j \op{ker}(\pi)/ F^{j+1} \op{ker}(\pi)$, is the totalization of the exact sequence
 \begin{displaymath}
  0 \to \Lambda^j V_g \otimes_k \Lambda^{\op{dim} W_g} W_g \otimes_k \mathcal O_{\mathbb{A}^n} \overset{\bullet \ \lrcorner \ \Delta^*s_g}{\to} \cdots \overset{\bullet \ \lrcorner \ \Delta^*s_g}{\to} \Lambda^j V_g \otimes_k \Lambda^{j+1} W_g \otimes_k \mathcal O_{\mathbb{A}^n} \overset{\bullet \ \lrcorner \ \Delta^*s_g}{\to} \Lambda^j V_g \otimes_k \mathcal J^j \to 0
 \end{displaymath}
 where the final term is in degree $-\op{dim} W_g$. Thus, $\op{ker}(\pi)$ is filtered by acyclic complexes and hence acyclic.  This implies that $\pi$ is a quasi-isomorphism as desired.
\end{proof}

\begin{definition}
 Let $\kappa: G \to \mathbb{G}_m$ be the character corresponding to $\Lambda^n V$. 
\end{definition}

\begin{lemma} \label{lemma: extended HH contains H hom}
 Assume that $K_{\chi}$ is finite. Then, there is an isomorphism
 \begin{displaymath}
  \op{HH}_{t}(\mathbb{A}^n,G,w) \cong \op{HH}_e^{(\kappa,n+t)}(\mathbb{A}^n,G,w).
 \end{displaymath}
\end{lemma}

\begin{proof}
 We have,
 \begin{align*}
  \op{HH}_{t}(\mathbb{A}^n,G,w) &  \cong \op{Hom}( (\op{Ind}^{G \times_{\mathbb{G}_m} G}_G \Delta_* \mathcal O_X)^{\vee}, \op{Ind}^{G \times_{\mathbb{G}_m} G}_G \Delta_* \mathcal O_X[t]) \\
  & \cong    \op{Hom}( \bigoplus_{g \in K_{\chi}} \mathcal K(s_g,t_g)^{\vee}, \op{Ind}^{G \times_{\mathbb{G}_m} G}_G \Delta_* \mathcal O_X[t]) \\ 
  & \cong    \op{Hom}( \bigoplus_{g \in K_{\chi}} \mathcal K(t_g^{\vee},s_g^{\vee}), \op{Ind}^{G \times_{\mathbb{G}_m} G}_G \Delta_* \mathcal O_X[t]) \\
  & \cong    \op{Hom}( \bigoplus_{g \in K_{\chi}} \mathcal O_{\Gamma^t(\sigma_g)} \otimes_k \Lambda^n V^{\vee}[-n], \op{Ind}^{G \times_{\mathbb{G}_m} G}_G \Delta_* \mathcal O_X[t]) \\
  & \cong    \op{Hom}( \op{Ind}^{G \times_{\mathbb{G}_m} G}_G \Delta_* \mathcal O_X, \op{Ind}^{G \times_{\mathbb{G}_m} G}_G \Delta_* \mathcal O_X \otimes_k \Lambda^n V [t+n]) \\
  &  =   \op{Hom}( \op{Ind}^{G \times_{\mathbb{G}_m} G}_G \Delta_* \mathcal O_X,\op{Ind}^{G \times_{\mathbb{G}_m} G}_G \Delta_* \mathcal O_X (\kappa) [t+n]) \\
  &  =   \op{HH}_e^{(\kappa,n+t)}(\mathbb{A}^n,G,w).
 \end{align*}
 All morphisms are computed in $\dabs [\mathsf{fact}(\mathbb{A}^n \times \mathbb{A}^n, G \times_{\mathbb{G}_m} G, (-w) \boxplus w)]$.
 
 The first line follows from Lemma~\ref{lemma: Tr is representable}. The second line follows from Lemma~\ref{lemma: resolution of the diagonal}.  The third line is Lemma~\ref{lemma: duality of Koszul factorizations}. The fourth line comes from Proposition~\ref{prop: Eisenbud stabilization}.  The fifth line is another application of Lemma~\ref{lemma: resolution of the diagonal}.  The six line is by definition as is the seventh line.
\end{proof}

\begin{definition}
 Let $(r_1,\ldots,r_c)$ be a sequence of elements of a commutative ring, $R$. We let 
 \begin{displaymath}
  \op{H}^{\bullet} (\mathbf{r}) 
 \end{displaymath}
 denote the cohomology of the Koszul complex for $(r_1,\ldots,r_c)$. We call $\op{H}^{\bullet} (\mathbf{r})$ the \newterm{Koszul cohomology} of $(r_1,\ldots,r_c)$. 
 
 In the case, $(r_1,\ldots,r_c) = (\partial_1 w,\ldots,\partial_n w)$ for $R = k[x_1,\ldots,x_n]$, we denote the Koszul cohomology by $\op{H}^{\bullet} (\op{d} \! w)$. The \newterm{Jacobian algebra} of $w$ is $\op{H}^{0} (\op{d} \! w)$ but we denote it by $\op{Jac}(w)$ for transparency.
\end{definition}

\begin{theorem} \label{thm: twisted HH*}
 Let $G$ act linearly on $\mathbb{A}^n$ and let $w \in \Gamma(\mathbb{A}^n, \mathcal O_{\mathbb{A}^n}(\chi))^{G}$. Assume that $K_{\chi}$ is finite and $\chi: G \to \mathbb{G}_m$ is surjective. Then,
 \begin{gather*}
  \op{HH}^{(\rho,t)}_e(\mathbb{A}^n,G,w) \cong \\ \left( \bigoplus_{\substack{g \in K_{\chi}, l \geq 0 \\ t - \op{dim} W_g = 2u }} \op{H}^{2l}(\op{d} \! w_g)(\rho-\kappa_g+(u-l)\chi) \oplus \bigoplus_{\substack{g \in K_{\chi}, l \geq 0 \\ t - \op{dim} W_g = 2u+1 }} \op{H}^{2l+1}(\op{d} \! w_g)(\rho-\kappa_g+(u-l)\chi) \right)^G
 \end{gather*}
 If, additionally, we assume the support of $(\op{d} \! w)$ is $\{0\}$, then we have
 \begin{displaymath}
  \op{HH}^{(\rho,t)}_e(\mathbb{A}^n,G,w) \cong \left( \bigoplus_{\substack{g \in K_{\chi} \\  t - \op{dim} W_g = 2u }} \op{Jac}(w_g)(\rho-\kappa_g+u\chi) \oplus \bigoplus_{\substack{g \in K_{\chi} \\  t- \op{dim} W_g = 2u+1 }} \op{Jac}(w_g)(\rho-\kappa_g+u\chi) \right)^G.
 \end{displaymath}
\end{theorem}

\begin{proof}
 We have
 \begin{align*}
  \op{HH}^{(\rho,t)}_e(\mathbb{A}^n,G,w) & := \op{Hom}_{\dabs [\mathsf{fact}(\mathbb{A}^n \times \mathbb{A}^n, G \times_{\mathbb{G}_m} G, (-w) \boxplus w)]}( \op{Ind}^{G \times_{\mathbb{G}_m} G}_{G} \Delta_* \mathcal O_{\mathbb{A}^n}, \op{Ind}^{G \times_{\mathbb{G}_m} G}_{G} \Delta_* \mathcal O_{\mathbb{A}^n}(\rho)[t]) \\
  & \cong \op{Hom}_{\dabs [\mathsf{fact}(\mathbb{A}^n \times \mathbb{A}^n, G, (-w) \boxplus w)]}( \op{Res}^{G \times_{\mathbb{G}_m} G}_{G} \op{Ind}^{G \times_{\mathbb{G}_m} G}_{G} \Delta_* \mathcal O_{\mathbb{A}^n},  \Delta_* \mathcal O_{\mathbb{A}^n}(\rho)[t]) \\
  & \cong \op{Hom}_{\dabs [\mathsf{fact}(\mathbb{A}^n, G, 0)]}(\mathbf{L} \Delta^* \op{Res}^{G \times_{\mathbb{G}_m} G}_{G} \op{Ind}^{G \times_{\mathbb{G}_m} G}_{G} \Delta_* \mathcal O_{\mathbb{A}^n}, \mathcal O_{\mathbb{A}^n}(\rho)[t]) \\
  & \cong \op{Hom}(\mathbf{L} \Delta^* \op{Ind}^{G \times_{\mathbb{G}_m} G}_{G} \Delta_* \mathcal O_{\mathbb{A}^n}, \mathcal O_{\mathbb{A}^n}(\rho)[t])\\
  &  \cong \op{Hom}(\bigoplus_{g \in K_{\chi}} i_{g*}\mathcal K(0,\op{d} \! w_g), \mathcal O_{\mathbb{A}^n}(\rho)[t]) \\
  & \cong \op{Hom}(\mathcal O_{\mathbb{A}^n}, \bigoplus_{g \in K_{\chi}} i_{g*}\mathcal K(0,\op{d} \! w_g)^{\vee}(\rho)[t]). \\
  & \cong \op{Hom}(\mathcal O_{\mathbb{A}^n}, \bigoplus_{g \in K_{\chi}} i_{g*}\mathcal K(\op{d} \! w_g,0)(\rho-\kappa_g)[t-\op{dim} W_g]).
\end{align*}

 The first line is by definition.  The second line is adjunction for $\op{Res}$ and $\op{Ind}$, Lemma~\ref{lemma: Res-Ind fact adjunction}.  The third line applies the adjunction, $\mathbf{L} \Delta^* \dashv \Delta_*$, Lemma~\ref{lemma: pull-push adjunction for factorizations}.  The fourth line is a slight notational respite obtained by viewing $\Delta$ as an equivariant for the diagonal embedding of $G$ into $G \times_{\mathbb{G}_m} G$.  The fifth line is Lemma~\ref{lemma: identifying the twisted sectors}.  The sixth line is just the equivalence $(-)^\vee$. We justify the seventh line in the next paragraph.

 Let $\widetilde{K}(0,\op{d} \! w_g)$ be the Koszul factorization on $\mathbb{A}^n$ associated to 
 \begin{align*}
  V_g \otimes_k \mathcal O_{\mathbb{A}^n} \overset{0}{\to} \mathcal O_{\mathbb{A}^n}
 \end{align*}
 and
 \begin{displaymath}
  \mathcal O_{\mathbb{A}^n} \overset{\op{d} \! w_g}{\to} V_g \otimes_k \mathcal O_{\mathbb{A}^n}(\chi).
 \end{displaymath}
 Using contraction with morphism,
 \begin{align*}
  W_g \otimes_k \mathcal O_{\mathbb{A}^n} & \to \mathcal O_{\mathbb{A}^n} \\
  w \otimes_k f & \mapsto fw,
 \end{align*}
 we have a exact sequence of Koszul factorizations, 
  \begin{displaymath}
  0 \to \Lambda^{\op{dim} W_g} W_g \otimes_k \widetilde{K}(0,\op{d} \! w_g) \to \cdots \to  W_g \otimes_k \widetilde{K}(0,\op{d} \! w_g) \to \widetilde{K}(0,\op{d} \! w_g) \to i_{g*}\mathcal K(0,\op{d} \! w_g) \to 0
 \end{displaymath}
Hence, $\mathcal K(0,\op{d} \! w_g)^{\vee}$ is quasi-isomorphic to the totalization of the complex 
 \begin{displaymath}
  0 \leftarrow \Lambda^{\op{dim} W_g} W_g^{\vee} \otimes_k \widetilde{K}(0,\op{d} \! w_g)^{\vee} \leftarrow \cdots \leftarrow  W_g^{\vee} \otimes_k \widetilde{K}(0,\op{d} \! w_g)^{\vee} \leftarrow 0
 \end{displaymath}
 This is, in turn quasi-isomorphic to $i_{g*}\mathcal K(\op{d} \! w_g,0) \otimes_k \Lambda^{\op{dim} W_g} W_g^{\vee} [-\op{dim} W_g]$. 
 
 The factorization, $\mathcal K(\op{d} \! w_g,0)(\rho-\kappa_g)$, has components
 \begin{align*}
  \mathcal K(\op{d} \! w_g,0)(\rho-\kappa_g)_{-1} & = \bigoplus_{l \geq 0} \Lambda^{2l+1} V_g^{\vee} \otimes_k \mathcal O_{(\mathbb{A}^n)^g}(\rho-\kappa_g-(l+1)\chi) \\
  \mathcal K(\op{d} \! w_g,0)(\rho-\kappa_g)_{0} & = \bigoplus_{l \geq 0} \Lambda^{2l} V_g^{\vee} \otimes_k \mathcal O_{(\mathbb{A}^n)^g}(\rho-\kappa_g-l\chi)
 \end{align*}
 with morphisms given by contraction with $\op{d} \! w_g$. The cohomology of $\mathcal K(\op{d} \! w_g,0)(\rho-\kappa_g)$ is
 \begin{align*}
  \op{H}^{2u} ( \mathcal K(\op{d} \! w_g,0)(\rho-\kappa_g) ) & \cong \bigoplus_{l \geq 0} \op{H}^{2l} (\op{d} \! w_g)(\rho-\kappa_g+(u-l)\chi) \\
  \op{H}^{2u+1} ( \mathcal K(\op{d} \! w_g,0)(\rho-\kappa_g) ) & \cong \bigoplus_{l \geq 0} \op{H}^{2l+1} (\op{d} \! w_g)(\rho-\kappa_g+(u-l)\chi).
 \end{align*}
 Thus, we have
 \begin{gather*}
  \op{Hom}(\mathcal O_{\mathbb{A}^n}, \bigoplus_{g \in K_{\chi}} i_{g*}\mathcal K(\op{d} \! w_g,0)(\rho-\kappa_g)[t-\op{dim} W_g]) \cong  \\ \left( \bigoplus_{\substack{g \in K_{\chi}, l \geq 0 \\ t - \op{dim} W_g = 2u }} \op{H}^{2l}(\op{d} \! w_g)(\rho-\kappa_g+(u-l)\chi) \oplus \bigoplus_{\substack{g \in K_{\chi}, l \geq 0 \\ t - \op{dim} W_g = 2u+1 }} \op{H}^{2l+1}(\op{d} \! w_g)(\rho-\kappa_g+(u-l)\chi) \right)^G
 \end{gather*}
 
 If $(\op{d} \! w)$ has support $\{0\}$, then so does $(\op{d} \! w_g)$ for all $g$. So all Koszul complexes only have cohomology in homological degree zero. 
\end{proof}

\begin{remark}
 By specializing to appropriate graded pieces, one can use Theorem~\ref{thm: twisted HH*} to extract both $\op{HH}^{\bullet}(\mathbb{A}^n,G,w)$ and $\op{HH}_{\bullet}(\mathbb{A}^n,G,w)$.
\end{remark}

\begin{corollary} \label{corollary: projective hypersurface case}
 Let $\mathbb{A}^n = \op{Spec} (\op{Sym} V)$ carry a $\mathbb{G}_m$ action with weight $(-1)$. Let $w \in \op{Sym} V$ be homogeneous of degree $d$. Then, we have isomorphisms
 \begin{displaymath}
  \op{HH}_t(\mathbb{A}^n, \mathbb{G}_m, w) \cong \begin{cases} \op{Jac}(w)_{d(\frac{n+t}{2})-n} & t \not = 0 \\
                                                  \op{Jac}(w)_{d(\frac{n}{2})-n} \oplus k^{\oplus d-1} & t = 0. 
                                                 \end{cases}
 \end{displaymath}
\end{corollary}

\begin{proof}
 We have $\kappa = -n$. In this case, $K_{\chi} \cong \Z/d\Z$. If $g \not = e$, then $V_g = \{0\}$, thus $\kappa_g = -n$ and $\op{dim} W_g = n$. For $g=e$, we have $\kappa_e = 0$ and $\op{dim} W_e = 0$. Applying Lemma~\ref{lemma: extended HH contains H hom} and Theorem~\ref{thm: twisted HH*}, we have
 \begin{displaymath}
  \op{HH}_{t}(\mathbb{A}^n,\mathbb{G}_m,w) \cong \op{Jac}(w)_{-n+d(\frac{n+t}{2})} \oplus \bigoplus_{g \not = e} \op{Jac}(w_g)_{d\frac{t}{2}}.
 \end{displaymath}
 We have $\op{Jac}(w_g) \cong k(0)$ so the latter term only contributes to $t=0$. 
\end{proof}

\begin{remark}
 This computation was first done by C\u{a}ld\u{a}raru and Tu, \cite[Example 6.4]{CT}. It is also performed, independently, by Polishchuk and Vaintrob \cite{PVnew}.
\end{remark}

\section{Implications for Hodge Theory} \label{section: implications for Hodge theory}

In this section, we give two applications of the ideas and computations of the previous sections to Hodge theory. To fully state the results, we recall some of the functoriality of Hochschild homology. Recall that $\op{perf}(\mathsf{C})$ consists of all compact objects in $\op{D}(\mathsf{C}^{\op{op}}\op{-Mod})$.

\begin{proposition} \label{prop: functorial push}
 Let $\mathsf C$ and $\mathsf D$ be saturated dg-categories over $k$. Let $F$ be an object of $\op{perf}(\mathsf C^{\op{op}} \otimes \mathsf D)$. Then, there is a homomorphism of vector spaces,
 \begin{displaymath}
  F_{\bullet} : \op{HH}_{\bullet}(\mathsf C) \to \op{HH}_{\bullet}(\mathsf D).
 \end{displaymath}
 Moreover, the assignment, $F \mapsto F_{\bullet}$, is natural in the following sense. Let $F_1 \in \op{perf}(\mathsf B^{\op{op}} \otimes \mathsf C)$ and $F_2 \in \op{perf}(\mathsf C^{\op{op}} \otimes \mathsf D)$ and let $F_2 \circ F_1$ denote the $\mathsf{B}$-$\mathsf{D}$ bimodule corresponding to the tensor product $F_1 \overset{\mathbf{L}}{\otimes}_{\mathsf C} F_2$. Then, $(F_2 \circ F_1)_{\bullet} \cong F_{2 \bullet} \circ F_{1 \bullet}$. 
\end{proposition}

\begin{proof}
 This is \cite[Lemma 1.2.1]{PV}.
\end{proof}

\begin{definition}
 Let $\mathsf C$ and $\mathsf D$ be saturated dg-categories over $k$. Let $F$ be an object of $\op{perf}(\mathsf C^{\op{op}} \otimes_k \mathsf D)$. We will call the linear map, $F_{\bullet}$, the \newterm{pushforward} by $F$.
 
 For an object $E \in \op{perf}(\mathsf C)$, we get an induced map,
\begin{displaymath}
 E_{\bullet}: k[0] \cong \op{HH}_{\bullet}(k) \to \op{HH}_{\bullet}(\mathsf C).
\end{displaymath}
 The map, $E_{\bullet}$, is called the  \newterm{Chern character map} and the element $E_{\bullet}(1)$ is called the \newterm{Chern character} of $E$. The map 
 \begin{displaymath}
  E \mapsto E_{\bullet}(1)
 \end{displaymath}
 is denoted by $\op{ch}$.
\end{definition}

There is also a natural pairing on Hochschild homology.

\begin{proposition} \label{prop: categorical pairing on HH}
 Let $\mathsf{C}$ be saturated dg-category over $k$. There is a natural pairing
 \begin{displaymath}
  \langle \cdot, \cdot \rangle : \op{HH}_{\bullet}(\mathsf C) \otimes_k \op{HH}_{\bullet}(\mathsf C) \to k
 \end{displaymath}
 satisfying
 \begin{displaymath}
  \chi\left(\oplus_{i \in \Z} \op{Hom}_{\op{perf}(\mathsf{C})}(E_1,E_2[i]) \right) = \langle \op{ch}(E_1),\op{ch}(E_2) \rangle
 \end{displaymath}
 for $E_1,E_2 \in \op{perf}(\mathsf{C})$.
\end{proposition}

\begin{proof}
 This pairing is constructed for smooth and proper dg-algebras in \cite[Section 1.2]{Shk2}. In this case, the equality
 \begin{displaymath}
  \chi\left(\oplus_{i \in \Z} \op{Hom}_{\op{perf}(\mathsf{C})}(E_1,E_2[i]) \right) = \langle \op{ch}(E_1),\op{ch}(E_2) \rangle
 \end{displaymath}
 is a special case of \cite[Theorem 1.3.1]{Shk2}. The pairing is also defined for a general saturated dg-category in \cite[Section 1.2]{PV}. As any saturated dg-category is Morita equivalent to a smooth and proper dg-algebra, the naturality of the pairing extends the result from algebras to categories. 
\end{proof}

\begin{definition}
 Let $\mathsf{C}$ be a saturated dg-category. We shall call the pairing
 \begin{displaymath}
  \langle \cdot, \cdot \rangle : \op{HH}_{\bullet}(\mathsf C) \otimes_k \op{HH}_{\bullet}(\mathsf C) \to k
 \end{displaymath}
 the \textsf{categorical pairing} on Hochschild homology.
\end{definition}

We will also need the following result due to Polishchuk and Vaintrob.

\begin{theorem} \label{thm: PV}
 Let $\mathbb{A}^n$ carry a linear action of $G$, an algebraic group, and let $w \in \Gamma(\mathbb{A}^n,\mathcal O_{\mathbb{A}^n}(\chi))^G$. Assume that $K_{\chi}$ is finite and $\chi: G \to \mathbb{G}_m$ is surjective. Furthermore, assume that $(\op{d} \! w)$ is supported at $\{0\} \in \mathbb{A}^n$. For a character, $\rho : G \to \mathbb{G}_m$, the twist functor,
\begin{displaymath}
 (\rho): \dabs \mathsf{fact}(\mathbb{A}^n,G,w) \to \dabs \mathsf{fact}(\mathbb{A}^n,G,w), 
\end{displaymath}
 induces a pushforward map,
 \[ 
 (\rho)_{\bullet}: \op{HH}_{\bullet}(\mathbb{A}^n,G,w) \to \op{HH}_{\bullet}(\mathbb{A}^n,G,w)
 \]
 which is multiplication by $\rho(g)^{-1}$ on $\op{Jac}(w_g)$ for $g \in K_{\chi}$. In other words, the decomposition of Theorem~\ref{thm: twisted HH*} is exactly the eigenspace decomposition for the action of $\widehat{G}$ on $\op{HH}_{\bullet}(\mathbb{A}^n,G,w)$.
\end{theorem}

\begin{proof}
 This is part of \cite[Theorem 2.6.1]{PVnew}, albeit stated in the notation used in this paper.
\end{proof}

\subsection{Another look at Griffiths' Theorem} \label{section: Griffiths}
 
In this section, we recall a celebrated result of Griffiths, reproved and understood in categorical language as a combination of Theorem~\ref{thm: twisted HH*}, the Hochschild-Kostant-Rosenberg isomorphism, and a theorem of Orlov \cite{Orl09}.

\begin{definition}
 Let $Z$ be a smooth complex projective hypersurface in $\P^{n-1}_{\C}$ defined by $w \in \C[x_1,\ldots,x_n]$. An element of $\text{H}^{2(n-2-k)}(Z;\C)$ is called \newterm{primitive} if it cups trivially with $H^k$, where $H$ is the class of a hyperplane section. We write
 \begin{displaymath}
  \op{H}_{\op{prim}}^{\bullet}(Z;\C)
 \end{displaymath}
 for the subspace of primitive classes. We will write
 \begin{displaymath}
  \op{H}_{\op{prim}}^{\bullet,\bullet}(Z)
 \end{displaymath}
 for the intersections of $\op{H}_{\op{prim}}^{\bullet}(Z;\C)$ with each bi-graded piece of the Dolbeault cohomology of $Z$.
\end{definition}

In our context, by the Lefschetz Hyperplane Theorem, all primitive cohomology classes lie in the middle dimensional cohomology, $\text{H}^{n-2}(Z;\C)$.  Furthermore, all elements are primitive when $n$ is odd. When $n$ is even, all Dolbeault classes of type $(p,n-2-p)$, $\op{H}^{p,n-2-p}(Z)$, with $p \neq \frac{n-2}{2}$ are primitive, while $\op{H}_{\op{prim}}^{\frac{n-2}{2},\frac{n-2}{2}}(Z)$ are just those classes lying in the kernel of the cup product with $H$. The following description is due to Griffiths.

\begin{theorem} \label{theorem: Griffiths}
 There is an isomorphism,
\begin{displaymath}
 \op{H}^{p,n-2-p}_{\op{prim}}(Z) \cong \op{Jac}(w)_{d(n-1-p)-n}.
\end{displaymath}
\end{theorem}

\begin{proof}
 This is \cite[Theorem 8.1]{Gri}.
\end{proof}

Comparing Griffiths' result with Theorem~\ref{thm: twisted HH*} we see a striking similarity.  Indeed, $\op{Jac}(w)_{d(n-1-p)-n}$, is also the summand of $ \op{HH}_{n-2-2p}(\mathbb A^n, \gm, w)$ corresponding to $g=e$. This is not a coincidence. To give a precise comparison, we will need to recall two results.

\begin{definition}
 Let $Z$ be a smooth, projective variety. Let $\mathsf{Inj}_{\op{coh}}(Z)$ denote the dg-category of bounded below chain complexes of injective sheaves on $Z$ with bounded and coherent cohomology. We denote the Hochschild homology of $\mathsf{Inj}_{\op{coh}}(Z)$ by $\op{HH}_{\bullet}(Z)$.
\end{definition}

\begin{definition}
 The \textsf{Mukai pairing} on $\op{H}^{\bullet}(Z; \C)$ is 
 \begin{displaymath}
  ( v, v' )_M := \int_Z v^{\vee} \cdot v' \cdot \op{td}(Z)
 \end{displaymath}
 where $v^{\vee} = \sum_{p,q} (-1)^p v_{p,q}$ if $v = \sum_{p,q} v_{p,q}$ is the Hodge decomposition. 
\end{definition}

The first result we use is the Hochschild-Kostant-Rosenberg isomorphism.  It allows one to reinterpret Dolbeault cohomology categorically.

\begin{theorem} \label{theorem: HKR}
 Let $Z$ be smooth projective variety. There are natural isomorphisms,
 \[
  \op{HH}_t(Z) \cong \bigoplus_{q-p=t} \op{H}^q(Z,\Omega_Z^p) \cong \bigoplus_{q-p=t} \op{H}^{p,q}(Z).
 \]
 We denote the isomorphism by $\phi_{\op{HKR}}: \op{HH}_{\bullet}(Z) \to \op{H}^{\bullet}(Z ; \C)$. Under the HKR isomorphism, we have 
 \begin{displaymath}
  \langle \alpha , \alpha' \rangle = ( \phi_{\op{HKR}}(\alpha) , \phi_{\op{HKR}}(\alpha') )_M. 
 \end{displaymath}
 The Chern character and classical Chern character agree under the HKR isomorphism
 \begin{displaymath}
  \phi_{\op{HKR}}(\op{ch}(\mathcal E)) = \op{ch}_{class}(\mathcal E).
 \end{displaymath}
 Furthermore, for an integral functor, $\Phi_{\mathcal K}: \dbcoh{X} \to \dbcoh{X}$, the action of $\Phi_{\mathcal K \bullet}$ under the HKR isomorphism is the cohomological integral transform, $\Phi^H_{\mathcal K}$, associated to $\op{ch}_{class}(\mathcal K) \in \op{H}^{\bullet}(X \times Y ; \C)$. 
\end{theorem}

\begin{proof}
 The HKR isomorphism in the affine case is due to \cite{HKR}. In this generality, it is due to Swan \cite[Corollary 2.6]{Swan2} and Kontsevich \cite{Kon}, see also \cite{Yek}. The preservation of the Chern character was stated in \cite{Markarian} and proven as \cite[Theorem 4.5]{Cal}. The equality of the pairings is \cite[Theorem 1]{Ram}. The equality
 \begin{displaymath}
  \phi_{\op{HKR}} \circ \Phi_{\mathcal K \bullet} = \Phi^H_{\mathcal K} \circ \phi_{\op{HKR}}
 \end{displaymath}
 is a consequence of \cite[Theorem 2]{Ram} and the definition of $\Phi^{muk}_*$ in \cite{Ram}.
\end{proof}

\begin{definition}
 Let $Z$ be a smooth, projective variety. Define the endofunctor,
 \begin{displaymath}
  \{1\} := L_{\mathcal O_Z} \circ T_{\mathcal O(1)}: \mathsf{Inj}_{\op{coh}}(Z) \to \mathsf{Inj}_{\op{coh}}(Z),
 \end{displaymath}
 where 
 \begin{displaymath}
  T_{\mathcal O(1)} (\mathcal E) := \mathcal E \otimes_{\mathcal O_Z} \mathcal O_Z(1)
 \end{displaymath}
 and, for $i \in \Z$,
 \begin{displaymath}
  L_{\mathcal O_Z(i)}(\mathcal E) := \op{Cone}\left( \op{Hom}(\widetilde{\mathcal O}_Z(i), \mathcal E) \otimes_k \widetilde{\mathcal O}_Z(i) \to \mathcal E \right)
 \end{displaymath}
 where $\widetilde{\mathcal O}_Z(i)$ is an injective resolution of $\mathcal O_Z(i)$. Let 
 \begin{displaymath}
  \varsigma(\mathcal O_Z(i)): \op{Id} \to L_{\mathcal O_Z(i)}
 \end{displaymath}
 denote the induced natural transformation.
\end{definition}

The second result we use is a theorem of Orlov \cite{Orl09}, generalized mildly to account for a larger grading group. Let $G$ be an Abelian affine algebraic group acting on $\mathbb{A}^n$. We assume that $G$ has rank one so that
\begin{displaymath}
 G \cong \mathbb{G}_m \times G_{\op{tors}}
\end{displaymath}
for $ G_{\op{tors}}$ a finite Abelian group.

\begin{definition}
 We say that $G$ acts \newterm{positively} on $\mathbb{A}^n$ if with respect to the induced $\mathbb{G}_m$-action all nonzero linear functions on $\mathbb{A}^n$ have positive degree. 
\end{definition}

We have a $\mathbb{G}_m$-equivariant isomorphism $\omega_{\mathbb{A}^n} \cong \mathcal O_{\mathbb{A}^n}(N)$ for $N$ equal to the sum of the degrees of $x_i$ if $\mathbb{A}^n = \op{Spec} k[x_1,\ldots,x_n]$.

\begin{theorem} \label{theorem: Orlov}
 Let $w \in \Gamma(\mathbb{A}^n,\mathcal O_{\mathbb{A}^n}(\chi))^G$ for a character $\chi: G \to \mathbb{G}_m$ with $\chi|_{\mathbb{G}_m} = d > 0$. Let $Y$ be the zero locus of $w$ on punctured affine space $\mathbb{A}^n\setminus \{0\}$. If $G = \mathbb{G}_m$ and $N=n$, let $Z$ denote the projective hypersurface determined by $w$. 
 
 Assume $w$ is not zero and that $Y$ is smooth. Further, assume that $G$ acts positively.
 
 \begin{itemize}
  \item If $d  < N$, then there exists morphisms in $\op{Ho}(\op{dg-cat}_k)$
  \begin{align*}
   \Phi & : \mathsf{Inj}_{\op{coh}}(\mathbb{A}^n,G,w) \to \mathsf{Inj}_{\op{coh}_G}(Y) \\
   \Phi^! & : \mathsf{Inj}_{\op{coh}_G}(Y) \to \mathsf{Inj}_{\op{coh}}(\mathbb{A}^n,G,w)
  \end{align*}
  and a semi-orthogonal decomposition
  \begin{displaymath}
   \op{D}^{\op{b}}(\op{coh}_G Y) = \left \langle \bigoplus_{\alpha|_{\mathbb{G}_m} = d - N} \mathcal O_Y (\alpha), \ldots, \bigoplus_{\alpha|_{\mathbb{G}_m} = -1} \mathcal O_Y (\alpha), [\Phi]\left( [\mathsf{Inj}_{\op{coh}}(\mathbb{A}^n,G,w)]  \right) \right \rangle.
  \end{displaymath}
  
  Moreover, if $G = \mathbb{G}_m$ and $N=n$, there are quasi-isomorphisms of bimodules
  \begin{align*}
   \Phi^! \circ \{1\} \circ \Phi & \cong (1) \\
   \Phi^! \circ \Phi & \cong \nabla
  \end{align*}
  and 
  \begin{displaymath}
   [\Phi^!] \mathcal O_Z(i) \cong 0
  \end{displaymath}
  for $d-N \leq i \leq -1$.
  
  \item If $d = N$, then there exists inverse morphisms in $\op{Ho}(\op{dg-cat}_k)$
  \begin{align*}
   \Phi & : \mathsf{Inj}_{\op{coh}}(\mathbb{A}^n,G,w) \to \mathsf{Inj}_{\op{coh}_G}(Y) \\
   \Psi & : \mathsf{Inj}_{\op{coh}_G}(Y) \to \mathsf{Inj}_{\op{coh}}(\mathbb{A}^n,G,w).
  \end{align*}
  If, in addition, $G = \mathbb{G}_m$ and $N=n$, there is a quasi-isomorphism of bimodules
  \begin{displaymath}
   \{1\} \circ \Phi \cong \Phi \circ (1).
  \end{displaymath}
  Moreover, for each $s \in k[x_1,\ldots,x_n]$ homogeneous of degree $i$, the natural transformations of exact functors,
  \begin{align*}
   s: \op{Id}_{\dabs[\mathsf{fact}(\mathbb{A}^n,\mathbb{G}_m,w)]} \to (i) \\
   s: \op{Id}_{\dbcoh{Z}} \to T_{\mathcal O_Z(i)}
  \end{align*}
  satisfy the identity
  \begin{displaymath}
   \Phi(s) = \varsigma(\mathcal O_Z) \circ \cdots \circ \varsigma(\mathcal O_Z(i-1)) \circ s : \op{Id} \to \Phi \circ (i) \circ \Phi^{-1} \cong L_{\mathcal O_Z} \circ \cdots L_{\mathcal O_Z(i-1)} \circ T_{\mathcal O_Z(i)}.
  \end{displaymath}
  
  \item If $d > N$, then there exists morphisms in $\op{Ho}(\op{dg-cat}_k)$
  \begin{align*}
   \Psi^! & : \mathsf{Inj}_{\op{coh}}(\mathbb{A}^n,G,w) \to \mathsf{Inj}_{\op{coh}_G}(Y) \\
   \Psi & : \mathsf{Inj}_{\op{coh}_G}(Y) \to \mathsf{Inj}_{\op{coh}}(\mathbb{A}^n,G,w)
  \end{align*}
  and a semi-orthogonal decomposition
  \begin{displaymath}
   \dabs [\mathsf{fact}(\mathbb{A}^n,G,w)]  = \left \langle \bigoplus_{ \alpha|_{\mathbb{G}_m} = -1 + d - N} k(\alpha), \ldots, \bigoplus_{ \alpha|_{\mathbb{G}_m} = 0} k(\alpha), [\Psi]\left( [\mathsf{Inj}_{\op{coh}_G}(Y)]  \right) \right \rangle.
  \end{displaymath}
  
   Moreover, if $G = \mathbb{G}_m$ and $N=n$, there are quasi-isomorphisms of bimodules
  \begin{align*}
   \Psi^! \circ (1) \circ \Psi & \cong \{1\} \\
   \Psi^! \circ \Psi & \cong \Delta_* \mathcal O_Z.
  \end{align*}
  and 
  \begin{displaymath}
   [\Psi^!] k(j) \cong 0
  \end{displaymath}
  for $d-N-1 \geq j \geq 0$.
 \end{itemize}
\end{theorem}

\begin{proof}
 In the case that $\mathbb{G}_m = G$, in \cite[Theorem 2.13]{Orl09}, Orlov constructs the triangulated functors and the semi-orthogonal decompositions of the triangulated categories.  The isomorphisms on the level of triangulated functors were constructed in \cite[Proposition 5.8]{BFK}.  Cald\u{a}r\u{a}ru and Tu \cite[Theorem 5.9]{CT} lifted these functors to dg-functors between appropriate enhancements. We indicate the extension to $G$ as in the statement of the theorem.
 
 Consider the following diagram of dg-categories:
 \begin{center}
 \begin{tikzpicture}
   [scale=1, cone/.style={->,thick,>=stealth}]
   \node (a) at (0,0) {$\mathsf{Inj}_{\op{coh}_G, \geq i}(U)$};
   \node (b) at (-3,-1) {$\mathsf{vect}(\mathbb{A}^n,G,w)^{\op{op}}\op{-Mod}$};
   \node (c) at (3,-1) {$\mathsf{Inj}_{\op{coh}_G}(Y)$};
   \draw[cone] (a) -- node[above] {$\Upsilon_i$} (b);
   \draw[cone] (a.south) -- node[below] {$\pi_i$} (c.west);
   \draw[cone] (c.north) -- node[above] {$\omega_i$} (a.east);
 \end{tikzpicture}
 \end{center}
 Here $U$ is zero locus of $w$ in $\mathbb{A}^n$. The dg-category $\mathsf{Inj}_{\op{coh}_G, \geq i}(U)$ consists of bounded below complexes of injective $G$-equivariant sheaves on $U$ whose cohomology lies in $\mathbb{G}_m$-degrees $\geq i$, is bounded, and finitely-generated. Let $\Upsilon_i$ denote the restriction of $\Upsilon$ to $\mathsf{Inj}_{\op{coh}_G, \geq i}(U)$ which is then naturally a dg-module for $\mathsf{vect}_{\op{coh}}(\mathbb{A}^n,G,w)$. It is easy to see that $\Upsilon_i$ is a quasi-functor.
 
 Finally, let $\pi$ the restriction along the inclusion $Y \to U$, $\pi_i$ the restriction of $\pi$ to $\mathsf{Inj}_{\op{coh}_G, \geq i}(U)$, and let $\omega_i$ denote the functor,
 \begin{displaymath}
  \omega_i(\mathcal F) := \bigoplus_{\substack {\alpha \in \widehat{G} \\ \alpha|_{\mathbb{G}_m} \geq i}} \op{H}^0(Y, \mathcal F(\alpha)). 
 \end{displaymath}
 Note that, as $\omega_i$ is right adjoint to $\pi_i$ at the level of the Abelian category of equivariant sheaves, the corresponding dg-functors are also adjoint. 
 
 Next, define $\mathsf{D}_i$ to be the quasi-essential image of $\omega_i$, in particular $\mathsf{D}_i$ is closed under quasi-isomorphism, and $\mathsf{P}_{\geq i}$ to be the full dg-subcategory of $\mathsf{Inj}_{\op{coh}_G, \geq i}(U)$ containing the injective resolutions of $\mathcal O_U(\alpha)$ for $\alpha|_{\mathbb{G}_m} \leq i$. Finally, let $\mathsf{T}_i$ be the full dg-subcategory containing all $\mathcal F$ that satisfy
 \begin{displaymath}
  \op{H}^{\bullet}\left (\op{Hom}_{\mathsf{Inj}_{\op{coh}_G, \geq i}(U)}(\mathcal F, \mathcal P) \right) = 0
 \end{displaymath}
 for all $\mathcal P \in \mathsf{P}_{\geq i}$. 
 
 As $\pi \circ \omega_i = \op{Id}$, the restriction of $\pi_i$ to $\mathsf{D}_i$ is a quasi-equivalence and $\omega_i$ its inverse. Following arguments of \cite{Orl09}, which we suppress, the restriction of $\Upsilon$ to $\mathsf{T}_i$ is a quasi-equivalence. Let $\nu_i$ be the inverse to $\Upsilon|_{\mathsf{T}_i}$ in $\op{Ho}(\op{dg-cat}_k)$. 
  One then sets 
 \begin{align*}
  \Phi_i & := \pi \circ \nu_i, \Phi := \Phi_1 \\
  \Phi_i^! & := \Upsilon \circ \omega_i, \Phi^! := \Phi^!_1 \\
  \Psi_i & := \Upsilon \circ \omega_i, \Psi := \Psi_1  \\
  \Psi_i^! & := \pi \circ \nu_{i-d+n}, \Psi^! := \Psi^!_1. 
 \end{align*}
 
 The proofs of the existence of the semi-orthogonal decompositions follow along the same arguments of \cite{Orl09} using the fact that
 \begin{displaymath}
  \mathbf{R}\!\op{Hom}_{\mathsf{Qcoh} U}(\bullet, \mathcal O_U) : \op{D}^{\op{b}}(\op{coh}_G U)^{\op{op}} \to \op{D}^{\op{b}}(\op{coh}_G U)
 \end{displaymath}
 is an equivalence satisfying
 \begin{displaymath}
  \mathbf{R}\!\op{Hom}_{\mathsf{Qcoh} U}(k,\mathcal O_U) \cong k(\nu)[-n]
 \end{displaymath}
 for $\nu \in \widehat{G}$ with $\nu|_{\mathbb{G}_m} = N$.
 
 In the case $\mathbb{G}_m = G$ and $n=N$, we have an equivalence $\op{Qcoh}_G Y \cong \op{Qcoh} Z$. The statements that 
 \begin{displaymath}
  [\Phi^!] \mathcal O_Z(i) \cong 0
 \end{displaymath}
 for $d-N \leq i \leq -1$ and 
 \begin{displaymath}
  [\Psi^!] k(j) \cong 0
 \end{displaymath}
 for $d-N-1 \geq j \geq 0$ follow immediately from \cite{Orl09}. 
 
 The only remaining statement to check is that concerning the existence of quasi-isomorphisms between the stated bimodules. It suffices to show that the corresponding dg-functors are naturally quasi-isomorphic.  

 Now, consider the following dg-functor,
 \begin{align*}
  M: \mathsf{Inj}_{\op{coh}_G, \geq i}(U) & \to \mathsf{Inj}_{\op{coh}_G, \geq 0}(U) \\
  \mathcal E & \mapsto \op{Cone} \left(\op{Hom}_{\mathsf{Inj}_{\op{coh}_G, \geq 1}(U)}(\widetilde{\mathcal O}_U,\mathcal E(1)) \otimes_k (\widetilde{\mathcal O}_U \overset{ev}{\to} \mathcal E(1) \right)
 \end{align*}
 where $(\widetilde{\mathcal O}_U$ is an injective resolution of $\mathcal O_U$. Note that we have a natural transformation $\eta: (1) \to M$. 
 
 Consider the diagram
 \begin{center}
  \begin{tikzpicture}[description/.style={fill=white,inner sep=2pt}]
   \matrix (m) [matrix of math nodes, row sep=3em, column sep=4em, text height=1.5ex, text depth=0.25ex]
   { \mathsf{Inj}_{\op{coh}_G, \geq i}(U)  & \mathsf{Inj}_{\op{coh}_G, \geq 0}(U)  \\ 
     \mathsf{Inj}_{\op{coh}}(Z) & \mathsf{Inj}_{\op{coh}}(Z) \\ };
   \path[->,font=\scriptsize]
   (m-1-1) edge node[above] {$M$} (m-1-2)
   (m-2-1) edge node[left] {$\omega_1$} (m-1-1)
   (m-1-2) edge node[right] {$\pi$} (m-2-2)
   (m-2-1) edge node[above] {$L_{\mathcal O_Z} \circ T_{\mathcal O_Z(1)}$} (m-2-2)
   ;
  \end{tikzpicture}
 \end{center}
 The composition equals
 \begin{displaymath}
  (\pi \circ M \circ \omega_i)(\mathcal E) := \op{Cone} \left( \op{Hom}_{\mathsf{Inj}_{\op{coh}_G, \geq 0}(U)}(\widetilde{\mathcal O}_U,\omega_i \mathcal E(1)) \otimes_k \pi \widetilde{\mathcal O}_U \overset{\op{ev}}{\to} (\pi \circ \omega_i)(\mathcal E(1))\right).
 \end{displaymath}
 Using the adjunction, $\pi \dashv \omega_i$, and the identity, $\pi \circ \omega_i \cong \op{Id}$, the composition is isomorphic to
 \begin{displaymath}
  \op{Cone} \left( \op{Hom}_{\mathsf{Inj}_{\op{coh}}(Z)}(\widetilde{\mathcal O}_Z, \mathcal E(1)) \otimes_k \widetilde{\mathcal O}_Z \overset{\op{ev}}{\to} \mathcal E(1) \right) = \mathcal E\{1\}.
 \end{displaymath}
 Thus, we have a natural isomorphism
 \begin{equation} \label{eq: formula for M}
  \pi \circ M \circ \omega_i \cong \{1\}.
 \end{equation}
 We will use this equation in both cases.

 Now, assume that $d \leq n$ and consider the composition
 \begin{displaymath}
  \Phi^! \circ \{1\} \circ \Phi = \Upsilon \circ \omega_1 \circ \{1\} \circ \pi \circ \nu_1.
 \end{displaymath}
 We can substitute 
 \begin{displaymath}
  \Upsilon \circ \omega_1 \circ \{1\} \circ \pi \circ \nu_1 \cong \Upsilon \circ \omega_1 \circ \pi \circ M \circ \omega_1 \circ \pi \circ \nu_1.
 \end{displaymath}
 Since the image of $\nu_1$ lies in $\mathsf{D}_1$ by \cite{Orl09}, we have
 \begin{displaymath}
  \omega_1 \circ \pi \circ \nu_1 \cong \nu_1.
 \end{displaymath}
 One can check, as in \cite[Lemma 5.7]{BFK}, that $M \circ \nu_1$ has quasi-essential image in $\mathsf{D}_1$, thus we have a natural quasi-isomorphism
 \begin{displaymath}
  M \circ \nu_1 \to \omega_1 \circ \pi \circ M \circ \nu_1.
 \end{displaymath}
 This gives a natural quasi-isomorphism
 \begin{displaymath}
  \Phi^! \circ \{1\} \circ \Phi \simeq \Upsilon \circ M \circ \nu_1.
 \end{displaymath}
 The composition
 \begin{displaymath}
  \Upsilon \circ (1) \circ \nu_1 \overset{\Upsilon(\eta_{\nu_1})}{\to} \Upsilon \circ M \circ \nu_1 
 \end{displaymath}
 is a quasi-isomorphism for all objects as $\Upsilon (\widetilde{\mathcal O}_U)$ is acyclic. Thus, using the above and Equation~\eqref{eq: formula for M}, we have a quasi-isomorphism
 \begin{displaymath}
  \Phi^! \circ \{1\} \circ \Phi \simeq \Upsilon \circ M \circ \nu_1 \simeq \Upsilon \circ (1) \circ \nu_1 = (1).
 \end{displaymath}
 
 Now, assume that $d \geq n$ and consider the composition
 \begin{displaymath}
  \Psi^! \circ (1) \circ \Psi = \pi \circ \nu_{1-d+n} \circ (1) \circ \Upsilon \circ \omega_{1}.
 \end{displaymath}
 One has a natural quasi-isomorphism
 \begin{displaymath}
  (1) \circ \Upsilon = \Upsilon \circ (1) \overset{\Upsilon(\eta)}{\to} \Upsilon \circ M.
 \end{displaymath}
 Thus, 
 \begin{displaymath}
  \pi \circ \nu_{1-d+n} \circ (1) \circ \Upsilon \circ \omega_1 \cong \pi \circ \nu_{1-d+n} \circ \Upsilon \circ M \circ \omega_1.
 \end{displaymath}
 As $\mathsf{D}_1 \subset \mathsf{T}_{1-d+n}$ by \cite{Orl09} and $M(\mathsf{D}_1)$ lies in $\mathsf{D}_1$, we have
 \begin{displaymath}
  \pi \circ \nu_{1-d+n} \circ \Upsilon \circ M \circ \omega_1 \cong \pi \circ M \circ \omega_1 \cong \{1\}
 \end{displaymath}
 where the last quasi-isomorphism is Equation~\eqref{eq: formula for M}.

 Finally, let us assume that $d=N=n$ and $G = \mathbb{G}_m$. Let $s \in k[x_1,\ldots,x_n]$ be homogeneous of degree $1$, the natural transformations of exact functors,
 \begin{align*}
  s: \op{Id}_{\dabs[\mathsf{fact}(\mathbb{A}^n,\mathbb{G}_m,w)]} \to (1) \\
  s: \op{Id}_{\dbcoh{Z}} \to T_{\mathcal O_Z(1)}.
 \end{align*}
 Let $\mathcal E$ be an object of $\op{Inj}_{\op{coh}}(Z)$ and consider $s: \mathcal E \to T_{\mathcal O_Z(1)}(\mathcal E)$. Applying $\omega_1$ gives a morphism 
 \begin{displaymath}
  \omega_1(s): \omega_1(\mathcal E) \to \omega_1(\mathcal E)_{\geq 2}(1).
 \end{displaymath}
 Composing with the  inclusion
 \begin{displaymath}
   \omega_1(\mathcal E)_{\geq 2}(1) \to \omega_1(\mathcal E)(1)
 \end{displaymath}
 equals 
 \begin{displaymath}
  s: \omega_1(\mathcal E) \to \omega_1(\mathcal E)(1).
 \end{displaymath}
 Apply the dg-functor
 \begin{displaymath}
  L_{\mathcal O_U}(\mathcal I) := \op{Cone}(\op{Hom}(\widetilde{\mathcal O}_U, \mathcal I) \otimes_k \widetilde{\mathcal O}_U \to \mathcal I).
 \end{displaymath}
 We get a map
 \begin{displaymath}
  \varsigma(\mathcal O_U)_{\omega_1(\mathcal E)} \circ s: \omega_1(\mathcal E) \to L_{\mathcal O_U}(\omega_1(\mathcal E)(1)).
 \end{displaymath}
 Since $\omega_1(\mathcal E)_{\geq 2}(1)$ is concentrated in homogeneous degrees $\geq 1$, we have
 \begin{displaymath}
  \op{H}^{\bullet}(\op{Hom}(\widetilde{\mathcal O}_U, \omega_1(\mathcal E)_{\geq 2}(1))) = \op{H}^{\bullet}(\op{Hom}(\mathcal O_U, \omega_1(\mathcal E)_{\geq 2}(1))) = 0.
 \end{displaymath}
 Thus, $\varsigma(\mathcal O_U)_{\omega_1(\mathcal E)_{\geq 2}(1)}: \omega_1(\mathcal E)_{\geq 2}(1) \to L_{\mathcal O_U}(\omega_1(\mathcal E)_{\geq 2}(1))$, is a quasi-isomorphism. Applying $\pi$, gives 
 \begin{displaymath}
  \Phi(s) = \varsigma(\mathcal O_Z)_{\mathcal E} \circ s: \mathcal E \to L_{\mathcal O_Z} \circ T_{\mathcal O_Z(1)}(\mathcal E)
 \end{displaymath}
 on $\dbcoh{Z}$. It is straightforward to check there are isomorphisms
 \begin{displaymath}
  \Phi \circ (i) \circ \Phi^{-1} \cong \{i\} \cong L_{\mathcal O_Z} \circ \cdots \circ L_{\mathcal O_Z(i-1)} \circ T_{\mathcal O_Z(i)}.
 \end{displaymath}
 We have two algebra homomorphisms 
 \begin{displaymath}
  S \to \bigoplus_{i \in \Z} \op{Nat}(\op{Id},\{i\})
 \end{displaymath}
 where $\op{Nat}$ denotes natural transformations. The first is given by conjugation by $\Phi$ while the second is
 \begin{displaymath}
  s \mapsto \varsigma(\mathcal O_Z) \circ \cdots \circ \varsigma(\mathcal O_Z(i-1)) \circ s.
 \end{displaymath}
 for $s \in S_i$. These agree on generators for $S$ and hence agree overall. 
\end{proof}

\begin{remark}
 From the arguments above, it is clear that in the case $G = \mathbb{G}_m$ and $d=N=n$, that $\Phi(k[1]) \cong \mathcal O_Z$.
\end{remark}

\begin{remark}
 One could also apply the results in \cite{BFK12} on VGIT for equivariant factorizations. Or, one could directly lift the statements of \cite{Orl09} using the results of \cite{Elagin}.
\end{remark}

\begin{remark}
 The case $G \not = \mathbb{G}_m$ will be used in \cite{BFKthesequel}. Henceforth, we will only apply Theorem~\ref{theorem: Orlov} under the assumption that $G = \mathbb{G}_m$ act in the usual manner on $\mathbb{A}^n$.
\end{remark}

\begin{corollary} \label{corollary: Orlov on HH}
 Let $w$ be a degree $d$ homogeneous polynomial in $k[x_1,\ldots,x_n]$ with its standard grading. Let $Z$ be the projective hypersurface defined by $w$. Assume that $Z$ is smooth.
 \begin{itemize}
  \item If $d < n$, we have a commutative diagram of vector spaces
   \begin{center}
   \begin{tikzpicture}[description/.style={fill=white,inner sep=2pt}]
    \matrix (m) [matrix of math nodes, row sep=3em, column sep=3em, text height=1.5ex, text depth=0.25ex]
    { \op{HH}_{\bullet}(Z) & \op{HH}_{\bullet}(Z) \\ 
      \op{HH}_{\bullet}(\mathbb{A}^n, \mathbb{G}_m, w) & \op{HH}_{\bullet}(\mathbb{A}^n, \mathbb{G}_m, w) \\ };
    \path[->,font=\scriptsize]
    (m-1-1) edge node[above] {$\{ 1 \}_{\bullet}$} (m-1-2)
    (m-2-1) edge node[left] {$\Phi_{\bullet}$} (m-1-1)
    (m-1-2) edge node[right] {$\Phi^!_{\bullet}$} (m-2-2)
    (m-2-1) edge node[above] {$(1)_{\bullet}$} (m-2-2)
    ;
   \end{tikzpicture}
   \end{center} 
   Moreover, 
   \begin{align*}
    \Phi^!_{\bullet} \circ \Phi_{\bullet} & = 1, 
   \end{align*}
  the functor $\Phi^!_{\bullet}$ is right adjoint to $\Phi_{\bullet}$ under the categorical pairing,
   and we have an orthogonal decomposition
   \begin{displaymath}
    \op{HH}_{\bullet}(Z) = \op{Im} \Phi_{\bullet} \oplus \bigoplus_{j=d-n-1}^{-1} \C \cdot \op{ch}(\mathcal O_Z(j)). 
   \end{displaymath}

  \item If $d=n$, we have a commutative diagram of vector spaces
   \begin{center}
   \begin{tikzpicture}[description/.style={fill=white,inner sep=2pt}]
    \matrix (m) [matrix of math nodes, row sep=3em, column sep=3em, text height=1.5ex, text depth=0.25ex]
    { \op{HH}_{\bullet}(Z) & \op{HH}_{\bullet}(Z) \\ 
      \op{HH}_{\bullet}(\mathbb{A}^n, \mathbb{G}_m, w) & \op{HH}_{\bullet}(\mathbb{A}^n, \mathbb{G}_m, w) \\ };
    \path[->,font=\scriptsize]
    (m-1-1) edge node[above] {$\{ 1 \}_{\bullet}$} (m-1-2)
    (m-2-1) edge node[left] {$\Phi_{\bullet}$} (m-1-1)
    (m-2-2) edge node[right] {$\Phi_{\bullet}$} (m-1-2)
    (m-2-1) edge node[above] {$(1)_{\bullet}$} (m-2-2)
    ;
   \end{tikzpicture}
   \end{center}
   and $\Phi_{\bullet}$ is an isomorphism.
   
  \item If $d > n$, we have a commutative diagram of vector spaces
   \begin{center}
   \begin{tikzpicture}[description/.style={fill=white,inner sep=2pt}]
    \matrix (m) [matrix of math nodes, row sep=3em, column sep=3em, text height=1.5ex, text depth=0.25ex]
    { \op{HH}_{\bullet}(Z) & \op{HH}_{\bullet}(Z) \\ 
      \op{HH}_{\bullet}(\mathbb{A}^n, \mathbb{G}_m, w) & \op{HH}_{\bullet}(\mathbb{A}^n, \mathbb{G}_m, w) \\ };
    \path[->,font=\scriptsize]
    (m-1-1) edge node[above] {$\{ 1 \}_{\bullet}$} (m-1-2)
    (m-1-1) edge node[left] {$\Psi_{\bullet}$} (m-2-1)
    (m-2-2) edge node[right] {$\Psi^!_{\bullet}$} (m-1-2)
    (m-2-1) edge node[above] {$(1)_{\bullet}$} (m-2-2)
    ;
   \end{tikzpicture}
   \end{center} 
   Moreover, 
   \begin{align*}
    \Psi^!_{\bullet} \circ \Psi_{\bullet} & = 1,
   \end{align*}
   the functor  $\Psi^!_{\bullet}$ is right adjoint to $\Psi_{\bullet}$ under the categorical pairing, and we have an orthogonal decomposition
    \begin{displaymath}
     \op{HH}_{\bullet}(\mathbb{A}^n, \mathbb{G}_m, w) = \op{Im} \Psi_{\bullet} \oplus \bigoplus_{j = 0}^{d-n-1} \C \cdot \op{ch}(k (j) ).
    \end{displaymath}
 \end{itemize}
\end{corollary}

\begin{proof}
 All statements but the adjunction and orthogonal decomposition are immediate consequences of Theorem~\ref{theorem: Orlov} and the functoriality for pushforwards, \cite[Section 1]{PV}. We check the adjunctions. 
 
 We only provide an argument for the case $d > n$. The case $d < n$ is analogous. We have a splitting 
 \begin{equation} \label{equation: splitting}
  \op{HH}_{\bullet}(\mathbb{A}^n, \mathbb{G}_m, w) = \op{Im} \Psi_{\bullet} \oplus \op{ker} \Psi^!_{\bullet}.
 \end{equation}
 Counting dimensions, we also have an orthogonal decomposition
 \begin{displaymath}
  \op{HH}_{\bullet}(\mathbb{A}^n, \mathbb{G}_m, w) = \op{Im} \Psi_{\bullet} \oplus \bigoplus_{j = 0}^{d-n-1} \C \cdot \op{ch}(k (j) ).
 \end{displaymath}
 Thus, 
 \begin{displaymath}
  \op{ker} \Psi^!_{\bullet} = \bigoplus_{j = 0}^{d-n-1} \C \cdot \op{ch}(k (j) )
 \end{displaymath}
 and the splitting of Equation~\eqref{equation: splitting} is orthogonal with respect to the Mukai pairing. The adjunction now follows via a straightforward linear algebra argument.
\end{proof}

\begin{remark}
 For the case, $d \leq n$, the argument can be significantly simplified using \cite[Theorem 7.1]{KuzBC}.  This result guarantees a splitting of $\op{HH}_{\bullet}(X)$  for any semi-orthogonal decomposition of $\dbcoh{X}$ at the triangulated level without having to prove anything at the level of dg-categories.   For the sake of this utility, we will appeal to this result in Section~\ref{section: cycles}.
\end{remark}

\begin{definition}
 Let $T: V \to V$ be a linear endomorphism of a vector space, $V$, over $\C$, and let $\lambda \in \C$. We denote the $\lambda$-eigenspace of $T$ by $E_{\lambda}(T)$.
\end{definition}

\begin{lemma} \label{lemma: eigenspace of monodromy}
 Under the HKR isomorphism, Theorem~\ref{theorem: HKR}, there is an equality
 \begin{displaymath}
  \phi_{\op{HKR}}\left( E_1(\{1\}_{\bullet}) \right) = \op{H}^{\bullet}_{\op{prim}}(Z ; \C).
 \end{displaymath}
\end{lemma}

\begin{proof}
 Let us first observe that 
 \begin{displaymath}
   \phi_{\op{HKR}}^{-1}\left( \op{H}^{\bullet}_{\op{prim}}(Z ; \C) \right) \subseteq E_1(\{1\}_{\bullet}).
 \end{displaymath}
 It easy to check, cf. \cite[Exercise 5.37]{H}, that, for $v \in \op{H}^{\bullet}(Z ; \C)$,
 \begin{displaymath}
  T_{\mathcal O_Z(1)}^H (v) = v \cdot \op{ch}_{class}(\mathcal O_Z(1)).
 \end{displaymath}
 If we assume that $v$ is primitive, then
 \begin{displaymath}
  v \cdot \op{ch}_{class}(\mathcal O_Z(1)) = v. 
 \end{displaymath}
 It is also easy to verify, cf.\cite[Exercise 8.15]{H}, that 
 \begin{displaymath}
  L_{\mathcal{O}_Z}^H(v) = v - ( \op{ch}_{class}(\mathcal O_Z), v )_M \op{ch}_{class}(\mathcal O_Z).
 \end{displaymath} 
 By definition, the pairing is expressed as 
 \begin{displaymath}
  ( \op{ch}_{class}(\mathcal O_Z), v )_M = \int_Z \op{ch}_{class}(\mathcal O_Z)^{\vee} \cdot v \cdot \op{td}(Z) = \int_Z v \cdot \op{td}(Z).
 \end{displaymath}
 As the Todd class, $\op{td}(Z)$, is of the form $1 + Hp(H)$ for some polynomial $p$, and $v$ is primitive, we have
 \begin{displaymath}
  \int_Z v \cdot \op{td}(Z) = \int_Z v.
 \end{displaymath}
 However, by the Lefschetz Hyperplane Theorem, primitive classes cannot have top dimensional components. Hence,
 \begin{displaymath}
  \int_Z v = 0
 \end{displaymath}
 and
 \begin{displaymath}
  L_{\mathcal{O}_Z}^H(v) = v.
 \end{displaymath}
 As the cohomological integral transform $\{1\}^H$ corresponds to $\{1\}_{\bullet}$ under the HKR isomorphism, Theorem~\ref{theorem: HKR}, we see that $\phi^{-1}_{\op{HKR}}(v) \in E_1(\{1\}_{\bullet})$.
 
 Next, let $\bigoplus_{i=0}^{n-2} \C \cdot H^i$ be the subspace of $\op{H}^{\bullet}(Z ; \C)$ corresponding to powers of the hyperplane class, $H$. Assume that $v = \sum_{i=0}^{n-2} a_i H^i$ lies in $\phi_{\op{HKR}}\left( E_1(\{1\}_{\bullet}) \right) = E_1(\{1\}^H)$. Then, 
 \begin{align*}
  a_0 & = a_0 - ( \op{ch}_{class}(\mathcal O_Z), v )_M \\
  a_i & = \sum_{j=0}^i \frac{a_j}{(i-j)!}, i > 0.
 \end{align*}
 This immediately implies that $v = 0$. As the induced map on cohomology $\{1\}^H$ preserves the splitting
 \begin{displaymath}
  \op{H}^{\bullet}(Z ; \C) = \op{H}^{\bullet}_{\op{prim}}(Z ; \C) \oplus \bigoplus_{i=0}^{n-2} \C \cdot H^i,
 \end{displaymath}
 we see that 
 \begin{displaymath}
  E_1(\{1\}^H) = \op{H}_{\op{prim}}^{\bullet}(Z ; \C)
 \end{displaymath}
 and
 \begin{displaymath}
  \phi^{-1}_{\op{HKR}}( \op{H}^{\bullet}_{\op{prim}}(Z ; \C) ) = E_1(\{1\}_{\bullet}).
 \end{displaymath}
\end{proof}

\begin{theorem} \label{theorem: eigenspaces preserved} 
 Let $w$ be a homogeneous polynomial of degree $d$ in $\C[x_1,\ldots,x_n]$. Assume that $w$ defines a smooth projective hypersurface, $Z$. 
 \begin{itemize}
  \item If we assume $d \leq n$, then the linear map, $\Phi_{\bullet}$, induces an isomorphism,
  \begin{displaymath}
   \Phi_{\bullet}: E_1(\{1\}_{\bullet}) \to E_1((1)_{\bullet}). 
  \end{displaymath}
  \item If we assume $d \geq n$, then the linear map, $\Psi^!_{\bullet}$, induces an isomorphism,
  \begin{displaymath}
   \Psi^!_{\bullet}: E_1(\{1\}_{\bullet}) \to E_1((1)_{\bullet}). 
  \end{displaymath}
 \end{itemize}
In particular, Orlov's theorem and the HKR isomorphism provide isomorphisms,
 \begin{displaymath}
  \op{H}^{p,n-2-p}_{\op{prim}}(Z) \cong \op{Jac}(w)_{d(n-1-p)-n}.
  \end{displaymath}
\end{theorem}

\begin{proof}
 Let us treat the case $d \leq n$ first.  Let $v \in E_1(\{1\}_{\bullet})$. By Lemma~\ref{lemma: eigenspace of monodromy}, $\phi_{\op{HKR}}(v) \in \op{H}_{\op{prim}}^{\bullet}(Z;\C)$. Thus, $v$ is orthogonal to $\op{ch}(\mathcal O_Z(j))$ under the Mukai pairing for each $j \in \Z$. By Corollary~\ref{corollary: Orlov on HH}, we have an orthogonal decomposition
 \begin{displaymath}
  \op{HH}_{\bullet}(Z) = \Phi_{\bullet} \op{HH}_{\bullet}(\mathbb{A}^n,\mathbb{G}_m,w) \oplus \bigoplus_{j = d-n}^{-1} \C \cdot \op{ch}(\mathcal O_Z(j)).
 \end{displaymath}
 Write $v = \Phi_{\bullet}v' \oplus v''$ with respect to this decomposition. Thus, for $j \in \Z$,
 \begin{displaymath}
  0 = (\phi_{\op{HKR}}(v), \phi_{\op{HKR}}(\op{ch}(\mathcal O_Z(j)) )_M = \langle v, \op{ch}(\mathcal O_Z(j)) \rangle = \langle v'', \op{ch}(\mathcal O_Z(j)) \rangle
 \end{displaymath}
 as $\phi_{\op{HKR}}(v) \in \op{H}_{\op{prim}}^{\bullet}(Z;\C)$ and $\phi_{\op{HKR}}(\op{ch}(\mathcal O_Z(j)) = \op{ch}_{class}(\mathcal O_Z(j)) \in \bigoplus_{i=0}^{n-2} \C \cdot H^i$ are orthogonal with respect to the Mukai pairing. Due to their exceptionality, the set of vectors $\op{ch}(\mathcal O_Z(d-n)), \ldots, \op{ch}(\mathcal O_Z(-1))$ forms an orthonormal basis for $\bigoplus_{j = d-n}^{-1} \C \cdot \op{ch}(\mathcal O_Z(j))$. Consequently, $v'' = 0$.
 
 Using Corollary~\ref{corollary: Orlov on HH} repeatedly, we have
 \[
 (1)_{\bullet}(v') =  \Phi^!_{\bullet} \{1\}_\bullet \Phi_{\bullet} v' =  \Phi^!_{\bullet}  \Phi_{\bullet} v' = v'
 \]
 i.e.\ $v' \in E_1((1)_{\bullet})$. Thus, $\Phi_{\bullet}$ maps $ E_1((1)_{\bullet})$ monomorphically into $E_1(\{1\}_{\bullet})$. Counting dimensions, we see this is an isomorphism.

 Now, let us turn our attention to $d \geq n$. By Theorem~\ref{theorem: Orlov}, we have an orthogonal decomposition
 \begin{displaymath}
  \op{HH}_{\bullet}(\mathbb{A}^n,\mathbb{G}_m,w)  = \Psi_{\bullet} \op{HH}_{\bullet}(Z) \oplus \bigoplus_{j=0}^{d-n-1} \C \cdot \op{ch}(k(j)).
 \end{displaymath}
 Assume that $v \in E_1(\{1\}_{\bullet})$. Write
 \begin{displaymath}
  ((1)_{\bullet} \circ \Psi_{\bullet})(v) = \Psi_{\bullet} v \oplus v'
 \end{displaymath}
 with respect to this decomposition. Let us compute 
 \begin{displaymath}
  \langle \op{ch}(k(j)), v' \rangle 
 \end{displaymath}
 for some $j \in \Z$. By orthogonality, we have
 \begin{displaymath}
  \langle \op{ch}(k(j)), v' \rangle = \langle \op{ch}(k(j)), ((1)_{\bullet} \circ \Psi_{\bullet})(v) \rangle.
 \end{displaymath}
 Since $(-1)$ is inverse to $(1)$ and $\Psi_{\bullet} \dashv \Psi^!_{\bullet}$, from Corollary~\ref{corollary: Orlov on HH}, we have
 \begin{displaymath}
  \langle \op{ch}(k(j)), ((1)_{\bullet} \circ \Psi_{\bullet})(v) \rangle = \langle (\Psi^!_{\bullet} \circ (-1)_{\bullet})(\op{ch}(k(j))), v \rangle.
 \end{displaymath}
 Using the functorial properties of pushforwards, we have
 \begin{displaymath}
  (\Psi^!_{\bullet} \circ (-1)_{\bullet})(\op{ch}(k(j))) = \op{ch}\left( \Psi^!(k(j-1)) \right).
 \end{displaymath}
 It is easy to check, in Orlov's equivalence, that $\Psi^! k(j-1)$ lies in the smallest triangulated subcategory of $\dbcoh{Z}$ generated by the objects $ \mathcal O_Z(j) $, $j \in \Z$. Note that we do \textit{not} need to pass to direct summands. Thus, 
 \begin{displaymath}
   \op{ch}(\Psi^! k(j-1)) \in \bigoplus_{j=0}^{n-2} \C \cdot \op{ch}(\mathcal O_Z(j)) = \phi^{-1}_{\op{HKR}}\left( \bigoplus_{j=0}^{n-2} \C \cdot H^j \right).
 \end{displaymath}
 By Lemma~\ref{lemma: eigenspace of monodromy}, $\phi_{\op{HKR}}(v) \in \op{H}^{\bullet}_{\op{prim}}(Z ; \C)$. Thus, $v$ is orthogonal to $\op{ch}(\Psi^! k(j-1))$ for any $j \in \Z$. Therefore, $v'= 0$ and we have a well-defined monomorphism
 \begin{displaymath}
  \Psi_{\bullet} : E_1( \{ 1 \}_{\bullet} ) \to E_1( (1)_{\bullet} ).
 \end{displaymath}
 Counting dimensions finishes the argument.
 
 Now, to see that
 \begin{displaymath}
  \op{H}^{p,n-2-p}_{\op{prim}}(Z) \cong \op{Jac}(w)_{d(n-1-p)-n},
 \end{displaymath}
 notice that  by Corollary~\ref{corollary: projective hypersurface case} and Theorem~\ref{thm: PV}, we have an isomorphism
 \begin{displaymath}
  E_1((1)_{\bullet}) \cap \op{HH}_{t}(\mathbb{A}^n,\mathbb{G}_m,w) \cong \op{Jac}_{d(\frac{n+t}{2}) - n}.
 \end{displaymath}
 By Theorem~\ref{theorem: eigenspaces preserved}, we have an isomorphism
 \begin{displaymath}
  E_1(\{1\}_{\bullet}) \cap \op{HH}_{t}(Z) \cong E_1((1)_{\bullet}) \cap \op{HH}_{t}(\mathbb{A}^n,\mathbb{G}_m,w).
 \end{displaymath}
 From Lemma~\ref{lemma: eigenspace of monodromy} and Theorem~\ref{theorem: HKR}, we have an isomorphism
 \begin{displaymath}
  \op{H}^{\bullet}_{\op{prim}}(Z) \cap \bigoplus_{q-p=t} \op{H}^{p,q}(Z) \cong E_1(\{1\}_{\bullet}) \cap \op{HH}_{t}(Z).
 \end{displaymath}
 Since we only have primitive cohomology in the middle degree, we must have $p+q = n-2$. Solving for $t$ gives $t = n-2-2p$. Plugging in gives the statement.
\end{proof}

\begin{remark}
 One can also define $\op{H}_{\op{prim}}^{\bullet}(Z)$ as the orthogonal to $\sum_{i \in \Z} k \cdot \op{ch}(\mathcal O_Z(i))$ with respect to the categorical pairing. This extends Theorem~\ref{theorem: eigenspaces preserved} to other algebraically closed fields of characteristic zero. \end{remark}

\begin{remark}
In addition to having interesting Eigenspaces, the determinant of $\{1\}_\bullet$ is the geometric genus of the hypersurface.
\end{remark}

\begin{definition}
 Let $Z$ be a smooth, projective hypersurface. Let $\mathcal K$ be the object
 \begin{displaymath}
  \mathcal K := \mathcal I_{\Delta} \otimes_{\mathcal O_{Z \times Z}} \pi_1^*\mathcal O_{Z}(1)[1].
 \end{displaymath}
 Define the graded ring
 \begin{displaymath}
  S(Z) := \bigoplus_{i \geq 0} \op{Hom}_{\dbcoh{Z \times Z}}(\Delta_* \mathcal O_Z, \mathcal K^{\ast i})
 \end{displaymath}
 where $K^{\ast i}$ denotes $i$-th self-convolution $\mathcal K$, cf. \cite[Section 5.1]{H}.
\end{definition}

\begin{lemma} \label{lemma: kernel of monodromy}
 Assume that $Z$ is Calabi-Yau. There is an isomorphism of functors
 \begin{displaymath}
  \{1\} \cong \Phi_{\mathcal K}: \dbcoh{Z} \to \dbcoh{Z}
 \end{displaymath}
 and an injective homomorphism of graded rings
 \begin{displaymath}
  \op{Jac}(w) \to S(Z)
 \end{displaymath}
 where $w$ is the defining polynomial of $Z$.
\end{lemma}

\begin{proof}
 It is straightforward to check that we have a quasi-isomorphism of kernels, $\mathcal K \cong \{1\}$. Using Orlov's equivalence from Theorem~\ref{theorem: Orlov}, we get a isomorphism of graded rings
 \begin{displaymath}
  \bigoplus_{i \geq 0} \op{Hom}_{\dabs [\mathsf{fact}(\mathbb{A}^n \times \mathbb{A}^n, \mathbb{G}_m \times_{\mathbb{G}_m} \mathbb{G}_m, (-w) \boxplus w)]}(\nabla, \nabla(i)) \to \bigoplus_{i \geq 0} \op{Hom}_{\dbcoh{Z \times Z}}(\Delta_* \mathcal O_Z, \mathcal K^{\ast i}).
 \end{displaymath}
 There is a natural homomorphism of graded rings
 \begin{displaymath}
  k[x_1,\ldots,x_n] \to \bigoplus_{i \geq 0} \op{Hom}_{\dabs [\mathsf{fact}(\mathbb{A}^n \times \mathbb{A}^n, \mathbb{G}_m \times_{\mathbb{G}_m} \mathbb{G}_m, (-w) \boxplus w)]}(\nabla, \nabla(i))
 \end{displaymath}
 given by multiplying by a polynomial. By Theorem~\ref{thm: twisted HH*}, this induces a monomorphism
 \begin{displaymath}
  \op{Jac}(w) \to \bigoplus_{i \geq 0} \op{Hom}_{\dabs [\mathsf{fact}(\mathbb{A}^n \times \mathbb{A}^n, \mathbb{G}_m \times_{\mathbb{G}_m} \mathbb{G}_m, (-w) \boxplus w)]}(\nabla, \nabla(i)).
 \end{displaymath}
 The total composition is the desired homomorphism $\op{Jac}(w) \to S(Z)$.
\end{proof}

\begin{remark} 
 A natural question to ask of Griffiths' Residue Theorem is: where do all the other graded pieces of the Jacobian algebra go? Lemma~\ref{lemma: kernel of monodromy} provides the answer in terms of the derived category of $Z$ for a Calabi-Yau hypersurface. The whole Jacobian algebra sits as a graded subring of morphisms in $\dbcoh{Z \times Z}$ from the identity functor to powers of $\{1\}$. Certain powers of $\{1\}$ are shifts of the Serre functor. Those graded pieces of the Jacobian algebra then appear in $\op{HH}_{\bullet}(Z) \cong \op{H}^{\bullet}(Z ; \C)$. 
 
 In the Fano case, we have to replace $S(Z)$ with the graded algebra
 \begin{displaymath}
  \bigoplus_{i \geq 0} \op{Hom}_{\dbcoh{Z \times Z}}(\mathcal P, \mathcal P \ast \{i\} \ast \mathcal P)
 \end{displaymath}
 where $\mathcal P = \Phi \circ \Phi^!$ is the kernel associated to the inclusion of $\dabs [\mathsf{fact}(\mathbb{A}^n,\mathbb{G}_m,w)] \to \dbcoh{Z}$ as an admissible subcategory, \cite{KuzBC}.
 
 In the general type case, we have different kernels, $\mathcal K_i = \Psi^! \circ (i) \circ \Psi$, for each $i$. The natural repository for the Jacobian algebra is the graded vector space
 \begin{displaymath}
  \bigoplus_{i \geq 0} \op{Hom}_{\dbcoh{Z \times Z}}(\Delta_* \mathcal O_Z, \mathcal K_i).
 \end{displaymath}

 In each situation, we have a categorical realization of Griffiths' fundamental result that sees the entire Jacobian algebra.
\end{remark}

\subsection{Using equivariant factorizations to study algebraic cycles} \label{section: cycles}

In this section we examine how algebraic classes behave under variation of the group action. Using Theorem~\ref{thm: PV}, the induction functor, and functoriality of push-forwards, Proposition~\ref{prop: functorial push}, one can precisely relate the algebraic classes under induction and restriction of the group action. The following is essentially due to Polishchuk and Vaintrob.

\begin{proposition} \label{cor: compare algebraic cycles}
 Let $\mathbb{A}^n$ carry a linear action of $G$, an Abelian algebraic group, and let $w \in \Gamma(\mathbb{A}^n,\mathcal O_{\mathbb{A}^n}(\chi))^G$. Assume that $K_{\chi}$ is finite and $\chi: G \to \mathbb{G}_m$ is surjective. Furthermore, assume that $(\op{d} \! w)$ is supported at $\{0\} \in \mathbb{A}^n$. Let $\phi: H \to G$ be an injective homomorphism of affine algebraic groups and assume that $\chi \circ \phi$ is surjective. Consider the functors,
\begin{align*}
 \op{Ind}_H^G  & : \mathsf{vect}(\mathbb{A}^n,H,w) \to \mathsf{vect}(\mathbb{A}^n,G,w) \\ 
 \op{Res}_H^G & : \mathsf{vect}(\mathbb{A}^n,G,w) \to \mathsf{vect}(\mathbb{A}^n,H,w),
\end{align*}
 and the induced maps,
\begin{align*}
 {\op{Ind}_H^G}_{\bullet} & : \op{HH}_{\bullet}(\mathbb{A}^n,H,w) \to \op{HH}_{\bullet}(\mathbb{A}^n,G,w) \\ 
 {\op{Res}_H^G}_{\bullet} & : \op{HH}_{\bullet}(\mathbb{A}^n,G,w) \to \op{HH}_{\bullet}(\mathbb{A}^n,H,w).
\end{align*}
 The composition is the linear map satisfying
\begin{align*}
 {\op{Ind}_H^G}_{\bullet} \circ {\op{Res}_H^G}_{\bullet}: \op{HH}_{\bullet}(\mathbb{A}^n,G,w) & \to \op{HH}_{\bullet}(\mathbb{A}^n,G,w) \\
 v & \mapsto \begin{cases} |G/H| v & v \in \op{Jac}(w_g) \text{ with } g \in K_{\chi \circ \phi} \\ 0 & v \in \op{Jac}(w_g) \text{ with } g \not\in K_{\chi \circ \phi}. \end{cases}
\end{align*}
\end{proposition}

\begin{proof}
 Let $K$ denote the kernel of $\widehat{\phi}: \widehat{G} \to \widehat{H}$. For $c \in K$, $c(g) = 1$ if and only if $g \in K_{\chi \circ \phi}$. From Lemma~\ref{lemma: facts about restriction induction}, we have an isomorphism of functors, $\op{Ind}_H^G \circ \op{Res}_H^G \cong p_*p^*$, where $p: G/H \times \mathbb{A}^n \to \mathbb{A}^n$ is the projection. Therefore, $\op{Ind}_H^G \circ \op{Res}_H^G \cong \bigoplus_{c \in K}(c)$. Note that $\bigoplus_{c \in K}(c)$ can be factored as a composition
\begin{displaymath}
 \mathsf{vect}(\mathbb{A}^n,G,w) \overset{\kappa}{\to} \coprod_{c \in K} \mathsf{vect}(\mathbb{A}^n,G,w) \overset{\oplus}{\to} \mathsf{vect}(\mathbb{A}^n,G,w)
\end{displaymath}
 where $\kappa$ maps to the factor corresponding to $c$ by the autoequivalence, $(c)$, and $\oplus$ is the functor that takes $\coprod\mathcal E_c$ to $\oplus\mathcal E_c$. Here $\coprod_{c \in K} \mathsf{vect}(\mathbb{A}^n,G,w)$ denotes the category whose objects are $|K|$-tuples of objects from $\mathsf{vect}(\mathbb{A}^n,G,w)$ and whose morphisms are $|K|$-tuples of morphisms $\mathsf{vect}(\mathbb{A}^n,G,w)$. Denote an object of $\coprod_{c \in K} \mathsf{vect}(\mathbb{A}^n,G,w)$ by $\oplus_{c \in K} \mathcal E_ce_c$ where we think of $e_c$ as orthogonal idempotents.

 A generator of $\mathsf{vect}(\mathbb{A}^n,G,w)$ exists by Lemma~\ref{corollary: generate with line bundles and singular locus}, Proposition~\ref{proposition: upsilon surjective}, and the assumption that the support of $(\op{d} \! w)$ is $\{0\}$. Choose a generator, $\mathcal G$, and let $A$ denote its dg-endomorphism complex. If we take $\oplus \mathcal G e_c$ as our generator of $\coprod_{c \in K} \mathsf{vect}(\mathbb{A}^n,G,w)$, we see its dg-endomorphism complex is $\tilde{A} = Ae_1 \oplus \cdots \oplus Ae_c$ where $e_c$ are (closed) orthogonal idempotents. It is easy to see that $\tilde{A} \overset{\mathbf{L}}{\otimes}_{\tilde{A}^e} \tilde{A} \cong \oplus_{c \in K} (A \overset{\mathbf{L}}{\otimes}_{A^e} A) e_c$. Thus, $\op{HH}_{\bullet}(\coprod_{c \in K} \mathsf{vect}(\mathbb{A}^n,G,w))$ is isomorphic to $\oplus_{c \in K} \op{HH}_{\bullet}(\mathbb{A}^n,G,w)e_c$. 

 Theorem~\ref{thm: PV} says that the action on the component of $\op{HH}_{\bullet}(\coprod_{c \in K} \mathsf{vect}(\mathbb{A}^n,G,w))$ corresponding to $\op{Jac}(w_g)$ is multiplication by $c(g)^{-1}$. 

 In terms of $\tilde{A}$ and $A$, $\oplus: \coprod_{c \in K} \mathsf{vect}(\mathbb{A}^n,G,w) \to \mathsf{vect}(\mathbb{A}^n,G,w)$ corresponds to the summing map $\tilde{A} \to A$ which takes $\oplus a_ce_c$ to $\sum a_c$. It is easy to see the induced action on Hochschild homology is again summation. 

 Now, we see that if $g \in K_{\chi \circ \phi}$, then each $c$ acts trivially and the summand corresponding to $\op{Jac}(w_g)$ gets multiplied by $|K| = |G/H|$. If $g \not \in K_{\chi \circ \phi}$, then $c(g)$ is nonzero and $\sum_{c \in K} c(g) = 0$. 
\end{proof}

Next, we prove a lemma that allows us to lift algebraic cycles via induction.

\begin{lemma} \label{lem: cycles under grading change}
 Let $\mathbb{A}^n$ carry a linear action of $G$, an Abelian algebraic group, and let $w \in \Gamma(\mathbb{A}^n,\mathcal O_{\mathbb{A}^n}(\chi))^G$. Assume that $K_{\chi}$ is finite and $\chi: G \to \mathbb{G}_m$ is surjective. Furthermore, assume that $(\op{d} \! w)$ is supported at $\{0\} \in \mathbb{A}^n$. Assume that the image of the Chern character,
\begin{displaymath}
 \op{ch}: K_0(\mathbb{A}^{nr},G,w^{\boxplus r}) \to \op{HH}_0(\mathbb{A}^{nr},G,w^{\boxplus r}),
\end{displaymath}
 spans, over $\C$, for all $r \geq 1$. Furthermore, assume that 
\begin{displaymath}
 \left( \op{Jac}(w_g)(-\kappa_e-\kappa_g-u\chi) \right)^G = 0 
\end{displaymath}
 for $t \not =0$ and $g \not = e$ where either $2u = \op{dim} W_g + t - n$ or $2u+1 = \op{dim} W_g + t - n$. Then, the image of the Chern character,
\begin{displaymath}
 \op{ch}: K_0(\mathbb{A}^{nr},G^{\times_{\mathbb{G}_m} r},w^{\boxplus r}) \to \op{HH}_0(\mathbb{A}^{nr},G^{\times_{\mathbb{G}_m} r},w^{\boxplus r}),
\end{displaymath}
 also spans, over $\C$. 
\end{lemma}

\begin{proof}
 If $\mathsf C_1, \ldots, \mathsf C_n$ are saturated dg-categories, then it is straightforward to verify 
 \begin{displaymath}
  \bigotimes_{i=1}^r \op{HH}_{\bullet}(\mathsf C_i) \cong \op{HH}_{\bullet}(\mathsf C_1 \circledast  \cdots \circledast \mathsf C_r) 
 \end{displaymath}
 where the isomorphism is given by taking tensor products over $k$. Thus, by Corollary~\ref{corollary: morita product of factorizations},
 \begin{displaymath}
  \op{HH}_{\bullet}(\mathbb{A}^{nr},G^{\times_{\mathbb{G}_m} r},w^{\boxplus r}) \cong \op{HH}_{\bullet}(\mathbb{A}^n,G,w)^{\otimes r}.
 \end{displaymath}
 In particular,
 \begin{equation}\label{eq: Morita product HH to tensor}
  \op{HH}_{0}(\mathbb{A}^{nr},G^{\times_{\mathbb{G}_m} r},w^{\boxplus r}) \cong \bigoplus_{i_1+\cdots+i_r=0} \op{HH}_{i_1}(\mathbb{A}^n,G,w) \otimes_k \cdots \otimes_k  \op{HH}_{i_r}(\mathbb{A}^n,G,w).
 \end{equation}
 To verify the claim, we need to find a basis of $\op{HH}_{\bullet}(\mathbb{A}^{nr},G^{\times_{\mathbb{G}_m} r},w^{\boxplus r})$ which are Chern characters of objects of $\dabs [\mathsf{fact}(\mathbb{A}^{nr},G^{\times_{\mathbb{G}_m} r},w^{\boxplus r})]$. We proceed by induction on $r$. 

 The base case, $r=1$, is covered under the assumptions of the lemma. Assume the lemma is true for all products of size $< r$, and consider the case of $r$. Under the isomorphism of Equation~\eqref{eq: Morita product HH to tensor}, it is enough to find a basis of decomposable vectors, i.e.\ those expressible as tensor products of elements of $\op{HH}_{\bullet}(\mathbb{A}^n,G,w)$. Let 
\begin{displaymath}
 v:= v_1 \otimes_k \cdots \otimes_k v_n \in \op{HH}_{0}(\mathbb{A}^{nr},G^{\times_{\mathbb{G}_m} r},w^{\boxplus r})
\end{displaymath}
 be a decomposable vector.  We have two cases: one, 
\begin{displaymath}
 \text{some } v_i \in \op{HH}_{0}(\mathbb{A}^n,G,w),
\end{displaymath}
 and, two, 
\begin{displaymath}
 \text{no } v_i \in \op{HH}_{0}(\mathbb{A}^n,G,w).
\end{displaymath}

 Let us consider case one first. In this case,
\begin{displaymath}
 v_1 \otimes_k \cdots \otimes_k \widehat{v}_i \otimes_k \cdots \otimes_k v_n \in \op{HH}_0((\mathbb{A}^n)^{\times r-1},G^{\times_{\mathbb{G}_m} r-1},w^{\boxplus r-1}),
\end{displaymath}
 under the isomorphism of Equation~\eqref{eq: Morita product HH to tensor}. By induction, there exists a factorization, $\mathcal E \in \dabs [\mathsf{fact}(\mathbb{A}^{n(r-1)},G^{\times_{\mathbb{G}_m} r-1},w^{\boxplus r-1})]$, with 
\begin{displaymath}
 \op{ch}(\mathcal E) = v_1 \otimes_k \cdots \otimes_k \widehat{v}_i \otimes_k \cdots \ok v_n
\end{displaymath}
 and $\mathcal E^{\prime} \in \dabs [\mathsf{fact}(\mathbb{A}^n,G,w)]$ with $\op{ch}(\mathcal E^{\prime}) = v_i$. Then, 
\begin{displaymath}
 \op{ch}(\mathcal E \boxtimes \mathcal E^{\prime}) = v_1 \otimes_k \cdots \otimes_k v_n.
\end{displaymath}
 This covers the first case. 

 Let us move to the second case. Note that, since we have assumed 
 \begin{displaymath}
  \left( \op{Jac}(w_g)(-\kappa_e-\kappa_g-u\chi) \right)^G = 0 
 \end{displaymath}
 for $t \not = 0$ and $g \not = e$, all of the $v_i \in \op{HH}_0(\mathbb{A}^n,G,w)$ lie in the untwisted sector corresponding to $g=e$. Consider, the diagonal homomorphism, $\phi: G \to G^{\times_{\mathbb{G}_m r}}$. By Proposition~\ref{cor: compare algebraic cycles}, we know the map, 
 \begin{displaymath}
  \left(\op{Ind}^{G^{\times_{\mathbb{G}_m}} r}_G \circ \op{Res}^{G^{\times_{\mathbb{G}_m}} r}_G\right)_{\bullet}: \op{HH}_{0}(\mathbb{A}^{nr},G^{\times_{\mathbb{G}_m} r},w^{\boxplus r}) \to \op{HH}_{0}(\mathbb{A}^{nr},G^{\times_{\mathbb{G}_m} r},w^{\boxplus r}),
 \end{displaymath}
 applied to $v$ is
 \begin{displaymath}
  \left(\op{Ind}^{G^{\times_{\mathbb{G}_m} r}}_{G \bullet} \circ \op{Res}^{G^{\times_{\mathbb{G}_m} r}}_{G \bullet} \right)(v) = |(G^{\times_{\mathbb{G}_m} r})/G| v.
 \end{displaymath}
 By assumption, we can find an $\mathcal E \in \dabs [\mathsf{fact}(A^{\otimes n},M,w^{\boxplus n})]$ with $\op{ch}(\mathcal E) = \op{Res}^{G^{\times_{\mathbb{G}_m} r}}_{G \bullet}(v)$. By Proposition~\ref{prop: functorial push}, we get
 \begin{displaymath}
  \op{ch}(\op{Ind}^{G^{\times_{\mathbb{G}_m} r}}_G \mathcal E) = \op{Ind}^{G^{\times_{\mathbb{G}_m} r}}_{G \bullet}(\op{ch}(\mathcal E)) = \left( \op{Ind}^{G^{\times_{\mathbb{G}_m} r}}_{G \bullet} \circ \op{Res}^{G^{\times_{\mathbb{G}_m} r}}_{G \bullet} \right)(v) = |(G^{\times_{\mathbb{G}_m} r})/G|v.
 \end{displaymath}
 Thus, over $\C$, we can find a spanning set of decomposable vectors in the image of the Chern class map. 
\end{proof}

\begin{remark} \label{rmk: HC}
 If we could define an appropriate rational structure on the Hochschild homology of $\mathsf{vect}(\mathbb{A}^n,G,w)$, the arguments of Lemma~\ref{lem: cycles under grading change} would generalize to show the following statement. Assume that
 \begin{displaymath}
  \op{ch}: K_0(\mathbb{A}^{nr},G,w^{\boxplus r}) \to \op{HH}_0(\mathbb{A}^{nr},G,w^{\boxplus r})_{\mathbb{Q}},
 \end{displaymath}
 spans, over $\Q$, for all $r \geq 1$. Furthermore, assume that 
\begin{displaymath}
 \left( \op{Jac}(w_g)(-\kappa_e-\kappa_g-u\chi) \right)^G = 0 
\end{displaymath}
 for $t \not =0$ and $g \not = e$. Then, the image of the Chern character,
\begin{displaymath}
 \op{ch}: K_0(\mathbb{A}^{nr},G^{\times_{\mathbb{G}_m} r},w^{\boxplus r}) \to \op{HH}_0(\mathbb{A}^{nr},G^{\times_{\mathbb{G}_m} r},w^{\boxplus r})_{\Q},
\end{displaymath}
 also spans, over $\Q$. As such, this gives a bootstrap procedure for proving the Hodge conjecture for Morita products of factorization categories by proving it for simpler grading groups.   In fact, recent work of Blanc \cite{Antony} may yield the appropriate rational structure.
\end{remark}

\begin{corollary} \label{cor: HC for MF}
 Consider $\mathbb{A}^n_{\C}$ with the standard $\mathbb{G}_m$-action. Let $w$ be the Fermat cubic or quartic polynomial. Then, the image of 
 \begin{displaymath}
  \op{ch}: K_0(\mathbb{A}^{nr}_{\C},\mathbb{G}_m^{\times_{\mathbb{G}_m} r},w^{\boxplus r}) \to \op{HH}_0(\mathbb{A}^{nr}_{\C},\mathbb{G}_m^{\times_{\mathbb{G}_m} r},w^{\boxplus r})
 \end{displaymath}
 spans over $\C$.
\end{corollary}

\begin{proof}
 The result is a consequence of the splitting result for Hochschild homology of derived categories under semi-orthogonal decomposition, \cite[Theorem 7.3]{Kuz09a}.

 We do this by applying Lemma~\ref{lem: cycles under grading change} for $G = \mathbb{G}_m$. To do so, we must check that
\begin{displaymath}
 \op{ch}: K_0(\mathbb{A}^{nr}_{\C},\mathbb{G}_m^{\times_{\mathbb{G}_m} r},w^{\boxplus r}) \to \op{HH}_0(\mathbb{A}^{nr}_{\C},\mathbb{G}_m^{\times_{\mathbb{G}_m} r},w^{\boxplus r})
\end{displaymath}
 spans. Appealing to Theorem~\ref{theorem: Orlov}, we have a semi-orthogonal decomposition,
\begin{displaymath}
 \dbcoh{Z_{w^{\boxplus r}}} = \langle \mathcal O_{Z_{w^{\boxplus r}}}(-rn+d),\ldots,\mathcal O_{Z_{w^{\boxplus r}}}(-1), \dabs [\mathsf{fact}(\mathbb{A}^{nr}_{\C},\mathbb{G}_m,w^{\boxplus r})] \rangle,
\end{displaymath}
 where $Z_{w^{\boxplus r}}$ is the associated projective hypersurface. Kuznetsov's result then states we have a decomposition
 \begin{displaymath}
  \op{HH}_0(Z_{w^{\boxplus r}}) = \bigoplus_{i=-rn+d}^{-1} \C \cdot \op{ch}(\mathcal O_{Z_{w^{\boxplus r}}}(i)) \oplus \op{HH}_0(\mathbb{A}^{nr}_{\C},\mathbb{G}_m,w^{\boxplus r}).
 \end{displaymath}
 Ran \cite{Ran} proved that for $d=3,4$, the image of 
\begin{displaymath}
 \op{ch}: K_0(\dbcoh{Z_{w^{\boxplus r}}}) \to \op{HH}_0(Z_{w^{\boxplus r}})
\end{displaymath}
 spans $\op{HH}_0(Z_{w^{\boxplus r}})$ over $\C$. Using Proposition~\ref{prop: functorial push}, we deduce that the image of 
\begin{displaymath}
 \op{ch}: K_0(\mathbb{A}^{nr}_{\C},\mathbb{G}_m,w^{\boxplus r}) \to \op{HH}_0(\mathbb{A}^{nr}_{\C},\mathbb{G}_m,w^{\boxplus r})
\end{displaymath}
 spans over $\C$. The vanishing condition on the twisted sectors of the Hochschild homology follows as the fixed locus of any $g \notin \mathbb{G}_m$ is the origin of $\mathbb{A}^n$. This verifies the hypotheses of Lemma~\ref{lem: cycles under grading change} so we may conclude that the image of 
\begin{displaymath}
 \op{ch}: K_0(\mathbb{A}^{nr}_{\C},\mathbb{G}_m^{\times_{\mathbb{G}_m} r},w^{\boxplus r}) \to \op{HH}_0(\mathbb{A}^{nr}_{\C},\mathbb{G}_m^{\times_{\mathbb{G}_m} r},w^{\boxplus r})
\end{displaymath}
 spans over $\C$ for all $r \geq 1$.
\end{proof}

\begin{remark}
 One may rephrase the conclusion of Corollary~\ref{cor: HC for MF} as: the Hodge conjecture over $\Q$ is true for $\dabs \mathsf{fact}(\mathbb{A}^{nr}_{\C},\mathbb{G}_m^{\times_{\mathbb{G}_m} r},w^{\boxplus r})$.
\end{remark}

We can apply Lemma~\ref{lem: cycles under grading change} to reprove the Hodge conjecture for arbitrary self-products of a certain K3 surface closely related to the Fermat cubic fourfold. We first recall a result of Kuznetsov.

\begin{proposition} \label{prop: Kuz Fermat cubic}
 Let $X$ be the Fermat cubic fourfold in $\mathbb{P}^5$. There exists a unique K3 surface, $Y$, such that there is a semi-orthogonal decomposition,
\begin{displaymath}
 \dbcoh{X} = \langle \mathcal O_X(-3),\mathcal O_X(-2),\mathcal O_X(-1), \dbcoh{Y} \rangle.
\end{displaymath}
\end{proposition}

\begin{proof}
 The Fermat cubic fourfold is a Pfaffian cubic. Thus, the existence of $Y$ is consequence of Kuznetsov's results on Homological Projective Duality, see \cite{Kuz09b} for the statement. As mentioned previously, Ran proved that the image of 
\begin{displaymath}
 \op{ch}: K_0(X) \to \op{HH}_0(X)
\end{displaymath}
 spans over $\Q$, \cite{Ran}. Using the splitting of Hochschild homology and naturality of pushforwards in Hochschild homology, we deduce that the image of 
\begin{displaymath}
 \op{ch}: K_0(Y) \to \op{HH}_0(Y)
\end{displaymath}
 spans over $\Q$. In particular, since $Y$ is a K3 surface, it must have Picard rank $20$. If we have two such K3's surfaces, $Y_1$ and $Y_2$, with 
\begin{align*}
 \dbcoh{X} & = \langle  \mathcal O_X(-3),\mathcal O_X(-2),\mathcal O_X(-1), \dbcoh{Y_1} \rangle \\ & = \langle  \mathcal O_X(-3),\mathcal O_X(-2),\mathcal O_X(-1), \dbcoh{Y_2} \rangle.
\end{align*}
 Then, we must have an equivalence,
\begin{displaymath}
 \dbcoh{Y_1} \cong \dbcoh{Y_2}.
\end{displaymath}
 However, K3 surfaces with Picard rank more than $11$ do not have non-trivial Fourier-Mukai partners \cite[Corollary 2.7.1]{HLOY}.
\end{proof}

\begin{corollary} \label{cor: HC for K3}
 Let $Y$ be the K3 surface appearing in Proposition~\ref{prop: Kuz Fermat cubic}. The Hodge conjecture holds for all self-products, $Y^{\times r}$, $r \geq 1$. 
\end{corollary}

\begin{proof}
 By \cite[Theorem 7.1]{KuzBC}, the projection functor, $\dbcoh{X} \to \dbcoh{Y}$, lifts to a dg-functor between enhancements. It is then straightforward to check that Theorem~\ref{theorem: Orlov} induces a quasi-equivalence between $\mathsf{Inj}_{\op{coh}}(\mathbb{A}^6_{\C},\mathbb{G}_m,w)$ and $\mathsf{Inj}_{\op{coh}}(Y)$, where
 \begin{displaymath}
  w = x_1^3 + \cdots + x_6^3.
 \end{displaymath}
 Thus, we have quasi-equivalences,
\begin{displaymath}
 \mathsf{Inj}_{\op{coh}}(Y)^{\circledast r} \simeq \mathsf{Inj}_{\op{coh}}(\mathbb{A}^6_{\C},\mathbb{G}_m,w)^{\circledast r} \simeq \mathsf{Inj}_{\op{coh}}(\mathbb{A}^{6r}_{\C},\mathbb{G}_m^{\times_{\mathbb{G}_m} r},w^{\boxplus r}).
\end{displaymath}
 The final quasi-equivalence is Corollary~\ref{corollary: morita product of factorizations}. To\"en, \cite[Section 8]{Toe}, proves that there is a quasi-equivalence
 \begin{displaymath}
  \mathsf{Inj}_{\op{coh}}(Y)^{\circledast r} \simeq \mathsf{Inj}_{\op{coh}}(Y^{\times r}).
 \end{displaymath}
 We know the Hodge conjecture for $\mathsf{Inj}_{\op{coh}}(\mathbb{A}^{6r}_{\C},\mathbb{G}_m^{\times_{\mathbb{G}_m} r},w^{\boxplus r})$ is true by Corollary~\ref{cor: HC for MF}. 
\end{proof}

\begin{remark}
 In the initial version of this paper, we claimed that Corollary~\ref{cor: HC for K3} was a new case of the Hodge conjecture. After the first version was released, we were informed by P. Stellari that this case is already known, see \cite{RM08}. We happily thank Stellari for this communication.
\end{remark}

\begin{remark}
 Ran's work was extended by N. Aoki, \cite{Aok}. Aoki's work relies on that of T. Shioda, \cite{Shi}. Shioda proves that the Hodge conjecture holds for Fermat hypersurfaces as long as a certain arithmetic condition is satisfied. Aoki gives a reinterpretation of this arithmetic condition.
 One can directly construct factorizations whose Chern characters span the classes, $\mathfrak D_d^{m-1}$, studied by Shioda-Aoki. Using their arithmetic argument, one can then prove \textit{directly} that the Hodge conjecture holds for categories of $\mathbb{G}_m$-equivariant factorizations of Fermat potentials of degree $d$ in $\mathbb{A}_{\C}^{2m}$ when $d$ is prime, $d=4$, or every prime divisor of $d$ is greater than $m+1$. We omit the details, though they can be found in the initial arXived version of this article.
\end{remark}

\end{document}